\documentclass[11pt,draft]{amsart}
\usepackage{amssymb, amstext, amscd, amsmath, amssymb}
\usepackage{mathtools, xypic, paralist, color, dsfont, rotating}
\usepackage{verbatim}
\usepackage{enumerate,enumitem}
\usepackage{float}
\usepackage{setspace}

\reversemarginpar

\usepackage{tikz}

\floatstyle{boxed} 
\restylefloat{figure}

\numberwithin{equation}{section}
\usepackage{quoting}
\quotingsetup{vskip=.1in}
\quotingsetup{leftmargin=.17in}
\quotingsetup{rightmargin=.17in}
\let\OLDthebibliography\thebibliography
\renewcommand\thebibliography[1]{
  \OLDthebibliography{#1}
  \setlength{\parskip}{0pt}
  \setlength{\itemsep}{2pt plus 0.5ex}
}

\makeatletter
\def\@cite#1#2{{\m@th\upshape\bfseries%
[{#1\if@tempswa{\m@th\upshape\mdseries, #2}\fi}]}}
\makeatother
\theoremstyle{plain}
\newtheorem{theorem}{Theorem}[subsection]
\newtheorem{corollary}[theorem]{Corollary}
\newtheorem{proposition}[theorem]{Proposition}
\newtheorem{lemma}[theorem]{Lemma}
\theoremstyle{definition}
\newtheorem{definition}[theorem]{Definition}
\newtheorem{example}[theorem]{Example}

\newtheorem{remark}[theorem]{Remark}

\theoremstyle{remark}
\renewcommand{\qedsymbol}{{\vrule height5pt width5pt depth1pt}}
\mathtoolsset{centercolon}
\def\al{\alpha}
\def\be{\beta}
\def\Ga{\Gamma}
\def\ga{\gamma}

\def\de{\delta}

\def\la{\lambda}
\def\La{\Lambda}
\def\om{\omega}
\def\Om{\Omega}
\def\si{\sigma}

\newcommand{\eps}{\varepsilon}

\newcommand\vphi{\varphi}

\newcommand\id{\mathop{\rm id}}

\newcommand{\cl}[1]{\mathcal{#1}}
\newcommand{\bb}[1]{\mathbb{#1}}
\newcommand{\fr}[1]{\mathfrak{#1}}

\newcommand{\T}{{\mathcal{T}}}
\newcommand{\M}{{\mathcal{M}}}

\newcommand{\foral}{\text{ for all }}
\newcommand{\qand}{\quad\text{and}\quad}

\newcommand{\ca}{\mathrm{C}^*}

\newcommand{\cenv}{\mathrm{C}^*_{\textup{env}}}
\newcommand{\tenv}{\mathcal{T}_{\textup{env}}}

\newcommand{\ol}{\overline}
\newcommand{\wt}{\widetilde}
\newcommand{\wh}{\widehat}

\newcommand{\sca}[1]{\left\langle#1\right\rangle} %\sca{a,b} =<a,b>
\newcommand{\nor}[1]{\left\Vert #1\right\Vert} %\nor{x}=||x||
\newcommand{\un}[1]{{\underline{#1}}} %underline math symbols
\addtocontents{toc}{\protect\setcounter{tocdepth}{1}}

\begin{document}
%%%%%%%%%%%%%%%%%%%%%%%%%%%%%%%%

%%%%%%%%%%%%%%%%%%%%%%%%%%%%%%%%
\title[Symmetrisations of operator spaces]
{Symmetrisations of operator spaces}

\author[G.K. Eleftherakis]{George K. Eleftherakis}
\address{Department of Mathematics\\Faculty of Sciences\\University of Patras\\26504 Patras\\Greece}
\email{gelefth@math.upatras.gr}

\author[E.T.A. Kakariadis]{Evgenios T.A. Kakariadis}
\address{School of Mathematics, Statistics and Physics\\ Newcastle University\\ Newcastle upon Tyne\\ NE1 7RU\\ UK}
\email{evgenios.kakariadis@newcastle.ac.uk}

\author[I.G. Todorov]{Ivan G. Todorov}
\address{School of Mathematical Sciences\\ University of Delaware\\ 501 Ewing Hall\\ Newark\\ DE 19716\\ USA} 
\email{todorov@udel.edu}



%%%%%%%%%%%%%%%%%%%%%%%%%%%%%%%%
\begin{abstract}
Let $\mathcal{A}$ be a unital C*-algebra, $\mathcal{S}$ be an operator $\mathcal{A}$-system and $\mathcal{E}$ be an operator space that is a left operator $\mathcal{A}$-module.
We introduce the symmetrisation of the pair $(\cl E,\cl S)$ as the Hausdorff completion of the balanced tensor product $\mathcal{E}^* \odot^{\mathcal{A}} \mathcal{S} \odot^{\mathcal{A}} \mathcal{E}$ with respect to a seminorm arising from the family of completely contractive completely positive $\mathcal{A}$-balanced trilinear maps.
We show that the symmetrisation is a selfadjoint operator space in the sense of W. Werner, possessing a universal mapping property for pairs of representations of $\cl S$ and $\cl E$, compatible with the $\cl A$-module actions. 
We point out cases where the symmetrisation is an operator system, and where it does not admit an Archimedean order unit.

We study separately the case where $\mathcal{A} = \mathbb{C}$;
in this case, we show that the symmetrisation seminorm is a norm, which is equivalent to, yet different from, the Haagerup tensor norm.
When $\mathcal{S} = \mathbb{C}$ we show that the symmetrisation is compatible with taking operator space duals.
In the case where $\mathcal{E}$ is a function space and $\mathcal{S} = \mathbb{C}$, we characterise the positive matricial cones of the symmetrisation in terms of positive semi-definiteness of naturally associated matrix-valued functions. 

As an application, we provide a characterisation of Morita equivalence in the operator system category involving tensorial decomposition where the analytic structure is provided by the symmetrisation.
This establishes an operator system counterpart of the factorisation Morita Theorem in other categories.
\end{abstract}

\maketitle

\tableofcontents

%%%%%%%%%%%%%%%%%%%%%%%%%%%%
\section{Introduction}\label{s_intr}
%%%%%%%%%%%%%%%%%%%%%%%%%%%%

The close interplay between order and metric properties occupies a cornerstone place in non-commutative analysis. 
For example, the matricial norms of an abstract operator system can be completely recovered through its Archime\-dean matrix order structure, providing a concrete realisation as a unital selfadjoint subspace of the space $\cl B(H)$ of all bounded linear operators on some Hilbert space $H$, due to the Choi-Effros Theorem \cite{choi-effros}.
Further, the metric structure of an operator space can be completely recovered via the order properties of a canonical operator system (called the \emph{Pauslen system}) associated with the operator space. 

These schemes do not account for non-unital operator spaces with a selfadjoint structure.
Nevertheless, in his seminal work W. Werner \cite{werner} has provided necessary and sufficient conditions for an involutive operator space to have a concrete realisation as a selfadjoint subspace of $\cl B(H)$ for some Hilbert space $H$. 
Similar studies in this direction have been undertaken by A. Karn \cite{karn}, C.-K. Ng \cite{ng}, and exploited further by T. Russell \cite{russell}.
Recently there has been a revived interest in studying selfadjoint operator spaces and lifting results from the unital to the non-unital setting.
This has been taken forwards with much success by M. Kennedy, S.-J. Kim and N. Manor \cite{kkm} by considering the \lq\lq duality'' with the category of pointed compact noncommutative convex sets.
A further motivation for exploiting possibly non-unital selfadjoint operator spaces comes from the recent work of A. Connes and W. D. van Suijlekom \cite{cvs}, where they compare operator systems up to their (non-unital selfadjoint) stabilisations.

The stable isomorphism used in \cite{cvs} is an incarnation of Morita equivalence for operator systems, studied by the authors in \cite{ekt}.
This type of equivalence, called \emph{$\Delta$-equivalence}, was introduced by the first author for operator algebras \cite{elef1} and was further exploited in several categories of operator spaces \cite{elef2, ek, ep, ept}.
It appears that plenty of structural data remain invariant under $\Delta$-equivalence; in this vein, several analogues of the Morita Theorems, known for rings and C*-algebras, were obtained.
In this work, we wish to provide a Morita Theorem for operator systems in terms of a tensorial factorisation.
For nonselfadjoint operator algebras, such a realisation, using the balanced Haagerup tensor product, has appeared in the work of D. Blecher, P. Muhly and V. Paulsen \cite{bmp}.
However, in the category of operator systems, the factorisation should account for positivity, and the Haagerup tensor product machinery seems to be \emph{a priori} insufficient to recapture a matrix cone structure that is compatible with the metric properties.
Indeed, although, for an operator space $\cl E$, the Haagerup tensor product $\cl E^*\otimes_{\rm h} \cl E$ can be endowed with the canonical family of matrix cones
\[
C_n 
:= 
\left\{ \left(\sum_{p=1}^k x_{p,i}^* \otimes x_{p,j}\right)_{i,j=1}^n \ : \ (x_{i,j})_{i,j} \in M_{k,n}(\cl E); k \in \bb N \right\},
\]
the matrix norms associated with that family do not recover the metric structure of $\cl E^* \otimes_{\rm h} \cl E$, thus making impossible completely isometric identifications. 
It is thus natural to ask for a metric structure on the algebraic tensor product $\cl E^* \odot \cl E$ that is compatible with the cones $C_n$ and adheres to a universal property.

A further problem is that a completely isometric flip involution fails to be readily available for the Haagerup tensor product.
For an example, consider $\cl B(H) \otimes_{\rm h} \cl B(H)$ in the free product C*-algebra $\cl B(H) \ast_{\bb C} \cl B(H)$ (for 
the latter identification see, for example, \cite[Theorem 5.13]{pisier_intr}). 
Then the flip map $a^* \otimes b \mapsto b^* \otimes a$ does not coincide with the canonical involution on $\cl B(H) \ast_{\bb C} \cl B(H)$.
Indeed, let $\pi_1$ be the identity representation of $\cl B(H)$, and $\pi_2$ be the representation given by $\pi_2(x) = uxu^*$ for a non-trivial unitary $u \in \cl B(H)$. 
By the universal property of the free product, the pair $(\pi_1,\pi_2)$ gives rise to a $*$-representation $\pi_1 \ast \pi_2$ of $\cl B(H) \ast_{\bb C} \cl B(H)$ on $H$.
If the flip map coincided with the involution, then
\begin{align*}
b^* u a u^* 
& = (\pi_1 \ast \pi_2)(b^* \otimes a) 
= (\pi_1 \ast \pi_2)(a^* \otimes b)^* 
= u b^* u^* a,
\end{align*}
for any $a, b \in \cl B(H)$, which fails for $b = u$ and $a \notin \{u^*\}'$. 
The problem arises due to the freedom in the choice of $\pi_1$ and $\pi_2$, which does not compile with the flip map unless $\pi_1 = \pi_2$.

In order to exploit further the $\Delta$-equivalence for operator systems, it is therefore necessary to exhibit a compatible 
tensor product.
This necessitates a construction, in which the flip map is an isometric involution. 
This \emph{symmetrisation} $\cl E^* \otimes_{\rm s} \cl E$ of an operator space $\cl E$ is an operator space, metrically equivalent to, but geometrically different from, the Haagerup tensor product $\cl E^*\otimes_{\rm h} \cl E$. 
The key idea is to define $\cl E^* \otimes_{\rm s} \cl E$ as a universal object through the concrete realisations $[\phi(\cl E)^* \phi(\cl E)]$ over the completely contractive maps $\phi$ of $\cl E$, thus inducing a universal property on $\cl E^* \otimes_{\rm s} \cl E$.
Moving further, we wish to expand this construction to include the tensor product $\cl E^* \odot \cl S \odot \cl E$, for an operator system $\cl S$, as well as their balanced (bimodule) variants.
This generalisation is necessary for achieving (balanced) tensorial factorisations in operator system Morita theory (see for example \cite{bmp}), and reveals further connections with the Haagerup tensor product.

Our construction resembles a tensor product; however, the corresponding norm is not a cross norm, thus failing one of Groth\'endieck's axioms \cite{Gro53} -- hence we refer to it simply as \emph{the symmetrisation}. 
The need of introducing a new metric structure is dictated by the essence of producing a \emph{topological} version of Morita equivalence as an analogue to the \emph{algebraic} version in ring theory.
This obstacle has been already identified in the seminal work of Rieffel \cite{rieffel} and is manifested in what is known as Morita Theorem IV for stable isomorphisms.
Indeed, although Morita equivalence in the category of rings has an equivalent through infinite matrices due to Camillo's Theorem \cite{cam84}, such a result is not at hand in the topological context without additional separability conditions.
There are several metrics one can induce on an algebraic tensor product, and the problem of identifying the one that is compatible with Morita contexts lies at the epicenter of the theory.

Although our original point of motivation has been to establish a Morita factorisation for $\Delta$-equivalent systems, we provide applications of the symmetrisation beyond the setting of \cite{ekt}.
We show that elements of the matricial cones of the symmetrisation of a function space can be characterised through the positive semi-definiteness of canonically associated two-variable matrix-valued functions, rendering, in this special case, canonical the approach we have taken here.
In addition, we show that the operator space dual compares with the symmetrisation of the dual of an operator space in a canonical way; the result is an order counterpart of the celebrated self-duality of the Haagerup tensor product \cite{bs}. 
Our work exhibits, as a byproduct, natural selfadjoint operator spaces that are not operator systems, that is, operator spaces with an involution structure that do not possess completely isometric complete order embeddings into the space of bounded operators on a Hilbert space whose image contains the identity operator, thus reinforcing the argument for their study, complementing the recent development \cite{kkm}.

\medskip

We next state our main results. 
Let $\cl A$ be a unital C*-algebra, $\cl S$ be an operator $\cl A$-system and $\cl E$ be a left operator $\cl A$-module with a unital module action.
Note that, if $\cl A \subseteq \cl S$, then $1_{\cl A} = 1_{\cl S}$.
We will call a pair of linear maps $(\phi, \psi)$ \emph{$\cl A$-admissible} if $\phi \colon \cl E \to \cl B(H,K)$ is a completely contractive map, $\psi \colon \cl S \to \cl B(K)$ is a unital completely positive map (where $H$ and $K$ are Hilbert spaces), and
\[
\psi(s \cdot a) \phi(x) = \psi(s) \phi(a \cdot x), \ \ s \in \cl S, a \in \cl A, x \in \cl E.
\]
For an element $u\in M_n(\cl E^* \odot^{\cl A} \cl S \odot^{\cl A} \cl E)$, where $\odot^{\cl A}$ designates the balanced algebraic tensor product, we define
\[
\left\|u\right\|^{(n)}_{\rm s \cl A}
:= 
\sup\{\|(\phi^*\cdot\psi\cdot \phi)^{(n)}(u)\| \ : \ (\phi,\psi) \mbox{ an $\cl A$-admissible pair}\}
\]
(see (\ref{eq_dot}) for the definition of the map $\phi^*\cdot\psi\cdot \phi$). 
In Lemma \ref{l_balnorm} we show that $\|\cdot\|_{\rm s \cl A}^{(n)}$ is a seminorm dominated by the $\cl A$-balanced Haagerup norm $\|\cdot\|_{{\rm h} \cl A}$, that is, we have 
\[
\|u\|_{\rm s \cl A}^{(n)} \leq \|u\|_{{\rm h} \cl A}^{(n)}, \ \ \ \ u\in M_n(\cl E^* \odot^{\cl A} \cl S \odot^{\cl A} \cl E).
\]
It follows that the family of seminorms $\{\|\cdot\|_{\rm s \cl A}^{(n)}\}_{n \in \bb N}$ defines an operator space structure on the corresponding quotient of $\cl E^* \odot^{\cl A} \cl S \odot^{\cl A} \cl E$, and we write $\cl E^* \otimes^{\cl A}_{\rm s} \cl S \otimes^{\cl A}_{\rm s} \cl E$ for the Hausdorff completion, which we call the \emph{balanced symmetrisation of $\cl E$ by $\cl S$ over $\cl A$}.
We will write:
\begin{enumerate}
\item $\cl E^* \otimes_{\rm s} \cl S \otimes_{\rm s} \cl E$ for $\cl E^* \otimes^{\bb C}_{\rm s} \cl S \otimes^{\bb C}_{\rm s} \cl E$, 
\item 
$\cl E^* \otimes_{\rm s}^{\cl A} \cl E$ for $\cl E^* \otimes^{\cl A}_{\rm s} \cl A \otimes^{\cl A}_{\rm s} \cl E$, and 
\item 
$\cl E^* \otimes_{\rm s} \cl E$ for $\cl E^* \otimes_{\rm s}^{\bb C} \bb C \otimes_{\rm s}^{\bb C} \cl E$.
\end{enumerate}

Since, in addition to the symmetrisation, we will need (the universal property of) the Haagerup tensor products, we will make use of the following notation.
If $\theta \colon \cl E^* \times \cl S \times \cl E \to \cl B(H)$ is an $\cl A$-balanced trilinear map we will write $\wt{\theta}_{\cl A}$ for the induced linearisation on the balanced tensor product, and we will write
\begin{align*}
\wt{\theta}_{\rm h \cl A} \colon \cl E^* \otimes^{\cl A}_{\rm h} \cl S \otimes^{\cl A}_{\rm h} \cl E \to \cl B(H)
\qand
\wt{\theta}_{\rm s \cl A} \colon \cl E^* \otimes^{\cl A}_{\rm s} \cl S \otimes^{\cl A}_{\rm s} \cl E \to \cl B(H)
\end{align*}
for the induced extensions of $\wt{\theta}_{\cl A}$ when they exist.
We will simply write $\wt{\theta}$, $\wt{\theta}_{\rm h}$ and $\wt{\theta}_{\rm s}$ when $\cl A = \bb{C}$.

The operator space $\cl E^* \odot^{\cl A} \cl S \odot^{\cl A} \cl E$ admits an isometric involution given by the flip map $x^* \odot^{\cl A} y \mapsto y^* \odot^{\cl A} x$ and a matrix ordered structure given by the cones
\begin{align*}
C_n^{\cl A} 
& := 
\{u\in M_n(\cl E^*\otimes_{\rm s}^{\cl A} \cl S \otimes_{\rm s}^{\cl A} \cl E)_h \ : \ 
(\phi^*\cdot\psi \cdot \phi)^{(n)}(u)\in M_n(\cl B(H))^+, \\
& \hspace{3.5cm} \text{ if } (\phi,\psi) \text{ is an $\cl A$-admissible pair for some } (H,K)\}.
\end{align*}
In Corollary \ref{c_consym} we show that
\[
C_n^{\cl A} = \{x^* \odot^{\cl A} s \odot^{\cl A} x \ : \ x \in M_{k,n}(\cl E), s \in M_k(\cl S)^+, k \in \bb N\}^{-\|\cdot\|_{\rm s \cl A}^{(n)}};
\]
this can be viewed as a synthesis type result: the \lq\lq externally defined'' cones $C_n^{\cl A}$ admit an \lq\lq inner'' approximation characterisation. 

In our study, we first consider the case where $\cl A = \bb{C}$ as it is of particular importance.
In Lemma \ref{l_norm} we show that
\[
\frac{1}{4} \|u\|_{{\rm h}}^{(n)} \leq \|u\|_{\rm s}^{(n)} \leq \|u\|_{{\rm h}}^{(n)}, \ \ \ \ u\in M_n(\cl E^* \odot \cl S \odot \cl E),
\]
that is, the unbalanced symmetrisation norm is equivalent to the Haagerup tensor norm on the tensor product $\cl E^* \odot \cl S \odot \cl E$.
Then in Theorem \ref{th_estareos} we show that $\cl E^* \otimes_{\rm s} \cl S \otimes_{\rm s} \cl E$ is a selfadjoint operator space in the sense of W. Werner \cite{werner}.
The examination of the \lq\lq unbalanced'' version first allows us to highlight the key aspects of the construction, while in the \lq\lq balanced'' version we emphasise the necessary modifications.
The connection is further made clear in Corollary \ref{c_iden}, where we show that
\[
\cl E^* \otimes_{\rm s}^{\cl A} \cl S \otimes_{\rm s}^{\cl A} \cl E
\simeq
\cl E^* \otimes_{\rm s} \cl S \otimes_{\rm s} \cl E / \cl J_{\cl E,\cl S}^{\otimes_{\rm s}^{\cl A}},
\]
up to a canonical completely isometric complete order isomorphism, where $\cl J_{\cl E,\cl S}^{\otimes_{\rm s}^{\cl A}}$ is a naturally defined kernel that annihilates the $\cl A$-trilinear maps (see also Lemma \ref{l_kerne}).
In more detail, we have
\[
\cl J_{\cl E,\cl S}^{\otimes_{\rm s}^{\cl A}} :=  {\rm span} \{ u \in (\cl E^* \otimes_{\rm s} \cl S \otimes_{\rm s} \cl E)_{h} \ : \ u \in D_1^{\cl A}\cap (-D_1^{\cl A}) \},
\]
where
\begin{align*}
D_n^{\cl A} := \{ u\in M_n(\cl E^*\otimes_{\rm s}\cl S \otimes_{\rm s} \cl E)_h \ : \ 
& (\phi^*\cdot\psi \cdot \phi)^{(n)}(u)\in M_n(\cl B(H))^+, \\
& \; \text{ for every $\cl A$-admissible pair } (\phi,\psi) \},
\end{align*}
for $n\in \bb{N}$.
We note that $\cl J_{\cl E,\cl S}^{\otimes_{\rm s}^{\cl A}}$ contains the subspace
\begin{align*}
\cl J_{\cl E,\cl S}^{\cl A}
& := 
\ol{\rm span}\{ y^*\otimes (b \cdot s \cdot a)\otimes x -  (y^*\cdot b) \otimes  s \otimes (a \cdot x)  \ : \  \label{eq_J} \\
& \hspace{7cm} a,b\in \cl A, x,y\in \cl E\}
\end{align*}
related to the balanced relations as well.
In particular, $\cl J_{\cl E,\cl S}^{\otimes_{\rm s}^{\cl A}}$ 
lies in the kernel of every map of the form
$\wt{\theta}_{\rm s}$, arising from an $\cl A$-balanced trilinear map. 
In the following theorem we collect the main structural properties of the balanced symmetrisation.

\medskip

\noindent
{\bf Theorem A.} \emph{(Theorem \ref{t_symunibal})
Let $\cl A$ be a unital C*-algebra, $\cl S$ be an operator $\cl A$-system and $\cl E$ be an operator $\cl A$-space.
Then the following hold:
\begin{enumerate}
\item If $\phi \colon \cl E\to \cl B(H,K)$ is a completely contractive map and $\psi \colon \cl S\to \cl B(K)$ is a completely positive map with the property that $(\phi, \psi)$ is an $\cl A$-admissible pair, then the map
\[
\phi^*\cdot\psi\cdot \phi \colon \cl E^* \otimes_{\rm s}^{\cl A} \cl S \otimes_{\rm s}^{\cl A} \cl E \to \cl B(H)
\]
is a completely contractive completely positive map. 
\item If $\wt{\theta}_{\rm s \cl A} \colon \cl E^*\otimes_{\rm s}^{\cl A} \cl S \otimes_{\rm s}^{\cl A} \cl E\to \cl B(H)$ is a completely contractive completely positive map, then there exist a completely contractive map $\phi \colon \cl E \to \cl B(H,K)$ and a unital completely positive map $\psi \colon \cl S\to \cl B(K)$ such that the pair $(\phi, \psi)$ is $\cl A$-admissible and 
\[
\wt{\theta}_{\rm s \cl A} = \phi^* \cdot \psi \cdot \phi.
\]
\item If $\cl A \subseteq \cl S$ and $\wt{\theta}_{\rm s \cl A} \colon \cl E^* \otimes_{\rm s}^{\cl A} \cl S \otimes_{\rm s}^{\cl A} \cl E\to \cl B(H)$ is a completely contractive completely positive map, then we can choose $\psi$ in (ii) so that $\pi := \psi|_{\cl A}$ is a $*$-representation, $(\phi, \pi)$ is an $\cl A$-representa\-tion of $\cl E$ and $(\psi, \pi)$ is an $\cl A$-represen\-ta\-tion of $\cl S$.
\item If $\cl A \subseteq \cl S$, $\wt{\theta}_{\rm s \cl A} \colon \cl E^* \otimes_{\rm s}^{\cl A} \cl S \otimes_{\rm s}^{\cl A} \cl E \to \cl B(H)$ is a completely isometric completely positive map, and $(\phi, \psi)$ is as in (ii), then $\phi$ is a complete isometry. 
\end{enumerate}
}

\medskip

The symmetrisation $\cl E^* \odot^{\cl A} \cl S \odot^{\cl A} \cl E$ may inherit properties from $\cl E$ and $\cl S$.
In Theorem \ref{t_opBsy} we show that if $\cl E$ is in addition a right operator $\cl B$-space for a unital C*-algebra $\cl B$, then $\cl E^* \otimes_{\rm s} \cl S \otimes_{\rm s} \cl E$ becomes a selfadjoint operator $\cl B$-space in a canonical fashion.
The bimodule structure descends to the quotient $\cl E^* \otimes_{\rm s}^{\cl A} \cl S \otimes_{\rm s}^{\cl A} \cl E$ as long as $\cl J_{\cl E,\cl S}^{\otimes_{\rm s}^{\cl A}}$ is a $\cl B$-bimodule.
This happens for example when $\cl E$ is a TRO $\cl M$, and $\cl A = [\cl M \cl M^*]$ and $\cl B = [\cl M^* \cl M]$.
In this special case the balanced symmetrisation coincides with the balanced Haagerup tensor product, and thus $\cl J_{\cl E,\cl S}^{\otimes_{\rm s}^{\cl A}}$ coincides with the subspace $\cl J_{\cl E,\cl S}^{\cl A}$ of the balanced relations.

We note here that the notion of a selfadjoint operator space bimodule is new and we provide the pertinent details in Subsection \ref{ss_soap}.
For this development we follow the viewpoint of W. Werner \cite{werner} by considering a unital C*-algebra in the place of $\bb C$, which results in the appropriate notion of the modular partial unitisation and the equivalent condition for a concrete representation in some $\cl B(H)$.

Our main tool for Theorem A is a factorisation result for a class of trilinear maps.
It is akin to the factorisation results proved by E. Christensen and A. Sinclair \cite{christensen-sinclair}, V. Paulsen and R. R. Smith \cite{paulsen-smith}, and A. Sinclair and R. R. Smith \cite{ss} 
(see for example \cite[Theorem 1.5.7, Subsecton 1.5.8]{blm}). 

\medskip

\noindent
{\bf Theorem B.} 
\emph{(Lemma \ref{l_ssgen})
Let $\cl A$ be a unital C*-algebra, $\cl E$ be a left operator $\cl A$-space, $\cl S$ be an operator $\cl A$-system and $\theta \colon \cl E^*\times\cl S\times\cl E\to \cl B(H)$ be a completely bounded completely positive $\cl A$-balanced trilinear map. 
Then the following hold:
\begin{enumerate}
\item There exist a completely bounded map $\phi \colon \cl E\to \cl B(H,K)$ and a unital completely positive map $\psi \colon \cl S\to \cl B(K)$ such that $\wt{\theta} = \phi^*\cdot \psi\cdot\phi$ and
\[
\psi(s)\phi(a\cdot x) = \psi(s\cdot a)\phi(x), \ \ \ x\in \cl E, s\in \cl S, a\in \cl A.
\]
Moreover, $\phi$ and $\psi$ can be chosen so that $\|\theta\|_{\rm cb} = \|\phi\|_{\rm cb}^2$.
\item If, in addition, $\cl A \subseteq \cl S$, then we can choose $\psi$ so that $\pi := \psi|_{\cl A}$ is a $*$-representation, $(\phi,\pi)$ is an $\cl A$-representation of $\cl E$, and $(\psi,\pi)$ is an $\cl A$-repre\-sen\-tation of $\cl S$.
\item If $\cl S$ is a C*-algebra, then the map $\psi$ in item (i) can be chosen to be a $*$-representation. 
\end{enumerate}
}

\medskip

Combining Theorem A with Theorem B endows the balanced symmetrisation with the following universal property.

\medskip

\noindent
{\bf Theorem C.} \emph{(Theorem \ref{t_univpro})
Let $\cl A$ be a unital C*-algebra, $\cl S$ be an operator $\cl A$-system and $\cl E$ be an operator $\cl A$-space.
Then the balanced symmetrisation space $\cl E^* \otimes_{\rm s}^{\cl A} \cl S \otimes_{\rm s}^{\cl A} \cl E$ has the following universal property:
\begin{enumerate}
\item the canonical trilinear map
\[
\iota \colon \cl E^* \times \cl S \times \cl E \to \cl E^* \otimes_{\rm s}^{\cl A} \cl S \otimes_{\rm s}^{\cl A} \cl E; (y^*, s, x) \mapsto y^* \odot^{\cl A} s \odot^{\cl A} x
\]
is completely contractive and completely positive with dense range, and
\item if $\theta \colon \cl E^* \times \cl S \times \cl E \to \cl B(H)$ is a completely contractive completely positive $\cl A$-balanced trilinear map, then there exists a unique completely contractive completely positive map $\wt{\theta}_{\rm s \cl A} \colon \cl E^* \otimes_{\rm s}^{\cl A} \cl S \otimes_{\rm s}^{\cl A} \cl E \to \cl B(H)$ that makes the following diagram
\[
\xymatrix{
\cl E^* \times \cl S \times \cl E \ar[rr]^{\theta} \ar[d]^{\iota} & & \cl B(H) \\
\cl E^* \otimes_{\rm s}^{\cl A} \cl S \otimes_{\rm s}^{\cl A} \cl E \ar[urr]^{\wt{\theta}_{\rm s \cl A}} & &
}
\]
commutative.
\end{enumerate}
}

\medskip

We note that the operator space $\cl E^* \odot^{\cl A} \cl S \odot^{\cl A} \cl E$ is not necessarily unital; in fact, the requirement that $\cl E^* \odot^{\cl A} \cl S \odot^{\cl A} \cl E$ is an operator system is a strong condition as there are several metrics to keep track of.
Such a context would require: (a) the existence of an Archimedean matrix order unit, and (b) the symmetrisation norms to coincide with the norms induced by that Archimedean matrix order unit, at every level.
This fails to happen in general (even when $\cl E$ is finite dimensional).
In Theorem \ref{t_notos} we show exactly when this is the case for $\cl A = \cl S = \bb C$, while we provide several cases where it does not hold:
\begin{enumerate}
\item If $\cl E$ is non-separable, then $\cl E^* \otimes_{\rm s} \cl E$ does not admit an order unit (Theorem \ref{t_nonsepnonun}).
\item If $\cl E \subseteq \cl B(H)$ with $\dim(\cl E^* \cl E) < \infty$ and $\dim(\cl E^* \cl E) < \dim(\cl E^* \otimes_{\rm s} \cl E)$, then $\cl E^* \otimes_{\rm s} \cl E$ is not an operator system (Corollary \ref{c_notos}).
\item In particular, if $\cl E = M_{I,n}$ for finite $n$ and any set $I$ with $n < |I|$, then $\cl E^* \otimes_{\rm s} \cl E$ is not an operator system (Remark \ref{r_appnotos}).
\end{enumerate}
Denoting by $R_n$ (resp. $C_n)$ the row (resp. column) $n$-dimensional operator space, item (iii) shows in particular that $R_n \otimes_{\rm s} C_n$ is not an operator system for $n > 1$.
A second example arises by considering the space of diagonal $n \times n$ matrices $\cl E = D_n$.
At the other extreme, we obtain that $C_n \otimes_{\rm s} R_n$ coincides with $M_n$ and is thus an operator system.
Note here that all three spaces $R_n \otimes_{\rm s} C_n$, $D_n \otimes_{\rm s} D_n$, and $C_n \otimes_{\rm s} R_n$ are topologically isomorphic via completely positive maps with completely positive inverses, and thus they all admit an Archimedean matrix order unit.
The subtlety resides in that the induced norms in the first two cases do not coincide with the corresponding symmetrisation norms.
The identification 
\[
C_n \otimes_{\rm s} R_n \simeq M_n = [C_n R_n]
\]
is a special case of the following more general result for $\cl M = R_n$.
Below we denote by $\cl A_{\cl S}$ the multiplier C*-algebra of the operator system $\cl S$. 

\medskip

\noindent
{\bf Theorem D.} \emph{(Theorem \ref{t_trosym})
Let $\cl S$ be an operator system.
Suppose that $\psi \colon \cl S \to \cl B(K)$ is a unital complete order embedding and $\cl M \subseteq \cl B(H,K)$ is a closed TRO such that $\cl M \cl M^* \psi(\cl S) \subseteq \psi(\cl S)$.
Then $[\cl M \cl M^*] \hookrightarrow \cl A_{\cl S}$ and the canonical completely contractive completely positive map
\[
\cl M^* \otimes_{\rm s}^{[\cl M \cl M^*]} \cl S \otimes_{\rm s}^{[\cl M \cl M^*]} \cl M 
\to 
[\cl M^* \psi(\cl S) \cl M]
\]
is a complete order isomorphism.
}

\medskip

In particular, under the assumptions of Theorem D, we have that
\[
\cl M^* \otimes_{\rm s}^{[\cl M \cl M^*]} \cl S \otimes_{\rm s}^{[\cl M \cl M^*]} \cl M
\simeq
\cl M^* \otimes_{\rm h}^{[\cl M \cl M^*]} \cl S \otimes_{\rm h}^{[\cl M \cl M^*]} \cl M.
\]
However, this does not hold in general.
Although the unbalanced symmetrisation norm is equivalent to the unbalanced Haagerup tensor norm, they fail to coincide, and we provide a finite dimensional example in this respect ($\cl E = M_2$), see Proposition \ref{p_pcn} and Example \ref{e_diffhs}. 
As a further example, we note that $R_n \otimes_{\rm s} C_n$ is isomorphic, but not isometrically isomorphic, to the operator space $R_n \otimes_{\rm h} C_n$ of the $n \times n$ trace class matrices.

A further property of the symmetrisation is that it is injective as long as we use the same C*-algebra for the module actions, see Proposition \ref{p_tensmaps}.
If we allow for different modular actions, then this fails even in finite dimensional cases.
Denoting by $\cl A_l(\cl E)$ the left multiplier C*-algebra of $\cl E$, in Example \ref{e_ccnotcis} we show that the canonical completely contractive completely positive map
\[
\cl E^* \otimes_{\rm s} \cl E \to \cl E^* \otimes_{\rm s}^{\cl A_{\ell}(\cl E)} \cl E
\]
may not be even isometric.
Moreover, in Example \ref{e_env} we show that the canonical map
\[
q \colon \cl E^* \otimes_{\rm s}^{\cl A_{\ell}(\cl E)} \cl E \to \tenv(\cl E)^* 
\otimes_{\rm s}^{\cl A_{\ell}(\tenv(\cl E))} \tenv(\cl E),
\]
where $\tenv(\cl E)$ is the ternary envelope (also known as the Shilov boundary) \cite{blecher-shilov},
may not be completely isometric; a counterexample is constructed for the $2$-level rigid operator systems of \cite{cvs}.
The identification of the C*-envelope of the symmetrisation remains an open question.

One of the main goals of this work is to provide a Morita Theorem II for operator systems.
In \cite{ekt} we have shown that the notion of $\Delta$-equivalence, emerging from previous works of the authors \cite{elef1,elef2,ek,ep,ept}, has an appropriate counterpart in the category of operator systems.
One of the key results is its identification with stable isomorphism, and it has an equivalent reformulation in terms of compatible trilinear maps, following closely the analogy with ring theory.
Here we provide its incarnation in terms of tensorial factorisation.

\medskip

\noindent
{\bf Theorem E.} \emph{(Theorem \ref{th_facDelta})
Let $\cl S$ and $\cl T$ be operator systems. 
The following are equivalent:
\begin{enumerate}
\item $\cl S$ and $\cl T$ are $\Delta$-equivalent;
\item there exist a closed $\cl A_{\cl T}$-balanced $\cl T$-semi-unital TRO $\cl M$, such that 
the TRO $\cl M^*$ is $\cl A_{\cl S}$-balanced $\cl S$-semi-unital, 
\[
\cl A_{\cl T} \stackrel{\pi_{\cl T}}{\simeq} [\cl M \cl M^*]
\qand
\cl A_{\cl S} \stackrel{\pi_{\cl S}}{\simeq} [\cl M^* \cl M],
\]
and a compatible pair $(\al, \be)$ of unital complete order isomorphisms
\[
\alpha \colon \cl M^*\otimes_{\rm s}^{\cl A_{\cl T}}\cl T \otimes_{\rm s}^{\cl A_{\cl T}}\cl M \to \cl S
\text{ and }
\beta \colon \cl M\otimes_{\rm s}^{\cl A_{\cl S}}\cl S \otimes_{\rm s}^{\cl A_{\cl S}}\cl M^* \to \cl T,
\]
that are bimodule maps over $\cl A_{\cl S}$ and $\cl A_{\cl T}$ respectively, such that
\begin{align*}
\al(x_1^* \otimes 1_{\cl T} \otimes x_2) & = \pi_{\cl S}^{-1}(x_1^* x_2) \cdot 1_{\cl S} 
\\
\be(x_1 \otimes 1_{\cl S} \otimes x_2^*) &= \pi_{\cl T}^{-1}(x_1 x_2^*) \cdot 1_{\cl T},
\end{align*}
for all $x_1, x_2 \in \cl M$;
\item there exist a non-degenerate 
$\cl A_{\cl T}$-balanced $\cl T$-semi-unital operator $\cl A_{\cl T}$-$\cl A_{\cl S}$-module $\cl E$, such that $\cl E^*$ is $\cl A_{\cl S}$-balanced $\cl S$-semi-unital, and a compatible pair $(\al, \be)$ of unital completely positive maps
\[
\alpha \colon \cl E^*\otimes_{\rm s}^{\cl A_{\cl T}}\cl T \otimes_{\rm s}^{\cl A_{\cl T}}\cl E \to \cl S
\text{ and }
\beta \colon \cl E\otimes_{\rm s}^{\cl A_{\cl S}}\cl S \otimes_{\rm s}^{\cl A_{\cl S}}\cl E^* \to \cl T.
\]
\end{enumerate}
If any (and thus all) of the items hold, then the maps $\al$ and $\be$ in items (ii) and (iii) are surjective.
}

\medskip

The compatibility clause for the pair $(\al,\be)$ is a canonical intertwining condition for tensor products and maps, modeling associativity of the multiplication when the objects are concretely represented, and was already exhibited in the compatibility of canonical maps in \cite{ekt}.
It is worth pointing out that condition (iii) does not, however, have a natural counterpart in nonselfadjoint categories. 
It is an abstract incarnation of homomorphic equivalence of operator systems which, motivated by quantum graph homomorphisms in quantum information theory \cite{stahlke, tt}, was introduced and shown to be identical to $\Delta$-equivalence in \cite{ekt}. 

We describe some further applications of the symmetrisation.
One class of particular interest is that of function spaces.
If $\Om$ is a compact Hausdorff space and $C(\Om)$ is the C*-algebra of all continuous functions on $\Om$ then, by the Grothendieck's inequality and the equivalence of the Haagerup and the symmetrisation norms, the spaces $C(\Om) \otimes_{\rm s} C(\Om)$, $C(\Om) \otimes_{\rm h} C(\Om)$ and the projective tensor product $C(\Om) \hat{\otimes} C(\Om)$ (known as the Varopoulos algebra on $\Om \times \Om$) are all boundedly isomorphic.
Their elements can be canonically identified with continuous functions on $\Om\times\Om$ by seeing the elementary tensors $f \otimes g$ as a two variable continuous function; up to such an identification, the matricial cones have a canonical interpretation via positive semi-definiteness of two variable functions.

\medskip

\noindent
{\bf Theorem F.} \emph{(Theorem \ref{th_fspa})
Let $u\in M_n(C(\Om)^* \otimes_{\rm s} C(\Om))$, and 
continue to write $u \colon \Omega \times \Omega \to M_n$
for the canonically associated function. 
The following are equivalent:
\begin{enumerate}
\item $u\in M_n(C(\Om)^* \otimes_{\rm s} C(\Om))^+$;
\item $u \colon \Omega \times \Omega \to M_n$ is positive semi-definite. 
\end{enumerate}
}

\medskip

We note that the function with constant value one on the diagonal is a natural candidate for an Archimedean matrix order unit for $C(\Om)^* \otimes_{\rm s} C(\Om)$.
However, for the constant function to be even continuous, the set $\Om$ has to be finite, in which case $C(\Om)$ is the space of diagonal matrices.
But as we have noted, even under the latter conditions, the canonical Archimedean matrix order unit does not induce the symmetrisation norm.

Finally, we describe our main result regarding the dual of the symmetrisation. 
We write $\cl E^{\rm d}$ for the operator space dual of the operator space $\cl E$ (as the symbol $\cl E^*$ is reserved for the adjoint operator space of $\cl E$). 

\medskip

\noindent
{\bf Theorem G.} \emph{(Theorem \ref{th_findimsd})
Let $\cl E$ be an operator space. 
Then the mapping 
\[
\iota \colon (\cl E^{\rm d})^* \otimes_{\rm s} \cl E^{\rm d} 
\to \left(\cl E^*\otimes_{\rm s} \cl E\right)^{\rm d}, 
\text{ given by }
\iota(\Psi^*\otimes \Phi)(y^*\otimes x) = \Psi^*(y^*) \Phi(x),
\]
is a complete order isomorphism onto its range. 
If $\cl E$ is finite dimensional, then this map is surjective.}

\medskip

Theorem G is a symmetrised version of the celebrated Blecher-Smith Theorem on the self-duality of the Haagerup tensor product \cite{bs}. 
Similar duality results have been proven for the Haagerup tensor product over Hilbert C*-modules by A. Stern and W. D. van Suijlekom \cite{ss21}, and T. Crisp and M. Rosbotham \cite{cr24}.
We note, however, that self-duality in our setting only holds with respect to the matricial order; in fact, we show in Remark \ref{r_notisom} that the map $\iota \colon \cl E^{{\rm d} *}\otimes_{\rm s} \cl E^{\rm d} \to \left(\cl E^*\otimes_{\rm s} \cl E\right)^{\rm d}$ may fail to be a complete isometry even for finite dimensional $\cl E$.

\medskip

We finish the current section with describing the organisation of the manuscript. 
In Section \ref{s_prel} we collect preliminaries from operator spaces theory, fix notation, and include the necessary details about (not necessarily unital) selfadjoint operator spaces from \cite{kkm, werner} for future reference.  
Section \ref{s_quatbim} contains the material on quotient theory for operator systems and selfadjoint operator spaces, equipped with a bimodule action over a unital C*-algebra, which will be needed subsequently. 
In passing, we provide the unitisation of a bimodule operator space over a C*-algebra akin to the one introduced by W. Werner \cite{werner}.

In Section \ref{s_univco} we provide the factorisation result for completely bounded completely positive balanced trilinear maps, and define the symmetrisation of an operator space by an operator system.
We establish the main properties of the symmetrisation and make comparisons with the Haagerup tensor product.
We also provide a characterisation for a symmetrisation to be an operator system by using the notion of semi-units, introduced therein.

In Section \ref{s_fundual} we study the symmetrisation of a function space, obtaining a canonical description of positivity in terms of positive semi-definiteness. 
In Section \ref{s_selfduality} we explore the connection of the dual of the symmetrisation with the symmetrisation of the duals.
We further show that there is an affine bijection from 
the quasi-state space $\rm{CCP}(\cl E^* \otimes_{\rm s} \cl E, \bb C)$ to the completely contractive maps from $\mathcal{E}$ to $\mathcal{E}^{* \rm{d}}$ that factor symmetrically through a Hilbert space.

In Section \ref{s_balsym} we present the balanced symmetrisation where the operator space and the operator systems are equipped with compatible module actions by a unital C*-algebra. 
We thus define the balanced symmetrisation; this is kept separate from the symmetrisation over the scalars in order to clarify similarities, as well as differences.
In Section \ref{s_morita} we provide the factorisation of $\Delta$-equivalent operator systems which complements our previous study \cite{ekt}.

\medskip

\noindent {\bf Acknowledgements.}
The authors acknowledge the support from the European Union - Next Generation EU (Implementation Body: HFRI. Project name: Noncommutative Analysis: Operator Systems and Nonlocality. HFRI Project Number: 015825), by the Heilbronn Institute for Mathematical Research (HIMR) and the UKRI EPSRC Additional Funding Programme for Mathematical Sciences, and by the London Mathematical Society (grant No.\ 42113).
Ivan Todorov was supported by NSF grants CCF-2115071 and DMS-2154459.

\smallskip

\noindent {\bf Data availability statement.}
For the purposes of publication of this article, we note that data sharing is not applicable as no datasets were generated or analysed during the underlying research.

\smallskip

\noindent {\bf Conflict of interest statement.}
On behalf of all authors, the corresponding author states that there is no conflict of interest.

\smallskip

\noindent {\bf Open access statement.}
For the purpose of open access, the second author has applied a Creative Commons Attribution (CC BY) license to any Author Accepted Manuscript version arising.

%%%%%%%%%%%%%%%%%%%%%%%%%%%%
\section{Preliminaries}\label{s_prel}
%%%%%%%%%%%%%%%%%%%%%%%%%%%%

In this section we recall background, set notation, and establish some preliminary results that will be used in the sequel. 

%%%%%%%%%%%%%%%%%%%%%%%%%%%%
\subsection{Operator spaces and systems}\label{ss_opspsy}
%%%%%%%%%%%%%%%%%%%%%%%%%%%%

As usual, if $V$ is a vector space, we write $M_{m,n}(V)$ for the linear space of all $m$ by $n$ matrices with entries in $V$,  set $M_n(V) = M_{n,n}(V)$, $M_{m,n} = M_{m,n}(\bb{C})$ and $M_n = M_{n,n}$.
If $V$ and $W$ are vector spaces and $\phi \colon V\to W$ is a linear map, we write 
\[
\phi^{(m,n)} \colon M_{m,n}(V)\to M_{m,n}(W)
\]
for the map given by 
\[
\phi^{(m,n)}((v_{i,j})_{i,j}) = (\phi(v_{i,j}))_{i,j},
\]
and set $\phi^{(n)} = \phi^{(n,n)}$.
The symbols $H$ and $K$ will be reserved to denote Hilbert spaces, and $\cl B(H,K)$ will denote the space of all bounded linear operators from $H$ into $K$; as usual, we set $\cl B(H) = \cl B(H,H)$.
We write $I_H$ for the identity operator in $\cl B(H)$, and set $I_n = I_{\bb{C}^n}$; if $H$ is clear from the context, we write $I = I_H$. 

We recall some basic notions and results in operator space theory, and refer the reader to \cite{blm, Pa, pisier_intr} for further background.
For an operator space $\cl X$, we will denote by $\|\cdot\|^{(n)}$ the corresponding norm on $M_n(\cl X)$, $n\in \bb{N}$
(written $\|\cdot\|^{(n)}_{\cl X}$ when $\cl X$ needs to be 
emphasised). 
We write $R(\cl X)$ (resp. $C(\cl X)$) for the row (resp. column) operator space over $\cl X$ (corresponding to some cardinality that will be clear from the context). 
If $\cl X$ and $\cl Y$ are operator spaces, a linear map $\phi \colon \cl X\to \cl Y$ is called \emph{completely bounded} if 
\[
\|\phi\|_{\rm cb} := \sup_{n\in \bb{N}} \|\phi^{(n)}\| < \infty;
\]
it is called \emph{completely contractive} if $\|\phi\|_{\rm cb}\leq 1$, and a \emph{complete isometry} if $\phi^{(n)}$ is an isometry for every $n$.
Given an operator space $\cl X$, we write $\cl X^*$ for the adjoint vector space, which becomes an operator space with the norms given by 
\[
\|(x_{i,j}^*)_{i,j}\|_{\cl X^*}^{(n)} := \|(x_{j,i})_{i,j}\|_{\cl X}^{(n)}, \ \ \ x_{i,j}^* \in \cl X^*, n \in \bb N.
\]
If $\cl X$ is concretely represented in some $\cl B(H,K)$, then $\cl X^*$ is (completely isometric to) its adjoint space in $\cl B(K,H)$.
There is a correspondence between completely bounded maps $\phi$ defined on $\cl X$ and completely bounded maps $\phi^*$ defined on $\cl X^*$, given by 
\[
\phi^*(x^*):= \phi(x)^* \text{ for } x^* \in \cl X^*;
\]
we have $\|(\phi^*)^{(n)}\| = \|\phi^{(n)}\|$, $n \in \bb N$ (see \cite[Section 1.2.25]{blm} for more details).

A \emph{$*$-vector space} is a complex vector space $\cl X$, equipped with an involution. 
In this case, the space $M_n(\cl X)$ admits an induced involution, given by 
\[
(x_{i,j})_{i,j}^* := (x_{j,i}^*)_{i,j};
\]
the elements of the real vector space 
\[
M_n(\cl X)_{h} := \{x \in M_n(\cl X) \mid x = x^*\}
\]
are called \emph{hermitian}. 
A \emph{matrix ordered $*$-vector space} is a complex vector space $\cl X$ endowed with an involution and a family $\{M_n(\cl X)^+\}_{n \in \bb N}$ such that the following hold: 
\begin{enumerate}
\item $M_n(\cl X)^+ \subseteq M_n(\cl X)_{h}$, $n \in \bb N$;
\item $M_n(\cl X)^+ \cap [ - M_n(\cl X)^+] = \{0\}$, $n \in \bb N$;
\item every $M_n(\cl X)^+$, $n \in \bb N$, is a cone in $M_n(\cl X)_{h}$;
\item $\alpha^* M_m(\cl X)^+ \alpha \subseteq M_n(\cl X)^+$ for all $n,m \in \bb N$ and all $\alpha \in M_{m,n}$.
\end{enumerate}
We will refer to condition (iv) as \emph{compatibility} of the matrix cones. 
The elements in $M_n(\cl X)^+$ are called \emph{positive}.
Recall that axiom (iii) above can be equivalently replaced by:
\begin{enumerate}
\item[(iii')] $M_n(\cl X)^+ \oplus M_m(\cl X)^+ \subseteq M_{n+m}(\cl X)^+$, $n, m \in \bb N$.
\end{enumerate}

A selfadjoint map $\phi \colon \cl X\to \cl Y$ between matrix ordered $*$-vector spaces $\cl X$ and $\cl Y$ is called \emph{$n$-positive} if $\phi^{(n)}(M_n(\cl X)^+)\subseteq M_n(\cl Y)^+$, \emph{positive} if it is $1$-positive, and \emph{completely positive} if it is $n$-positive for every $n\in \bb{N}$. 
A map $\phi \colon \cl X\to \cl Y$ is called a \emph{complete order embedding} if it is completely positive, injective, and 
\[
\phi^{(n)}(M_n(\cl X)^+) = \phi^{(n)}(M_n(\cl X))\cap M_n(\cl Y)^+ \foral n\in \bb{N};
\] 
in this case, we write $\cl X\subseteq_{\rm c.o.i.}\cl Y$. 

The positive functionals on a matrix ordered $*$-vector space are automatically completely positive (see for example \cite[Proposition 13.2]{Pa}).
We recall the canonical duality between the linear maps $\phi \colon \cl X \to M_n$ and the linear functionals $s \colon M_n(\cl X) \to \bb C$ of a matrix ordered $*$-vector space $\cl X$ (see for example \cite[Chapter 13]{Pa}).  
For a linear map $\phi \colon \cl X \to M_n$, we define the linear functional $s_\phi \colon M_n(\cl X) \to \bb C$ by
\begin{equation}\label{eq_1n} 
s_\phi( (x_{i,j})_{i,j} ) := \sum_{i,j=1}^n \sca{\phi(x_{i,j}) e_j, e_i} = \sca{\phi^{(n)}( (x_{i,j})_{i,j} ) e, e},
\end{equation} 
where $e = e_1 \oplus \cdots \oplus e_n \in \bb C^{n^2}$ for the canonical basis $\{e_1, \dots, e_n\}$ of $\bb C^n$.
Reversely, let $\{E_{i,j}\}_{i,j}$ be the set of the canonical matrix units in $M_n$.
For a linear functional $s \colon M_n(\cl X) \to \bb C$, we define the linear map $\phi_s  \colon \cl X \to M_n$ by 
\[
\sca{\phi_s(x) e_j, e_i} = s(x \otimes E_{i,j}), \ \ \ x\in \cl X.
\]
As pointed out in \cite[Chapter 13]{Pa}, these operations are mutual inverses; we record the following result, known as Choi's Theorem.

%%%%%%%%%%%%%%%%%%%%%%%%%%%%
\begin{theorem}\label{t_choi non-unital} \cite[Proposition 13.2]{Pa}
Let $\cl X$ be a matrix ordered $*$-vector space and $\phi \colon \cl X \to M_n$ be a linear map.
The following are equivalent:
\begin{enumerate}
\item $s_\phi(M_n(\cl X)^+) \geq 0$;
\item $\phi$ is $n$-positive;
\item $\phi$ is completely positive.
\end{enumerate} 
\end{theorem} 

Let $\cl S$ be a matrix ordered $*$-vector space. An element $e\in \cl S_h$ is called an \emph{order unit} for $\cl S$ if for every $x\in \cl S_h$ there exists an $r > 0$ such that $re+x\ge 0$.
An order unit $e\in \cl S$ is called \emph{Archimedean} if, for $x\in \cl S_h$, the condition $re+x\ge 0$ for all $r>0$ implies that $x\ge 0$. 
An element $e \in \cl S$ is called an \emph{Archimedean matrix order unit} if $e \otimes I_n$ is an Archimedean order unit on $M_n(\cl S)$ for every $n \in \bb N$. 
An \emph{operator system} is a matrix ordered $*$-vector space $\cl S$ with an Archimedean matrix order unit.
An operator system $\cl S$ with Archimedean matrix order unit $e$ admits an operator space structure given by
\[
\nor{x}_e^{(n)} := \inf \left\{ r \geq 0 \ : \ \begin{pmatrix} r e \otimes I_n & x \\ x^* & r e \otimes I_n \end{pmatrix}  
\in M_{2n}(\cl S)^+ \right\},
\ \ \ x \in M_n(\cl S),
\] 
(see for example \cite[Proposition 13.3]{Pa}).
By virtue of the Choi-Effros Theorem, $\cl S$ admits a unital complete order embedding in some $\cl B(H)$, that is completely isometric with respect to these norms.

The morphisms in the operator system category are unital completely positive maps.
We note that often the terms \emph{completely contractive} and \emph{completely bounded} will be abbreviated \lq\lq c.c.'', and \lq\lq c.b.'', respectively, while the terms \emph{completely positive}, \emph{completely contractive completely positive} and \emph{unital completely positive} will be abbreviated \lq\lq c.p.'', \lq\lq c.c.p.'', and \lq\lq u.c.p.'', respectively.
Recall that the contractive functionals on an operator space are completely contractive; see for example \cite[Proposition 3.8]{Pa}.

We denote by $[\cl X]$ the closed linear span of a subset $\cl X$ of a normed space. 
We say that an operator space $\cl E \subseteq \cl B(H,K)$ is \emph{non-degenerate} if 
\[
I_H \in [\cl E^* \cl E] \qand I_K \in [\cl E \cl E^*].
\]
We say that a completely bounded map $\phi \colon \cl E\to \cl B(H,K)$ is \emph{non-degenerate} if $\phi(\cl E)$ acts non-degenerately, that is 
\[
[\phi(\cl E)H] = K \qand [\phi(\cl E)^* K] = H.
\] 
Given a completely bounded map $\phi \colon \cl E \to \cl B(H, K)$ we can write
\[
\phi(x) = \begin{pmatrix} P_{K'} \phi(x) |_{H'} & 0 \\ 0 & 0 \end{pmatrix}, \ \ \  x \in \cl E,
\]
where $K' := [\phi(\cl E) H]$ and $H' := [\phi(\cl E)^* K]$, and thus pass to a non-degenerate completely bounded compression.

We next record a variation of 
Haagerup-Paulsen-Wittstock's Representation Theorem for completely bounded maps that will be used in the sequel (see for example \cite[Theorem 8.2 and Theorem 8.4]{Pa}).
If $\cl E$ is an operator subspace of a unital C*-algebra $\cl A$ and $\phi \colon \cl E \to \cl B(H)$ is a completely bounded map, then there exists a completely bounded map $\psi \colon \cl A \to \cl B(H)$ such that 
\[
\psi|_{\cl E} = \phi \qand 
\|\psi\|_{\rm cb} = \|\phi\|_{\rm cb}.
\]
On the other hand, if $\cl A$ is a unital C*-algebra and $\phi \colon \cl A \to \cl B(H)$ is a completely bounded map, then there exists a Hilbert space $K$, a unital $*$-representation $\rho \colon \cl A \to \cl B(K)$, and bounded operators $W_1, W_2 \colon H \to K$ such that
\[
\|\phi\|_{\rm cb} = \|W_1\| \cdot \|W_2\|
\qand
\phi(a) = W_2^* \rho(a) W_1, \text{ for all } a \in \cl A.
\]
If in addition $\|\phi\|_{\rm cb} = 1$, then $W_1$ and $W_2$ can be taken to be isometries.
Furthemore, if $\phi$ is completely contractive, then $W_1$ and $W_2$ can be taken to be contractions.
Indeed, if $\|\phi\|_{\rm cb} = 0$, then we can chooose $W_1 = W_2 = 0$.
Otherwise, considering $\phi' := \|\phi\|_{\rm cb}^{-1} \phi$, we obtain isometries $V_1, V_2$ so that $\phi'(a) = V_2^* \rho(a) V_1$. 
By setting $W_i := \|\phi\|_{\rm cb}^{1/2} V_i$ we get
\[
\|W_i\| = \|\phi\|_{\rm cb}^{1/2} \leq 1
\]
and 
\[
\phi(a) = \|\phi\|_{\rm cb} \phi'(a) = \|\phi\|_{\rm cb} V_2^* \rho(a) V_1 = W_2^* \rho(a) W_1, 
\]
for all $a \in \cl A$.
We will also use the following variation for corner maps (see for example \cite[Theorem 1.2.8]{blm}).

%%%%%%%%%%%%%%%%%%%%%%%%%%%%
\begin{theorem}\label{t_hpw} (Haagerup, Paulsen, Wittstock)
Let $H_1$ and $H_2$ be Hilbert spaces, $\cl A$ be a unital C*-algebra, and $\phi \colon \cl A \to \cl B(H_1, H_2)$ be a completely contractive map.
Then there exist a Hilbert space $K$, a unital $*$-representation $\rho \colon \cl A \to \cl B(K)$ and operators $W_1 \colon H_1 \to K$, $W_2 \colon H_2 \to K$ such that 
\[
\nor{\phi}_{\rm cb} = \|W_1\| \cdot \|W_2\|
\qand
\phi(a) = W_2^* \rho(a) W_1 \text{ for all } a \in \cl A.
\]
If in addition $\nor{\phi}_{\rm cb} = 1$, then the operators $W_1$ and $W_2$ can be chosen to be isometries.
\end{theorem}

\begin{proof}
Consider the map $\wt{\phi} \colon \cl A \to \cl B(H_1 \oplus H_2)$, given by 
\[
\wt{\phi}(a) = \begin{pmatrix} 0 & 0 \\ \phi(a) & 0 \end{pmatrix}, \ \ a \in \cl A.
\]
By Wittstock's Representation Theorem, there exist a Hilbert space $K$, a unital $*$-representa\-tion $\rho \colon \cl A \to \cl B(K)$ and bounded operators 
$\wt{W}_i \colon H_1 \oplus H_2 \to K$, for $i=1,2$, such that
\[
\| \wt{\phi} \|_{\rm cb} = \| \wt{W}_1 \| \cdot \| \wt{W}_2 \|
\qand
\wt{\phi}(a) = \wt{W}_2^* \rho(a) \wt{W}_1, \text{ for all } a \in \cl A.
\]
By setting $W_i := \wt{W}_i |_{H_i}$, for $i=1,2$, we have
\[
\phi(a) = W_2^* \rho(a) W_1, \text{ for all } a \in \cl A.
\]
If $\|\phi\|_{\rm cb} = 1$, then $\|\wt{\phi}\|_{\rm cb} = 1$ as well, and we can choose $\wt{W}_1, \wt{W}_2$ to be isometries; then $W_1$ and $W_2$ are isometries as well.
Finally, we have
\[
\nor{W_1} \cdot \nor{W_2} \leq \| \wt{W}_1 \| \cdot \| \wt{W}_2 \| = \| \wt{\phi} \|_{\rm cb} = \nor{\phi}_{\rm cb} \leq \nor{W_1} \cdot \nor{W_2},
\]
and thus $\nor{\phi}_{\rm cb} = \|W_1\| \cdot \|W_2\|$, as required.
\end{proof}

%%%%%%%%%%%%%%%%%%%%%%%%%%%%
\subsection{Selfadjoint operator spaces}\label{ss_nonunital}
%%%%%%%%%%%%%%%%%%%%%%%%%%%%

Matrix ordered operator spaces that are completely isometric to selfadjoint subspaces of C*-algebras were first studied in \cite{werner}.
A further detailed study was undertaken in \cite{kkm, russell}. 
They are sometimes referred to as \emph{non-unital operator systems} (see for example \cite{kkm}).
Here we call them \emph{selfadjoint operator spaces}, reflecting more closely the viewpoint of \cite{werner} (see also \cite{russell}). 
It is clarified through \cite{kkm, russell, werner} that such operator spaces can be viewed as either concrete or abstract. 
As is the case with operator systems and general operator spaces, if unclear from the context which of the two is intended, this will be explicitly stated. 

An operator space will be called a  \emph{matrix ordered operator space} if it is a matrix ordered $*$-vector space such that the involution at every $M_n(\cl X)$ is isometric, and every cone $M_n(\cl X)^+$ is closed.
If $\cl X$ and $\cl Y$ are matrix ordered operator spaces, we denote by ${\rm CCP}(\cl X, \cl Y)$ the (convex) set of all completely contractive completely positive maps from $\cl X$ into $\cl Y$.
We note that, if $\cl X$ is an operator system, then the matrix cone structure defines a natural family of matricial norms (see Subsection \ref{ss_opspsy}); thus, operator systems are matrix ordered operator spaces in a canonical fashion (as in \cite[Chapter 13]{Pa}).

Henceforth, let $\cl X$ be a matrix ordered operator space.
For $\eps > 0$ and $\alpha \in M_n^+$, let $\alpha_\eps := \alpha + \eps I_n$; thus, 
\[
\alpha_\eps^{-1/2} \in M_n^+
\qand
\alpha_\eps^{-1/2} x \alpha_\eps^{-1/2} \in M_n(\cl X),\text{ for } x \in M_n(\cl X).
\]
We equip the vector space $\cl X^\# = \cl X \oplus \bb C$ with the involution given by 
\[
(x,\lambda)^* := (x^*,\bar{\lambda})
\]
and notice that
\[
M_n(\cl X^\#)_h = M_n(\cl X)_h \oplus (M_n)_h.
\]
The \emph{partial unitisation} $\cl X^\#$ of a matrix ordered operator space $\cl X$ is the matrix ordered $*$-vector space $\cl X^\# = \cl X \oplus \bb C$ with the matrix cone structure $\{M_n(\cl X^\#)^+\}_{n \in \bb N}$ given as follows: an element $(x, \alpha) \in M_n(\cl X^\#)_h$  is in $M_n(\cl X^\#)^+$ if and only if $\alpha \geq 0$ and
\[
\varphi(\alpha_\eps^{-1/2} x \alpha_\eps^{-1/2}) \geq  -1 \foral \eps > 0 \text{ and all } 
\varphi \in {\rm CCP}(M_n(\cl X), \bb C).
\]
By \cite[Lemma 4.8]{werner}, the partial unitisation is an operator system with unit $(0,1)$.

%%%%%%%%%%%%%%%%%%%%%%%%%%%%
\begin{remark}
When $\alpha \geq 0$, the condition $\varphi(\alpha_\eps^{-1/2} x \alpha_\eps^{-1/2}) \geq -1$ for all $\vphi \in {\rm CCP}(M_n(\cl X), \bb C)$ is equivalent to the condition $\varphi(\alpha_\eps^{-1/2} x \alpha_\eps^{-1/2}) \geq -1$ for all $\vphi \in {\rm CCP}(M_n(\cl X), \bb C)$  with $\|\vphi\| = 1$, since normalisations of completely positive maps are completely positive.
The former formulation appears in \cite{kkm}, and the latter in \cite{werner}.
The reader is also addressed to a correction on \cite{werner} provided at \cite[Lemma 2.10]{cvs}.
\end{remark}

By \cite[Lemma 4.9 (a)]{werner}, a completely contractive completely positive map $\phi \colon \cl X \to \cl Y$ extends to a unital completely positive map
\begin{equation}\label{eq_sharp}
\phi^\# \colon \cl X^\# \to \cl Y^\#; \ (x, \alpha) \mapsto (\phi(x), \alpha).
\end{equation}
If $\phi$ is a completely isometric complete order isomorphism, then $\phi^\#$ is a unital complete order isomorphism (see \cite[Theorem 2.8 (3)]{kkm} or apply \cite[Theorem 4.9 (a)]{werner} for $\phi$ and $\phi^{-1}$).
In the case where $\phi \colon \cl X \to \cl Y$ is completely contractive completely positive for a matrix ordered operator space $\cl X$ and an operator system $\cl Y$, we have that the extension
\[
\phi^\# \colon \cl X^\# \to \cl Y; \ (x, \alpha) \mapsto \phi(x) + \alpha \cdot 1_{\cl Y},
\]
is unital and completely positive \cite[Theorem 4.9 (c)]{werner}.

Let $\iota \colon \cl X \to \cl X^\#$ be the (inclusion) map, given by $\iota(x) = (x,0)$.
An application of the separation Hahn-Banach Theorem for cones shows that $\iota$ is (completely contractive and) completely positive, that is, the positive cones of $\cl X$ are included in the respective positive cones of $\cl X^\#$. 
A matrix ordered operator space $\cl X$ will be called a \emph{selfadjoint operator space} if the canonical inclusion map $\cl X \to \cl X^\#$ is completely isometric.
In this case we refer to $\cl X^\#$ as \emph{the unitisation} of $\cl X$.
It is shown in \cite[Lemma 4.8 and Theorem 4.15]{werner} (see also \cite[Theorem 2.15]{kkm}) that a matrix ordered operator space is a selfadjoint operator space if and only if for all $x \in M_n(\cl X)$ and $n \in \bb N$ we have
\[
\|x\|^{(n)} 
= \nu_n(x)
:= \sup \left\{ \left| \varphi \begin{pmatrix} 0 & x \\ x^* & 0 \end{pmatrix} \right| \ : \ \varphi \in {\rm CCP}(M_{2n}(\cl X), \bb C) \right\}.
\]

Selfadjoint operator spaces form a category with morphisms given by the completely contractive completely positive maps 
\cite{kkm}. 
A completely isometric map $\phi \colon \cl S \to \cl T$ between two selfadjoint operators spaces is called an \emph{embedding} if the unitisation $\phi^\# \colon \cl S^\# \to \cl T^\#$ is a unital complete order embedding in the category of operator systems.
The unitization $\phi^\#$ of a completely isometric complete order embedding $\phi \colon \cl S \to \cl T$ between selfadjoint operator spaces is not necessarily a complete order embedding (see \cite[Example 2.14]{kkm} for a counterexample).

Let $\tau \colon \cl X^\# \to \bb C$ be the (projection) map given by $\tau((x,\alpha)) = \alpha$; then $\tau$ is unital and (completely) positive.
By \cite[Lemmas 4.5, 4.8, 4.9(c) and Corollary 4.4]{werner}, the unitisation $\cl X^\#$ of a selfadjoint operator space $\cl X$ is an operator system that satisfies the following (universal) property:
\begin{enumerate}
\item there is a split exact sequence
\[
0 \longrightarrow \cl X \stackrel{\iota}{ \longrightarrow} \cl X^\# \stackrel{\tau}{ \longrightarrow} \bb C \to 0,
\]
so that $\iota$ is a complete order monomorphism, and $\tau$ is unital;
\item for each operator system $(\cl Y, e)$ and every 
completely contractive completely positive map $\phi \colon M_n(\cl X) \to \cl Y$ the extension
\[
\phi^\# \colon \iota^{(n)}(M_n( \cl X)) + \bb C (1 \otimes I_n) \to \cl Y; 
\iota^{(n)}(x) + \la (1 \otimes I_n) \mapsto \phi(x) + \la e
\]
is (unital and) completely positive. 
\end{enumerate}
Furthermore, by \cite[Corollary 4.4]{werner}, $\iota^{(n)}(M_n(\cl X)) + \bb C (1 \otimes I_n)$ is the partial unitisation of $M_n(\cl X)$,  that is, there is a canonical complete order embedding 
$M_n(\cl X)^\# \hookrightarrow M_n(\cl X^\#)$.
In this way we can extend completely contractive completely positive maps first from $M_n(\cl X)$ to $M_n(\cl X)^\#$ by the properties of the unitisation, and then from $M_n(\cl X)^\#$ to $M_n(\cl X^\#)$  by Arveson's Extension Theorem.

%%%%%%%%%%%%%%%%%%%%%%%%%%%%
\begin{remark}\label{r_sos unit}
Let $\cl X$ be a selfadjoint operator space and let $e$ be an Archi\-medean matrix order unit with respect to the cone structure of $\cl X$.
This alone may not guarantee that $\cl X$ is an operator system with $e$ as a unit, and it is a consequence of the fact that $\cl X$ comes \emph{a priori} with a matrix-norm structure.
First of all, $e$ needs to have $\nor{\cdot}_{\cl X}$-norm equal to $1$.
Further, recall that the element $e$ induces a matrix-norm structure $\{\nor{\cdot}^{(n)}_e\}_{n\in \bb{N}}$, given by 
\[
\nor{x}^{(n)}_e := \inf \left\{r \geq 0 \ : \ \begin{pmatrix} r e\otimes I_n & x \\ x^* & r e\otimes I_n \end{pmatrix} \in M_{2n}(\cl X)^+ \right\}, \ \ x \in M_n(\cl X).
\]
By the Choi-Effros Theorem, there exists a unital complete order embedding $\phi \colon \cl X \to \cl B(H)$ such that $\phi(e) = I_H$.
Although $\phi$ is a complete isometry with respect to the family $\{\nor{\cdot}^{(n)}_e\}_{n\in \bb{N}}$ of matricial norms, it may fail to be a complete isometry with respect to the family $\{\nor{\cdot}^{(n)}_{\cl X}\}_{n\in \bb{N}}$.
Therefore, we will say that a selfadjoint operator space $\cl X$ is an operator system if it admits an Archimedean matrix order unit $e$ such that 
\[
\nor{x}_{\cl X}^{(n)} = \nor{x}^{(n)}_e, \ \ x \in M_n(\cl X), n \in \bb N.
\]
We will encounter this situation later on, for example in the setting of Theorem \ref{t_notos}, Remark \ref{r_appnotos} and Remark \ref{r_troscand}.
\end{remark}

%%%%%%%%%%%%%%%%%%%%%%%%%%%%
\subsection{TRO-envelope}
%%%%%%%%%%%%%%%%%%%%%%%%%%%%

A \emph{ternary ring of operators (TRO)} is a subspace $\cl M\subseteq \cl B(H,K)$, for some Hilbert spaces $H$ and $K$, such that $\cl M \cl M^* \cl M \subseteq \cl M$.
If in addition $\cl M$ is closed, then $\M$ is a TRO if and only if $[\cl M \cl M^* \cl M] = \cl M$.
A ternary morphism $\theta \colon \cl M \to \cl N$ between TRO's is a linear map that satisfies
\[
\theta(m_1 m_2^* m_3) = \theta(m_1) \theta(m_2)^* \theta(m_3), \ \ \ m_1, m_2, m_3 \in \cl M;
\]
ternary morphisms are completely contractive \cite[Lemma 8.3.2]{blm}.
Every ternary morphism $\theta \colon \cl M \to \cl B(H)$ induces $*$-representations
\[
\rho \colon [\cl M \cl M^*] \to \cl B(H); \ xy^* \mapsto \theta(x) \theta(y)^*
\]
and
\[
\si \colon [\cl M^* \cl M] \to \cl B(H); \ x^*y \mapsto \theta(x)^* \theta(y),
\]
that satisfy
\[
\theta(a \cdot x \cdot b) = \rho(a) \theta(x) \si(b), \ \ a \in [\cl M \cl M^*], b \in [\cl M^* \cl M], x \in \cl M,
\]
see for example \cite[Proof of Corollary 8.3.5]{blm}.
For a completely isometric map $\iota \colon \cl E \to \cl M$ into a TRO $\cl M$, let $\cl T(\iota(\cl E))$ be the TRO densely spanned by the products
\[
\iota(x_1) \iota(x_2)^* \iota(x_3) \iota(x_4)^* \cdots \iota(x_{2n})^* \iota(x_{2n+1})
\ \mbox{ for } 
n\geq 0, \ x_1,\dots, x_{2n+1} \in \cl E;
\]
we say that $(\cl T(\iota(\cl E)),\iota)$ is a \emph{TRO cover} of $\cl E$.

If $\cl E$ is unital, set $\cl S(\cl E) := \cl E + \cl E^*$; if not, set
\[
\cl S(\cl E) := \left\{ \begin{pmatrix} \la & x_1 \\ x_2^* & \mu \end{pmatrix} \mid x_1, x_2 \in \cl E, \la, \mu \in \bb C\right\},
\]
which is the \emph{Paulsen operator system} of $\cl E$.
The operator space $\cl E$ embeds completely isometrically into the operator system $\cl S(\cl E)$, that is, $\cl S(\cl E)$ is independent of the representation of $\cl E$ (see for example \cite[Lemma 1.3.15]{blm}); therefore $\cl E$ embeds into the injective envelope $\cl I(\cl S(\cl E))$ of $\cl S(\cl E)$.
The well-known Choi-Effros construction \cite{choi-effros} endows the injective envelope with a C*-algebraic structure.
Writing 
\begin{equation}\label{eq_I12}
\cl I(\cl S(\cl E)) = \begin{pmatrix} \cl I_{11}(\cl E) & \cl I_{12}(\cl E) \\ \cl I_{21}(\cl E)^* & \cl I_{22}(\cl E) \end{pmatrix}
\end{equation}
in a matrix form, the injective envelope $\cl I(\cl E)$ of $\cl E$ is defined to be the corner $\cl I_{12}(\cl E)$; in particular, it coincides with $\cl I(\cl E + \cl E^*)$ when $\cl E$ is unital (see \cite[Paragraph 4.4.2]{blm}).
The TRO cover of $\cl E$ inside $\cl I(\cl S(\cl E))$ will be denoted by $\tenv(\cl E)$.
M. Hamana \cite{Ham99} showed that $\tenv(\cl E)$ possesses the following universal property: given any 
TRO cover $(\cl M, j)$ of $\cl E$ there exists a (necessarily unique) surjective ternary morphism $\theta \colon \cl M \to \tenv(\cl E)$ such that $\theta(j(x)) = x$, for all $x\in \cl E$.
The operator space $\tenv(\cl E)$ is called the \emph{ternary} (or \emph{TRO}) \emph{envelope} of $\cl E$.
If $\cl E$ is unital, then the embedding $\cl E \hookrightarrow \cl I(\cl E)$ is unital and, the TRO envelope is a C*-algebra, denoted by $\cenv(\cl E)$ (see for example \cite[Remark 8.3.12, item (5)]{blm}).
The existence of the C*-envelope was established by M. Hamana \cite{Ham79}, following Arveson's quantisation program \cite{Arv69}
(see also \cite{blecher-shilov}).
For an alternative proof of Hamana's Theorem, using boundary subsystems, the reader is directed to \cite{Kak11-2}; the arguments in \cite{Kak11-2} suffice to show the existence of the TRO envelope as well, even though this is not mentioned therein.

A. Connes and W. D. van Suijlekom \cite{cvs} introduced the C*-envelope of a selfadjoint operator space.
A pair $(\cl C, \iota)$ is called a \emph{C*-cover} of a selfadjoint operator space $\cl E$ if $\cl C$ is a C*-algebra, $\iota \colon \cl E \to \cl C$ is an embedding and $\cl C = \ca(\iota(\cl E))$.
A C*-cover $(\cl C, \iota)$ is \emph{minimal} if for any C*-cover $(\cl C', \iota')$ there is a surjective $*$-homomorphism $\pi \colon \cl C' \to \cl C$ such that $\pi \circ \iota' = \iota$.
It  is shown in \cite[Corollary 7.3]{kkm} that the minimal C*-cover of a (unital) operator system in the category of selfadjoint operator spaces coincides with its C*-envelope.

We next recall some facts about multipliers of an operator space, and refer the reader to \cite[Section 4.5]{blm} for more details.
The \emph{left multiplier algebra} $\cl M_\ell(\cl E)$ of an operator space $\cl E$ is defined in \cite[Paragraph 4.5.1]{blm}, and its C*-diagonal $\cl A_{\ell}(\cl E)$ -- called the \emph{left C*-multiplier} of $\cl E$ -- can be defined as the C*-algebra whose elements are completely bounded maps $u \colon \cl E \to \cl E$ for which there exist a Hilbert space $H$, an operator $S \in \cl B(H)$, and a (linear) complete isometry $\si \colon \cl E \to \cl B(H)$ such that $S^* \si(\cl E) \subseteq \si(\cl E)$ and $\si(u (x)) = S \si(x)$ for all $x \in \cl E$.
By \cite[Proposition 4.5.8]{blm},
\[
\cl A_{\ell}(\cl E) \simeq \{a \in \cl I_{11}(\cl E) \ : \ a \cl E \subseteq \cl E \text{ and } a^*\cl E \subseteq \cl E \}.
\]
The \emph{right C*-multiplier} $\cl A_{r}(\cl E)$ of $\cl E$ is described analogously, that is, 
\[
\cl A_{r}(\cl E) \simeq \{a \in \cl I_{22}(\cl E) \ : \ \cl E a \subseteq \cl E \text{ and } \cl E a^* \subseteq \cl E \},
\]
as the diagonal of the \emph{right C*-multiplier} $\cl M_r(\cl E)$ of $\cl E$.
If $\cl S$ is an operator system, we have $\cl A_{\ell}(\cl S) = \cl A_r(\cl S) \subseteq \cl S$ due to the existence of the involution and the unit; in this case we simplify the notation to $\cl A_{\cl S}$.

For a TRO $\cl M$, the C*-algebra $\cl A_{\ell}(\cl M)$ is $*$-isomorphic to the multiplier algebra of $[\cl M \cl M^*]$; respectively, the C*-algebra $\cl A_{r}(\cl M)$ is $*$-isomorphic to the multiplier algebra of $[\cl M^* \cl M]$ (see for example \cite[Corollary 8.4.2]{blm}).
If, in addition, $\cl M$ is non-degenerate, then 
\[
\cl A_{\ell}(\cl M) \simeq [\cl M \cl M^*] \qand \cl A_{r}(\cl M) \simeq [\cl M^* \cl M], 
\]
since $[\cl M \cl M^*]$ and $[\cl M^* \cl M]$ are unital. %, and thus equal to their multiplier algebras.

%%%%%%%%%%%%%%%%%%%%%%%%%%%%
\section{Quotients and bimodules}\label{s_quatbim}
%%%%%%%%%%%%%%%%%%%%%%%%%%%%

In this section we collect some results on the general theory of bimodules and quotients for operator systems and for selfadjoint operator spaces. 
All bimodule actions will be unital over a unital C*-algebra.
We will be interested in bimodule structures that descend to a quotient.

%%%%%%%%%%%%%%%%%%%%%%%%%%%%
\subsection{Quotients of operator systems}
%%%%%%%%%%%%%%%%%%%%%%%%%%%%

Let $(\cl S, \{C_n\}_{n=1}^{\infty}, e)$ be an operator system.
Recall from \cite{kptt_adv} that a subspace $\cl J$ of an operator system $\cl S$ is called a \emph{kernel} of $\cl S$ if there exist a Hilbert space $H$ and a unital completely positive map $\phi \colon \cl S \to \cl B(H)$ such that $\cl J = \ker \phi$. 
The quotient vector space $\cl S/ \cl J$ is then endowed with the standard involution map and equipped with a family of matrix cones given by
\begin{align*}
C_n(\cl S/ \cl J) := \{u + M_n(\cl J) \in M_n(\cl S/\cl J)_h \ : \ & \, \forall \eps > 0 \ \exists \, w \in M_n(\cl J) \\
& \hspace{-0.5cm} \text{ such that } \eps(e \otimes 1_n) + u + w \in C_n\}.
\end{align*}
It follows that $e + \cl J$ is an Archimedean matrix order unit for $\cl S/ \cl J$, and the triple $(\cl S/ \cl J, \{C_n(\cl S/ \cl J)\}_{n=1}^\infty, e+ \cl J)$ is the induced quotient operator system of \cite[Definition 3.5]{kptt_adv}.
By \cite[Proposition 3.6]{kptt_adv}, the latter operator system has the following universal property: if $\phi \colon \cl S \to \cl B(H)$ is a unital completely positive map with $\phi|_{\cl J} = 0$, then there exists a unital completely positive map $\phi' \colon \cl S/ \cl J \to \cl B(H)$ making the diagram
\[
\xymatrix{
\cl S \ar[rr]^{\phi} \ar[d] & & \cl B(H) \\
\cl S/ \cl J \ar[urr]_{\phi'} & &
}
\]
commutative.

For every $n \in \bb N$, we can define the matrix norms on $M_n(\cl S/ \cl J)$, given by
$$\|u + M_n(\cl J)\|_{\rm osp}^{(n)}  := \sup \{ \|\phi^{(n)}(u)\| \ : \ \phi \textup{ c.c. map on }\cl S \textup{ with } \phi|_{\cl J} = 0 \}
$$
and
$$\|u + M_n(\cl J)\|_{\rm osy}^{(n)}  := \sup \{ \|\phi^{(n)}(u)\| \ : \ \phi \textup{ u.c.p. map on }\cl S \textup{ with } \phi|_{\cl J} = 0 \},$$
as well as the matrix norm structure on $\cl S/ \cl J$ induced by $e + \cl J$.
We see that $\| \cdot \|_{\rm osy}^{(n)} \leq \| \cdot \|_{\rm osp}^{(n)}$ for every $n \in \bb N$, and by \cite[Example 4.4]{kptt_adv} we have that equality may not hold in general.
We note that, on the other hand, $u + M_n(\cl J) \in C_n(\cl S/ \cl J)$ if and only if $\phi^{(n)}(u) \geq 0$ for every unital completely positive map $\phi$ on $\cl S$ with $\phi|_{\cl J} = 0$.
Moreover, if $\phi$ is a unital completely positive map on $\cl S$, then
\[
\begin{pmatrix} \la & \phi^{(n)}(u) \\ \phi^{(n)}(u)^* & \la \end{pmatrix} \geq 0
\Leftrightarrow 
\|\phi^{(n)}(u)\| \leq \la.
\]
Therefore $\| \cdot \|_{e + \cl J}^{(n)} = \| \cdot \|_{\rm osy}^{(n)}$.
Consequently $C_n(\cl S/ \cl J)$ is closed with respect to $\| \cdot \|_{\rm osy}^{(n)}$, and thus also with respect to $\| \cdot \|_{\rm osp}^{(n)}$, $n\in \bb{N}$.

%%%%%%%%%%%%%%%%%%%%%%%%%%%%
\begin{proposition}\label{p_quoos}
Let $\cl J$ be a kernel of an operator system $\cl S$, and let $q \colon \cl S \to \cl S/ \cl J$ be the canonical quotient map.
Then
\[
\ol{q^{(n)}(C_n)}^{\| \cdot \|_{\rm osy}^{(n)}} 
= 
C_n(\cl S/ \cl J), \ \ \ n \in \bb N.
\]
\end{proposition}

\begin{proof}
It is immediate from the definition that
\[
q^{(n)}(C_n)\subseteq C_n(\cl S/ \cl J);
\]
since $C_n(\cl S/ \cl J)$ is closed, we have 
\[
\ol{q^{(n)}(C_n)}^{\| \cdot \|_{\rm osy}^{(n)}} \subseteq C_n(\cl S/ \cl J).
\]

For the reverse inclusion, let $u + M_n(\cl J) \in C_n(\cl S/ \cl J)$ and, for $\eps >0$, let $w \in M_n(\cl J)$ be such that $\eps (e \otimes 1_n) + u + w \in C_n$.
From the first paragraph we have 
\begin{align*}
\eps (e \otimes 1_n) + u + M_n(\cl J) 
& = q^{(n)}(\eps (e \otimes 1_n) + u + w) \\
& \in q^{(n)}(C_n) \subseteq C_n(\cl S/ \cl J) 
\end{align*}
while clearly
\begin{align*}
\left\|(u + M_n(\cl J) ) - (\eps (e \otimes 1_n) + u + M_n(\cl J)) \right\|_{\rm osy}^{(n)} 
= 
\eps,
\end{align*}
and the proof is complete.
\end{proof}

Note that the statement in Proposition \ref{p_quoos} is still true if $\| \cdot \|_{\rm osy}^{(n)}$ is replaced by $\| \cdot \|_{\rm osp}^{(n)}$.

%%%%%%%%%%%%%%%%%%%%%%%%%%%%
\subsection{Operator $\cl A$-systems}
%%%%%%%%%%%%%%%%%%%%%%%%%%%%

We collect some facts about operator bimodules.
%All actions we consider here are by a unital C*-algebra, and the actions themselves are unital.
%
Let $\cl A$ be a unital C*-algebra. 
An operator space $\cl E$ will be called a \emph{left} (resp. \emph{right}) \emph{operator $\cl A$-space} if $\cl E$ is a left (resp. right) operator $\cl A$-module with unital action, that is, the modular actions are unital and completely contractive. 
An \emph{$\cl A$-representation} of a left operator $\cl A$-space (or a \emph{C*-representation} if $\cl A$ is clear from the context) is a pair $(\phi,\pi)$, where $\phi \colon \cl E\to \cl B(H,K)$ is a completely bounded map, $\pi \colon \cl A \to \cl B(K)$ is a unital $*$-representation, for some Hilbert spaces $H$ and $K$, and 
\begin{equation}
\phi(a\cdot x) = \pi(a)\phi(x), \ \ \ x\in \cl E, a\in \cl A.
\end{equation}
An $\cl A$-representation of $\cl E$ will be called \emph{completely contractive} (resp. \emph{completely isometric}) if $\phi$ is completely contractive (resp. completely isometric). 
An operator space $\cl E$ is a left operator $\cl A$-space if and only if it has a completely isometric $\cl A$-representation; likewise for the right operator modules (see \cite[Theorem 4.6.7]{blm} in conjunction with \cite[Definition 3.1.1]{blm}).
There is a similar theory for operator systems that are bimodules over a unital C*-algebra \cite[Chapter 15]{Pa}, which we now review.

%%%%%%%%%%%%%%%%%%%%%%%%%%%%
\begin{definition}
Let $\cl S \subseteq \cl B(H)$ be an operator system that is a bimodule over a unital C*-algebra $\cl A$.
We say that $\cl S$ is a \emph{concrete operator $\cl A$-system} if there exists a unital $*$-representation $\pi \colon \cl A \to \cl B(H)$ such that
\[
\pi(a) s = a \cdot s \mbox{ and }
s \pi(a) = s \cdot a,
\foral s \in \cl S, a \in \cl A.
\]
\end{definition}

Note that, if $\cl S \subseteq \cl B(H)$ is a concrete operator $\cl A$-system, then
\[
a^* \cdot C_m \cdot a = \pi^{(m,n)}(a)^* C_m \pi^{(m,n)}(a) \subseteq C_n
\foral
a \in M_{m,n}(\cl A),
\]
and that, for all $a \in \cl A$, we get 
\[
a \cdot e = \pi(a) e = \pi(a) = e \pi(a) = e \cdot a,
\]
where we use that the unit $e \in \cl S$ coincides with $I_H$.

%%%%%%%%%%%%%%%%%%%%%%%%%%%%
\begin{definition}\label{d_opeAsy}
Let $(\cl S,\{C_n\}_{n\in \bb{N}},e)$ be an operator system that is a bimodule over a unital C*-algebra $\cl A$.
We say that $\cl S$ is an \emph{abstract operator $\cl A$-system} if the following hold:
\begin{enumerate}
\item $a^* \cdot C_m \cdot a \subseteq C_n$ for all $a \in M_{m,n}(\cl A)$;
\item $a \cdot e = e \cdot a$ for all $a \in \cl A$;
\item $(a \cdot s)^* = s^* \cdot a^*$ for all $s \in \cl S$ and 
all $a \in \cl A$.
\end{enumerate}
\end{definition}

An abstract operator $\cl A$-system has a concrete realisation, as seen from the following result. 

%%%%%%%%%%%%%%%%%%%%%%%%%%%%
\begin{theorem}\label{t_opA-sys}
\cite[Theorem 15.12]{Pa}
Let $\cl S$ be an operator system that is a bimodule over a unital C*-algebra $\cl A$.
If $\cl S$ is an abstract operator $\cl A$-system, then there exists a unital complete order embedding $\ga \colon \cl S \to \cl B(H)$ and a unital $*$-representation $\pi \colon \cl A \to \cl B(H)$ such that
\[
\pi(a) \ga(s) = \ga(a \cdot s)
\qand
\ga(s) \pi(a) = \ga(s \cdot a)
\]
for all $a \in \cl A$ and $s \in \cl S$.
\end{theorem}

In Subsection \ref{ss_constpro} below, we will need 
bits of quotient theory for operator systems in the presence of a modular action, which we now develop.

%%%%%%%%%%%%%%%%%%%%%%%%%%%%
\begin{theorem}\label{th_modquo}
Let $\cl A$ be a unital C*-algebra, $\cl S$ be an operator $\cl A$-system, and $\cl J\subseteq \cl S$ be a kernel, such that $\cl A \cdot \cl J \subseteq\cl J$. 
Then the quotient operator system $\cl S/\cl J$ is an operator $\cl A$-system in a canonical way. 
\end{theorem}

\begin{proof}
Set $C_n = M_n(\cl S)^+$ and $C_n(\cl S/ \cl J) = M_n(\cl S/\cl J)^+$, $n\in \bb{N}$. 
Since the space $\cl J$ is selfadjoint, 
we have that it is an $\cl A$-bimodule, and hence
$\cl S/\cl J$ is an $\cl A$-bimodule with unital actions, given by
\[
a\cdot (x + \cl J) := a\cdot x + \cl J, \ (x + \cl J)\cdot b := x\cdot b + \cl J, \ \ \ a,b\in \cl A, x\in \cl S.
\]
We will show that $\cl S/ \cl J$ satisfies the three axioms of an operator $\cl A$-system (see Definition \ref{d_opeAsy}).

For the first axiom, let $u \in M_m(\cl S)$ be such that $u + M_m(\cl J) \in C_m(\cl S/\cl J)$, and let $a \in M_{m,n}(\cl A)$, for some $n,m\in \bb{N}$. 
For a fixed $\eps > 0$, let $w \in M_m(\cl J)$ be such that $\eps (e \otimes I_m) + u + w \in C_m$. 
Since $\cl S$ is an operator $\cl A$-system, we have
\[
\eps a^*a \cdot (e \otimes I_n) + a^*\cdot u \cdot a + a^*\cdot w \cdot a \in a^* \cdot C_m \cdot a \subseteq C_n,
\]
where we have used that 
\[
a^* \cdot (e \otimes I_n) \cdot a = a^*a \cdot (e \otimes I_n)
\]
by axiom (ii).
Therefore we obtain 
\begin{align*}
\eps \|a\|^2 \cdot (e \otimes I_n) + a^*\cdot u \cdot a + a^*\cdot w \cdot a 
& = \\
& \hspace{-3cm} = 
\eps a^*a \cdot (e \otimes I_n) + a^*\cdot u \cdot a + a^*\cdot w \cdot a + \\
& \hspace{-1cm} + 
\eps (\|a\|^2 (e \otimes I_n) - a^*a \cdot (e \otimes I_n)) \in C_n,
\end{align*}
where we have used that
\[
\|a\|^2 (e \otimes I_n) - a^*a \cdot (e \otimes I_n)
\in C_n.
\]
It follows that 
\[
a^*\cdot (u + M_m(\cl J)) \cdot a \in C_n(\cl S/\cl J),
\]
as required.

For the second axiom, we have
\[
a \cdot (e + \cl J) = a \cdot e + \cl J = e \cdot a + \cl J = (e + \cl J) \cdot a
\]
for all $a \in \cl A$.

For the third axiom, we have
\[
(a \cdot (s + \cl J))^* = (a \cdot s + \cl J)^* = (a \cdot s )^* + \cl J = s^* \cdot a ^* + \cl J = (s^* + \cl J) \cdot a^*
\]
for all $a \in \cl A$ and $s \in \cl S$, and the proof is complete.
\end{proof}

%%%%%%%%%%%%%%%%%%%%%%%%%%%%
\begin{remark}\label{r_contaut}
It follows from Theorem \ref{th_modquo} that the $\cl A$-module action on the quotient operator system $\cl S/\cl J$ is automatically completely contractive.
This fact that can be also verified directly, by using the definition of the quotient matricial order in the operator system category. 
\end{remark}

%%%%%%%%%%%%%%%%%%%%%%%%%%%%
\subsection{Quotients of selfadjoint operator spaces}
%%%%%%%%%%%%%%%%%%%%%%%%%%%%

Let $\cl X$ be a selfadjoint operator space with family $\{C_n\}_{n \in \bb N}$ of matricial cones. 
The notion of a kernel is extended in \cite[Definition 8.1, Remark 8.2]{kkm} to selfadjoint operator spaces:  a \emph{kernel} of a selfadjoint operator space $\cl X$ is a subspace $\cl J\subseteq \cl X$ such that $\cl J = \ker \phi$ for a completely contractive completely positive map $\phi \colon \cl X \to \cl B(H)$.
It follows that a kernel is itself a selfadjoint space.

%%%%%%%%%%%%%%%%%%%%%%%%%%%%
\begin{remark}\label{r_kernel}
As in \cite[Proposition 3.1]{kptt_adv}, we have that the following statements are equivalent for a selfadjoint operator space $\cl X$ and a selfadjoint subspace $\cl J\subseteq \cl X$:
\begin{enumerate}
\item $\cl J$ is a kernel; 
\item $\cl J = \ker \phi$ for a completely bounded completely positive map $\phi \colon \cl X \to \cl B(H)$;
\item $\cl J = \bigcap_{\al \in \bb{A}} \ker f_\al$ for a set $\{f_\al\}_{\al \in \bb{A}}$ of (completely) contractive and (completely) positive functionals on $\cl X$.
\end{enumerate}
Indeed, (i) is equivalent to (ii) after normalisation. 
If (i) holds, then (iii) holds for the family $\{f_\alpha \circ \phi\}_{\alpha \in \bb A}$, where $\{f_\alpha\}_{\alpha \in \bb A}$ is the family of all states on $\cl B(H)$.
If (iii) holds, then (i) is obtained by setting $\phi = \oplus_{\alpha \in \bb A} f_{\alpha}$.
\end{remark}

The space $\cl X/ \cl J$ is equipped with a natural involution given by 
\[
(x + \cl J)^* = x^* + \cl J
\foral x \in \cl X.
\] 
For every $n \in \bb N$, we define
\[
\|u + \cl J\|_{\rm sos}^{(n)}
:=
\sup\{ \|\phi^{(n)}(u)\| \ : \ \phi \text{ is a c.c.p. map with }  \phi|_{\cl J} = 0\}.
\]
Since $\cl J$ is a kernel, we have that each $\|\cdot\|_{\rm sos}^{(n)}$ is a norm.
We equip $\cl X/ \cl J$ with the matricial cones
\[
C_n(\cl X/\cl J) := \ol{q^{(n)}(M_n(\cl X)^+)}^{\nor{\cdot}_{\rm sos}^{(n)}}, \ \ \ n\in \bb{N}.
\] 
It is shown in \cite[Section 8]{kkm} that $\cl X/ \cl J$ is then a selfadjoint operator space with the following universal property: if $\phi \colon \cl X \to \cl B(H)$ is a completely contractive completely positive map with $\phi|_{\cl J} = 0$, then there exists a (necessarily unique) completely contractive completely positive map $\wt{\phi}$ such that the diagram
\[
\xymatrix{
\cl X \ar[rr]^{\phi} \ar[d]^{q} & & \cl B(H) \\
\cl X/ \cl J \ar[urr]_{\wt{\phi}} & &
}
\]
is commutative.
Due to the universal property of the quotient, we have that $u + M_n(\cl J) \in C_n(\cl X/ \cl J)$ if and only if $\phi^{(n)}(u) \geq 0$ for every completely contractive completely positive map $\phi$ of $\cl X$ with $\phi|_{\cl J} = 0$.

%%%%%%%%%%%%%%%%%%%%%%%%%%%%
\begin{proposition}\label{p_quosos}
Let $\cl X$ be a selfadjoint operator space and let $\cl J \subseteq \cl X$ be a kernel.
Consider the partial unitisation
\[
0 \longrightarrow \cl X \stackrel{\iota}{ \longrightarrow} \cl X^\# \stackrel{\tau}{ \longrightarrow} \bb C \to 0
\]
of $\cl X$.
Then $\iota(\cl J)$ is a kernel of $\cl X^\#$, and the canonical map
\[
\cl X/ \cl J \to \cl X^\#/\iota(\cl J); \ x + \cl J \mapsto \iota(x) + \iota(\cl J)
\]
is a completely isometric complete order embedding, that promotes to a complete order isomorphism $(\cl X/ \cl J)^\# \simeq \cl X^\# / \iota(\cl J)$.
\end{proposition}

\begin{proof}
Let $\phi \colon \cl X \to \cl Y$ be a completely contractive completely positive map, such that $\cl J = \ker \phi$.
Then $\iota(\cl J) = \ker \phi^\#$ for the unital completely positive extension $\phi^\# \colon \cl  X^\# \to \cl Y^\#$.
Hence $\cl X^\#/ \iota(\cl J)$ is an operator system.
In order to make a distinction we write
\[
q_{\cl J} \colon \cl X \to \cl X/ \cl J
\qand
q_{\iota(\cl J)} \colon \cl X^\# \to \cl X^\# / \iota(\cl J)
\]
for the corresponding quotient maps.

By the universal property of $\cl X/\cl J$, there exists a completely contractive completely positive map $\wt{q_{\iota(\cl J)} \circ \iota}$ that makes the diagram
\[
\xymatrix{
\cl X \ar[rr]^{\iota} \ar[d]^{q_{\cl J}} & & \cl X^\# \ar[rr]^{q_{\iota(\cl J)}} & & \cl X^{\#}/\iota(\cl J) \\
\cl X/\cl J  \ar[urrrr]_{\wt{q_{\iota(\cl J)} \circ \iota}} & & & &
}
\]
commutative.
We note that if $\phi$ is a completely contractive completely positive map on $\cl X$ with $\phi|_{\cl J} = 0$, then $\phi^\#|_{\iota(\cl J)} = 0$, and thus
\[
\|\phi^{(n)}(x)\| = \|(\phi^\#)^{(n)}(\iota^{(n)}(x))\| \leq \|\iota^{(n)}(x) + M_n(\iota(\cl J))\|_{\rm osy}^{(n)}
\]
for every $x \in M_n(\cl X)$.
Taking suprema gives that
\[
\|x + M_n(\cl J)\|_{\rm sos}^{(n)} \leq \|\iota^{(n)}(x) + M_n(\iota(\cl J))\|_{\rm osy}^{(n)},
\]
and therefore $\wt{q_{\iota(\cl J)} \circ \iota}$  is a complete isometry.

To complete the proof of the first part, let $u + M_n(\cl J) \in M_n(\cl X/\cl J)$ be such that $\iota^{(n)}(u) + M_n(\iota(\cl J)) \geq 0$.
Let $\phi$ be a completely contractive completely positive map such that $\phi|_{\cl J} = 0$, and let $\phi^\#$ be the unital completely positive extension on $\cl X^\#$, so that $\phi^\#|_{\iota(\cl J)} = 0$.
Since $\iota^{(n)}(u) + M_n(\iota(\cl J)) \geq 0$, we have 
\[
\phi^{(n)}(u) = (\phi^\#)^{(n)}(\iota^{(n)}(u)) \geq 0,
\]
where we used the description of the cones for a quotient operator system.
By the description of the cones in the quotient selfadjoint operator space we conclude that $u + M_n(\cl J) \in C_n(\cl X/\cl J)$.
Hence the completely isometric completely positive map $\wt{q_{\iota(\cl J)} \circ \iota}$ is a complete order embedding. 

Finally we show that $\cl X^\# / \iota(\cl J)$ has the universal property of the partial unitisation.
Since $\cl X^\#/\iota(\cl J) \simeq (\iota(\cl X)/\iota(\cl J)) \oplus \bb C$ we have a split exact sequence
\[
0 \longrightarrow \cl X/ \cl J \longrightarrow \cl X^\#/\iota(\cl J) \longrightarrow \bb C \to 0,
\]
where we have shown that the first arrow is a complete order monomorphism, and the second arrow unital.
Now suppose that $(\cl Y,e)$ is a (unital) operator system and $\phi \colon M_n(\cl X/\cl J) \to \cl Y$ is a completely contractive completely positive map.
We have to show that the extension $\phi^\# \colon M_n(\cl X^\#/\iota(\cl J)) \to \cl Y$ given by
\[
\phi^\#(\iota^{(n)}(x) + \la(1 \otimes I_n) + M_n(\iota(\cl J))) = \phi(x + M_n(\cl J)) + \la(e \otimes I_n),
\]
for $x \in M_n(X)$ and $\la \in \bb C$, is completely positive.
Towards this end we note that 
\[
\phi^\# \circ q^{(n)}_{\iota(\cl J)} = \left( \phi \circ \wt{q^{(n)}_{\iota(\cl J)} \circ \iota^{(n)}} \circ q^{(n)}_{\cl J} \right)^\#,
\]
and thus $\phi^\# \circ q^{(n)}_{\iota(\cl J)}$ is completely positive.
By the identification of the matricial cones in $\cl X^\#/\iota(\cl J)$ from $\cl X^\#$ obtained in Proposition \ref{p_quoos}, we then obtain that $\phi^\#$ is completely positive, and the proof is complete. 
\end{proof}

%%%%%%%%%%%%%%%%%%%%%%%%%%%%
\subsection{Selfadjoint operator $\cl A$-spaces}\label{ss_soap}
%%%%%%%%%%%%%%%%%%%%%%%%%%%%

Henceforth we assume that $\cl X$ is a selfadjoint operator space and $\cl A$ is a unital C*-algebra such that $\cl X$ is an $\cl A$-bimodule with unital modular actions. 
We wish to obtain a concrete $\cl A$-representation of $\cl X$, and to point out a suitable $\cl A$-bimodule structure on quotient spaces of $\cl X$, in the presence of some natural compatibility conditions.

%%%%%%%%%%%%%%%%%%%%%%%%%%%%
\begin{definition}
Let $\cl X \subseteq \cl B(H)$ be a selfadjoint operator space and $\cl A$ be a unital C*-algebra such that $\cl X$ is an $\cl A$-bimodule with unital modular actions.
We say that $\cl X$ is a \emph{concete selfadjoint operator $\cl A$-space} if there is a $*$-representation $\pi \colon \cl A \to \cl B(H)$ such that
\[
\pi(a) x = a \cdot x
\qand
x \pi(a) = x \cdot a
\]
for all $x \in \cl X$ and $a \in \cl A$.
\end{definition}

We also have an abstract definition of a selfadjoint operator $\cl A$-space.

%%%%%%%%%%%%%%%%%%%%%%%%%%%%
\begin{definition}\label{d_sosbim}
Let $\cl X$ be a selfadjoint operator space (resp. matrix ordered operator space) and $\cl A$ be a unital C*-algebra such that $\cl X$ is an $\cl A$-bimodule with unital modular actions.
We say that $\cl X$ is an \emph{abstract selfadjoint operator $\cl A$-space} (resp. \emph{matrix ordered operator $\cl A$-space}) if:
\begin{enumerate}
\item $a^* \cdot M_m(\cl X)^+ \cdot a \subseteq M_n(\cl X)^+$ for all $a \in M_{m,n}(\cl A)$;
\item $\cl X$ is an operator $\cl A$-bimodule, that is, the modular actions are completely contractive;
\item $(a \cdot x)^* = x^* \cdot a^*$ for all $a \in \cl A$ and all $x \in \cl X$.
\end{enumerate}
\end{definition}

An \emph{$\cl A$-representation} of a matrix ordered operator $\cl A$-space (or a \emph{C*-repre\-sentation} if $\cl A$ is clear from the context) is a pair $(\phi,\pi)$, where $\phi \colon \cl X \to \cl B(H)$ is a completely positive map, $\pi \colon \cl A \to \cl B(H)$ is a unital $*$-represen\-tation for some Hilbert space $H$, and 
\begin{equation}\label{eq_Apair}
\phi(a\cdot x) = \pi(a)\phi(x), \ \ \ x\in \cl X, a\in \cl A.
\end{equation}
An $\cl A$-representation of $\cl X$ will be called \emph{completely contractive} (resp. \emph{completely isometric}) if $\phi$ is completely contractive (resp. completely isometric). 

%%%%%%%%%%%%%%%%%%%%%%%%%%%%
\begin{remark} \label{r_comp}
We pass to non-degenerate compressions of C*-represen\-tations as follows.
If $\cl X$ is a selfadjoint operator $\cl A$-space, and $(\phi, \pi)$ is a completely contractive $\cl A$-representation we can take the non-degenerate compression $\phi'$ of $\phi$ on the subspace $[\phi(\cl X) H]$ so that
\[
\phi(x) = \begin{bmatrix} \phi'(x) & 0 \\ 0 & 0 \end{bmatrix}, \ \ \ x \in \cl X.
\]
It follows that $[\phi(\cl X) H]$ is reducing for $\pi$ and $\pi(1_{\cl A}) \phi(x) \xi = \phi(x) \xi$ since the modular action is unital.
Hence we can restrict $\pi$ to the \emph{unital} $*$-representation $\pi'$ on $[\phi(\cl X) H]$.
Thus the pair $(\phi', \pi')$ is a completely contractive $\cl A$-representa\-tion, such that $\phi'$ is non-degenerate, $\|\phi'^{(n)}(x)\| = \|\phi^{(n)}(x)\|$ for all $x \in M_n(\cl X)$, and $\pi'$ is unital.
Note, however, that faithfulness of $\pi$ does not imply faithfulness of $\pi'$.
\end{remark}

It is clear that, if $\cl A, \cl X \subseteq \cl B(H)$ with $\cl X^* = \cl X$ and $\cl A \cl X \subseteq \cl X$, then $\cl X$ is a selfadjoint operator $\cl A$-space.
Our goal is to show that selfadjoint operator 
$\cl A$-spaces are characterised by the existence of a concrete representation.
Towards this end, we follow the approach of the partial unitisation from \cite{werner}.

Let $\cl X$ be a matrix ordered operator $\cl A$-space for a unital C*-algebra $\cl A$.
Let $\cl X^\#_{\cl A}$ be the direct linear space sum $\cl X \oplus \cl A$ endowed with the involution
\[
(x, a)^* := (x^*, a^*), \ \ x \in \cl X, a \in \cl A.
\]
For every $n \in \bb N$ we write
\[
M_n(\cl X^\#_{\cl A})_h := M_n(\cl X)_h \oplus M_n(\cl A)_h.
\]
The matrix cone structure $\{M_n(\cl X^\#_{\cl A})^+\}_{n \in \bb N}$ is given as follows: an element $(x, a) \in M_n(\cl X^\#_{\cl A})_h$ is in $M_n(\cl X^\#_{\cl A})^+$ if and only if $a \geq 0$ and
\[
\varphi(a_\eps^{-1/2} x a_\eps^{-1/2}) \geq -1 \foral \eps > 0 \text{ and all } 
\varphi \in {\rm CCP}(M_n(\cl X), \bb C),
\]
where $a_\eps:= a + \eps (1_{\cl A} \otimes I_n)$.
Moreover we define the $\cl A$-modular actions
\[
b \cdot (x, a) := (b \cdot x, ba) \qand (x, a) \cdot b := (x \cdot b, ab)
\]
for all $a,b \in M_{n}(\cl A)$, $x \in M_n(\cl X)$ and $n \in \bb N$.
We refer to $\cl X^\#_{\cl A}$ as the \emph{$\cl A$-modular partial unitisation} of $\cl X$.

%%%%%%%%%%%%%%%%%%%%%%%%%%%%
\begin{remark}
Let $\cl A$ be a unital C*-algebra and $\cl X$ be a matrix ordered operator $\cl A$-space.
If $a\in M_n(\cl A)^+$, then the condition $\varphi(a_\eps^{-1/2} x a_\eps^{-1/2}) \geq -1$ for all $\vphi \in {\rm CCP}(M_n(\cl X), \bb C)$ is equivalent to the condition $\varphi(a_\eps^{-1/2} x a_\eps^{-1/2}) \geq -1$ for all $\vphi \in {\rm CCP}(M_n(\cl X), \bb C)$  with $\|\vphi\| = 1$, since normalisations of completely positive maps are completely positive.

Furthermore, if $\vphi \in {\rm CCP}(M_n(\cl X), \bb C)$, then for $a \in M_{m,n}(\cl A)$ with $\|a\| \leq 1$ we can define the functional $\vphi_a$ given by
\[
\vphi_a(x) := \vphi(a^* x a), \ \ x \in M_m(\cl X).
\]
Since the modular actions are completely contractive we get $\|\vphi_a\| \leq \|a\|^2 \leq 1$, so that $\vphi_a$ is (completely) contractive.
Since $a^* M_m(\cl X)^+ a \subseteq M_n(\cl X)^+$, we conclude that $\vphi_a \in {\rm CCP}(M_n(\cl X), \bb C)$.
\end{remark}

For the next proposition we will use two key observations.
First recall that, if $\cl X$ is a matrix ordered operator space and $x \in M_n(\cl X)_h$, then $x \geq 0$ if and only if $\vphi(x) \geq 0$ for all $\vphi \in {\rm CCP}(M_n(\cl X), \bb C)$; see for example \cite[Lemma 3.1]{werner}.
Moreover, if $\vphi \in {\rm CCP}(M_n(\cl X), \bb C)$, then we can write $\vphi = (\vphi_{i,j})_{i,j}$ where $\vphi_{i,j}^* = \vphi_{j,i}$ are bounded functionals on $\cl X$, $\vphi_{i,i} \geq 0$ and $\sum_{i=1}^n \|\vphi_{i,i}\| \leq \|\vphi\| \leq 1$; see for example \cite[Lemma 2.7]{werner}.

%%%%%%%%%%%%%%%%%%%%%%%%%%%%
\begin{proposition}\label{p_pumooas}
Let $\cl A$ be a unital C*-algebra and $\cl X$ be a matrix ordered operator $\cl A$-space.
Then the $\cl A$-modular partial unitisation $\cl X^\#_{\cl A}$ is an operator $\cl A$-system that contains the partial unitisation $\cl X^\#$ of $\cl X$.
\end{proposition}

\begin{proof}
We first note that, if $\cl X^\#_{\cl A}$ is shown to be an operator system, then by definition the inclusion map $\cl X^\# \hookrightarrow \cl X^\#_{\cl A}$ is automatically a unital complete order embedding.
We hereafter show that $\cl X^\#_{\cl A}$ is an operator $\cl A$-system.
The proof follows the arguments of \cite[Theorem 4.8]{werner}.

Note that, if $(\pm x, \pm a) \in M_n(\cl X^\#_{\cl A})^+$, then $\pm a \geq 0$, and so $a =0$.
Moreover, for every $\vphi \in {\rm CCP}(M_n(\cl X), \bb C)$ and $\eps>0$, we have 
\[
-\eps \leq \vphi((\eps^{1/2}(\pm a)_\eps^{-1/2}) (\pm x) (\eps^{1/2} (\pm a)_\eps^{-1/2})).
\]
As $a=0$ we have  $\eps^{1/2}(\pm a)_\eps^{-1/2} = \pm 1_{\cl A}$, and we conclude that $\vphi(\pm x) \geq -\eps$.
This implies that $\pm x \in M_n(\cl X)^+$ and so $x = 0$.
It follows that 
$$M_n(\cl X^\#_{\cl A})^+ \cap (-M_n(\cl X^\#_{\cl A})^+) = \{0\}.$$

By definition, we have $M_n(\cl X^\#_{\cl A})^+ \subseteq M_n(\cl X^\#_{\cl A})_h$ for all $n \in \bb N$.
For $(x, a) \in M_n(\cl X^\#_{\cl A})^+$, set $(\wt{x}, \wt{a}) \in M_{n+m}(\cl X^\#_{\cl A})$ by padding with zeroes, and note that 
\[
\wt{a}_\eps = a_\eps \oplus \eps(1_{\cl A} \otimes I_{m}),
\]
so that
\[
\wt{a}_\eps^{-1/2} \wt{x} \wt{a}_\eps^{-1/2} = 
\begin{pmatrix}
a_\eps^{-1/2} x a_\eps^{-1/2} & 0 \\ 0 & 0
\end{pmatrix}.
\]
For $\vphi \in {\rm CCP}(M_{n+m}(\cl X), \bb C)$ we can write $\vphi_{n}$ for the restriction of $\vphi$ to the $n \times n$ corner of $M_{n+m}(\cl X)$, which is in ${\rm CCP}(M_n(\cl X), \bb C)$.
Then
\[
\vphi(\wt{a}_\eps^{-1/2} \wt{x} \wt{a}_\eps^{-1/2}) = \vphi_{n}(a_\eps^{-1/2} x a_\eps^{-1/2}) \geq -1,
\]
since $(x, a) \geq 0$.
This shows that 
\begin{equation}\label{eq_ninm1}
\begin{pmatrix}
M_n(\cl X^\#_{\cl A})^+ & 0 \\ 0 & 0
\end{pmatrix}
 \subseteq M_{n+m}(\cl X^\#_{\cl A})^+.
\end{equation}
Similarly, 
\begin{equation}\label{eq_ninm2}
\begin{pmatrix}
0 & 0 \\ 0 & M_m(\cl X^\#_{\cl A})^+ 
\end{pmatrix}
\subseteq M_{n+m}(\cl X^\#_{\cl A})^+.
\end{equation}
Equations (\ref{eq_ninm1}) and (\ref{eq_ninm2}), together with the fact that 
$M_{n+m}(\cl X^\#_{\cl A})^+$ is a cone show that 
\[
M_n(\cl X^\#_{\cl A})^+ \oplus M_m(\cl X^\#_{\cl A})^+ \subseteq M_{n+m}(\cl X^\#_{\cl A})^+.
\]

Next we show that $b^* M_m(\cl X^\#_{\cl A})^+ b \subseteq M_n(\cl X^\#_{\cl A})^+$ for all $b \in M_{m,n}(\cl A)$.
Applying for scalar matrices $M_{m,n} \subseteq M_{m,n}(\cl A)$ 
then yields that $\cl X^\#_{\cl A}$ is a matrix ordered $*$-vector space.
Towards this end, let $(x, a) \in M_n(\cl X^\#_{\cl A})^+$ and $b \in M_{n,m}(\cl A)$.
Then $b^*(x, a) b = (b^* x b, b^* a b)$ and $b^* a b \geq 0$, since $a \geq 0$.
If $b=0$, then $b^* (x, a) b = 0$.
Otherwise, set
\[
\eps_1 := \eps \|b^* b\|^{-1} 
\text{ and }
c := a_{\eps_1}^{1/2}( b ( b^* a b)_{\eps_1}^{-1/2}) \in M_{m,n}(\cl A).
\]
We then have
\[
c^* (a_{\eps_1}^{-1/2} x a_{\eps_1}^{-1/2}) c = (b^* a b)_\eps^{-1/2} (b^* x b) (b^* a b)_\eps^{-1/2},
\]
and
\[
b^* a_{\eps_1} b = b^*ab + \eps \|b^* b\|^{-1} b^*b \leq (b^*ab)_\eps.
\]
Therefore
\begin{align*}
\|c\|^2 
& = \|(b^*ab)_\eps^{-1/2} (b^* a_{\eps_1} b) (b^*ab)_\eps^{-1/2}\| \\
& \leq \|(b^*ab)_\eps^{-1/2} (b^* a_{\eps} b) (b^*ab)_\eps^{-1/2}\| = 1,
\end{align*}
and therefore the functional $\vphi_c(x) := \vphi(c^* x c)$ is in ${\rm CCP}(M_m(\cl X), \bb C)$ whenever $\vphi \in {\rm CCP}(M_n(\cl X), \bb C)$.
We conclude
\[
\vphi((b^*ab)_\eps^{-1/2} (b^* x b) (b^*ab)_\eps^{-1/2})
=
\vphi_c(a_{\eps_1}^{-1/2} x a_{\eps_1}^{-1/2}) \geq -1,
\]
for all $\vphi \in {\rm CCP}(M_n(\cl X), \bb C)$, since $(x, a) \in M_m(\cl X^\#_{\cl A})^+$.
Hence we get $b^*(x, a) b \geq 0$, as required.

Finally, we show that $(0, 1_{\cl A})$ is an Archimedean matrix order unit for $\cl X^\#_{\cl A}$, concluding the proof that $\cl X^\#_{\cl A}$ is an operator $\cl A$-system.
For the matrix order, let $(x, a) \in M_n(\cl X^\#_{\cl A})_h$ and let $r = \|x\| + \|a\|$.
Then
\[
(x, a) + r(0, 1_{\cl A} \otimes I_n) = (x, \|x\| 1_{\cl A} \otimes I_n) + (0, a + \|a\| 1_{\cl A} \otimes I_n).
\]
Then we obtain
\[
(0, a + \|a\| 1_{\cl A} \otimes I_n) \geq 0
\text{ and }
x = x^*;
\]
further, for $\vphi \in {\rm CCP}(M_n(\cl X), \bb C)$, we have $|\vphi(x)| \leq \|x\|$.
Hence we get 
\begin{align*}
\vphi( (\|x\| 1_{\cl A} \otimes I_n)_\eps^{-1/2} x (\|x\| 1_{\cl A} \otimes I_n)_\eps^{-1/2} )
& = \\
& \hspace{-3cm} = (\|x\| + \eps)^{-1} \vphi(x) \geq -(\|x\| + \eps)^{-1} \|x\| \geq -1.
\end{align*}
Therefore $(x, \|x\| 1_{\cl A} \otimes I_n) \geq 0$, and so $(x, a) + r (0, 1_{\cl A} \otimes I_n)$ is positive as a sum of two positive elements.

For the Archimedean property, let $(x, a) \in M_n(\cl X^\#_{\cl A})_h$ and suppose that 
\[
(x, a + r 1_{\cl A} \otimes I_n) = (x, a) + r(0, 1_{\cl A} \otimes I_n) \geq 0,
\foral r > 0.
\]
Note that $(a + r 1_{\cl A} \otimes I_n)_\eps = a_{r + \eps}$ for every $\eps > 0$, and thus for every $\vphi \in {\rm CCP}(M_n(\cl X), \bb C)$ we have
\[
\vphi(a_{r+\eps}^{-1/2} x a_{r + \eps}^{-1/2}) = \vphi((a + r 1_{\cl A} \otimes I_n)_\eps^{-1/2} x (a + r 1_{\cl A} \otimes I_n)_\eps^{-1/2}) \geq -1.
\]
Since this holds for all $r, \eps >0$, and since $\lim_{r \to 0} a_{r+\eps}=a_\eps$, we derive the required $\vphi(a_{\eps}^{-1/2} x a_{\eps}^{-1/2}) \geq -1$ for all $\eps>0$, that is $(x, a) \geq 0$.
\end{proof}

%%%%%%%%%%%%%%%%%%%%%%%%%%%%
\begin{remark}\label{r_ciemA}
If $\cl X$ is an abstract selfadjoint operator $\cl A$-space, then we obtain the (completely isometric complete order) embeddings
\[
\cl X \hookrightarrow \cl X^\# \hookrightarrow \cl X_{\cl A}^\#.
\]
Hence, if $\phi$ is a completely contractive completely positive map of $\cl X$, then it has a unital completely positive extension on $\cl X^\#$, and a further unital completely positive extension $\phi^\#_{\cl A}$ on $\cl X_{\cl A}^\#$.
We are going to use this scheme to induce the bimodule structure to appropriate quotients.
\end{remark}

We now arrive at the main theorem of this subsection.

%%%%%%%%%%%%%%%%%%%%%%%%%%%%
\begin{theorem}\label{t_soap}
Let $\cl A$ be a unital C*-algebra and $\cl X$ be a selfadjoint operator space that is a bimodule over 
$\cl A$ with unital modular actions.
Then $\cl X$ is a selfadjoint operator $\cl A$-space if and only if there is a (non-degenerate) completely isometric complete order embedding $\ga \colon \cl X \to \cl B(H)$ and a unital $*$-representation $\pi \colon \cl A \to \cl B(H)$ such that $\pi(a) \ga(x) = \ga(a \cdot x)$ for all $x \in \cl X$.
\end{theorem}

\begin{proof}
Given such a pair $(\ga, \pi)$ it follows that $a^* M_m(\cl X)^+ a \subseteq M_n(\cl X)^+$ for all $a \in M_{m,n}(\cl A)$ since $\ga$ is a complete order isomorphism onto its range.
Moreover the modular actions are completely contractive since $\ga$ is complete isometric and $\pi$ is completely contractive.

Conversely, if $\cl X$ is a selfadjoint operator $\cl A$-space, then we have the (completely isometric complete order) embeddings
\[
\iota \colon \cl X \hookrightarrow \cl X^\# \hookrightarrow \cl X^\#_{\cl A},
\]
since $\cl X$ is a selfadjoint operator space (see Remark \ref{r_ciemA}). 
Let
\[
\ga \colon \cl X^\#_{\cl A} \to \cl I(\cl X^\#_{\cl A})
\qand
\pi \colon \cl A \to \cl I(\cl X^\#_{\cl A})
\]
be the $\cl A$-representation of $\cl X^\#_{\cl A}$, where $\ga$ is a complete order embedding, since $\cl X^\#_{\cl A}$ is an operator $\cl A$-system by Proposition \ref{p_pumooas}.
By the definition of the bimodule actions we see that $(\ga \circ \iota, \pi)$ is an $\cl A$-representation, and considering its non-degenerate compression completes the proof. 
\end{proof}

Note that in Theorem \ref{t_soap} we do not require $\ga$ be an embedding, although it can be chosen to be so.
This is consistent with the case where $\cl A = \bb C$, as there are completely isometric complete order embeddings of selfadjoint spaces that are not embeddings, see \cite[Example 2.14]{kkm}.

%%%%%%%%%%%%%%%%%%%%%%%%%%%%
\begin{proposition}\label{p_qasosem}
Let $\cl X$ be a selfadjoint operator $\cl A$-space and let $\cl J$ be a kernel of $\cl X$.
Consider the diagram
\[
0 \longrightarrow \cl X \stackrel{\iota_\cl A}{ \longrightarrow} \cl X_{\cl A}^\# \stackrel{\tau_\cl A}{ \longrightarrow} \cl A \to 0.
\]
We have that 
$\iota_{\cl A}(\cl J)$ is a kernel of $\cl X_{\cl A}^\#$, the quotient $\cl X_{\cl A}^\# / \iota_{\cl A}(\cl J)$ is an operator system, and the canonical map
\[
\cl X/ \cl J \to \cl X_{\cl A}^\# / \iota_{\cl A}(\cl J); \ x + \cl J \mapsto \iota_{\cl A}(x) + \iota_{\cl A}(\cl J)
\]
is an embedding.
\end{proposition}

\begin{proof}
Let $\phi$ be a completely contractive completely positive map of $\cl X$ such that $\cl J = \ker \phi$.
Let $\phi^\#$ be its unique unital completely positive extension on $\cl X^\#$.
Since $\cl X^\# \subseteq \cl X^\#_{\cl A}$, by Arveson's Extension Theorem there exists a unital completely positive extension $\phi^\#_{\cl A}$ of $\phi^\#$ on $\cl X_{\cl A}^\#$.
We have that 
\[
\iota_\cl A(\cl J) = \ker (\phi^\#_{\cl A} \oplus \tau_{\cl A}).
\]
Thus $\iota_{\cl A}(\cl J)$ is a kernel of $\cl X_{\cl A}^\#$, and we can consider the quotient operator system $\cl X_{\cl A}^\# / \iota_{\cl A}(\cl J)$.

Since $\iota(\cl J) = \iota_{\cl A}(\cl J)$, by definition we have that the canonical quotient map 
\[
\cl X^\# / \iota(\cl J) \to \cl X^\#_{\cl A} / \iota_{\cl A}(\cl J)
\]
is a complete order embedding between operator systems.
Therefore the extension of the canonical quotient map
\[
\cl X/ \cl J \to \cl X_{\cl A}^\# / \iota_{\cl A}(\cl J)
\]
to its unitisation is a complete order embedding, since $(\cl X/ \cl J)^\# \simeq \cl X^\# / \iota(\cl J)$ by Proposition \ref{p_quosos}, and the proof is complete.
\end{proof}

%%%%%%%%%%%%%%%%%%%%%%%%%%%%
\begin{corollary}\label{c_qasosbm}
Let $\cl X$ be a selfadjoint operator $\cl A$-space and let $\cl J$ be a kernel of $\cl X$ such that $\cl A \cdot \cl J \subseteq \cl J$.
Then $\cl X/ \cl J$ is a selfadjoint operator $\cl A$-space in a canonical fashion.
\end{corollary}

\begin{proof}
By using selfadjointness and the third axiom for selfadjoint operator $\cl A$-spaces we have $\cl J \cdot \cl A \subseteq \cl J$, and thus the bimodule actions descend to the quotient $\cl X/ \cl J$.
On the other hand, $\iota_{\cl A}(\cl J)$ is also a kernel of $\cl X^\#_{\cl A}$ and thus by Theorem \ref{th_modquo} we have that $\cl X^\#_{\cl A} / \iota_{\cl A}(\cl J)$ is an operator $\cl A$-system.

By choosing a concrete realisation of $\cl X^\#_{\cl A} / \iota_{\cl A}(\cl J)$ provided by Theorem \ref{t_opA-sys}, and the completely isometric complete order embedding of $\cl X/ \cl J$ in $\cl X_{\cl A}^\# / \iota_{\cl A}(\cl J)$, we obtain a realisation of $\cl X/ \cl J$ as a concrete selfadjoint operator $\cl A$-space, and Theorem \ref{t_soap} completes the proof.
\end{proof}

%%%%%%%%%%%%%%%%%%%%%%%%%%%%
\section{Symmetrisation}\label{s_univco}
%%%%%%%%%%%%%%%%%%%%%%%%%%%%

In this section we construct a canonical selfadjoint operator space associated with an operator space $\cl E$ and an operator system $\cl S$, which we call the \emph{symmetrisation of $\cl E$ by $\cl S$}.
We show that this selfadjoint operator space enjoys a natural universal property, and we establish its injectivity. 
At the base of our approach is a factorisation result for completely contractive completely positive balanced trilinear maps.
We will return to discussing a balanced form of 
symmetrisation in Section \ref{s_balsym}.

%%%%%%%%%%%%%%%%%%%%%%%%%%%%
\subsection{Multilinear maps}
%%%%%%%%%%%%%%%%%%%%%%%%%%%%

For $n\in \bb{N}$, we write $[n] = \{1,2,\dots,n\}$. 
The algebraic tensor product of vector spaces $V$ and $W$ will be denoted by $V\odot W$. 
If $x = (x_{i,j})_{i,j} \in M_{m,n}(V)$ and $y = (y_{i,j})_{i,j}\in M_{n,k}(W)$, let $x\odot y$ denote the matrix $(z_{i,j})_{i,j} \in M_{m,k}(V\odot W)$ with 
\[
z_{i,j} = \sum_{p=1}^n x_{i,p}\otimes y_{p,j}, \ \ \ i\in [m], j\in [k].
\]
The canonical identification $\bb{C}\odot V \odot \bb{C} = V$ yields a map 
\[
M_{m,n}\times M_{n,k}(V) \times M_{k,\ell} \to M_{m,\ell}(V); \ (\alpha,x,\beta) \mapsto 
\alpha\cdot x \cdot \beta := \alpha\odot x \odot \beta.
\] 

Let $V, V_i$, $i = 1,\dots,N$, be complex vector spaces. 
A multilinear map $\theta \colon V_1\times V_2\times \cdots \times V_N \to V$ gives rise to a multilinear map (denoted in the same way)
\[
\theta \colon M_{k_1,k_2}(V_1) \times M_{k_2,k_3}(V_2) \times \cdots \times M_{k_{N-1},k_N}(V_N) \to M_{k_1,k_N}(V)
\]
by letting 
\begin{align*}
\theta\left((v_{i_1,i_2}^{(1)})_{i_1,i_2},(v_{i_2,i_3}^{(2)})_{i_2,i_3}, \dots, (v_{i_{N-1},i_{N}}^{(N)})_{i_{N-1}, i_N}\right) 
& = \\
& \hspace{-5cm} = 
\left(\sum_{i_{N-1} = 1}^{k_{N-1}} \cdots \sum_{i_2 =1}^{k_2} \theta\left(v_{i_1,i_2}^{(1)},v_{i_2,i_3}^{(2)}, \dots, v_{i_{N-1},i_{N}}^{(N)}\right)\right)_{i_1,i_N}.
\end{align*}
Let $\cl A_j$, $j = 1, \dots, N$, be algebras and suppose that each $V_i$ is an $\cl A_j$-$\cl A_{j+1}$-bimodule for $j = 1, \dots, N-1$, and $V$ is an $\cl A_1$-$\cl A_N$-module; then we say that $\theta$ is a module map over $\cl A_1,\dots,\cl A_N$, if 
\[
\theta(a_1\cdot v_1\cdot a_2, v_2\cdot a_3, \dots, v_N\cdot a_N) 
= 
a_1\cdot \theta(v_1, a_2 \cdot v_2, \dots, a_{N-1} \cdot v_N)\cdot a_N,
\]
for all $v_i\in V_i$, $a_j\in \cl A_j$.
Here, we will be concerned exclusively with the cases where $N=2,3$.

Assume that $V_1$ and $V_2$ are operator spaces.
The norm of a bilinear map $\theta \colon V_1 \times V_2 \to V$ is defined by letting
\[
\|\theta\| := \inf\{ C \in \bb R^+ \ : \ \|\theta(v_1, v_2)\| \leq C \cdot \|v_1\|_{V_1} \cdot \|v_2\|_{V_2} \}.
\]
The map $\theta$ is called \emph{completely bounded} if 
its completely bounded norm (abbreviated cb-norm)
\[
\|\theta\|_{\rm cb} := \sup_{n\in \bb{N}} \|\theta \colon M_n(V_1) \times M_n(V_2) \to M_n(V) \|
\]
is finite; see, for example, \cite[Subsection 1.5.4]{blm}.
These definitions extend to multilinear maps, see for example \cite[Section 9.4]{er}. 

For two operator spaces $\cl E$ and $\cl F$ we denote by $\cl E \otimes_{\rm h} \cl F$ their \emph{Haagerup tensor product};
by its definition, the Haagerup tensor product norm of an element $u \in M_n(\cl E \odot \cl F)$ is given by
\[
\|u\|_{\rm h}^{(n)} := \inf \{ \nor{x} \cdot \nor{y} \ : \ u = x \odot y, \ x \in M_{n,k}(\cl E), y \in M_{k,n}(\cl F), k \in \bb N \},
\]
see for example \cite[Section 5]{pisier_intr}.
The Haagerup tensor product linearises bilinear completely bounded maps $u \colon \cl E \times \cl F\to \cl B(H)$ to completely bounded maps $\wt{u} \colon \cl E \otimes_{\rm h} \cl F\to \cl B(H)$ with preservation of cb-norm, see for example \cite[Subsection 1.5.4]{blm}.
Its associativity allows to extend this property to multilinear completely bounded maps, see for example \cite[Section 9.4]{er}.

We note that, if $\theta \colon V \times V \to W$ is a bilinear map for $*$-vector spaces $V$ and $W$, then the polarisation identity reads
\begin{align*}
4 \theta(v_1^*, v_2) 
& = \theta( (v_1 + v_2)^*, v_1 + v_2) - \theta((v_1-v_2)^*, v_1-v_2)  + \\ 
& \hspace{.5cm} + i \theta( (v_1 - i v_2)^*, v_1 - i v_2) - i \theta( (v_1 + i v_2)^*, v_1 + i v_2).
\end{align*}
Consequently, if $\theta(v^*, v)^* = \theta(v^*, v)$ for all $v \in V$, then $\theta(v_1^*, v_2)^* = \theta(v_2^*, v_1)$ for all $v_1, v_2 \in V$.

%%%%%%%%%%%%%%%%%%%%%%%%%%%%
\subsection{Factorisation of trilinear maps}
%%%%%%%%%%%%%%%%%%%%%%%%%%%%

%Let $\cl E$ be an (abstract) operator space.  
%We denote by $\cl E^*$ the adjoint of $\cl E$, i.e., $\cl E^*$ is the (abstract) operator space, completely isometric to $\phi(\cl E)^*$, for any complete isometry $\phi \colon \cl E\to \cl B(H,K)$. 
%For a linear map $\phi \colon \cl E \to \cl B(H,K)$, let $\phi^* \colon \cl E^*\to \cl B(K,H)$ be the map given by 
%\[
%\phi^*(x^*) := \phi(x)^*, \ \ \ x\in \cl E.
%\]
Let $\cl E$ and $\cl F$ be (abstract) operator spaces.  
Given linear maps $\phi \colon \cl E\to \cl B(H)$ and $\psi \colon \cl F\to \cl B(H)$, we will write $\psi\cdot\phi \colon \cl F\odot\cl E \to \cl B(H)$ for the linear map given by 
\begin{equation}\label{eq_dot}
(\psi\cdot\phi)(y\otimes x) := \psi(y)\phi(x), \ \ \ x\in \cl E, y\in \cl F.
\end{equation}
If $\cl G$ is a(nother) operator space and $\theta \colon \cl G\to \cl B(H)$ is a linear map, we set 
$\theta \cdot \psi \cdot \phi := \theta \cdot (\psi \cdot \phi) = (\theta \cdot \psi) \cdot\phi$.
We note that, if $u\in M_n(\cl G\odot \cl F\odot\cl E)$, $\alpha \in M_{m,n}$ and $\beta\in M_{n,m}$, then 
\begin{equation}\label{eq_modpr}
(\theta\cdot \psi\cdot\phi)^{(m)}(\alpha u \beta) = \alpha (\theta\cdot \psi\cdot\phi)^{(n)}(u) \beta.
\end{equation}

Let $\cl S$ and $\cl T$ be operator systems and let $\cl E$ be an operator space.
Recall the operator space structure of the adjoint $\cl E^*$
from Subsection \ref{ss_opspsy}.
We say that a trilinear map $\theta \colon \cl E^* \times \cl S \times \cl E \longrightarrow \cl T$ is \emph{positive}, if 
\[
\theta(x^*, \cl S^+, x) \subseteq \cl T^+, \ \  x \in \cl E.
\]
We record the following remark for positive maps for further use.

%%%%%%%%%%%%%%%%%%%%%%%%%%%%
\begin{remark}\label{r_adjoint} 
If $\cl E$ is an operator space, $\cl S$ and $\cl T$ are operator systems and $\theta \colon \cl E^* \times \cl S \times \cl E \longrightarrow \cl T$ is a positive map, then
\[
\theta(y^*, s, x)^* = \theta(x^*, s^*, y), \ \ x,y \in \cl E, s \in \cl S
\]
(that is, positive trilinear maps are automatically compatible with taking adjoints).
Indeed, first note that $\theta(x^*, s, x)^* = \theta(x^*, s, x)$ for all $s \geq 0$ and $x \in \cl E$.
Next, for $s \in \cl S_h$, write $s = s_1 - s_2$ for $s_1, s_2 \geq 0$ in $\cl S$ and we have
\begin{align*}
\theta(x^*, s, x)^*
& =
\theta(x^*, s_1, x)^* - \theta(x^*, s_2, x)^* \\
& = 
\theta(x^*, s_1, x) - \theta(x^*, s_2, x)
=
\theta(x^*, s, x),
\end{align*}
for all $x \in \cl E$.
Next, let $s \in \cl S_h$ and consider the bilinear map
\[
\theta_s \colon \cl E^* \times \cl E \to \cl T; (x^*,y) \mapsto \theta(x^*, s, y).
\]
Due to the positivity of $\theta$ we have
\[
\theta_s(x^*, x)^* = \theta(x^*, s, x)^* = \theta(x^*, s, x) = \theta_s(x^*, x).
\]
The polarisation identity for $\theta_s$ then implies that
\[
\theta(y^*, s, x)^* = \theta_s(y^*, x)^* = \theta_s(x^*,y) = \theta(x^*, s, y).
\]
Finally, for an arbitrary $s \in \cl S$, write $s = s_1 + is_2$ for $s_1, s_2 \in \cl S_h$ and compute
\begin{align*}
\theta(y^*, s, x)^*
& =
\theta(y^*, s_1, x)^* - i \theta(y^*, s_2, x)^* \\
& =
\theta(x^*, s_1, y) - i \theta(x^*, s_2, y)
=
\theta(x^*, s^*, y),
\end{align*}
for all $x,y \in \cl E$. 
\end{remark}

We will consider trilinear maps that are positive at every matrix level.

%%%%%%%%%%%%%%%%%%%%%%%%%%%%
\begin{definition}
Let $\cl S$ and $\cl T$ be operator systems and let $\cl E$ be an operator space.
We say that a trilinear map $\theta \colon \cl E^* \times \cl S \times \cl E \to \cl T$ is \emph{completely positive} if the induced trilinear map (denoted by the same symbol)
\[
\theta \colon M_{n,m}(\cl E^*)\times M_m(\cl S) \times M_{m,n}(\cl E) \to M_n(\cl T)
\]
is positive for all $n,m\in \bb{N}$, that is, if
\begin{equation}\label{eq_npos}
\theta(x^*, s, x)\in M_n(\cl T)^+, \ \ x \in M_{m,n}(\cl E), s \in M_m(\cl S)^+, n,m\in \bb{N}.
\end{equation}

Let $\cl A$ be a unital C*-algebra.
If $\cl E$ is a left operator $\cl A$-space and $\cl S$ is an operator $\cl A$-system, we say that a map $\theta \colon \cl E^* \times \cl S \times \cl E \longrightarrow \cl T$ is \emph{$\cl A$-balanced} if 
\[
\theta(y^*, a \cdot s \cdot b, x) = \theta(y^*\cdot a, s, b\cdot x), \ \ \ x, y\in \cl E, s\in \cl S, a,b\in \cl A.
\]
We write $\wt{\theta}_{\cl A}$ for the linearisation of $\theta$ to a map
\[
\wt{\theta}_{\cl A} \colon \cl E^* \odot^{\cl A} \cl S \odot^{\cl A} \cl E\to \cl B(H).
\]
When $\cl A = \bb C$ we will simply write $\wt{\theta}$.
\end{definition}

There is a straightforward way for producing completely bounded completely positive $\cl A$-balanced trilinear maps.

%%%%%%%%%%%%%%%%%%%%%%%%%%%%
\begin{definition}
Let $\cl A$ be a unital C*-algebra, $\cl E$ be a left operator $\cl A$-space, $\cl S$ be an operator $\cl A$-system.
Let $\phi \colon \cl E\to \cl B(H,K)$ be a completely bounded map and $\psi \colon \cl S\to \cl B(K)$ be a unital completely positive map.
We say that the pair $(\phi,\psi)$ is \emph{$\cl A$-admissible} if 
\[
\psi(s \cdot a ) \phi(x) = \psi(s) \phi(a \cdot x), \ \ s \in \cl S, a \in \cl A, x \in \cl E.
\]
\end{definition}

If $(\phi,\psi)$ is an $\cl A$-admissible pair for $\cl E$ and $\cl S$, then we can define the map
\[
\theta \colon \cl E^* \times \cl S \times \cl E \to \cl B(H); \ (x^*,s,y) \mapsto \phi(x)^* \psi(s) \phi(y).
\]
It is readily verified that $\theta$ is a completely bounded $\cl A$-balanced trilinear map with $\nor{\theta}_{\rm cb} \leq \nor{\phi}_{\rm cb}^2$, that is by definition completely positive.
Moreover, its linearisation 
$\wt{\theta}_{\cl A}$ satisfies $\wt{\theta}_{\cl A} = \phi^*\cdot \psi\cdot\phi$.

The next lemma is a factorisation result which shows that maps described in the previous paragraph form the only class of examples of completely bounded completely positive $\cl A$-balanced trilinear maps.
It is akin to the factorisation results proved by E. Christensen and A. Sinclair \cite{christensen-sinclair}, V. Paulsen and R. R. Smith \cite{paulsen-smith}, and A. Sinclair and R. R. Smith \cite{ss}
(see for example \cite[Theorem 1.5.7, Subsection 1.5.8]{blm}). 

%%%%%%%%%%%%%%%%%%%%%%%%%%%%
\begin{lemma}\label{l_ssgen}
Let $\cl A$ be a unital C*-algebra, $\cl E$ be a left operator $\cl A$-space, $\cl S$ be an operator $\cl A$-system and $\theta \colon \cl E^*\times\cl S\times\cl E\to \cl B(H)$ be a completely bounded completely positive $\cl A$-balanced trilinear map. 
Then the following hold:
\begin{enumerate}
\item[(i)]
There exist a completely bounded map $\phi \colon \cl E\to \cl B(H,K)$ and a unital completely positive map $\psi \colon \cl S\to \cl B(K)$ such that the pair $(\phi,\psi)$ is $\cl A$-admissible and 
$\wt{\theta}_{\cl A} = \phi^*\cdot \psi\cdot\phi$.
Moreover, $\phi$ and $\psi$ can be chosen so that $\|\theta\|_{\rm cb} = \|\phi\|_{\rm cb}^2$.

\item[(ii)]
If, in addition, $\cl A \subseteq \cl S$, then we can choose $\psi$ 
in item (i) so that $\pi := \psi|_{\cl A}$ is a $*$-representation, $(\phi,\pi)$ is an $\cl A$-representation of $\cl E$, and $(\psi,\pi)$ is an $\cl A$-representation of $\cl S$.

\item[(iii)]
If $\cl S$ is a C*-algebra, then the map $\psi$ in item (i) can be chosen to be a $*$-representation. 
\end{enumerate}
\end{lemma}

\begin{proof}
We equip the algebraic tensor product $\cl E\odot H$ with the sesquilinear form given by 
\[
\langle x\otimes \xi, y \otimes \eta \rangle := \langle \theta(y^*,1_{\cl S},x) \xi, \eta \rangle_H, \ \ \ x,y\in \cl E, \xi,\eta\in H,
\]
and let 
\[
N := \{u\in \cl E\odot H \ : \ \langle u,u\rangle = 0\}.
\]
A standard argument shows that $N$ is a linear subspace of $\cl E\odot H$ and that the quotient $\cl E\odot H/N$ is an inner product space. 
We write $K$ for the resulting Hilbert space completion with inner product $\sca{\cdot, \cdot}_K$ and associated norm $\|\cdot\|_K$. 

Next we define a unital completely positive map $\psi \colon \cl S \to \cl B(K)$ such that
\begin{equation}\label{eq_psis}
\sca{\psi(s) (x \otimes \xi + N), y \otimes \eta + N}_K = \sca{ \theta( y^*, s, x ) \xi, \eta}_H,
\end{equation}
$s\in \cl S$, $x,y\in \cl E$, $\xi,\eta\in H$. 
Towards this end, we first assume that $s \in \cl S^+$ and define a sesquilinear form on $\cl E \odot H$ by letting
\[
\sca{x \otimes \xi, y \otimes \eta}_s := \sca{ \theta( y^*, s, x ) \xi, \eta}_H.
\] 
For $u = \sum_{k=1}^n x_k \otimes \xi_k$, let $X = [x_1,\dots,x_n]\in M_{1,n}(\cl E)$ and $\xi = (\xi_i)_{i=1}^n\in H^n$.
We then have
\begin{align*}
\sca{u,u}_s
& =
\sum_{k, \ell=1}^n \sca{\theta(x_\ell^*, s, x_k) \xi_k, \xi_\ell}
=
\sca{\theta(X^*, s, X) \xi, \xi} \geq 0,
\end{align*} 
where we used that $s \in \cl S^+$ and $\theta$ is completely positive; thus the form $\sca{\cdot,\cdot}_s$ is positive semi-definite.
Moreover, since $0 \leq s \leq \|s\| 1$ and the map $\theta$ is completely positive, we have 
\begin{align*}
\sca{u, u}_s
& =
\left\langle \theta\left(X^*, s, X\right)\xi,\xi \right\rangle_H
\leq
\|s\|
\left\langle \theta\left(X^*, 1, X\right)\xi,\xi \right\rangle_H
=
\|s\| \sca{u,u}.
\end{align*}
In particular, if $\sca{u,u} = 0$, then $\sca{u,u}_s = 0$, and thus we obtain an induced positive sesquilinear form (denoted in the same way) by passing to the quotient of $\cl E \odot H$ by $N$.
By the Cauchy-Schwarz inequality, we have 
$$
|\sca{u + N,v + N}_s|^2
\leq 
\|u + N\|_s^2 \cdot \|v + N\|_s^2
\leq 
\|s\|^2 \cdot \|u + N\|_K^2 \cdot \|v + N\|_K^2
$$
for all $u, v \in \cl E \odot H$, 
and now the existence of a unique operator $\psi(s) \in \cl B(K)$ satisfying
\[
\sca{\psi(s) (u +N), v + N}_K = \sca{u + N, v + N}_s,
\]
and thus (\ref{eq_psis}), follows from the Riesz Representation Theorem. 

Since the positive elements span $\cl S$, we can now derive a well-defined linear map $\psi(s)$, satisfying (\ref{eq_psis}), for every $s \in \cl S$.
Indeed, if $s_1, s_2 \geq 0$ then we have $s_1 + s_2 \geq 0$, and equation (\ref{eq_psis}) yields
\begin{align*}
\sca{\psi(s_1 + s_2) (x \otimes \xi + N), y \otimes \eta + N}_K 
& = \\
& \hspace{-6cm} = \sca{ \theta( y^*, s_1 + s_2, x ) \xi, \eta}_H \\
& \hspace{-6cm} =\sca{ \theta( y^*, s_1, x ) \xi, \eta}_H + \sca{ \theta( y^*, s_2, x ) \xi, \eta}_H \\
& \hspace{-6cm} = \sca{\psi(s_1) (x \otimes \xi + N), y \otimes \eta + N}_K  + \sca{\psi(s_2) (x \otimes \xi + N), y \otimes \eta + N}_K, \\
& \hspace{-6cm} = \sca{(\psi(s_1) + \psi(s_2)) (x \otimes \xi + N), y \otimes \eta + N}_K,
\end{align*}
showing that $\psi(s_1 + s_2) = \psi(s_1) + \psi(s_2)$. 
For an arbitrary selfadjoint $s \in \cl S$ we can write $s = s_1 - s_2$ for $s_1, s_2 \geq 0$ and we define 
$\psi(s) = \psi(s_1) - \psi(s_2)$.
The operator $\psi(s)$ is well defined: if $s= s_1 - s_2 = t_1 - t_2$ for $s_1, s_2, t_1, t_2 \geq 0$, then $s_1 + t_2 = s_2 + t_1 \geq 0$ and so
\begin{align*}
\psi(s_1) + \psi(t_2) = \psi(s_1 + t_2) = \psi(s_2 + t_1) = \psi(s_2) + \psi(t_1),
\end{align*}
showing that $\psi(s_1) - \psi(s_2) = \psi(t_1) - \psi(t_2)$.
For an arbitrary $s \in \cl S$, write $s = s_1 + is_2$ for 
\[
s_1 = \frac{s + s^*}{2}
\text{ and }
s_2 = \frac{s - s^*}{2i},
\]
and set $\psi(s) = \psi(s_1) + i\psi(s_2)$.
Since $(s+t)_1 = s_1 + t_1$ and $(s+t)_2 = s_2 + t_2$ we have that $\psi$ is additive. 
We thus have a well-defined linear map $\psi(s)$, $s\in \cl S$, satisfying (\ref{eq_psis}). 
Equation (\ref{eq_psis}) now implies that $\psi$ is linear; by its definition, it is unital. 

We next claim that $\psi$ is completely positive.
Towards this end, it suffices to show that 
\[
\left\langle \psi^{(m)}(S) u,u \right\rangle
\geq 0 \text{ if } S = (s_{i,j})_{i,j =1}^m \in M_m(\cl S)^+ , u = (u_i)_{i=1}^m\in (\cl E\odot H)^{(m)}.
\] 
Write $u_i = \sum_{p=1}^{n} x_p^{(i)}\otimes \xi_p^{(i)}$, where $x_p^{(i)}\in \cl E$ and $\xi_p^{(i)} \in H$, $p = 1,\dots,n$, $i = 1,\dots,m$; we can assume that the number of elementary tensors in the expression for each $u_i$ is the same by adding zeros if necessary.
For each $i = 1,\dots,m$, let 
\[
\eta_i = (\xi_p^{(i)})_{p=1}^n \in H^{(n)}
\qand
\eta = (\eta_i)_{i=1}^m\in H^{(mn)}.
\]
Let $X_i = (x_p^{(i)})_{p=1}^n \in M_{1,n}(\cl E)$ and 
\[
X = \begin{pmatrix} X_1 & 0 & \cdots & 0 \\ 0 & X_2 & \cdots & 0 \\ \vdots & \vdots & \cdots & \vdots \\ 0 & 0 & \cdots & X_m \end{pmatrix} \in M_{m, mn}( \cl E).
\] 
We use the same symbol $\theta$ for the multivariable extension
\[
\theta \colon M_{mn, m}(\cl E^*) \times M_{m,m}(\cl S) \times M_{m, mn}(\cl E) \to \cl B(H^{(mn)});
\]
we thus have 
\begin{align*}
\theta(X^*, S, X)
& =
\bigg( \sum_{k, \ell =1}^m \theta( (X^*)_{i,k}, s_{k,\ell}, X_{\ell,j}) \bigg)_{i,j=1}^m \\
& =
\bigg( \theta( (X^*)_{i,i}, s_{i,j}, X_{j,j}) \bigg)_{i,j=1}^m 
=
\bigg( \theta(X_i^*, s_{i,j}, X_j) \bigg)_{i,j=1}^m.
\end{align*}
On the other hand, 
\begin{align*}
\theta(X_i^*, s_{i,j}, X_j)
& =
\theta\left( \begin{pmatrix} x_1^{(i) *} \\ \vdots \\ x_n^{(i) *} \end{pmatrix}, s_{i,j}, \begin{pmatrix}x_1^{(j)} & \cdots & x_n^{(j)} \end{pmatrix} \right) \\
& =
\bigg( \theta(x_p^{(i) *}, s_{i,j}, x_q^{(j)}) \bigg)_{p,q=1}^n.
\end{align*}
Consequently, we get the required
\begin{align*}
\left\langle \psi^{(m)}(S)u,u \right\rangle
& = \\
& \hspace{-2.2cm} =
\sum_{i,j=1}^m \left\langle \psi(s_{i,j})u_j,u_i\right\rangle
 = 
\sum_{i,j=1}^m \sum_{p,q=1}^n \left\langle \psi(s_{i,j}) \left(x_q^{(j)}\otimes \xi_q^{(j)}\right), 
x_p^{(i)}\otimes \xi_p^{(i)} \right\rangle \\
& \hspace{-2.2cm} =
\sum_{i,j=1}^m \sum_{p,q=1}^n 
\left\langle \theta\left(x_p^{(i)*},s_{i,j},x_q^{(j)}\right)\xi_q^{(j)},\xi_p^{(i)}\right\rangle_H 
 =
\sum_{i,j=1}^m \left\langle \theta(X_i^*,s_{i,j},X_j)\eta_j,\eta_i \right\rangle_{H^n} \\
& \hspace{-2.2cm} = 
\left\langle \theta(X^*,S,X)\eta,\eta \right\rangle_{H^{mn}} \geq 0.
\end{align*}

Next we consider the map $\phi \colon \cl E\to \cl B(H,K)$ given by 
\[
\phi(x)\xi = x\otimes \xi + N, \ \ \ x\in \cl E, \ \xi\in H.
\]
For well-definedness, a direct computation gives that
\[
\| \phi(x) \xi \|_K
=
| \sca{\theta(x^*, 1, x)\xi, \xi}_H |^{1/2}
\leq
\|\theta\|^{1/2} \cdot \nor{x} \cdot \nor{\xi}.
\]
By definition, $\phi(\cl E)H$ spans a dense subset of $K$, and we have
\[
\phi(x)^* (y \otimes \eta + N) = \theta(x^*, 1_{\cl S}, y) \eta.
\]
Clearly $\phi$ is a linear map.
We verify that $\phi$ is completely bounded.
Towards this end, let
\[
X = (x_{i,j})_{i,j=1}^m \in M_m(\cl E) \qand \xi = (\xi_i)_{i=1}^m \in H^{(m)}.
\]
We then have
\begin{align*}
\| \phi^{(m)}(X)\xi \|^2_{K^{(m)}} 
& = \\
& \hspace{-2.2cm} =
\langle\phi^{(m)}(X)\xi,\phi^{(m)}(X)\xi \rangle_{K^{(m)}} 
 = 
\sum_{i,j,k = 1}^m \hspace{-0.15cm} \left\langle\phi(x_{k,j})\xi_j, \phi(x_{k,i})\xi_i \right\rangle_K \\
& \hspace{-2.2cm} = 
\sum_{i,j,k = 1}^m \left\langle x_{k,j} \otimes \xi_j + N, x_{k,i} \otimes \xi_i + N\right\rangle_K 
 = 
\sum_{i,j,k = 1}^m \left\langle \theta(x_{k,i}^*, 1_{\cl S}, x_{k,j})\xi_j, \xi_i \right\rangle_H \\
& \hspace{-2.2cm} =
\left\langle \left( \sum_{k=1}^m \theta(x_{k,i}^*, 1_{\cl S}, x_{k,j}) \right)_{i,j} \xi, \xi \right\rangle_{H^{(m)}} 
 =
\left\langle \theta(X^*,1_{\cl S} \otimes I_m,X)  \xi, \xi \right\rangle_{H^{(m)}} \\
& \hspace{-2.2cm} \leq 
\|\theta\|^{(m)} \|X\|^2\|\xi\|^2_{H^{(m)}}.
\end{align*}
Thus $\|\phi^{(m)}\|^2\leq \|\theta\|^{(m)}$ and therefore $\|\phi\|_{\rm cb}^2\leq \|\theta\|_{\rm cb}$, implying that $\phi$ is completely bounded.
Finally, note that, if $x, y\in \cl E$, $s\in \cl S$ and $\xi,\eta\in H$, then 
\begin{align*}
\langle \phi(y)^*\psi(s) \phi(x)\xi,\eta\rangle_H
& = 
\langle \psi(s) \phi(x)\xi,\phi(y)\eta\rangle_K \\
& = 
\langle \psi(s) (x\otimes \xi + N), y \otimes \eta + N \rangle_K \\
& = 
\langle\theta(y^*,s,x)\xi,\eta\rangle_H.
\end{align*}
Therefore
\[
\theta(y^*,s,x) = \phi(y)^*\psi(s) \phi(x), \ \ \ x, y \in \cl E, s \in \cl S.
\]
A similar computation holds at every matrix level, giving 
\[
\|\theta\|_{\rm cb} \leq \|\phi\|_{\rm cb}^2 \cdot \|\psi\|_{\rm cb} = \|\phi\|_{\rm cb}^2,
\]
and hence $\|\theta\|_{\rm cb} = \|\phi\|_{\rm cb}^2$.

Using the representation obtained in the first part of the proof, we write $\wt{\theta}_{\cl A} = \phi^*\cdot\psi\cdot \phi$. 
For all elements $a,b\in \cl A$, $x,y\in \cl E$, $s\in \cl S$ and all vectors $\xi,\eta\in H$, we then have 
\[
\langle\theta(y^*\cdot b, s, a \cdot x)\xi,\eta \rangle_H = \langle\theta(y^*, b \cdot s \cdot a, x)\xi,\eta\rangle_H.
\]
By setting $b = 1_{\cl A}$ we have
\[
\langle \psi(s) \phi(a \cdot x) \xi, \phi(y) \eta \rangle_K = \langle \psi(s \cdot a) \phi(x) \xi, \phi(y) \eta\rangle_K.
\]
Since $\phi(\cl E)H$ spans a dense subspace of $K$, this implies 
\begin{equation}\label{eq_admiagai}
\psi(s)\phi(a\cdot x) = \psi(s\cdot a)\phi(x), \ \ \ x\in \cl E, s\in \cl S, a\in \cl A.
\end{equation}
This completes the proof of item (i).

For item (ii), assume that $\cl A \subseteq \cl S$.
Let $\pi \colon \cl A \to \cl B(K)$ be the map given by $\pi(a) := \psi(a)$, $a\in \cl A$. 
Since $\psi$ is unital, applying (\ref{eq_admiagai}) for $s = 1_{\cl S}$ yields 
\[
\phi(a\cdot x) = \pi(a)\phi(x), \ \ \ x\in \cl E, a\in \cl A,
\]
and therefore
\[
\psi(s\cdot a)\phi(x) = \psi(s) \phi(a \cdot x) = \psi(s)\pi(a)\phi(x), \ \ \ s\in \cl S, a\in \cl A, x\in \cl E.
\]
Since $[\phi(\cl E) H] = K$ we get
\begin{equation}\label{eq_multdom}
\psi(s\cdot a) = \psi(s)\pi(a), \ \ \ s\in \cl S, a\in \cl A.
\end{equation}
In particular, $\pi$ is a $*$-homomorphism.

For item (iii), assume that $\cl S$ is a unital 
C*-algebra and that $(\psi,\pi)$ is an $\cl A$-admissible pair. 
Let $q \colon \cl A \to \cl S$ be the map, 
given by $q(a) = a\cdot 1_{\cl S}$; we have that $q$ is a $*$-homomorphism. 
Using the Stinespring Theorem, let $\rho \colon \cl S\to \cl B(\wh{K})$ be a unital $*$-representa\-tion and $V \colon K\to \wh{K}$ be an isometry such that 
\[
\psi(s) = V^* \rho(s) V, \ \ \ s \in \cl S.
\]
Let $\wh{\phi} \colon \cl E\to \cl B(H,\wh{K})$ be the map given by $\wh{\phi}(x) := V \phi(x)$. 
By construction, we have 
\[
\wt{\theta}_{\cl A} = \wh{\phi}^* \cdot \rho \cdot \wh{\phi}.
\]
Setting $\pi = \psi \circ q$, equation (\ref{eq_multdom}) implies that $q(\cl A)$ is in the multiplicative domain of $\psi$; that is, $\pi$ is a $*$-homomorphism, and $(\psi, \pi)$ is an $\cl A$-representation.
Since $\rho \circ q$ is a $*$-homomorphism dilating the $*$-homomorphism $\pi$, we have that $\pi$ is a direct summand of $\rho \circ q$.
We have that 
%$$\pi(a) = V^*(\rho\circ q)(a) V, \ \ \ a\in \cl A,$$
%and hence 
$$V\pi(a) = VV^*(\rho\circ q)(a) V, \ \ \ a\in \cl A,$$
and since $VV^*$ commutes with the image of $\rho\circ q$, 
we have 
$$V\pi(a) = (\rho\circ q)(a) V, \ \ \ a\in \cl A.$$
Thus, 
\begin{align*}
\rho(s)\wh{\phi}(a\cdot x) 
& = \rho(s)V\phi(a\cdot x)
=  \rho(s)V\pi(a)\phi(x) \\
& = \rho(s)(\rho\circ q)(a) V\phi(x)
= \rho(s\cdot a)\wh{\phi}(x)
\end{align*}
for all $x\in \cl E$, $s\in \cl S$, $a\in \cl A$,
that is, the pair $(\wh{\phi}, \rho)$ is $\cl A$-admissible.
\end{proof}

The decomposition of Lemma \ref{l_ssgen} remains valid under the weaker complete boundedness condition, when $\cl A = \bb C$.

%%%%%%%%%%%%%%%%%%%%%%%%%%%%
\begin{proposition} \label{p_cbbalance}
Let $\cl E$ be an operator space, $\cl S$ be an operator system and $\theta \colon \cl E^*\times\cl S\times\cl E\to \cl B(H)$ be a trilinear map.
If $\theta$ is completely bounded (resp. completely contractive), then there are completely bounded (resp. completely contractive) completely positive trilinear maps $\theta_m$, $m=1,2,3,4$, such that $\wt{\theta} = \sum_{m=1}^4 \wt{\theta}_{m}$.
\end{proposition}

\begin{proof}
Let $\theta$ be a completely bounded trilinear map.
By the CSPS Factorisation Theorem (see for example \cite[Theorem 1.5.7]{blm}), there exist Hilbert spaces $K_1$ and $K_2$, and completely bounded maps 
\[
\phi_1 \colon \cl E\to \cl B(H,K_1),
\psi \colon \cl S\to \cl B(K_2,K_1),
\phi_2 \colon \cl E\to \cl B(H,K_2),
\]
such that $\wt{\theta} = \phi_1^{*}\cdot \psi \cdot\phi_2$.
Without loss of generality we may assume that $\cl S$ is a concrete operator system, and using Arveson's Extension Theorem, we extend $\psi$ to a completely bounded map on $\ca(\cl S)$.
Using Wittstock's Representation Theorem, we write 
\[
\psi(s) = W_1^*\rho(s)W_2, \ \ \ s\in \cl S,
\]
for a unital $*$-representation $\rho \colon \ca(\cl S)\to \cl B(K)$, and $W_1 \in \cl B(K_1, K)$ and $W_2 \in \cl B(K_2, K)$. 
Letting
\[
\phi_1' := W_1 \phi_1, \psi':= \rho|_{\cl S}, \phi_2' := W_2 \phi_2,
\]
we obtain that $\phi_1'$ and $\phi_2'$ are completely bounded maps, $\psi'$ is unital completely positive, and $\wt{\theta} = \phi_1^{'*} \cdot \psi' \cdot\phi_2'$.
Using the polarisation identity, we then write $\wt{\theta} = \sum_{m=1}^4 \wt{\theta}_{m}$ for
\[
\wt{\theta}_{m} := \frac{1}{4} (\phi_1' + i^m \phi_2')^* \cdot \psi' \cdot (\phi_1' + i^m \phi_2'), \ \ \ m=1,2,3,4.
\]
Notice that
\[
\|\wt{\theta}_{m}\|_{\rm cb} 
\leq 
\frac{1}{4} \| \phi_1' + i^m \phi_2' \|_{\rm cb}^2 \cdot \|\psi'\|
\leq
\frac{1}{4} \left( \|W_1\| \cdot \| \phi_1\|_{\rm cb} + \|W_2\| \cdot \| \phi_2 \|_{\rm cb} \right)^2.
\]
If, in particular, $\theta$ is completely contractive, then $\phi_1$, $\phi_2$ and $\psi$ are completely contractive.
Hence we can choose $W_1, W_2$ to be contractions, and thus $\|\wt{\theta}_{m}\|_{\rm cb} \leq 4^{-1} \cdot 2^2 = 1$.
\end{proof}

%%%%%%%%%%%%%%%%%%%%%%%%%%%%
\subsection{The symmetrisation norm}\label{ss_sn}
%%%%%%%%%%%%%%%%%%%%%%%%%%%%

For $u\in M_n(\cl E^*\odot\cl S \odot\cl E)$, we define 
the \emph{symmetrisation norm} of $u$ by
\begin{equation}\label{eq_normde}
\|u\|_{\rm s}^{(n)} 
:= 
\sup\{\|\wt{\theta}^{(n)}(u)\| \ : \ \theta \mbox{ c.c.p. trilinear map on $\cl E^* \times \cl S \times \cl E$}\}.
\end{equation}
We note that, although the class of completely contractive completely positive trilinear maps may not form a set, the class of real values $\|\wt{\theta}^{(n)}(u)\|$ is inside $\bb R^+$, and thus the supremum is well defined.

%%%%%%%%%%%%%%%%%%%%%%%%%%%%
\begin{lemma}\label{l_norm}
Let $\cl E$ be an operator space and $\cl S$ be an operator system. 
Then the map $\|\cdot\|_{\rm s}^{(n)} \colon M_n(\cl E^*\odot\cl S \odot\cl E) \to \bb{R}^+$ is a norm for every $n \in \bb{N}$. 
Moreover, 
\begin{equation}\label{eq_esymh}
\frac{1}{4} \|u\|_{\rm h}^{(n)} \leq \|u\|_{\rm s}^{(n)}\leq \|u\|_{\rm h}^{(n)}, \ \ \ \ u\in M_n(\cl E^*\odot\cl S \odot\cl E).
\end{equation}
\end{lemma}

\begin{proof} 
The subadditivity and the homogeneity of $\|\cdot\|_{\rm s}^{(n)}$ are clear. 
In addition, \cite[Theorem 5.1]{pisier_intr} implies the second inequality in (\ref{eq_esymh}). 
It remains to show the first inequality in (\ref{eq_esymh}); this will also imply that $\|\cdot\|_{\rm s}^{(n)}$ is a norm, 
as $\|\cdot\|_{\rm h}^{(n)}$ is so.

Towards this end, let $\theta \colon \cl E^* \times \cl S \times \cl E \to \cl B(H)$ be a completely contractive trilinear map.
By Proposition \ref{p_cbbalance}, there exist completely contractive completely positive maps $\theta_m$, $m=1,2,3,4$, such that $\wt{\theta} = \sum_{m=1}^4 \wt{\theta}_m$.
Therefore for $u \in M_n(\cl E^*\odot\cl S \odot\cl E)$ we have
\[
\| \wt{\theta}^{(n)}(u) \| \leq \sum_{m=1}^4 \|\wt{\theta}_m(u)\| \leq 4 \|u\|^{(n)}_{\rm s}.
\]
By \cite[Theorem 5.1]{pisier_intr}, taking suprema over all completely contractive trilinear maps $\theta$ implies that $\|u\|_{\rm h}^{(n)} \leq 4 \|u\|^{(n)}_{\rm s}$, as required.
\end{proof}

Due to Lemma \ref{l_ssgen}, we can use decompositions of trilinear maps for the symmetrisation norm.
We introduce some terminology in this respect.

%%%%%%%%%%%%%%%%%%%%%%%%%%%%
\begin{definition}
Let $\cl E$ be an operator space and $\cl S$ be an operator system. 
We call a pair $(\phi,\psi)$ of maps \emph{admissible} if it is $\bb{C}$-admissible, that is, 
for a pair $(H,K)$ of Hilbert spaces, 
$\phi \colon \cl E\to \cl B(H,K)$ is completely contractive and 
$\psi \colon \cl S\to \cl B(K)$ is unital and completely positive. 
We will say that $(\phi,\psi)$ is \emph{associated} with $(H,K)$. 
An admissible pair $(\phi,\psi)$ associated with $(H,K)$ will be called \emph{non-degenerate} if $\phi$ is non-degenerate, that is, $[\phi(\cl E) H] = K$ and $[\phi(\cl E)^* K] = H$.
\end{definition}

%%%%%%%%%%%%%%%%%%%%%%%%%%%%
\begin{remark}\label{ex_existad}
\rm
There is an abundance of admissible pairs for an operator space $\cl E$ and an operator system $\cl S$.
As a first example, let $\tau \colon \cl S \to \bb C$ be a state on $\cl S$; then for every $\phi \colon \cl E \to \cl B(H,K)$ we can define $\psi:=\tau(\cdot) I_K$, for which we get that $(\phi, \psi)$ is an admissible pair.

For another example, let $\phi \colon \cl E \to \cl B(H, K)$ be a completely contractive map and $\psi \colon \cl S \cl \to \cl B(L)$ be a unital completely positive map.
Then, setting 
\[
\wt{\phi}(\cdot) := \phi(\cdot) \otimes I_L
\text{ and }
\wt{\psi}(\cdot) := I_{K} \otimes \psi(\cdot),
\] 
we have that $(\wt{\phi}, \wt{\psi})$ is an admissible pair associated with $(H \otimes L, K \otimes L)$.
\end{remark}

The following observation allows us to consider only (non-degenerate) admissible pairs.

%%%%%%%%%%%%%%%%%%%%%%%%%%%%
\begin{proposition}\label{p_nd}
Let $\cl E$ be an operator space and $\cl S$ be an operator system. 
Then 
\begin{align*}
\|u\|_{\rm s}^{(n)} 
& = 
\sup\{\|(\phi^*\cdot\psi\cdot \phi)^{(n)}(u)\| \ : \ (\phi,\psi) \mbox{ adm. pair}\} \\
& = 
\sup\{\|(\phi^*\cdot\psi\cdot \phi)^{(n)}(u)\| \ : \ (\phi,\psi) \mbox{ non-degenerate adm. pair}\},
\end{align*}
for all $u\in M_n(\cl E^*\odot\cl S \odot\cl E)$.
\end{proposition}

\begin{proof}
The first equality follows from Lemma \ref{l_ssgen} and the fact that, for every admissible pair 
$(\phi,\psi)$, the map $\phi^*\times\psi\times\phi$ on $\cl E^*\times\cl S\times \cl E$, given by  
\[
(\phi^*\times\psi\times\phi)(x^*,s,y) := \phi^*(x^*)\psi(s)\phi(y), \ \ \ x, y \in \cl E, s \in \cl S,
\] 
is completely contractive and completely positive.
For the second equality, let $(\phi, \psi)$ be an admissible pair associated with $(H, K)$, 
let $K' = [\phi(\cl E) H]$ and $H' = [\phi(\cl E)^* K]$, and set
\[
\phi'(\cdot) = P_{K'} \phi(\cdot) |_{H'} \qand \psi'(\cdot) = P_{K'} \psi(\cdot) |_{K'}.
\]
Since
\[
\phi(x) = \begin{pmatrix} \phi'(x) & 0 \\ 0 & 0 \end{pmatrix}, \ \ \  x \in \cl E,
\]
with respect to the decompositions $H = H'\oplus (H')^{\perp}$ and 
$K = K'\oplus (K')^{\perp}$, we see that
\[
(\phi^* \cdot \psi \cdot \phi)^{(n)}(u)
=
\begin{pmatrix}
((\phi')^* \cdot \psi' \cdot \phi')^{(n)}(u) & 0 \\
0 & 0
\end{pmatrix}.
\]
By construction, $(\phi', \psi')$ is a non-degenerate admissible pair, and the proof is complete.
\end{proof}

Let $\cl E^*\otimes_{\rm s}\cl S \otimes_{\rm s} \cl E$ be the completion of $\cl E^*\odot\cl S\odot \cl E$ with respect to $\|\cdot\|_{\rm s}$.

%%%%%%%%%%%%%%%%%%%%%%%%%%%%
\begin{proposition}\label{p_estareos}
Let $\cl E$ be an operator space and $\cl S$ be an operator system. 
Then $\cl E^*\otimes_{\rm s}\cl S \otimes_{\rm s} \cl E$, equipped with the family $\{\|\cdot\|_{\rm s}^{(n)}\}_{n\in \bb{N}}$ of matricial norms, is an operator space.
\end{proposition}

\begin{proof}
The proof is straightforward as the norm is given through a collection of representations.
In short, suppose that $u\in M_n(\cl E^*\otimes_{\rm s}\cl S \otimes_{\rm s}\cl E)$, $\alpha \in M_{m,n}$ and $\beta\in M_{n,m}$. 
Using (\ref{eq_modpr}), we have 
\begin{align*}
\left\|\alpha u \beta\right\|^{(m)}_{\rm s}  
& =  
\sup\{\|(\phi^*\cdot\psi \cdot \phi)^{(m)}(\alpha u \beta)\| \ : \  
(\phi,\psi) \mbox{ admissible pair}\}\\
& = 
\sup\{\|\alpha(\phi^*\cdot\psi \cdot \phi)^{(n)}(u)\beta\| \ : \  (\phi,\psi) \mbox{ admissible pair}\}\\
& \leq 
\left\|\alpha\right\| \left\|\beta\right\| \left\|u\right\|^{(n)}_{\rm s}.
\end{align*}

Now let, in addition, $v\in M_m(\cl E^*\otimes_{\rm s}\cl S \otimes_{\rm s}\cl E)$. 
Then 
\begin{align*}
\left\|
\left(\begin{matrix}
u & 0\\
0 & v
\end{matrix}
\right)
\right\|^{(n + m)}_{\rm s} & = \\
& \hspace{-2cm} = 
\sup\left\{
\left\|
\left(\begin{matrix}
(\phi^*\cdot\psi\cdot \phi)^{(n)}(u) & 0\\
0 & (\phi^*\cdot\psi\cdot \phi)^{(m)}(v)
\end{matrix}
\right)
\right\|  \ : \  (\phi,\psi) \mbox{ adm.\ pair}\right\}\\
& \hspace{-2cm} \leq 
\max\left\{\left\|u\right\|^{(n)}_{\rm s},\left\|v\right\|^{(m)}_{\rm s}\right\}.
\end{align*}
On the other hand, 
\[
\left\|u\right\|^{(n)}_{\rm s}
= 
\left\|
\left(\begin{matrix}
u & 0\\
0 & 0
\end{matrix}
\right)
\right\|^{(n + m)}_{\rm s}
\leq 
\left\|
\left(\begin{matrix}
u & 0\\
0 & v
\end{matrix}
\right)
\right\|^{(n + m)}_{\rm s},
\]
where we used that
\[
\left\|
\left(\begin{matrix}
(\phi^*\cdot\psi\cdot \phi)^{(n)}(u) & 0\\
0 & 0
\end{matrix}
\right)
\right\| 
\leq
\left\|
\left(\begin{matrix}
(\phi^*\cdot\psi\cdot \phi)^{(n)}(u) & 0\\
0 & (\phi^*\cdot\psi\cdot \phi)^{(m)}(v)
\end{matrix}
\right)
\right\| 
\]
for every admissible pair $(\phi,\psi)$.
Therefore, by symmetry, we obtain that
\[
\left\|
\left(\begin{matrix}
u & 0\\
0 & v
\end{matrix}
\right)
\right\|^{(n + m)}_{\rm s}  
= 
\max\left\{\left\|u\right\|^{(n)}_{\rm s},\left\|v\right\|^{(m)}_{\rm s}\right\},
\]
and the proof is complete. 
\end{proof}

%%%%%%%%%%%%%%%%%%%%%%%%%%%%
\begin{definition}
Let $\cl E$ be an operator space and $\cl S$ be an operator system.
The (complete) operator space $\cl E^*\otimes_{\rm s}\cl S \otimes_{\rm s} \cl E$ will be called the \emph{symmetrisation of $\cl E$ by $\cl S$}.
\end{definition}

We will refer to $\{\|\cdot\|_{\rm s}^{(n)}\}_{n\in \bb{N}}$ as the \emph{symmetric operator space structure}, and for convenience we set 
\[
\cl E^*\otimes_{\rm s} \cl E := \cl E^*\otimes_{\rm s} \bb{C} \otimes_{\rm s} \cl E.
\] 
We note that, by Lemma \ref{l_norm}, 
$\cl E^*\otimes_{\rm s}\cl S \otimes_{\rm s} \cl E$ is completely boundedly isomorphic 
to the Haagerup tensor product $\cl E^*\otimes_{\rm h}\cl S \otimes_{\rm h} \cl E$. 
We will later see that the symmetrisation norm and the Haagerup tensor norm differ in general.

%%%%%%%%%%%%%%%%%%%%%%%%%%%%
\begin{remark}\label{r_norest} 
Let $\theta \colon \cl E^* \times \cl S \times \cl E \to \cl B(H)$ be a completely bounded trilinear map and let $\wt{\theta} \colon \cl E^* \odot \cl S \odot \cl E \to \cl B(H)$ be the induced map on the tensor product.
We can then extend to a completely bounded map
\[
\wt{\theta}_{\rm h} \colon \cl E^* \otimes_{\rm h} \cl S \otimes_{\rm h} \cl E \to \cl B(H)
\]
on the Haagerup tensor product, for which we get $\| \wt{\theta}_{\rm h} \|_{\rm cb} = \| \theta \|_{\rm cb}$.
Using Lemma \ref{l_norm} we can also induce a completely bounded map
\[
\wt{\theta}_{\rm s} \colon \cl E^* \otimes_{\rm s} \cl S \otimes_{\rm s} \cl E \to \cl B(H)
\]
on the symmetrisation, for which we get
\[
\| \wt{\theta}_{\rm h} \|_{\rm cb}
\leq 
\| \wt{\theta}_{\rm s} \|_{\rm cb}
\leq
4 \cdot \| \wt{\theta}_{\rm h} \|_{\rm cb}.
\]
If, in addition, $\theta$ is completely positive, then
\[
\| \theta \|_{\rm cb}
=
\| \wt{\theta}_{\rm h} \|_{\rm cb}
=
\| \wt{\theta}_{\rm s} \|_{\rm cb}.
\]
Indeed, without loss of generality assume that $\| \wt{\theta}_{\rm h} \|_{\rm cb} = 1$.
Then $\theta$ is a completely contractive completely positive map, and by the definition of the symmetrisation norm we obtain 
$\| \wt{\theta}_{\rm s} \|_{\rm cb} \leq 1 = \| \wt{\theta}_{\rm h} \|_{\rm cb}$,
as required.
\end{remark}

%%%%%%%%%%%%%%%%%%%%%%%%%%%%
\subsection{The selfadjoint operator space structure}\label{ss_sss}
%%%%%%%%%%%%%%%%%%%%%%%%%%%%

We now show that the operator space $\cl E^*\otimes_{\rm s} \cl S \otimes_{\rm s} \cl E$ is in fact a selfadjoint operator space.
Towards this end, for $u = \sum_{i=1}^N y_i^*\otimes s_i \otimes x_i \in \cl E^*\odot \cl S\odot \cl E$, set 
\[
u^* := \sum_{i=1}^N x_i^*\otimes s_i^*\otimes y_i.
\] 
The map $u \mapsto u^*$ is a (well-defined) isometric involution; indeed,
\begin{eqnarray*}
\|u\|_{\rm s} \hspace{-0.2cm}
& = & \hspace{-0.2cm}
\sup\left\{\left\|\sum_{i=1}^N \hspace{-0.03cm}\phi(y_i)^*\psi(s_i)\phi(x_i)\right\| \ : \ (\phi,\hspace{-0.05cm}\psi) \mbox{ admissible pair}\right\}\\
\hspace{-0.2cm} & = & \hspace{-0.2cm}
\sup\left\{\left\|\sum_{i=1}^N \hspace{-0.03cm}\phi(x_i)^*\psi(s_i^*)\phi(y_i)\right\| \ : \ (\phi,\hspace{-0.05cm}\psi) \mbox{ admissible pair}\right\}
\hspace{-0.09cm} = \hspace{-0.07cm} \|u^*\hspace{-0.04cm}\|_{\rm s}\hspace{-0.02cm}.
\end{eqnarray*}
We hence extend it to an (isometric) involution on $\cl E^*\otimes_{\rm s}\cl S \otimes_{\rm s}\cl E$. 
For $n\in \bb{N}$, let
\begin{align*}
C_n 
& := 
\{u\in M_n(\cl E^*\otimes_{\rm s}\cl S \otimes_{\rm s} \cl E)_h \ : \ 
(\phi^*\cdot\psi \cdot \phi)^{(n)}(u)\in M_n(\cl B(H))^+, \\
& \hspace{3.5cm} \mbox{ if } (\phi,\psi) \mbox{ is an admissible pair for some } (H,K)\}.
\end{align*}
By definition, if $(\phi, \psi)$ is an admissible pair, then
\[
(\phi^* \cdot \psi \cdot \phi)^{(n)}(C_n) \subseteq M_n(\cl B(H))^+, \ \ \ n \in \bb N.
\]
We point out that, since $M_n(\cl B(H))^+$ is closed in the operator norm, $C_n$ is closed in the symmetrisation norm.

%%%%%%%%%%%%%%%%%%%%%%%%%%%%
\begin{lemma}\label{l_mos}
Let $\cl E$ be an operator space and $\cl S$ be an operator system. 
Then the following hold:
\begin{enumerate}
\item the family $\{C_n\}_{n\in \bb{N}}$ is a matrix order structure on the $*$-vector space $\cl E^*\otimes_{\rm s} \cl S \otimes_{\rm s}\cl E$;
\item the family $\{C_n'\}_{n \in \bb{N}}$ given by
\[
C_n' := \left\{x^*\odot s\odot x \ : \  x\in M_{k,n}(\cl E), s\in M_k(\cl S)^+, k\in \bb{N}\right\}^{-\|\cdot\|_{\rm s}},
\]
is a matrix order structure on the $*$-vector space $\cl E^*\otimes_{\rm s} \cl S \otimes_{\rm s}\cl E$;
\item $C_n' \subseteq C_n$ for every $n \in \bb{N}$.
\end{enumerate}
\end{lemma}

\begin{proof}
For item (i), for every $n$, the set $C_n$ is a cone consisting of hermitian elements by its definition. 
Let $n\in \bb{N}$ and $u\in C_n \cap (-C_n)$. 
Then 
\[
(\phi^*\cdot\psi\cdot \phi)^{(n)}(u)\in M_n(\cl B(H))^+ \cap [-M_n(\cl B(H))^+] = \{0\}
\]
for all admissible pairs $(\phi, \psi)$ on some $(H,K)$.
Thus, $\|u\|_{\rm s}^{(n)} = 0$, and hence $u = 0$ by Lemma \ref{l_norm} and Proposition \ref{p_nd}, implying that $C_n$ is a proper cone. 
Moreover we have that $C_n \oplus C_m \subseteq C_{n+m}$ for every $n,m \in \bb N$.
Furthermore, since
\[
(\phi^* \cdot \psi \cdot \phi)^{(m)}(\al^* u \al)
=
\al^* (\phi^* \cdot \psi \cdot \phi)^{(n)}(u) \al, 
\text{ for all } \al \in M_{n,m}, u \in C_n,
\]
the family $\{C_n\}_{n\in \bb{N}}$ is compatible, and thus defines a matrix ordering.

For item (iii), if $x\in M_{m,n}(\cl E)$ and $s\in M_m(\cl S)^+$, then trivially
\[
(x^*\odot s\odot x)^* = x^*\odot s\odot x.
\]
If $(\phi,\psi)$ is an admissible pair associated with the pair $(H,K)$ of Hilbert spaces, then
\[
(\phi^*\cdot\psi\cdot \phi)^{(n)}(x^*\odot s\odot x) = \phi^{(m,n)}(x)^* \psi^{(m)}(s)\phi^{(m,n)}(x) \in M_n(\cl B(H))^+.
\]
As this holds for all admissible pairs, we obtain that $x^*\odot s\odot x \in C_n$.

For item (ii), if $x_i\in M_{k_i,n}(\cl E)$, $s_i\in M_{k_i}(\cl S)^+$, $i = 1,2$, then 
\[
x_1^* \odot s_1 \odot x_1 + x_2^* \odot s_2 \odot x_2 
= 
\left(\begin{matrix}
x_1\\
x_2
\end{matrix}\right)^*
\odot
\left(\begin{matrix}
s_1 & 0\\
0 & s_2
\end{matrix}\right)
\odot
\left(\begin{matrix}
x_1\\
x_2
\end{matrix}\right).
\]
This shows that every $C_n'$ is a cone, and that $C_n' \oplus C_m' \subseteq C_{n+m}'$ for every $n,m \in \bb N$.
Further, if $\alpha\in M_{n,l}$, $x\in M_{k,n}(\cl E)$ and $s\in M_k(\cl S)^+$, then 
\[
\alpha^*(x^*\odot s\odot x)\alpha = (x\cdot\alpha)^*\odot s\odot (x\cdot\alpha).
\]
Clearly $C_n' \cap [-C_n'] \subseteq C_n \cap [-C_n] = \{0\}$ for every $n \in \bb N$.
It follows that the family $\{C_n'\}_{n\in \bb{N}}$ is compatible and is hence a matrix ordering.
\end{proof}

%%%%%%%%%%%%%%%%%%%%%%%%%%%%
\begin{remark} \label{r_span}
By Lemma \ref{l_mos} we have that the cones densely span the symmetrisation at every matrix level.
Indeed, every element $y^* \odot s \odot x$ can be chosen so that $s$ is a square matrix, and because $\cl S$ is an operator system, every such element is spanned by elements where the middle term is assumed to be positive.
Now by using the polarisation identity we can see that every element $y^* \odot s \odot x$ with $s \geq 0$ is a linear combination of elements of the form $x^* \odot s \odot x$ with $s \geq 0$, which are positive.
\end{remark}

%%%%%%%%%%%%%%%%%%%%%%%%%%%%
\begin{theorem}\label{th_estareos} 
Let $\cl E$ be an operator space and $\cl S$ be an operator system. 
Let $\{D_n\}_{n\in \bb{N}}$ be a family of matricial cones such that $D_n \subseteq C_n$ for all $n \in \bb{N}$.
Then $\cl E^*\otimes_{\rm s}\cl S \otimes_{\rm s}\cl E$, equipped with the family $\{D_n\}_{n\in \bb{N}}$, is a selfadjoint operator space. 
\end{theorem}

\begin{proof} 
Let $\{\nu_n^D\}_{n\in \bb{N}}$ for $n \in \bb N$ be the seminorms 
arising from the family $\{D_n\}_{n\in \bb{N}}$
as in Subsection \ref{ss_nonunital}.
By \cite[Lemma 4.8]{werner} we have to show that $\| \cdot \|_{\rm s}^{(n)} = \nu_n^D(\cdot)$ for all $n \in \bb N$.
We will show that
\[
\nor{u}_{\rm s} \leq \nu_n^C(u) \leq \nu_n^D(u) \leq \nor{u}_{\rm s},
\]
for all $u \in M_n(\cl E^* \odot_{\rm s} \cl S \odot_{\rm s} \cl E)$.
Since the involution is isometric, we have that
\[
\nu_n^D(u) 
\leq 
\nor{ \begin{pmatrix} 0 & u \\ u^* & 0 \end{pmatrix} }_{\rm s}^{(2n)}
=
\max\{ \|u\|_{\rm s}^{(n)}, \|u^*\|_{\rm s}^{(n)} \} 
=
\|u\|_{\rm s}^{(n)},
\]
while the inclusion $D_n \subseteq C_n$ for all $n \in \bb N$ yields $\nu_n^C(u) \leq \nu_n^D(u)$.

For the remaining inequality, let $(\phi, \psi)$ be an admissible pair associated with $(H, K)$, such that
\begin{align*}
\nor{ \begin{pmatrix} 0 & u \\ u^* & 0 \end{pmatrix} }_{\rm s}^{(2n)} - \eps
& \leq 
\nor{ (\phi^* \cdot \psi \cdot \phi)^{(2n)} \begin{pmatrix} 0 & u \\ u^* & 0 \end{pmatrix} }.
\end{align*}
Note that
\begin{align*}
(\phi^* \cdot \psi \cdot \phi)^{(2n)} \begin{pmatrix} 0 & u \\ u^* & 0 \end{pmatrix}
& =
\begin{pmatrix} 0 & (\phi^* \cdot \psi \cdot \phi)^{(n)}(u) \\ (\phi^* \cdot \psi \cdot \phi)^{(n)}(u)^* & 0 \end{pmatrix}, 
\end{align*}
and the latter element is a selfadjoint operator in $\cl B(H^{(2n)})$. 
For a vector $\xi \in H^{(2n)}$, set 
\[
\omega_{\xi, \xi}(T) := \sca{T \xi, \xi}, \ \ \ T \in \cl B(H^{(2n)}).
\]
Then for a suitably chosen unit vector $\xi \in H^{(2n)}$ we have 
\begin{align*}
\nor{ (\phi^* \cdot \psi \cdot \phi)^{(2n)} \begin{pmatrix} 0 & u \\ u^* & 0 \end{pmatrix} } - \eps 
& \leq
\left| \omega_{\xi, \xi} \circ (\phi^* \cdot \psi \cdot \phi)^{(2n)} \begin{pmatrix} 0 & u \\ u^* & 0 \end{pmatrix} \right|.
\end{align*}
Since $\omega_{\xi, \xi} \in {\rm CCP}(M_n(\cl B(H^{(2)}), \bb C)$, and $\phi^* \cdot \psi \cdot \phi$ is completely positive on $\cl E^* \otimes_{\rm s} \cl S \otimes_{\rm s} \cl E$ with respect to $\{C_n\}_n$, we have that the map $\omega_{\xi, \xi} \circ (\phi^* \cdot \psi \cdot \phi)^{(2n)}$ is completely contractive completely positive on $M_{2n}(\cl E^* \otimes_{\rm s} \cl S \otimes_{\rm s} \cl E)$ with respect to $\{C_{2nm}\}_{m \in \bb{N}}$.
Thus, by the definition of $\nu_n^C(u)$, we have
\[
\left| \omega_{\xi, \xi} \circ (\phi^* \cdot \psi \cdot \phi)^{(2n)} \begin{pmatrix} 0 & u \\ u^* & 0 \end{pmatrix} \right|
\leq \nu_n^C(u).
\]
In total we have
\begin{align*}
\|u\|_{\rm s}^{(n)} - 2 \eps
& =
\nor{ \begin{pmatrix} 0 & u \\ u^* & 0 \end{pmatrix} }_{\rm s}^{(2n)} - 2 \eps \\
& \leq
\nor{ (\phi^* \cdot \psi \cdot \phi)^{(2n)} \begin{pmatrix} 0 & u \\ u^* & 0 \end{pmatrix} } - \eps
\leq
\nu_n^C(u).
\end{align*}
Letting $\eps \to 0$ gives that $\|u\|_{\rm s}^{(n)} \leq \nu_n^C(u)$, and the proof is complete.
\end{proof}

We can now give an explicit description of the positive cones of $\cl E^*\otimes_{\rm s}\cl S\otimes_{\rm s}\cl E$, which supersedes Lemma \ref{l_mos}.

%%%%%%%%%%%%%%%%%%%%%%%%%%%%
\begin{proposition}\label{p_synth}
Let $\cl E$ be an operator space and $\cl S$ be an operator system.
Then, for every $n \in \bb N$, we have 
\begin{equation}\label{eq_Cn}
C_n = \left\{x^*\odot s\odot x \ : \  x\in M_{k,n}(\cl E), s\in M_k(\cl S)^+, k\in \bb{N}\right\}^{-\|\cdot\|_{\rm s}}.
\end{equation}
\end{proposition}

\begin{proof}
Following the notation of Lemma \ref{l_mos} we write
\[
C_n' 
:= 
\left\{x^*\odot s\odot x \ : \  x\in M_{k,n}(\cl E), s\in M_k(\cl S)^+, k\in \bb{N}\right\}^{-\|\cdot\|_{\rm s}}.
\]
By Lemma \ref{l_mos}, the family $\{C_n'\}$ defines a matrix order structure on $\cl E^* \otimes_{\rm s} \cl S \otimes_{\rm s} \cl E$ satisfying $C_n' \subseteq C_n$, for $n \in \bb N$; thus the identity map
\[
\id \colon \left(\cl E^* \otimes_{\rm s} \cl S \otimes_{\rm s} \cl E, \{C_n'\}_{n\in \bb{N}}\right) 
\to \left(\cl E^* \otimes_{\rm s} \cl S \otimes_{\rm s} \cl E, \{C_n\}_{n\in \bb{N}}\right)
\]
is completely isometric and completely positive.

By Theorem \ref{th_estareos}, there exists a completely isometric complete order embedding
\[
\wt{\theta} \colon (\cl E^* \otimes_{\rm s} \cl S \otimes_{\rm s} \cl E, \{C_n'\}_{n\in \bb{N}}) \to \cl B(H).
\]
Let $\theta \colon \cl E^* \times \cl S \times \cl E \to \cl B(H)$ be its delinearisation.
By the definition of the cones $C_n'$, the map $\theta$ is completely positive, and thus we can write $\wt{\theta} = \phi^* \cdot \psi \cdot \phi$ by Lemma \ref{l_ssgen}.
Therefore $\wt{\theta}$ is completely positive with respect to the family $\{C_n\}_{n\in \bb{N}}$ by the definition of the cones $C_n$.
Hence 
\[
\wt{\theta}(C_n) \subseteq \wt{\theta}(\cl E^* \otimes_{\rm s} \cl S \otimes_{\rm s} \cl E) \cap M_n(\cl B(H))^+ = \wt{\theta}(C_n'),
\]
as $\wt{\theta}$ is a complete order embedding.
Since $\wt{\theta}$ is completely isometric, we get the required inclusions $C_n \subseteq C_n'$ for all $n \in \bb{N}$.
\end{proof}

As an immediate consequence, we obtain that the symmetrisation is injective with respect to complete order embeddings.

%%%%%%%%%%%%%%%%%%%%%%%%%%%%
\begin{theorem} \label{t_injective}
Let $\cl E_1 \subseteq \cl E_2$ be operator spaces and $\cl S_1 \subseteq \cl S_2$ be operator systems.
Then the inclusion map
\[
\cl E_1^* \odot \cl S_1 \odot \cl E_1 \to \cl E_2^* \odot \cl S_2 \odot \cl E_2
\]
extends to a completely isometric complete order embedding of the respective symmetrisations.
\end{theorem}

\begin{proof}
We will use the notation $\cl E_i^* \odot_{\rm s} \cl S_i \odot_{\rm s} \cl E_i$ for the corresponding algebraic tensor products, when considered as operator subspaces of $\cl E_i^* \otimes_{\rm s} \cl S_i \otimes_{\rm s} \cl E_i$, $i = 1,2$. 
Let 
\[
\iota \colon \cl E_1^* \odot_{\rm s} \cl S_1 \odot_{\rm s} \cl E_1 
\to \cl E_2^* \odot_{\rm s} \cl S_2 \odot_{\rm s} \cl E_2
\]
be the inclusion map.
We check that $\iota$ is a complete isometry. 
Towards this end, let $n \in \bb N$ and $u\in M_n(\cl E_1^* \otimes_{\rm s} \cl S_1 \otimes_{\rm s} \cl E_1)$. 
If $(\wt{\phi}, \wt{\psi})$ is an admissible pair for $(\cl E_2, \cl S_2)$, then $(\wt{\phi}|_{\cl E_1}, \wt{\psi}|_{\cl S_1})$ is an admissible pair for $(\cl E_1, \cl S_1)$. 
Thus, 
\begin{equation}\label{eq_inesy}
\|\iota^{(n)}(u)\|_{\rm s}^{(n)}\leq \|u\|_{\rm s}^{(n)}. 
\end{equation}
Let $\eps > 0$ and $(\phi, \psi)$ be an admissible pair for $(\cl E_1, \cl S_1)$, associated with a pair $(H,K)$ of Hilbert spaces, such that 
\[
\|u\|_{\rm s}^{(n)} - \eps < \|(\phi^*\cdot\psi\cdot \phi)^{(n)}(u)\|.
\]
Using Arveson's Extension Theorem, let $\wt{\phi}$ (resp. $\wt{\psi}$) be a completely contractive (resp. unital completely positive) extension of $\phi$ (resp. $\psi$). 
We have that $(\wt{\phi}, \wt{\psi})$ is an admissible pair for $(\cl E_2, \cl S_2)$ and 
\[
(\phi^*\cdot\psi\cdot\phi)^{(n)}(u) = (\wt{\phi}^*\cdot\wt{\psi}\cdot\wt{\phi})^{(n)}(\iota^{(n)}(u)).
\]
It follows that 
\begin{align*}
\|u\|_{\rm s}^{(n)} - \eps
<
|(\wt{\phi}^* \cdot \wt{\psi} \cdot \wt{\phi})^{(n)}(\iota^{(n)}(u))\|
 \leq
\|\iota^{(n)}(u)\|_{\rm s}^{(n)}.
\end{align*}
Taking $\eps \to 0$, together with (\ref{eq_inesy}), shows that $\iota$ is a complete isometry. 
Hence $\iota$ can be extended to a complete isometry (denoted in the same way)
\[
\iota \colon \cl E_1^* \otimes_{\rm s} \cl S_1 \otimes_{\rm s} \cl E_1 
\to \cl E_2^* \otimes_{\rm s} \cl S_2 \otimes_{\rm s} \cl E_2.
\]

By Proposition \ref{p_synth}, the map $\iota$ is completely positive. 
To complete the proof let $u\in M_n(\cl E_1^* \otimes_{\rm s} \cl S_1 \otimes_{\rm s} \cl E_1)$ such that $\iota^{(n)}(u)\in M_n(\cl E_2^* \otimes_{\rm s} \cl S_2 \otimes_{\rm s}\cl E_2)^+$. 
Let $(\phi, \psi)$ be an admissible pair for $(\cl E_1,\cl S_1)$, and let $(\wt{\phi},\wt{\psi})$ be an admissible pair for $(\cl E_2,\cl S_2)$ consisting of extensions of $\phi$ and $\psi$, as in the first paragraph.
We then have
\[
(\phi^*\cdot\psi\cdot\phi)^{(n)}(u) 
= 
(\wt{\phi}^*\cdot\wt{\psi}\cdot\wt{\phi})^{(n)}(\iota(u))
\in 
M_n(\cl B(H))^+.
\]
Since this holds for all admissible pairs $(\phi, \psi)$, by the definition of positivity we get $u \in M_n(\cl E_1^* \otimes_{\rm s} \cl S_1 \otimes_{\rm s} \cl E_1)^+$.
This shows that $\iota$ is a complete order embedding, and the proof is complete.
\end{proof} 

In the following theorem we collect some consequences of Lemma \ref{l_ssgen} for the symmetrisation.

%%%%%%%%%%%%%%%%%%%%%%%%%%%%
\begin{theorem}\label{t_symuni}
Let $\cl S$ be an operator system and $\cl E$ be an operator space. 
Then the following hold:
\begin{enumerate}
\item
If $\phi \colon \cl E\to \cl B(H,K)$ is a completely contractive map and $\psi \colon \cl S\to \cl B(K)$ is a completely positive map, then the map
\[
\phi^*\cdot\psi\cdot \phi \colon \cl E^* \otimes_{\rm s} \cl S \otimes_{\rm s} \cl E \to \cl B(H)
\]
is completely contractive and completely positive. 

\item
If a map $\wt{\theta}_{\rm s} \colon \cl E^*\otimes_{\rm s}\cl S \otimes_{\rm s}\cl E\to \cl B(H)$ is completely contractive and completely positive, then there exist a completely contractive map $\phi \colon \cl E \to \cl B(H,K)$ and a unital completely positive map $\psi \colon \cl S\to \cl B(K)$ such that 
\[
\wt{\theta}_{\rm s} = \phi^* \cdot \psi \cdot \phi.
\]

\item
If $\wt{\theta}_{\rm s} \colon \cl E^*\otimes_{\rm s}\cl S \otimes_{\rm s}\cl E\to \cl B(H)$ is a completely isometric completely positive map, and $\phi$ and $\psi$ are as in (ii), then $\phi$ is a complete isometry. 
\end{enumerate}
\end{theorem}

\begin{proof}
Item (i) follows from the definition of the operator space and the matrix order structure on $\cl E^*\otimes_{\rm s}\cl S \otimes_{\rm s} \cl E$. 

For item (ii), suppose that $\wt{\theta}_{\rm s} \colon \cl E^*\otimes_{\rm s}\cl S \otimes_{\rm s}\cl E\to \cl B(H)$ is completely contractive completely positive, and let $\theta$ be the associated trilinear map on $\cl E^* \times \cl S \times \cl E$.
By Remark \ref{r_norest}, $\|\theta\|_{\rm cb} = \|\wt{\theta}_{\rm s}\|_{\rm cb} \leq 1$, 
that is, $\theta$ is completely contractive.
Moreover, for every $x \in M_{m,n}(\cl E)$ and $s\in M_m(\cl S)^+$ we have
\[
\theta(x^*,s,x) = \wt{\theta}_{\rm s}(x^*\odot s \odot x)\in M_n(\cl B(H))^+,
\]
and thus $\theta$ is completely positive.
The conclusion follows from Lemma \ref{l_ssgen}.

For item (iii), let $x \in M_n(\cl E)$. 
If $\phi_0 \colon \cl E\to \cl B(H_0,K_0)$ is a complete isometry, then $\phi_0$ participates in an admissible pair $(\phi_0,\psi_0)$ by Remark \ref{ex_existad}, and using (ii) we have  
\begin{align*}
\|x\|^2 
& = 
\|\phi_0^{(n)}(x)^*\phi_0^{(n)}(x)\|
\leq \|x^* \odot (1_{\cl S} \otimes I_n) \odot x\|_{\rm s}^{(n)}\\ 
& = 
\|\wt{\theta}^{(n)}_{\rm s}(x^*\odot (1_{\cl S} \otimes I_n) \odot x)\|
=  \|\phi^{(n)}(x)^*\psi^{(n)}(1_{\cl S} \otimes I_n) \phi^{(n)}(x)\|\\
& \leq 
\|\psi^{(n)}(1_{\cl S} \otimes I_n)\| \|\phi^{(n)}(x)\|^2 = \|\phi^{(n)}(x)\|^2 \leq \|x\|^2.
\end{align*}
Thus, $\phi$ is a complete isometry. 
\end{proof}

%%%%%%%%%%%%%%%%%%%%%%%%%%%%
\begin{corollary}\label{c_concreo}
Let $\cl E$ be an operator space and $\cl S$ be an operator system.
Then there exist Hilbert spaces $H$ and $K$, a complete isometry $\phi \colon \cl E\to \cl B(H,K)$ and a completely positive map $\psi \colon \cl S\to \cl B(K)$, such that
\[
\cl E^*\otimes_{\rm s} \cl S \otimes_{\rm s}\cl E \simeq [\phi(\cl E)^*\psi(\cl S)\phi(\cl E)]
\]
completely isometrically.
Moreover, if $\cl E$ and $\cl S$ are separable, then the Hilbert spaces $H$ and $K$ can be chosen to be separable.
\end{corollary} 

\begin{proof}
Let $\wt{\theta}_{\rm s} \colon \cl E^* \otimes_{\rm s} \cl S \otimes_{\rm s} \cl E \to \cl B(H)$ be a 
(completely isometric) complete order embedding.
By item (iii) of Theorem \ref{t_symuni}, we can write $\wt{\theta}_{\rm s} = \phi^* \cdot \psi \cdot \phi$ for a complete isometry $\phi \colon \cl E \to \cl B(H, K)$ and a unital completely positive map $\psi \colon \cl S \to \cl B(K)$ with $K = [\phi(\cl E) H]$.
By construction, we have
\[
\cl E^*\otimes_{\rm s} \cl S \otimes_{\rm s}\cl E 
\simeq 
[\wt{\theta}_{\rm s}(\cl E^*\otimes_{\rm s} \cl S \otimes_{\rm s}\cl E)] 
= 
[\phi(\cl E)^*\psi(\cl S)\phi(\cl E)].
\]

Now assume that $\cl E$ and $\cl S$ are separable, and let $\ca(\wt{\theta}_{\rm s})$ be the C*-algebra generated by the image of a complete order embedding $\wt{\theta}_{\rm s}$ of the symmetrisation $\cl E^* \otimes_{\rm s} \cl S \otimes_{\rm s} \cl E$.
We have that ${\rm ran}(\wt{\theta}_{\rm s})$, and thus $\ca(\wt{\theta}_{\rm s})$, is separable.
Let 
\[
\rho \colon \ca(\wt{\theta}_{\rm s}) \to \cl B(H')
\]
be a faithful $*$-representation for a separable Hilbert space $H'$.
Applying item (iii) of Theorem \ref{t_symuni} to $\rho \circ \wt{\theta}_{\rm s}$ we get a decomposition into $(\phi')^* \cdot \psi' \cdot \phi'$  for a complete isometry $\phi' \colon \cl E \to \cl B(H', K')$ and a completely positive map $\psi \colon \cl E \to \cl B(K')$ with $K' = [\phi(\cl E) H']$.
As both $\cl E$ and $H'$ are separable, we get that $K'$ is separable, and by construction we have
\[
\cl E^*\otimes_{\rm s} \cl S \otimes_{\rm s}\cl E 
\simeq 
[\rho \circ \wt{\theta}_{\rm s}(\cl E^*\otimes_{\rm s} \cl S \otimes_{\rm s}\cl E)] 
= 
[\phi'(\cl E)^*\psi'(\cl S) \phi'(\cl E)],
\]
as required.
\end{proof}

%%%%%%%%%%%%%%%%%%%%%%%%%%%%
\begin{remark} \label{r_nd}
Let $\theta$ be a trilinear map as in the statement of Lemma \ref{l_ssgen}.
By construction, we can choose an admissible pair $(\phi,\psi)$ such that $\wt{\theta}_{\rm s} = \phi^*\cdot\psi\cdot \phi$ and $[\phi(\cl E)H] = K$.
If, in addition, $\theta$ satisfies the condition $[\theta(\cl E^*, 1_{\cl S}, \cl E) H] = H$, then $[\phi(\cl E)^* K] = H$, and thus $\phi$ is non-degenerate.

Conversely, suppose that $(\phi, \psi)$ is an admissible pair as in the statement of Theorem \ref{t_symuni}, and suppose that $\phi$ is non-degenerate.
Then the induced map $\wt{\theta}_{\rm s} = \phi^* \cdot \psi \cdot \phi$ on $\cl E^* \otimes_{\rm s} \cl S \otimes_{\rm s} \cl E$ satisfies
\[
[\wt{\theta}_{\rm s}(\cl E^* \otimes 1_S \otimes \cl E) H]
=
[\phi(\cl E)^* \phi(\cl E) H]
=
[\phi(\cl E)^* K]
=
H.
\]
Hence the delinearisation $\theta$ on $\cl E^* \times \cl S \times \cl E$ satisfies $[\theta(\cl E^*, 1_{\cl S}, \cl E) H] = H$.
\end{remark}

%%%%%%%%%%%%%%%%%%%%%%%%%%%%
\begin{remark}\label{r_Mn}
Suppose that the Hilbert space $H$ in Lemma \ref{l_ssgen} has dimension $n\in  \bb{N}$, and $\theta$ is completely bounded and $n$-positive in the sense that condition (\ref{eq_npos}) is fulfilled for the fixed $n$ and arbitrary $m\in \bb{N}$, that is
\[
\theta(x^*, s, x)\in M_n(\bb C), \ \ x \in M_{m,n}(\cl E), s \in M_m(\cl S)^+, m\in \bb{N}.
\]
By the universal property of the Haagerup tensor product, $\theta$ linearises to a completely bounded map $\wt{\theta}_{\rm h}$ on $\cl E^* \otimes_{\rm h} \cl S \otimes_{\rm h} \cl E$.
An application of Lemma \ref{l_norm} then gives a linearisation $\wt{\theta}_{\rm s}$ of $\theta$ to a completely bounded map on $\cl E^* \otimes_{\rm s} \cl S \otimes_{\rm s} \cl E$.
Since $\theta$ is $n$-positive, so is $\wt{\theta}_{\rm s}$.
Since $\cl E^* \otimes_{\rm s} \cl S \otimes_{\rm s} \cl E$ is a selfadjoint operator space and $H$ is $n$-dimensional, Theorem \ref{t_choi non-unital} yields that $\wt{\theta}_{\rm s}$ is completely positive. 
\end{remark}

Next we show that positivity can be traced by using just finite dimensional representations.

%%%%%%%%%%%%%%%%%%%%%%%%%%%%
\begin{lemma} \label{l_finen}
Let $\cl E$ be an operator space and $\cl S$ be an operator system.
Let $(\phi, \psi)$ be an admissible pair with $\phi \colon \cl E \to \cl B(H, K)$ and $\psi \colon \cl E \to \cl B(K)$.
Let $(P_\beta)_{\beta\in \bb{B}}$ be a net of finite rank projections with $\textup{wot-}\lim_\beta P_\beta = I_K$, and 
set $\psi_\be(\cdot) := P_\be \psi(\cdot) P_\be$ and 
$\phi_\be(\cdot) := P_\be \phi(\cdot)$.
Then $(\phi_\be, \psi_\be)$ is an admissible pair and
\[
\textup{wot-}\lim_\be (\phi_\be^* \cdot \psi_\be \cdot \phi_\be)^{(n)} (u) 
=
(\phi^* \cdot \psi \cdot \phi)^{(n)}(u), \ \ 
u \in M_n(\cl E^* \otimes_{\rm s} \cl S \otimes_{\rm s} \cl E).
\]
\end{lemma}

\begin{proof}
It is clear from the definitions that $\psi_\be$ is a unital completely positive map, and thus $(\phi_\be, \psi_\be)$ is an admissible pair, for all $\be \in \bb B$.
Hence $\phi^* \cdot \psi_\be \cdot \phi$ is completely contractive, for all $\be \in \bb B$.

Assume that $u = y^* \odot s \odot x$ for $x,y \in M_{m,n}(\cl E)$ and $s \in M_m(\cl S)$.
Then it is clear that
\begin{align*}
\textup{wot-}\lim_\be (\phi_\be^* \cdot \psi_\be \cdot \phi_\be)^{(n)} (y^* \odot s \odot x)
& = \\
& \hspace{-4cm} =
\textup{wot-}\lim_\be \phi^{(m,n)}(y)^* (P_\be \otimes I_n) \psi^{(n)}(s) (P_\be \otimes I_n) \phi^{(m,n)}(x) \\
& \hspace{-4cm} =
\phi^{(m,n)}(y)^* \psi^{(n)}(s) \phi^{(m,n)}(x)
=
(\phi^* \cdot \psi \cdot \phi)^{(n)} (y^* \odot s \odot x).
\end{align*}
Now let $u\in M_n(\cl E \otimes_{\rm s} \cl S \otimes_{\rm s} \cl E)$ be arbitrary. 
Fix unit vectors $\xi,\eta\in H$ and $\eps > 0$, and let $x,y \in M_{m,n}(\cl E)$ and $s \in M_m(\cl S)$ be such that 
\[
\|y^* \odot s \odot x - u\|_{\rm s} < \eps.
\] 
Choose $\beta_0\in \bb{B}$ such that
\[
\left|\langle ( (\phi^* \cdot \psi_\be \cdot \phi)^{(n)} (y^* \odot s \odot x) - 
(\phi^* \cdot \psi \cdot \phi)^{(n)}(y^* \odot s \odot x) ) \xi, \eta \rangle \right| < \eps,
\]
for all $\beta\geq \beta_0$. 
For $\beta\geq \beta_0$ we now have 
\begin{align*}
\left|\langle ( (\phi_\be^* \cdot \psi_\be \cdot \phi_\be)^{(n)}(u) - 
(\phi^* \cdot \psi \cdot \phi)^{(n)}(u) ) \xi, \eta \rangle \right| 
& \leq \\
& \hspace{-6.5cm} \leq 
\left|\langle ( (\phi_\be^* \cdot \psi_\be \cdot \phi_\be)^{(n)}(u) - 
(\phi_\be^* \cdot \psi_\be \cdot \phi_\be)^{(n)}(y^* \odot s \odot x) ) \xi, \eta \rangle \right| \\
& \hspace{-6.2cm} + 
\left|\langle ( (\phi_\be^* \cdot \psi_\be \cdot \phi_\be)^{(n)}(y^* \odot s \odot x) - 
(\phi^* \cdot \psi \cdot \phi)^{(n)}(y^* \odot s \odot x) )\xi,\eta\rangle\right|\\
& \hspace{-6.2cm} + 
\left|\langle ((\phi^* \cdot \psi \cdot \phi)^{(n)}(y^* \odot s \odot x) - 
(\phi^* \cdot \psi \cdot \phi)^{(n)}(u)\xi,\eta\rangle\right|  \\
& \hspace{-6.5cm} \leq
2 \|x^* \odot s \odot y - u\|_{\rm s} + \eps
\leq 3\eps,
\end{align*}
as required.
\end{proof}

%%%%%%%%%%%%%%%%%%%%%%%%%%%%
\begin{proposition} \label{p_finen}
Let $\cl E$ be an operator space and $\cl S$ be an operator system.
The following are equivalent:
\begin{enumerate}
\item $u \in M_n(\cl E \otimes_{\rm s} \cl S \otimes_{\rm s} \cl E)^+$;
\item $(\phi^* \cdot \psi \cdot \phi)^{(n)} (u) \geq 0$ for all the admissible pairs $(\phi, \psi)$ such that $\phi \colon \cl E \to \cl B(\bb C^k, K)$ and $\psi \colon \cl S \to \cl B(K)$, $k \in \bb N$;
\item $(\phi^* \cdot \psi \cdot \phi)^{(n)} (u) \geq 0$ for all the admissible pairs $(\phi, \psi)$ such that $\phi \colon \cl E \to \cl B(H, \bb C^m)$ and $\psi \colon \cl S \to M_m$, $m \in \bb N$;
\item $(\phi^* \cdot \psi \cdot \phi)^{(n)} (u) \geq 0$ for all the admissible pairs $(\phi, \psi)$ such that 
$\phi \colon \cl E \to M_{m,k}$ and $\psi \colon \cl S \to M_m$, $k, m \in \bb N$.
\end{enumerate}
\end{proposition}

\begin{proof}
The implications [(i)$\Rightarrow$(ii)$\Rightarrow$(iv)] are immediate.
It suffices to show that [(iv)$\Rightarrow$(iii)$\Rightarrow$(i)].

Assume that item (iv) holds and let $(\phi, \psi)$ be an arbitrary admissible pair with $\phi \colon \cl E \to \cl B(H, \bb{C}^m)$ and $\psi \colon \cl S \to M_m$.
Let $(P_\beta)_{\beta\in \bb{B}}$ be a net of projections of finite rank with $\textup{wot-}\lim_\be P_\beta= I_H$, and set $\phi_\be(\cdot) := \phi(\cdot) P_\be$.
Then $(\phi_\be, \psi)$ is an admissible pair and
\[
\textup{wot-}\lim_\be (\phi_\be^* \cdot \psi \cdot \phi_\be)^{(n)} (u) = (\phi^* \cdot \psi \cdot \phi)^{(n)}(u).
\]
Since positivity is preserved by limits in the weak operator topology, this shows that if $(\phi_\be^* \cdot \psi \cdot \phi_\be)^{(n)} (u) \geq 0$ for all $\be$, then $(\phi^* \cdot \psi \cdot \phi)^{(n)}(u) \geq 0$, from which item (iii) follows.

Assume that item (iii) holds and let $(\phi, \psi)$ be an arbitrary admissible pair with $\phi \colon \cl E \to \cl B(H, K)$ and $\psi \colon \cl E \to \cl B(K)$.
Let $(P_\beta)_{\beta\in \bb{B}}$ be a net of projections of finite rank with $\textup{wot-}\lim_\be P_\beta= I_K$, and set $\psi_\be(\cdot) := P_\be \psi(\cdot) P_\be$.
Then $(\phi, \psi_\be)$ is an admissible pair, and by Lemma \ref{l_finen} we have
\[
\textup{wot-}\lim_\be (\phi^* \cdot \psi_\be \cdot \phi)^{(n)} (u) = (\phi^* \cdot \psi \cdot \phi)^{(n)}(u).
\]
Since positivity is preserved by limits in the weak operator topology, if we have $(\phi^* \cdot \psi_\be \cdot \phi)^{(n)} (u) \geq 0$ for all $\be$, then $(\phi^* \cdot \psi \cdot \phi)^{(n)}(u) \geq 0$, from which item (i) follows.
\end{proof}

%%%%%%%%%%%%%%%%%%%%%%%%%%%%
\subsection{The symmetrisation norm versus the Haagerup norm}
%%%%%%%%%%%%%%%%%%%%%%%%%%%%

We will see that the operator space structure on $\cl E^*\otimes_{\rm s} \cl S \otimes_{\rm s} \cl E$ differs from the Haagerup tensor product $\cl E^*\otimes_{\rm h} \cl S \otimes_{\rm h} \cl E$.
Towards this end we will give an equivalent formula for the symmetrisation norm when $\cl E$ is a C*-algebra and $\cl S = \bb C$.
For a Hilbert space $H$ and a linear map $\Phi \colon \cl B(H)\to \cl B(H)$,
we set 
\[
\|\Phi\|_{+} := \sup \{ \|\Phi(T)\| \ : \ T\in \cl B(H), 0\leq T\leq I \}.
\]
For $u = \sum_{i=1}^N a_i^* \otimes b_i \in \cl B(H)^* \odot \cl B(H)$, let $\Phi_u \colon \cl B(H)\to \cl B(H)$ be given by 
\[
\Phi_u(T) = \sum_{i=1}^N a_i^* T b_i, \ \ \ T \in \cl B(H);
\]
recall that $\Phi_u(T)$ is well defined, 
by considering the bilinear map
\[
\cl B(H)^* \times \cl B(H) \to \cl B(H); \ (a^*,b) \mapsto a^* T b.
\]
The latter map is bilinear and completely bounded with cb-norm at most $\|T\|$, and $\Phi_u(T)$ is well defined as is the image of $u$ under the linearisation of this bilinear map.

%%%%%%%%%%%%%%%%%%%%%%%%%%%%
\begin{proposition}\label{p_pcn}
Let $\cl A$ be a unital C*-algebra and $u = \sum_{i=1}^N a_i^* \otimes b_i$ be an element of $\cl A^* \odot\cl A$. 
Then
\begin{equation}\label{eq_pit}
\|u\|_{\rm s} 
= 
\sup\left\{\left\|\Phi_{(\rho^*\otimes\rho)(u)}\right\|_{+} \ : \ \rho \colon \cl A\to \cl B(H) \mbox{ unital $*$-hom.}\right\}.
\end{equation}
Moreover, if $\cl A$ is separable, then (\ref{eq_pit}) holds with the supremum taken over $*$-represen\-tations on separable Hilbert spaces. 
\end{proposition}

\begin{proof}
Given a unital $*$-representation $\rho \colon \cl A\to \cl B(H)$ and $\eps > 0$, let $T\in \cl B(H)$ be a positive contraction, such that 
\[
\left\|\Phi_{(\rho^*\otimes\rho)(u)}\right\|_{+} - \eps < \|\Phi_{(\rho^* \otimes \rho)(u)}(T)\|.
\] 
Let $\phi \colon \cl A\to \cl B(H)$ be given by $\phi(\cdot) := T^{1/2}\rho(\cdot)$.
Clearly, $\phi$ is completely contractive, and
\[
\Phi_{(\rho^*\otimes\rho)(u)}(T) 
= 
\sum_{i=1}^N \rho(a_i)^*T\rho(b_i) 
= 
(\phi^*\cdot\phi)(u).
\]
Therefore,
\[
\left\|\Phi_{(\rho^*\otimes\rho)(u)}\right\|_{+} - \eps < \|u\|_{\rm s}.
\]
Since $\eps$ is arbitrary, the right hand side of (\ref{eq_pit}) is dominated by $\|u\|_{\rm s}$.

For the reverse inequality, let $\phi \colon \cl A\to \cl B(K)$ be a complete contraction and, using Wittstock's Theorem, write $\phi(\cdot) = W_2^* \rho(\cdot) W_1$ for a unital $*$-representa\-tion $\rho \colon \cl A\to \cl B(H)$ and contractions $W_1, W_2 \in \cl B(K,H)$. 
Using the positive contraction $T := W_2W_2^*$, we have 
\begin{align*}
\|(\phi^*\cdot\phi)(u)\| 
& = 
\left\|W_1^*\left(\sum_{i=1}^N \rho(a_i)^* W_2 W_2^* \rho(b_i)\right)W_1\right\| \\
& \leq 
\left\|\sum_{i=1}^N \rho(a_i)^* T \rho(b_i)\right\|
\leq  
\left\|\Phi_{(\rho^*\otimes\rho)(u)}\right\|_{+}.
\end{align*}
Taking supremum over all complete contractions $\phi$ yields the equality (\ref{eq_pit}). 

Next assume that $\cl A$ is separable.
By the first part of the proof, 
\[
\sup\left\{\left\|\Phi_{(\rho^*\otimes\rho)(u)}\right\|_{+} \ : \ \rho \colon \cl A\to \cl B(H) \mbox{ unital $*$-repn., $H$ separable}\right\}
\leq
\|u\|_{\rm s}.
\]
For the reverse inequality, let again $\phi \colon \cl A \to \cl B(K)$ be a complete contraction, $\rho \colon \cl A \to \cl B(H)$ be a $*$-representation, and $W_1, W_2 \in \cl B(K, H)$ be contractions such that $\phi(\cdot) = W_2^* \rho(\cdot) W_1$.
Since $\cl A$ is separable we can decompose the $*$-representation $\rho = \oplus_{\la \in \La} \rho_\la$ over a set $\La$ such that every $\rho_\la$ acts on a separable Hilbert space $H_\la$ with $H = \oplus_{\la \in \La} H_\la$.
We have
\begin{align*}
\left\|(\phi^*\cdot\phi)(u)\right\| 
& =
\left\| \sum_{i=1}^N W_1^* \rho(a_i)^* T \rho(b_i) W_1\right\| 
\leq
\left\| \sum_{i=1}^N \rho(a_i)^* T \rho(b_i) \right\|.
\end{align*}
For $\eps >0$ let $\xi \in H$ such that
\[
\left\| \sum_{i=1}^N \rho(a_i)^* T \rho(b_i) \right\| - \eps
<
\left\| \sum_{i=1}^N \rho(a_i)^* T \rho(b_i)\xi \right\|.
\]
Since $\sum_{i=1}^N \rho(a_i)^* T \rho(b_i)$ is bounded there exists an at most countable set $\La'_\xi \subseteq \La$ such that, for $H' := \oplus_{\la \in \La'_\xi} H_\la$, we have
\[
\sum_{i=1}^N \rho(a_i)^* T \rho(b_i) \xi
=
P_{H'} \left(\sum_{i=1}^N \rho(a_i)^* T \rho(b_i) \right) P_{H'} \xi.
\]
By definition, $H'$ is separable and reducing for $\rho$.
We then have
\begin{align*}
\left\|(\phi^*\cdot\phi)(u)\right\| - \eps
& \leq
\left\| \sum_{i=1}^N \rho(a_i)^* T \rho(b_i)\xi \right\| 
=
\left\| \sum_{i=1}^N P_{H'} \rho(a_i)^* T \rho(b_i) P_{H'}\xi \right\| \\
& \leq
\left\| \sum_{i=1}^N P_{H'} \rho(a_i)^* P_{H'} T P_{H'} \rho(b_i) \right\| 
\leq
\left\| \Phi_{(\rho^*|_{H'} \otimes \rho|_{H'})}(u) \right\|_+,
\end{align*}
where we used that $P_{H'} T P_{H'}$ is a positive contraction in $\cl B(H')$.
Since $H'$ is separable this quantity is dominated by the right hand side of (\ref{eq_pit}) when the supremum is taken over separable Hilbert spaces, and taking $\eps \to 0$ completes the proof.
\end{proof}

The difference between $\| \cdot \|_{\rm s}$ and $\| \cdot \|_{\rm h}$ becomes clear from Proposition \ref{p_pcn}, since $\|u\|_{\rm h}$ coincides with the cb-norm of $\Phi_u \colon \cl B(H) \to \cl B(H)$ for $u = \sum_{i=1}^N a_i^* \otimes b_i$ when $\cl A \subseteq \cl B(H)$, that is,
$\nor{u}_{\rm h} = \nor{\Phi_u}_{\rm cb}$ 
(see for example \cite[Theorem 5.12]{pisier_intr}).
Let us provide a concrete example.

%%%%%%%%%%%%%%%%%%%%%%%%%%%%
\begin{example}\label{e_diffhs}
The separably acting unital $*$-representations of $M_2$ are all unitarily equivalent to sub-representations of
$\pi \colon M_2\to \cl B(\bb{C}^2\otimes \ell^2)$, where $\pi(x) = x\otimes I$, $x\in M_2$. 
If $\si$ is a sub-representation of $\pi$ acting on $K \subseteq \bb C^2 \otimes \ell^2$ and $ u = \sum_{i=1}^N a_i^* \otimes b_i \in M_2^* \odot M_2$, then 
\[
\Phi_{(\si^* \otimes \si)(u)}(P_K T P_K) = P_K \Phi_{(\pi^* \otimes \pi)(u)}(T) P_K,
\textup{ for all } T \in \cl B(\bb{C}^2\otimes \ell^2),
\]
and thus
\begin{equation}\label{eq_plus}
\| \Phi_{(\si^* \otimes \si)(u)} \|_+ \leq \|\Phi_{(\pi^* \otimes \pi)(u)}\|_+.
\end{equation}
Therefore, 
\begin{equation} \label{eq_s=id}
\nor{u}_{\rm s} = \|\Phi_{(\pi^* \otimes \pi)(u)}\|_+, \ \ \ u \in M_2^* \odot M_2.
\end{equation}

Let 
$p = \left(\smallmatrix 1 & 0\\ 0 & 0\endsmallmatrix\right)$ and, for $t\in (0,1)$, 
let  $u_t\in M_2$ be the unitary matrix
\[
u_t = \begin{pmatrix} t & \sqrt{1 - t^2} \\ \sqrt{1-t^2} & -t \end{pmatrix},
\]
so that $u_te_1 = (t,\sqrt{1-t^2})$. 
Set $q_t := u_tpu_t^*$.
Since both $p$ and $q_t$ are projections and the Haagerup norm is a cross-norm  \cite[p. 142]{er}, we have 
\begin{equation}\label{eq_h=1}
\|p\otimes q_t\|_{\rm h} = 1.
\end{equation}
Let 
\[
\gamma(t) 
:= 
\sup\left\{\|(p\otimes I) T (u_tp \otimes I)\| \ : \ T\in \cl B(\bb{C}^2\otimes \ell^2), 0\leq T\leq I\right\}.
\]
By (\ref{eq_plus}), Proposition \ref{p_pcn} and the fact that $u_t^* \otimes I$ is a unitary, we have 
\begin{align*}
\|p\otimes q_t\|_{\rm s}
& =
\sup\left\{\|(p\otimes I) T (u_t p u_t^* \otimes I)\| \ : \ T \in \cl B(\bb{C}^2\otimes \ell^2), 0 \leq T \leq 1\right\} \\
& =
\sup\left\{\|(p\otimes I) T (u_t p \otimes I)\| \ : \ T \in \cl B(\bb{C}^2\otimes \ell^2), 0 \leq T \leq 1\right\}\\
& =
\ga(t).
\end{align*}
Suppose that the symmetrisation norm and the Haagerup norm coincided; then equations (\ref{eq_s=id}) and (\ref{eq_h=1}) would imply that $\gamma(t) = 1$.

Let $\eps = \min\left\{\frac{1}{2}, 2t^2\right\}$, $T\in M_2\otimes \cl B(\ell^2)$ be a positive contraction, and $\xi_t,\eta_t \in \bb{C}^2\otimes\ell^2$ be unit vectors such that 
\[
1 - \eps < |\sca{ (p\otimes I) T (u_tp \otimes I) \xi_t, \eta_t} | \leq 1.
\]
Replacing $\xi_t$ with $(p \otimes I) \xi_t$, and $\eta_t$ with $(p \otimes I) \eta_t$, we assume that 
\[
(p \otimes I)\xi_t = \xi_t, \ (p \otimes I)\eta_t = \eta_t, \textup{ and } 1 - \eps < |\langle (u_t\otimes I)\xi_t, T\eta_t \rangle| \leq 1.
\]
Replacing further $\eta_t$ with a suitable unimodular multiple, we can assume that $\langle (u_t\otimes I)\xi_t, T\eta_t \rangle \in \bb{R}^+$. 
Then we have 
\begin{equation}\label{eq_xpu0}
\|T\eta_t - (u_t\otimes I)\xi_t \|^2
=
\|T\eta_t\|^2 - 2\sca{(u_t \otimes I)\xi_t, T\eta_t} + \|(u_t \otimes I) \xi_t\|^2
< 2\eps.
\end{equation}
Write $\xi_t = (\xi_t^{(i)})_{i\in \bb{N}}$ and $\eta_t = (\eta_t^{(i)})_{i\in \bb{N}}$, 
where $\xi_t^{(i)}, \eta_t^{(i)}\in \bb{C}^2$,  
and let the scalars $\lambda_i(t),\mu_i(t) \in \bb{C}$ satisfying
\[
\lambda_i(t) e_1 =  \xi_t^{(i)} \qand \mu_i(t)e_1 = \eta_t^{(i)}, \ \ \ i \in \bb{N},
\]
since $(p \otimes I)\xi_t = \xi_t$ and $(p \otimes I)\eta_t = \eta_t$.
Letting $\lambda(t) = (\lambda_i(t))_{i\in \bb{N}}$ and $\mu(t) = (\mu_i(t))_{i\in \bb{N}}$ (as vectors in $\ell^2$), we see that
\[
\|\la(t)\| = \|\xi_t\| = 1
\qand
\|\mu(t)\| = \|\eta_t\| = 1.
\]
Then 
\[
(u_t\otimes I)\xi_t = \left(t\lambda_i(t)e_1 + \sqrt{1-t^2}\lambda_i(t)e_2\right)_{i\in \bb{N}}
\in \bb C^2 \otimes \ell^2.
\]
After applying the canonical shuffle, consider $T$ as an element 
\[
T = 
\left(\begin{matrix}
A & B \\
B^* & C
\end{matrix}\right)
\in M_2(\cl B(\ell^2)).
\] 
Then the inequality (\ref{eq_xpu0}) reads
\[
\left\|\left(\begin{matrix}
A & B \\
B^* & C
\end{matrix}\right)\left[\begin{matrix}
\mu(t) \\
0
\end{matrix}\right]
- 
\left[\begin{matrix}
t \lambda(t) \\
\sqrt{1-t^2}\lambda(t)
\end{matrix}\right] \right\|^2 < 2\eps,
\]
implying that
\[
\left\|A\mu(t) - t \lambda(t)\right\|^2 < 2\eps
\ \mbox{ and } \ \left\|B^*\mu(t) - \sqrt{1-t^2}\lambda(t)\right\|^2 < 2\eps.
\]
By applying the triangle inequality we have
\begin{align}\label{eq_mumu}
|\sca{A \mu(t), \mu(t)}|
& \leq 
|\sca{t\la(t), \mu(t)}| + |\sca{A\mu(t) - t \lambda(t), \mu(t)}| \nonumber\\
& \leq
t |\sca{\la(t), \mu(t)}| + \sqrt{2 \eps},
\end{align}
and
\begin{align}\label{eq_mula}
\left|\sca{\la(t), \sqrt{1-t^2} \la(t)}\right|
& \leq
|\sca{\la(t), B^* \mu(t)}| + \left|\sca{\la(t), B^*\mu(t) - \sqrt{1-t^2}\lambda(t)}\right| \nonumber\\
& \leq
|\sca{\la(t), B^* \mu(t)}| + \sqrt{2 \eps}.
\end{align}
On the other hand, using the positivity of $T$ and the fact that $C$ is a contraction, we have
\begin{align*}
\left|\sca{\la(t), B^* \mu(t)}\right| 
& =
|\sca{B\la(t), \mu(t)}| \\
& \leq
|\sca{A \mu(t), \mu(t)}| \cdot |\sca{C \la(t), \la(t)}| \\
& \leq 
|\sca{A \mu(t), \mu(t)}|,
\end{align*}
see for example \cite[Exercise 3.2 (i)]{Pa}.
Using (\ref{eq_mumu}) and (\ref{eq_mula}), we get
\begin{align*}
\sqrt{1-t^2}
& = 
|\langle\lambda(t),\sqrt{1-t^2}\lambda(t) \rangle |
\leq 
\left|\langle\lambda(t),B^* \mu(t) \rangle\right| + \sqrt{2\eps}\\
& \leq 
\left|\langle A\mu(t), \mu(t)\rangle\right| + \sqrt{2\eps}
\leq t \left|\langle\lambda(t),\mu(t)\rangle \right|  + 2\sqrt{2\eps}\\
& \leq 
t + 2\sqrt{2\eps} \leq 5t,
\end{align*}
a condition that fails if $t < \frac{1}{\sqrt{26}}$. 
\end{example}

We next provide a more general result that showcases the difference between the symmetrisation norm and the Haagerup tensor norm.

%%%%%%%%%%%%%%%%%%%%%%%%%%%%
\begin{theorem}\label{t_diffhs}
Let $\cl A$ be a unital C*-algebra with dimension greater or equal than $2$.
Then the completely contractive map $\cl A^* \otimes_{\rm h} \cl A \to \cl A^* \otimes_{\rm s} \cl A$ fixing $\cl A^* \odot \cl A$ is not completely isometric.
\end{theorem}

%The existence of the completely contractive map $\cl A^* \otimes_{\rm h} \cl A \to \cl A^* \otimes_{\rm s} \cl A$ fixing $y^* \otimes x$ is provided by the universal property of the Haagerup tensor product.
In order to show that the map from 
Corollary \ref{t_diffhs} does not admit a completely contractive inverse, we will need a technical lemma.

%%%%%%%%%%%%%%%%%%%%%%%%%%%%
\begin{lemma}\label{l_technical}
Let $H$ be a Hilbert space, 
$\cl X \subseteq \cl B(H)$ be a selfajoint operator space, $\cl E$ be a unital operator space, and $\iota_{\rm env} \colon \cl E \to \cenv(\cl E)$ be the canonical unital completely isometric embedding.
Let $\Phi \colon \cl X \to \cl E$ be a surjective completely isometric map and suppose that $\Phi(u) = 1_{\cl E}$ for a positive element $u \in \cl X$.
Then $\iota_{\rm env} \circ \Phi$ is a completely isometric completely positive map.
Consequently, $\iota_{\rm env}(\cl E) = \iota_{\rm env} \circ \Phi(\cl X)$ is an operator subsystem of $\cenv(\cl E)$.
\end{lemma}

\begin{proof}
Without loss of generality, we assume that $\cl E$ is concretely represented inside $\cenv(\cl E) \subseteq \cl B(K)$ as a unital operator subspace.
We will show that the map $\Phi \colon \cl X \to \cl B(K)$ is completely positive.

Let $\cl M$ be the TRO generated by $\cl X$ inside $\cl B(H)$.
By the universal property of the ternary envelope, there is a 
canonical ternary morphism $\cl M \to \cl T_{\rm env}(\cl E)$.
Since $\cl E$ is unital, its ternary envelope coincides with $\cenv(\cl E)$, see for example \cite[Remark 8.3.12, item (5)]{blm}; thus the map $\Phi$ extends to a ternary morphism
\[
\wt{\Phi} \colon \cl M \to \cenv(\cl E),
\]
see for example \cite[Remark 8.3.12, item (2)]{blm}.
Since $\wt{\Phi}$ is a ternary morphism, we obtain $*$-representations
\[
\rho \colon [\cl M \cl M^*] \to \cl B(K); \ yx^* \mapsto \wt{\Phi}(y) \wt{\Phi}(x)^*,
\]
and
\[
\si \colon [\cl M^* \cl M] \to \cl B(K); \ y^*x \mapsto \wt{\Phi}(y)^* \wt{\Phi}(x),
\]
so that
\[
\wt{\Phi}(a \cdot x \cdot b) = \rho(a) \wt{\Phi}(x) \si(b), \ \ \  a \in [\cl M \cl M^*], b \in [\cl M^* \cl M], x \in \cl M.
\]

We claim that $u^{1/2} \in [\cl M \cl M^*]$ and that $\rho(u^{1/2}) = I_K$; and similarly that $u^{1/2} \in [\cl M^* \cl M]$ and that $\si(u^{1/2}) = I_K$.
Therefore we will have
\[
\wt{\Phi}(u^{1/2} \cdot x \cdot u^{1/2}) = \rho(u^{1/2}) \wt{\Phi}(x) \si(u^{1/2}) = \wt{\Phi}(x),
\]
for all $x \in \cl M$.
We will show that this is the case for $\rho$; similar arguments apply for $\si$.

Towards this end, first note that $u^2 = u u^* \in [\cl M \cl M^*]$ since $u \geq 0$.
Since $[\cl M \cl M^*]$ is a C*-algebra, we get $u^{1/2} = (u^2)^{1/4} \in [\cl M \cl M^*]$.
Since $\rho$ is a $*$-representation, we thus get
\[
\rho(u^{1/2})^4 = \rho(u^2) = \rho(u u^*) = \wt{\Phi}(u) \wt{\Phi}(u)^* = \Phi(u) \Phi(u)^* = I_K \cdot I_K = I_K;
\]
since $u^{1/2} \geq 0$, we have $\rho(u^{1/2}) = I_K$.

Recall that ternary morphisms are completely contractive, see \cite[Lemma 8.3.2]{blm}.
Hence, by Wittstock's Extension Theorem, we can extend $\wt{\Phi}$ to a completely contractive map $\wh{\Phi} \colon \cl B(H) \to \cl B(K)$, and set
\[
\Psi \colon \cl B(H) \to \cl B(K); \ x \mapsto \wh{\Phi}(u^{1/2} \cdot x \cdot u^{1/2}).
\]
Note here that
\[
\|u^{1/2} \| = \| u \|^{1/2} = \|\Phi(u)\|^{1/2} = 1.
\]
%and therefore 
%\begin{align*}
%\| \Psi^{(n)}([x_{ij}]) \| 
%& = \|\wh{\Phi}^{(n)}([u^{1/2} x_{ij} u^{1/2}]) \| \leq \| [u^{1/2} x_{ij} u^{1/2}] \| \\
%& = \| (u^{1/2} \otimes I_n) [x_{ij}] (u^{1/2} \otimes I_n) \| \leq \| [x_{ij}] \|.
%\end{align*}
Hence $\Psi$ is completely contractive and satisfies
\[
\Psi(I_H) = \wh{\Phi}(u^{1/2} \cdot I_H \cdot u^{1/2}) = \wh{\Phi}(u) = \Phi(u) = I_K.
\]
Thus $\Psi$ is unital completely contractive, and therefore it is a unital completely positive map by \cite[Proposition 3.5]{Pa}.
Note that, if $x \in \cl X$, then
\[
u^{1/2} \cdot x \cdot u^{1/2} \in [\cl M \cl M^* \cl X \cl M^* \cl M] \subseteq [\cl M \cl M^* \cl M \cl M^* \cl M] = \cl M.
\]
Therefore we have
\begin{align*}
\Psi(x) 
& = 
\wh{\Phi}(u^{1/2} \cdot x \cdot u^{1/2}) 
= 
\wt{\Phi}(u^{1/2} \cdot x \cdot u^{1/2}) \\
& = 
\rho(u^{1/2}) \wt{\Phi}(x) \si(u^{1/2}) 
= 
\Phi(x).
\end{align*}
Thus $\Phi$ is a completely positive map as the restriction of the completely positive map $\Psi$ to $\cl X$.

To finish the proof, we observe that 
\[
\cl E^* = \Phi(\cl X)^* = \Psi(\cl X)^* = \Psi(\cl X^*) = \Psi(\cl X) = \Phi(\cl X) = \cl E,
\]
since $\Psi$ is selfadjoint and $\Psi|_{\cl X} = \Phi$.
Hence $\cl E$ is a selfadjoint operator subspace of $\cl B(K)$ that contains $I_K$, and thus it is an operator system.
\end{proof}

We now return to the proof of Theorem \ref{t_diffhs}.\\ 

\noindent
{\bf Proof of Theorem \ref{t_diffhs}.}
We will write $\cl A^* \ast_1 \cl A$ for the free product of $\cl A$ with itself.
By the Christensen--Effros--Sinclair--Pisier Theorem \cite[Theorem 17.11]{Pa}, the linear map
\[
\cl A^* \otimes_{\rm h} \cl A \to \cl A^* *_1 \cl A; \ y^* \otimes x \mapsto y^* * x
\]
is a complete isometry, mapping $1_{\cl A} \otimes 1_{\cl A}$ to the unit of $\cl A^* *_1 \cl A$.
Hence $\cl A^* \otimes_{\rm h} \cl A$ is a unital operator space.
Let  
\[
\iota_{\rm env} \colon \cl A^* \otimes_{\rm h} \cl A \to \cenv(\cl A^* \otimes_{\rm h} \cl A) \subseteq \cl B(K)
\]
be the canonical unital completely isometric map.
To reach a contradiction, suppose that the map 
\[
\Phi \colon \cl A^* \otimes_{\rm s} \cl A \to \cl A^* \otimes_{\rm h} \cl A; \ y^* \otimes_{\rm s} x \mapsto y^* \otimes_{\rm h} x
\]
is completely isometric (where we are using the notation $\otimes_{\rm s}$ and $\otimes_{\rm h}$ to differentiate between the symmetrisation and the Haagerup tensor product).
Then $\Phi(u) = 1_{\cl A} \otimes_{\rm h} 1_{\cl A}$, for $u = 1_{\cl A} \otimes_{\rm s} 1_{\cl A}$ which is a positive element in the symmetrisation by Proposition \ref{p_synth}.
By Lemma \ref{l_technical}, $\iota_{\rm env} (\cl A^* \otimes_{\rm h} \cl A) = \iota_{\rm env} \circ \Phi(\cl A^* \otimes_{\rm s} \cl A)$ is an operator system with
\begin{align*}
\iota_{\rm env}(y^* \otimes_{\rm h} x)^*
& =
(\iota_{\rm env} \circ \Phi)(y^* \otimes_{\rm s} x)^* 
= 
(\iota_{\rm env} \circ \Phi)( (y^* \otimes_{\rm s} x)^* ) \\
& = (\iota_{\rm env} \circ \Phi)(x^* \otimes_{\rm s} y) 
=
\iota_{\rm env} (x^* \otimes_{\rm h} y).
\end{align*}

On the other hand, by the universal property of the Haagerup tensor product we derive a completely contractive map
\[
\Psi \colon \iota_{\rm env}(\cl A^* \otimes_{\rm h} \cl A) \to \cl A^* \otimes_{\rm h} \cl A \to \cl A^* \otimes_{\min} \cl A; \ \iota_{\rm env}(y^* \otimes_{\rm h} x) \mapsto y^* \otimes_{\min} x.
\]
Since the map $\Psi$ is unital and $\iota_{\rm env} (\cl A^* \otimes_{\rm h} \cl A)$ is an operator system, it follows that 
$\Psi$ is completely positive by \cite[Proposition 3.5]{Pa}, and thus selfadjoint.
In particular,
\begin{align*}
y^* \otimes_{\min} x
& = 
(y \otimes_{\min} x^*)^* = \Psi( \iota_{\rm env}(y \otimes_{\rm h} x^*) )^* \\
& =
\Psi( \iota_{\rm env} (y \otimes_{\rm h} x^*)^* )
=
\Psi( \iota_{\rm env}(x \otimes_{\rm h} y^*) )
=
x \otimes_{\min} y^*, 
\end{align*}
for all $x,y \in \cl A$, which is a contradiction since $\dim \cl A \geq 2$.
\hfill{$\qedsymbol$}

%%%%%%%%%%%%%%%%%%%%%%%%%%%%
\subsection{Semi-units}
%%%%%%%%%%%%%%%%%%%%%%%%%%%%

In \cite{ekt}, a concrete operator space $\cl E \subseteq \cl B(H, K)$ was called \emph{semi-unital} if $I_H \in [\cl E^* \cl E]$; equivalently, if there is a pair $((\underline{y}_i)_{i\in \bb{I}}, (\underline{x}_i)_{i\in \bb{I}})$ of nets of finitely supported columns over $\cl E$ such that $\underline{y}_i^* \underline{x}_i \longrightarrow_{i\in \bb{I}} I_H$ in the norm topology. 
Here we will work with a version of this notion for abstract operator spaces, which we now define.

Let $\cl E$ be an operator space and $\phi \colon \cl E\to \cl B(H,K)$ be a complete isometry.
A pair $((\underline{y}_i)_{i\in \bb{I}}, (\underline{x}_i)_{i\in \bb{I}})$ of nets of finitely supported columns over $\cl E$ is called a \emph{$\phi$-semi-unit} for $\cl E$, if 
\[
\phi^{(\infty)}(\underline{y}_i)^*\phi^{(\infty)}(\underline{x}_i) \longrightarrow_{i\in \bb{I}} I_H
\]
in the norm topology.
We say that $\cl E$ is \emph{semi-unital} if it admits a $\phi$-semi-unit for a complete isometry $\phi$. 
A $\phi$-semi-unit $((\underline{y}_i)_{i\in \bb{I}}, (\underline{x}_i)_{i\in \bb{I}})$ is called \emph{symmetric} if $\underline{y}_{i} = \underline{x}_{i}$ for every $i\in \bb{I}$.

Let, in addition, $\cl S$ be an operator system and $(\phi,\psi)$ be an admissible pair associated with a pair $(H,K)$ of Hilbert spaces, such that $\wt{\theta}_{\rm s} := \phi^* \cdot \psi \cdot \phi$ is a completely isometric complete order embedding of $\cl E^* \otimes_{\rm s} \cl S \otimes_{\rm s} \cl E$. 
A pair $((\underline{y}_i)_{i\in \bb{I}}, (\underline{x}_i)_{i\in \bb{I}})$ of nets of finitely supported columns over $\cl E$ is called a \emph{$(\phi,\psi)$-semi-unit for $\cl E$ relative to $\cl S$} if 
\[
\phi^{(\infty)}(\underline{y}_i)^*\phi^{(\infty)}(\underline{x}_i) \longrightarrow_{i\in \bb{I}} I_H.
\]
We say that $\cl E$ is \emph{$\cl S$-semi-unital} if there exists a $(\phi,\psi)$-semi-unit for some admissible pair $(\phi,\psi)$ that gives rise to a completely isometric complete order embedding  $\wt{\theta}$ of $\cl E^* \otimes_{\rm s} \cl S \otimes_{\rm s} \cl E$; it follows that if $\cl E$ is $\cl S$-semi-unital, then the selfadjoint operator space $\cl E^* \otimes_{\rm s} \cl S \otimes_{\rm s} \cl E$ is an operator system. 
By item (iii) of Theorem \ref{t_symuni}, if $\cl E$ is $\cl S$-semi-unital, then $\cl E$ is semi-unital.
By the proof of Proposition \ref{p_nd} we can assume that the pair $(\phi,\psi)$ associated to the semi-unit is non-degenerate.

Recall that, due to the injectivity of the symmetrisation (Theorem \ref{t_injective}), there exists a completely isometric complete order embedding
\[
\iota \colon \cl E^* \otimes_{\rm s} \cl E \to \cl E^* \otimes_{\rm s} \cl S \otimes_{\rm s} \cl E; \ y^* \otimes \la \otimes x \mapsto y^* \otimes \la \cdot 1_{\cl S} \otimes x.
\]
We will say that the embedding $\iota$ is \emph{unital} provided that both the selfadjoint operator spaces are operator systems and $\iota$ preserves the Archimedean matrix order units.

%%%%%%%%%%%%%%%%%%%%%%%%%%%%
\begin{proposition}\label{p_semi-unit}
Let $\cl E$ be an operator space and $\cl S$ be an operator system. 
The following are equivalent:
\begin{enumerate}
\item $\cl E$ is $\cl S$-semi-unital;
\item the canonical embedding $\iota \colon \cl E^* \otimes_{\rm s} \cl E \to 
\cl E^* \otimes_{\rm s} \cl S \otimes_{\rm s} \cl E$ is unital.
\end{enumerate}
In particular, if $\cl E$ is $\cl S$-semi-unital, then it admits a symmetric $\cl S$-semi-unit.
\end{proposition}

\begin{proof}
\noindent
[(i)$\Rightarrow$(ii)]:
Suppose that $\cl E$ is $\cl S$-semi-unital. 
By definition, there exist a Hilbert space $H$ and an embedding $\wt{\theta}_{\rm s} = \phi^* \cdot \psi \cdot \phi$ of $\cl E^* \otimes_{\rm s} \cl S \otimes_{\rm s} \cl E$ in $\cl B(H)$, and a pair $((\underline{y}_i)_{i\in \bb{I}}, (\underline{x}_i)_{i\in \bb{I}})$ of nets such that
\[
\phi^{(\infty)}(\underline{y}_i)^*\phi^{(\infty)}(\underline{x}_i) \longrightarrow_{i\in \bb{I}} I_H.
\]
Hence $\cl E^* \otimes_{\rm s} \cl S \otimes_{\rm s} \cl E$ is an operator system.
Due to the embedding so is the space $\cl E^* \otimes_{\rm s} \bb{C} \otimes_{\rm s} \cl E$ with the same unit, which by definition is preserved by the inclusion.

Since the unit is in 
$(\cl E^* \otimes_{\rm s} \cl S \otimes_{\rm s} \cl E)^+$, Proposition \ref{p_synth} implies that the $\cl S$-semi-unit can be chosen to be symmetric. 
%by considering a net in the cones of $\cl E^* \otimes_{\rm s} \bb{C} \otimes_{\rm s} \cl E$.

\smallskip

\noindent
[(ii)$\Rightarrow$(i)]:
Let $\wt{\theta}_{\rm s} \colon \cl E^* \otimes_{\rm s} \cl S \otimes_{\rm s} \cl E\to \cl B(H)$ be a completely isometric complete order isomorphism onto an operator system in $\cl B(H)$. 
Write $\wt{\theta}_{\rm s} = \phi^*\cdot\psi\cdot\phi$ by virtue of Theorem \ref{t_symuni} so that $\phi$ is a complete isometry. 
By the assumption in (ii), we have
\[
\wt{\theta}_{\rm s}(1) = (\phi^*\cdot\psi\cdot\phi)(1) = (\phi^*\cdot\phi)(1).
\]
The conclusion follows from the fact that $[(\phi^*\cdot\phi)(\cl E^* \otimes_{\rm s} \cl E)] = [\phi(\cl E)^*\phi(\cl E)]$. 
\end{proof}

The existence of a semi-unit for the symmetrisation is a rather strong assumption as we are about to show.
There are several obstructions in this respect.
The first one is the existence of an order unit.

%%%%%%%%%%%%%%%%%%%%%%%%%%%%
\begin{theorem}\label{t_nonsepnonun} 
If $\cl E$ is a non-separable operator space, then the selfadjoint operator space $\cl E^*\otimes_s\cl E$ does not admit an order unit, and thus it is not an operator system.
\end{theorem}

\begin{proof} 
We will show that for every positive element $e \in \cl E^* \otimes_{\rm s} \cl E$ there exists an $x \in \cl E$ such that the inequality $x^* \otimes x \leq r e$ does not hold for any $r > 0$.
Towards this end, fix a positive element $e \in \cl E^* \otimes_{\rm s} \cl E$.
Let $\wt{\theta}_{\rm s} \colon \cl E^* \otimes_{\rm s} \cl E \rightarrow \cl B(H)$ be a complete order embedding.
By Corollary \ref{c_concreo}, there exists a complete isometry $\phi\colon \cl E\rightarrow \cl B(H, K)$ such that $\wt{\theta}_{\rm s} = \phi^* \cdot \phi$.
Let $(y^n_{i})_{i=1}^{k_n}, (x^n_{i})_{i=1}^{k_n} \subseteq \cl E$, $n\in \bb{N}$, such that 
\[
e = \lim_{n\to \infty}\sum_{i=1}^{k_n} (y^n_{i})^* \otimes x^n_{i},
\]
and set
\[
\Omega := \ol{\rm span} \{ x^n_{i} \ : \ i=1,\dots,k_n, n\in \bb N \}.
\]
Since $\cl E$ is non-separable there exists $0 \neq x\in \cl E$ such that $x \notin \Omega$.

By the Hahn-Banach Theorem, there exists a contractive functional $\vphi$ on $\cl B(H, K)$, such that $\vphi( \phi(x) ) \neq 0$ and $\vphi \circ \phi(\Omega) = \{0\}$.
Since the contractive functional $\vphi \circ \phi \colon \cl E \to \bb C$ is automatically completely contractive, by Theorem \ref{t_symuni} the map $(\vphi \circ \phi)^* \cdot (\vphi \circ \phi)$ is a completely contractive completely positive functional on the symmetrisation $\cl E^* \otimes_{\rm s} \cl E$ such that
\[
((\vphi \circ \phi)^* \cdot (\vphi \circ \phi))(x^* \otimes x) 
= 
\varphi (\phi(x))^*\varphi (\phi(x)) 
= 
| \varphi(\phi(x)) |^2 > 0.
\]
If there existed an $r >0$ such that $x^*\otimes x \leq r e$, then we would have 
\begin{align*}
0 
& < 
((\vphi \circ \phi)^* \cdot (\vphi \circ \phi) )(x^* \otimes x) 
\leq 
r ((\vphi \circ \phi)^* \cdot (\vphi \circ \phi) )(e) \\
& = 
r \lim_{n\to \infty} \sum_{i=1}^{k_n} \varphi(\phi(y^n_{i}))^* \varphi(\phi(x^n_{i}))=0,
\end{align*}
a contradiction.  
\end{proof}

The second obstruction to the symmetrisation being an operator system resides in the matching of the norms as explained in Remark \ref{r_sos unit}, although an Archimedean matrix order unit may exist.

%%%%%%%%%%%%%%%%%%%%%%%%%%%%
\begin{theorem} \label{t_notos}
Let $\cl E \subseteq \cl B(H)$ be an operator space.
Then $\cl E^* \otimes_{\rm s} \cl E$ is an operator system if and only if there exist Hilbert spaces $H'$ and $L$, a unital $*$-representation $\rho \colon \cl B(H) \to \cl B(H')$ and an isometry $W \colon L \to H'$ such that the completely contractive completely positive map
\[
\wt{\theta}_{\rm s} \colon \cl E^* \otimes_{\rm s} \cl E \to \cl B(L); \ y^* \otimes x \mapsto W^* \rho(y^* x) W
\]
is a completely isometric complete order embedding with 
\[
I_L \in \wt{\theta}_{\rm s}(\cl E^* \otimes_{\rm s} \cl E) = [W^* \rho(\cl E^* \cl E) W].
\]
\end{theorem}

\begin{proof}
Suppose that $\cl E^* \otimes_{\rm s} \cl E$ is an operator system with an Archimedean matrix order unit $e$.
Then, by Remark \ref{r_sos unit},
\[
\nor{\cdot}_{\rm s}^{(n)} = \nor{\cdot}_e^{(n)}, 
\foral 
n \in \mathbb{N},
\]
and there exists a unital completely isometric complete order embedding 
\[
\wt{\theta}_{\rm s} \colon \cl E^* \otimes_{\rm s} \cl E \to \cl B(L).
\]
Since $\wt{\theta}_{\rm s}$ is completely isometric with respect to the symmetrisation norm, by Theorem \ref{t_symuni} there exists a completely isometric map $\phi \colon \cl E \to \cl B(L,K)$ such that $\wt{\theta}_{\rm s} = \phi^* \cdot \phi$.
We use the same sumbol to denote the completely contractive extension
\[
\phi \colon \cl B(H) \to \cl B(L,K),
\]
obtained by Wittstock's Extension Theorem.
By Theorem \ref{t_hpw}, there exists a unital $*$-representation $\rho \colon \cl B(H) \to \cl B(H')$ and isometries $W_1 \colon L \to H'$ and $W_2 \colon K \to H'$ such that
\[
\phi(x) = W_2^* \rho(x) W_1, \ \ x \in \cl B(H).
\]
Consequently, 
\[
\wt{\theta}_{\rm s}(y^* \otimes x) 
= 
W_1^* \rho(y)^* W_2 W_2^* \rho(x) W_1, \ \ x,y \in \cl E.
\]

Let $\phi' \colon \cl E \to \cl B(L, H')$ be the map, given by $\phi'(x) = \rho(x) W_1$.
By the definition of the symmetrisation we obtain a completely contractive completely positive map
\[
\wt{\theta}'_{\rm s} 
:= 
(\phi')^* \cdot \phi' \colon \cl E^* \otimes_{\rm s} \cl E \to \cl B(L); \ y^* \otimes x \mapsto W_1^* \rho(y^* x) W_1.
\]
We claim that $\wt{\theta}_{\rm s} = \wt{\theta}'_{\rm s}$, which will complete the proof of the forward direction.

Towards this end, we claim that $\wt{\theta}_{\rm s}(u) \leq \wt{\theta}_{\rm s}'(u)$ for all $u \in (\cl E^* \otimes_{\rm s} \cl E)^+$.
By Proposition \ref{p_synth}, it suffices to prove the inequality for $u = x^* \odot x$, where $x \in M_{k,1}(\cl E)$, $k \in \mathbb{N}$.
Since $W_2 W_2^*$ is a projection, 
\begin{align*}
\wt{\theta}_{\rm s}(x^* \odot x)
& =
(\phi^{(k,1)}(x))^* \phi^{(k,1)}(x) \\
& =
W_1 \rho^{(k,1)}(x)^* (W_2 W_2^* \otimes I_k) \rho^{(k,1)}(x) W_1 \\
& \leq
W_1 \rho^{(k,1)}(x)^* \rho^{(k,1)}(x) W_1 \\
& =
\wt{\theta}_{\rm s}'(x^* \odot x).
\end{align*}
This shows that the selfadjoint map $\wt{\theta}_{\rm s}' - \wt{\theta}_{\rm s}$ is positive on $\cl E^* \otimes_{\rm s} \cl E$ and therefore 
\begin{equation}\label{eq_twon}
\|\wt{\theta}_{\rm s}' - \wt{\theta}_{\rm s}\| \leq 2 \|(\wt{\theta}_{\rm s}' - \wt{\theta}_{\rm s})(e)\|
\end{equation}
by \cite[Proposition 2.1]{Pa}.
However, $\wt{\theta}_{\rm s}(e) \leq \wt{\theta}_{\rm s}'(e)$ and thus
\[
I_L = \wt{\theta}_{\rm s}(e) \leq \wt{\theta}_{\rm s}'(e) \leq \|\wt{\theta}_{\rm s}'(e)\| I_L \leq I_L,
\]
since $\wt{\theta}_{\rm s}'$ is contractive and $\nor{e}_{\rm s} = 1$ by assumption.
Hence $\wt{\theta}_{\rm s}'(e) - \wt{\theta}_{\rm s}(e) = 0$, giving, by (\ref{eq_twon}), that $\wt{\theta}_{\rm s}' - \wt{\theta}_{\rm s} = 0$ as required.

For the converse, suppose that the constructed $\wt{\theta}_{\rm s}$ is a completely isometric complete order embedding, and set $e = \wt{\theta}_{\rm s}^{-1}(I_L)$.
Since $I_L$ is an Archimedean matrix order unit for $\wt{\theta}_{\rm s}(\cl E^* \otimes_{\rm s} \cl E)$, the element $e$ is an Archimedean matrix order unit for $\cl E^* \otimes_{\rm s} \cl E$.
Moreover, for every $u \in M_n(\cl E^* \otimes_{\rm s} \cl E)$, we have
\begin{align*}
\nor{u}_{\rm s}^{(n)}
& =
\|\wt{\theta}_{\rm s}^{(n)}(u)\|_{\cl B(L^{(n)})}
=
\|\wt{\theta}_{\rm s}^{(n)}(u)\|_{I_L}^{(n)}
=
\nor{u}_e^{(n)}
\end{align*}
where we used that $\wt{\theta}_{\rm s}$ is a complete order embedding in the last equality of the norms induced by $I_L$ and $e$, respectively.
By Remark \ref{r_sos unit}, $\cl E^* \otimes_{\rm s} \cl E$ is an operator system.
\end{proof}

%%%%%%%%%%%%%%%%%%%%%%%%%%%%
\begin{corollary}\label{c_notos}
Let $\cl E \subseteq \cl B(H)$ be an operator space such that $\dim(\cl E^* \cl E) < \infty$.
If $\dim(\cl E^* \cl E) < \dim(\cl E^* \otimes_{\rm s} \cl E)$, then $\cl E^* \otimes_{\rm s} \cl E$ is not an operator system.
\end{corollary}

\begin{proof}
Assume, towards a contradiction, that $\cl E^* \otimes_{\rm s} \cl E$ is an operator system.
By Theorem \ref{t_notos}, there exists a unital completely isometric complete order embedding of the form
\[
\cl E^* \otimes_{\rm s} \cl E \to \cl B(L); \ y^* \otimes x \mapsto V^* \rho(y^* x) V
\]
for a unital $*$-representation $\rho$ of $\cl B(H)$ and an isometry $V$.
This shows that 
\[
\dim(\cl E^* \otimes_{\rm s} \cl E) = \dim(V^*\rho(\cl E^* \cl E) V) \leq \dim(\cl E^*\cl E),
\]
a contradiction.
\end{proof}

We close with a note on the spaces $C_n$, $R_n$ and $D_n$.

%%%%%%%%%%%%%%%%%%%%%%%%%%%%
\begin{remark} \label{r_appnotos}
Corollary \ref{c_notos} shows that the symmetrisation $\cl E^* \otimes_{\rm s} \cl E$ is not an operator system for $\cl E = M_{I, n}$ for any finite $n$, and every set $I$ with $n < |I|$, since
\begin{align*}
\dim(M_{I, n}^* M_{I, n}) 
& = \dim(M_n) 
= n^2 \\
& < |I|^2 
\leq \dim( M_{I, n} )^2 \\
& \leq \dim(M_{I, n}^* \otimes M_{I, n}) 
\leq \dim(M_{I, n}^* \otimes_{\rm s} M_{I, n}).
\end{align*}
In particular $R_{I} \otimes_{\rm s} C_{I}$ is not an operator system for any set $I$ with $|I|\geq 2$.

A second example arises by considering the space $D_n$ of the diagonal $n$-by-$n$ matrices whenever $n \neq 1$.
In this case we have 
\[
\dim (D_n^* D_n) = \dim(D_n) = n < n^2 = \dim(D_n)^2 \leq \dim( D_n^* \otimes_{\rm s} D_n),
\]
and therefore $D_n \otimes_{\rm s} D_n$ is not an operator system, whenever $n \neq 1$.

However both $R_n \otimes_{\rm s} C_n$ and $D_n^* \otimes_{\rm s} D_n$ admit an Archimedean matrix order unit.
Indeed, by Lemma \ref{l_norm} there exists a canonical topological isomorphism
\[
R_n \otimes_{\rm s} C_n \simeq R_n \otimes_{\rm h} C_n \simeq M_n^{\rm d},
\]
where $M_n^{\rm d}$ denotes $M_n$ with the trace norm \cite[Corollary 5.11]{pisier_intr}.
The positive matrices induce a matrix ordered structure on $M_n^{\rm d}$ which becomes an operator system with unit $I_n$ (see \cite[Theorem 6.2]{ptt}).
Due to Proposition \ref{p_synth}, the canonical topological isomorphism is completely positive with a completely positive inverse, and thus $I_n$ defines an Archimedean matrix order unit in $R_n \otimes_{\rm s} C_n$.
This gives another example where the symmetrisation norm differs from the Haagerup tensor norm.

For $D_n^* \otimes_{\rm s} D_n$ recall that the map
\[
\phi \colon C_n \to D_n; \phi(w)(k) = w_k
\]
is a completely contractive isomorphism with cb-norm equal to $n^{-1/2}$, since the supremum norm is dominated by the $\ell^2$-norm.
Therefore the map 
\[
\phi^* \odot \phi \colon R_n \otimes_{\rm s} C_n \to D_n^* \otimes_{\rm s} D_n
\]
is a completely contractive isomorphism, and by construction it is completely positive with a completely positive inverse.
Hence an Archimedean matrix order unit of $R_n \otimes_{\rm s} C_n$ passes to $D_n^* \otimes_{\rm s} D_n$.
In Subsection \ref{ss_fsys} we will identify $D_n^* \otimes_{\rm s} D_n$ with the functions on $[n] \times [n]$ and we will see that the characteristic function on the diagonal is an Archimedean matrix order unit.

Theorem \ref{t_notos} does not cover the case $C_{I} \otimes_{\rm s} R_{I}$.
We will show in Remark \ref{r_troscand} that $C_{I} \otimes_{\rm s} R_{I}$ is completely isometrically completely order isomorphic to $\cl K_I$, and thus it is an operator system if and only if $I$ is finite.
\end{remark}

%%%%%%%%%%%%%%%%%%%%%%%%%%%%
\section{Function spaces}\label{s_fundual}
%%%%%%%%%%%%%%%%%%%%%%%%%%%%

In this section, we describe explicitly the positive cones of the symmetrisation of a function space.
For the rest of the section, we fix a compact Hausdorff space $\Omega$.
Suppose that $\cl E \subseteq C(\Omega)$ is a function space. 
By Theorem \ref{t_injective}, $\cl E^*\otimes_{\rm s}\cl E$ can be viewed as a selfadjoint subspace of $C(\Om)^* \otimes_{\rm s} C(\Om)$. 
Without loss of generality, we can therefore restrict our attention to the case where $\cl E = C(\Omega)$.
The main result of this section, Theorem \ref{th_fspa}, shows that the positivity of an element $u$ of a matrix space over the symmetrisation $C(\Om)^* \otimes_{\rm s} C(\Om)$ is equivalent to the positive semi-definiteness of the matrix-valued function on $\Om\times \Om$ canonically associated with $u$. 

%%%%%%%%%%%%%%%%%%%%%%%%%%%%
\subsection{Positive semi-definiteness}\label{ss_fsys}
%%%%%%%%%%%%%%%%%%%%%%%%%%%%

For a Borel probability measure $\mu$ on $\Omega$, set $H_{\mu} = L^2(\Omega,\mu)$ and let $\pi_{\mu} \colon C(\Omega) \to \cl B(H_{\mu})$ be the $*$-representation given by 
\[
\pi_{\mu}(f)(\xi) = f\xi, \ \ \ f\in C(\Omega), \xi\in H_{\mu}.
\]
For a function $u \in C(\Omega\times\Omega)$, let $T_{u}^\mu \in \cl B(H_{\mu})$ be the integral operator given by 
\[
T_{u}^\mu(\eta)(x) := \int_{\Om} u(x,y) \eta(y) d\mu(y), \ \ \ \eta\in H_{\mu}, x\in \Omega.
\]

We recall that a continuous function $u \colon \Omega \times \Omega \to M_n$ is called \emph{positive semi-definite} if, for every $x_1, \dots, x_r \in \Omega$, we have that the matrix $(u(x_p,x_q))_{p,q=1}^r$ is positive semi-definite in $M_{r n}$. 
We write $u_{i,j} \colon \Om\to \bb{C}$ for the (scalar-valued) continuous functions given by 
\[
u_{i,j}(x,y) := E_{i,i} u(x,y) E_{j,j},
\]
so that $u = (u_{i,j})_{i,j=1}^n$.
Given a Borel propability measure $\mu$ on $\Omega$, we write 
\[
T_u^\mu := (T_{u_{i,j}}^\mu)_{i,j=1}^n
\]
for the induced operator acting on $H_\mu^{(n)}$.

%%%%%%%%%%%%%%%%%%%%%%%%%%%%%%%%%%%%%%%%
\begin{remark} 
The definition of positive semi-definiteness for a function $u \colon \Omega \times \Omega \to M_n$ does not require that the points are distinct, that is, the following are equivalent:
\begin{enumerate}
\item $(u(x_p,x_q))_{p,q=1}^r$ is positive in $M_{r n}$ for every $x_1, \dots, x_r \in \Omega$, $r \in \bb N$;
\item $(u(x_p,x_q))_{p,q=1}^r$ is positive in $M_{r n}$ for every distinct $x_1, \dots, x_r \in \Omega$, $r \in \bb N$.
\end{enumerate}
Indeed, the implication [(i)$\Rightarrow$(ii)] is trivial. 
For the converse, suppose that (ii) holds, and let the points $x_0, x_1, x_2, \dots, x_r$ for $x_1, \dots, x_2$ distinct and $x_0 = x_1$.
Then we can write
\begin{align*}
(u(x_p,x_q))_{p,q=0}^r
& =
\begin{pmatrix}
u(x_1, x_1) & R \\
C & B
\end{pmatrix}
\end{align*}
where $B = (u(x_p,x_q))_{p,q=1}^r$, $R$ is the first row of $B$ and $C$ is the first column of $B$.
Since $B$ is positive, we have $C = R^*$.
For $\zeta \in \bb C^{nr+n}$ we can write $\zeta = (\xi, \eta)$ with $\xi \in \bb C^n$ and $\eta \in \bb C^{nr}$.
Set $\eta' = \xi' + \eta \in \bb C^{nr}$ with $\xi' = (\xi, 0, \dots, 0)$. 
We then compute
\begin{align*}
\sca{ (u(x_p,x_q))_{p,q=0}^r \zeta, \zeta}
& =
\sca{u(x_1,x_1)\xi, \xi} + \sca{R \eta, \xi} + \sca{C\xi, \eta} + \sca{B \eta, \eta} \\
& =
\sca{B\xi', \xi'} + \sca{B\eta, \xi'} + \sca{B \xi', \eta} + \sca{B\eta, \eta}\\
& =
\sca{B \eta', \eta'} \geq 0.
\end{align*}
Therefore $(u(x_p,x_q))_{p,q=0}^r$ is positive.
Inductively (and by using a canonical shuffle) we deduce the same for any collection of points that has more than one repetition.
\end{remark}

The following lemma is rather well known, but we were not able to locate a precise reference; we include a proof for sake of completeness.

%%%%%%%%%%%%%%%%%%%%%%%%%%%%%%%%%%%%%%%%
\begin{lemma}\label{l_pdeag}
Let $u \colon \Omega\times\Omega \to M_n$ be a continuous function.
The following are equivalent:
\begin{enumerate}
\item $u$ is positive semi-definite;
\item $T_u^\mu$ is positive for every finite positive Borel measure $\mu$ on $\Omega$;
\item $T_u^\mu$ is positive for every Borel probability measure $\mu$ on $\Omega$.
\end{enumerate}
\end{lemma}

\begin{proof}
\noindent
[(i)$\Rightarrow$(ii)].
Suppose that $u$ is positive semi-definite and first consider a Borel finite measure of the form $\mu = \sum_{p=1}^r \la_p \de_{x_p}$ for $x_1, \dots, x_r \in \Omega$ and $\la_1, \dots, \la_r > 0$.
For $\eta_1, \dots, \eta_n \in C(\Om)$, let $h \in \ell^2([r] \times [n])$ be the vector
\[
h \colon [r] \times [n] \to \bb C; \ h_{p,i} := \la_p \eta_i(x_p).
\]
Then a direct computation yields
\begin{align}\label{eq_cocofu}
\sca{T_u^\mu (\eta_i)_{i=1}^n, (\eta_i)_{i=1}^n}
& = \nonumber \\
& \hspace{-2.5cm} =
\sum_{i,j = 1}^n \sca{T_{u_{i,j}}^\mu \eta_j, \eta_i} 
=
\sum_{i,j=1}^n \int_{\Omega\times\Omega} u_{i,j}(x,y) \eta_j(y) \ol{\eta_i(x)} d\mu(y) d\mu(x) \nonumber \\
& \hspace{-2.5cm} =
\sum_{i,j=1}^n \sum_{p,q=1}^r \la_p \la_q u_{i,j}(x_p, x_q) \eta_j(x_q) \ol{\eta_i(x_p)}
=
\sum_{i,j =1}^n \sum_{p,q=1}^r u_{i,j}(x_p, x_q) h_{q,j} \ol{h_{p,i}} \\
& \hspace{-2.5cm} =
\sca{ (u_{i,j}(x_p, x_q))_{(p,i), (q,j)} h, h} 
=
\sca{ (u(x_p, x_q))_{p,q} h, h } \geq 0, \nonumber
\end{align}
since $u$ is positive semi-definite.
More generally, we see that if $(\mu_i)_i$ is a net of finite positive Borel measures with weak* limit $\mu$, then $(T_u^{\mu_i})_i$ converges to $T_u^\mu$ in the weak operator topology.
Using the fact that every positive finite Borel measure on $\Omega$ is the weak* limit of positive combinations of Dirac measures, we conclude that $\sca{T_u^\mu (\eta_i)_{i=1}^n, (\eta_i)_{i=1}^n}\geq 0$ for every positive finite Borel measure $\mu$. 
Since $C(\Om)$ is dense in $H_{\mu}$, we conclude that the operator $T_u^\mu$ is positive. 

\smallskip

\noindent
[(ii)$\Rightarrow$(i)].
Let $x_1, \dots, x_r \in \Omega$ be distinct points and $h \in \ell^2([r] \times [n])$; we have to show that
\[
\sca{(u(x_p, x_q))_{p,q} h, h } \geq 0.
\]
Consider the finite Borel measure $\mu = \sum_{p=1}^r \de_{x_p}$.
By assumption we have $T_u^\mu \geq 0$ in $\cl B(H_\mu^{(n)})$. 
Choose $\eta_1, \dots, \eta_n \in C(\Om)$ with
\[
\eta_i(x_p) := h_{p,i}, \ p=1, \dots, r, i=1, \dots, n.
\]
Then $\eta_i \in H_\mu$, and reading (\ref{eq_cocofu}) in reverse, together with the positivity $T_u^\mu$, shows that $\sca{(u(x_p, x_q)_{p,q} h, h } \geq 0$.

\smallskip

\noindent
[(ii)$\Leftrightarrow$(iii)]. It follows by a normalisation of a positive finite Borel measure to a Borel probability measure.
\end{proof}

%%%%%%%%%%%%%%%%%%%%%%%%%%%%%%%%%%%%%%%%
\subsection{The symmetrisation cones}\label{ss_symcon}
%%%%%%%%%%%%%%%%%%%%%%%%%%%%%%%%%%%%%%%%

By Lemma \ref{l_norm}, the symmetrisation space $C(\Om)^* \otimes_{\rm s} C(\Om)$ is completely boundedly isomorphic to the Haagerup tensor product $C(\Omega)\otimes_{\rm h} C(\Omega)$. 
On the other hand, by the Groth\'endieck Inequality, the space $C(\Omega)\otimes_{\rm h} C(\Omega)$ is boundedly isomorphic to the Banach space projective tensor product $C(\Omega) \hat{\otimes} C(\Omega)$, which is an algebra of continuous functions over $\Omega \times \Omega$ (known as the \emph{Varopoulos algebra} \cite{var}). 
The latter isomorphism is realised via the following identification: for a point $x \in \Omega$, we view the point mass measure $\delta_x$ as an element of the dual $C(\Omega)^{\rm d}$ in the canonical way; an element $u\in C(\Omega) \hat{\otimes} C(\Omega)$ gives rise to the continuous function (denoted in the same way) $u \colon \Omega\times\Omega\to \bb{C}$, given by 
\[
u(x, y) = (\delta_x \otimes \delta_y)(u), \ \ \ x,y \in \Om.
\]
We thus identify the elements of $C(\Om)^* \otimes_{\rm s} C(\Om)$ as continuous functions on $\Omega\times\Omega$. 
We note that, if $x, y \in \Omega$, then $\de_x \otimes \de_y$ is a (completely) contractive functional on $C(\Omega) \otimes_{\rm h} C(\Omega)$ and thus
\begin{equation}\label{eq_infhai}
\|u\|_\infty = \sup_{x,y \in \Om} |u(x,y)| = \sup_{x,y \in \Om} |(\de_x \otimes \de_y)(u)| \leq \|u\|_{\rm h}.
\end{equation}

We will use the identification between the elements of $C(\Om)^* \otimes_{\rm s} C(\Om)$ and the continuous functions on $\Omega\times\Omega$ described above. 
With an element 
\[
u = (u_{i,j})_{i,j=1}^n \in M_n(C(\Om)^* \otimes_{\rm s} C(\Om)),
\]
we associate the function (denoted in the same way) $u \colon \Omega\times \Omega \to M_n$, given by 
\[
u(x,y) = (u_{i,j}(x,y))_{i,j=1}^n, \ \ \ x,y \in \Om.
\] 
Further, if $H$ and $K$ are Hilbert space, and $\xi\in H$ and $\eta\in K$, we let $\theta_{\eta,\xi}\in \cl B(H,K)$ be the rank one operator, given by $\theta_{\eta,\xi}(\zeta) = \sca{\zeta,\xi}\eta$. 

%%%%%%%%%%%%%%%%%%%%%%%%%%%%%%%%%%%%%%%%
\begin{lemma}\label{l_intop}
Let $\Omega$ be a compact Hasudorff space, $\mu$ be a Borel probability measure on $\Om$ and $H$ be a Hilbert space.
For $h_1, h_2 \in H$ and $\xi_1, \xi_2 \in C(\Omega)$, let $\phi_1, \phi_2 \colon C(\Om)\to \cl B(H_{\mu},H)$ be the completely contractive maps, given by 
\[
\phi_i(f) := \theta_{h_i, \xi_i} \pi_{\mu}(f), \ \ f \in H_\mu, \ \ i = 1,2.
\]
Let $u\in M_n(C(\Om)^* \otimes_{\rm s} C(\Om))$ and $\omega \in C(\Omega\times\Omega, M_n)$ be the function given by 
\begin{equation}\label{eq_Tom0-}
\omega(x,y) := \langle h_2,h_1\rangle (\xi_1(x) \otimes I_n) u(x,y) (\overline{\xi_2(y)} \otimes I_n), \ \ \ x,y\in \Omega.
\end{equation}
Then 
\begin{equation}\label{eq_Tom}
(\phi_1^*\cdot \phi_2)^{(n)}(u) = T_{\omega}^\mu.
\end{equation}
\end{lemma}

\begin{proof}
We set $W_i := \theta_{h_i, \xi_i}$ for brevity, $i = 1,2$. 
Suppose first that $u = F^*\odot G$, for some $F, G \in M_{m,n}(C(\Omega))$, so that
\[
u_{i,j}(x,y) = \sum_{k=1}^m \ol{F_{k, i}(x)} G_{k,j}(y), \ \ \  i,j=1, \dots, n.
\]
Thus
\begin{equation}\label{eq_omom}
\om_{i,j}(x,y) = \sum_{k=1}^m \sca{h_2, h_1} \xi_1(x) \ol{F_{k,i}(x)} G_{k,j}(y) \ol{\xi_2(y)}.
\end{equation}
For $\eta_{1,i}, \eta_{2,i} \in H_{\mu}$, $i=1, \dots, n$, we set
\[
\eta_1 := (\eta_{1,i})_{i=1}^n
\qand
\eta_2 := (\eta_{2,i})_{i=1}^n,
\]
and, using (\ref{eq_omom}), we compute
\begin{align*}
\langle (\phi_1^*\cdot \phi_2)(u)\eta_2, \eta_1\rangle
& = \\
& \hspace{-2.7cm} = 
\left\langle (W_2 \otimes I_m) \pi_{\mu}^{(m,n)}(G)\eta_2, 
(W_1 \otimes I_m) \pi_{\mu}^{(m,n)}(F) \eta_1\right\rangle \\
& \hspace{-2.7cm} = 
\left\langle \left(\sum_{j=1}^n W_2 G_{k,j} \eta_{2,j}\right)_{k=1}^m, 
\left(\sum_{i=1}^n W_1 F_{k,i} \eta_{1,i}\right)_{k=1}^m \right\rangle \\
& \hspace{-2.7cm} = 
\left\langle \left(\sum_{j=1}^n \sca{G_{k,j} \eta_{2,j}, \xi_2} h_2\right)_{k=1}^m, \left(\sum_{i=1}^n \sca{F_{k,i} \eta_{1,i}, \xi_1} h_1\right)_{k=1}^m \right\rangle \\
& \hspace{-2.7cm} = 
\sum_{k=1}^m \sum_{i,j=1}^n \sca{G_{k,j} \eta_{2,j}, \xi_2} \sca{\xi_1, F_{k,i} \eta_{1,i}} \sca{h_2, h_1} \\
& \hspace{-2.7cm} = 
\sum_{k=1}^m \sum_{i,j=1}^n \int_{\Om \times \Om} \sca{h_2, h_1} \xi_1(x) \ol{F_{k,i}(x)} G_{k,j}(y)  \ol{\xi_2(y)} \eta_{2,j}(y) \ol{\eta_{1,i}(x)} d\mu(x) d\mu(y) \\
& \hspace{-2.7cm} = 
\sum_{i,j=1}^n \int_{\Om \times \Om} \om_{i,j}(x,y) \eta_{2,j}(y) \ol{\eta_{1,i}(y)} d\mu(x) d\mu(y) \\
& \hspace{-2.7cm} = 
\sum_{i,j=1}^n \sca{T_{\om_{i,j}}^\mu \eta_{2,j}, \eta_{1,i} }
=
\sca{T_\om^\mu \eta_2, \eta_1}.
\end{align*}
It follows that (\ref{eq_Tom}) is fulfilled when $u = F^* \odot G$. 

Next suppose that $u\in M_n(C(\Om)^* \otimes_{\rm s} C(\Om))$ is arbitrary and let $(u_k)_{k\in \bb{N}}\in M_n(C(\Omega)^* \odot C(\Omega))$ be a sequence, such that $\lim_k \|u - u_k\|_{\rm s}^{(n)} = 0$. 
By Lemma \ref{l_norm}, $\lim_{k\to\infty} \|u - u_k\|_{\rm h}^{(n)} = 0$ and so 
\begin{equation}\label{eq_phi12a}
\lim_{k\to\infty} \| (\phi_1^* \cdot \phi_2)(u_k) - (\phi_1^* \cdot \phi_2)(u) \| = 0.
\end{equation}
On the other hand, by (\ref{eq_infhai}), $\lim_{k\to\infty} u_k(x,y) = u(x,y)$ for all $x,y\in \Omega$, and thus, letting $\omega_k$ correspond to $u_k$ via (\ref{eq_Tom0-}), we have 
\[
\lim_{k\to\infty} \om_k(x,y)= \om(x,y), \ \ \ x,y\in \Omega.
\]
Let $C > 0$ be such that $\|u_k\|_{\rm s}^{(n)} \leq C$, $k\in \bb{N}$. 
In particular, using Lemma \ref{l_norm}, 
\[
\|u_k(x,y)\| = \|(\delta_x\otimes\delta_y)^{(n)}(u_k)\| \leq \|u_k\|_{\rm h}^{(n)} \leq 4 \|u_k\|_{\rm s}^{(n)} \leq 4 C, \ \ \ k\in \bb{N},
\]
and thus the $\om_k$ are uniformly bounded with
\[
\|\om_k\|_\infty 
\leq 
|\sca{h_1, h_2}| \|\xi_1\|_\infty \|\xi_2\|_\infty \|u_k\|_\infty
\leq
4C \hspace{0.05cm} |\hspace{-0.1cm}\sca{h_1, h_2}\hspace{-0.1cm}| \|\xi_1\|_\infty  \|\xi_2\|_\infty.
\]
The finiteness of the measure $\mu$ and the Lebesgue Dominated Convergence Theorem imply that $\text{wot-}\lim_{k\to\infty} T_{\om_k}^\mu = T_\om^\mu$. 
Equation (\ref{eq_phi12a}), together with the first part of the proof, now imply (\ref{eq_Tom}), that is
\[
(\phi_1^* \cdot \phi_2)^{(n)}(u)
=
\text{wot-}\lim_{k\to\infty} (\phi_1^* \cdot \phi_2)^{(n)}(u_k)
=
\text{wot-}\lim_{k\to\infty} T_{\om_k}^{\mu}
=
T_{\om}^{\mu},
\]
and the proof is complete.
\end{proof}

The maps  of Lemma \ref{l_intop} arise naturally in the context of the symmetrisation.
For a set ${\sf M} = \{\mu_\al\}_{\al \in \bb A}$ of Borel probability measures on $\Omega$, we denote the sum of the corresponding $*$-representations on $H_{\sf M} := \oplus_{\al \in \bb A} H_{\mu_\al}$ by 
\[
\pi_{\sf M} := \oplus_{\al \in \bb A} \pi_{\mu_\al}.
\]
Let $H$ and $K$ be Hilbert spaces, and $V \colon H_{\sf M} \to H$ and $W \colon H_{\sf M} \to K$ be bounded operators such that $\|V\| \cdot \|W\| \leq 1$.
Then the map
\[
\Phi_{{\sf M}, W, V} \colon C(\Omega) \to \cl B(H, K), 
\text{ given by } \Phi_{{\sf M}, W, V}(f) := W \pi_{\sf M}(f) V^*
\]
is completely contractive.
We will simply write $\Phi_{\mu, W,V}$ for $\Phi_{\{\mu\},W,V}$.
In the case where $H = H_{\sf M}$, we set 
\[
\Phi_{{\sf M}, W} := \Phi_{{\sf M}, W, I_{H_{\sf M}}}.
\]

Conversely, if $\Phi \colon C(\Omega) \to \cl B(H,K)$ is a completely contractive map, then by the Haagerup-Paulsen-Wittstok Theorem \cite[Theorem 1.2.8]{blm}, there exist a unital $*$-repre\-sen\-tation $\pi \colon C(\Omega) \to \cl B(L)$, and contractions $V \colon L \to H$ and $W \colon L \to K$ such that 
\[
\Phi(f) = W \pi(f) V^*, \ \ \  f\in C(\Om).
\]
We can write $\pi = \oplus_{\al \in \bb A} \pi_\al$ for cyclic representations $\pi_\al$, with cyclic vectors $\xi_\al$.
By the Riesz-Markov-Kakutani Theorem and the GNS construction, there exist a family $\{\mu_\al\}_{\al \in \bb A}$ of Borel probability measures and a family $\{U_\al\}_{\al \in \bb A}$ of unitary operators, such that $\pi_\al(\cdot) = U_\al \pi_{\mu_\al}(\cdot) U_\al^*$, $\alpha\in \bb{A}$. 
Set $U = \oplus_{\al \in \bb A} U_\al$,  $W' = W U$, $V' = V U$ and $M=\{\mu_\al\}_{\al \in \bb A}$. Then $W'$ and $V'$ are contractions and $\Phi = \Phi_{{\sf M}, W', V'}$.

%%%%%%%%%%%%%%%%%%%%%%%%%%%%%%%%%%%%%%%%
\begin{proposition}\label{p_rank1}
Let $u\in M_n(C(\Om)^* \otimes_{\rm s} C(\Om))$. 
The following are equivalent:
\begin{enumerate}
\item $u$ is positive semi-definite;
\item $(\Phi_{\mu, W}^* \cdot \Phi_{\mu,W})^{(n)}(u) \geq 0$ for every Borel probability measure $\mu$ and every rank one contraction $W$;
\item $(\Phi_{\mu, W}^* \cdot \Phi_{\mu,W})^{(n)}(u) \geq 0$ for every Borel probability measure $\mu$ and every finite rank contraction $W$.
\end{enumerate}
\end{proposition}

\begin{proof}
\noindent 
[(i)$\Leftrightarrow$(ii)]. 
Suppose that item (i) holds.
First let $\xi\in C(\Om)$, $h\in H$ with $\|h\| = 1$, and set $W := \theta_{h, \xi}$; thus, $W\in \cl B(H_\mu,H)$.
By Lemma \ref{l_intop}, 
\[
(\Phi_{\mu, W}^* \cdot \Phi_{\mu,W})^{(n)}(u) = T_\om^\mu, 
\]
where 
\[
\om(x,y) := (\xi(x) \otimes I_n) u(x,y) (\ol{\xi(y)} \otimes I_n), \ \ \ 
x,y\in \Om.
\]
Hence it suffices to show that $T_\om^\mu \geq 0$; equivalently, due to Lemma \ref{l_pdeag}, that $\om$ is positive semi-definite.
For $x_1, \dots, x_r \in \Omega$, we have
\begin{align*}
(\om(x_p,x_q))_{p,q=1}^r 
& = \\
& \hspace{-2cm} =
( \xi(x_p) u(x_p,x_q) \ol{\xi(x_q)} )_{p,q=1}^r \\
& \hspace{-2cm} =
{\rm diag} \{\xi(x_p) \otimes I_n\}_{p=1}^r
\cdot
(u(x_p,x_q))_{p,q=1}^n
\cdot
({\rm diag} \{\xi(x_p) \otimes I_n\}_{p=1}^r )^*
\geq 0,
\end{align*}
since $u$ is positive semi-definite, showing that $\om$ is positive semi-definite.

Next take an arbitrary $\xi\in H_{\mu}$ and let $(\xi_k)_{k\in \bb{N}}\subseteq C(\Om)$ be a sequence, converging to $\xi$. 
Letting $W_k = \theta_{h,\xi_k}$, we have $\lim_{k \to \infty} \|W_k - W\| = 0$, and hence 
\[
\lim_{k \to \infty} \|\Phi_{\mu, W_k} - \Phi_{\mu, W}\|_{\rm cb} = 0.
\]
It then follows that
\[
\lim_{k \to \infty} (\Phi_{\mu, W_k}^* \cdot \Phi_{\mu,W_k})^{(n)}(u) = (\Phi_{\mu, W}^* \cdot \Phi_{\mu,W})^{(n)}(u),
\]
and so $(\Phi_{\mu, W}^* \cdot \Phi_{\mu,W})^{(n)}(u) \geq 0$. 

Conversely, if item (ii) holds, then by choosing $W = \theta_{1, 1}$ for $1 \in H_\mu$, and using Lemma \ref{l_intop}, we have 
\[
0 \leq (\Phi_{\mu, W}^* \cdot \Phi_{\mu,W})^{(n)}(u) = T_u^\mu.
\]
Since this holds for every Borel probability measure $\mu$, $u$ is positive semi-definite by Lemma \ref{l_pdeag}.

\smallskip

\noindent
[(ii)$\Leftrightarrow$(iii)].
Suppose that item (ii) holds and let $W$ be a finite rank contraction.
Write $W = \sum_{k=1}^R \theta_{h_k, \xi_k}$ with $h_k$ pairwise orthogonal.
Similarly to the first paragraph of the proof of [(i)$\Leftrightarrow$(ii)] above, it suffices to assume that $\xi_k\in C(\Om)$, $k = 1,\dots,R$. 
We claim that
\begin{equation}\label{eq_capPhi}
\Phi_{\mu, W}^* \cdot \Phi_{\mu, W} = \sum_{k=1}^R \Phi_{\mu, \theta_{h_k, \xi_k}}^* \cdot \Phi_{\mu, \theta_{h_k, \xi_k}},
\end{equation}
therefore concluding
\begin{align*}
(\Phi_{\mu, W}^* \cdot \Phi_{\mu, W})^{(n)}(u)
& = \sum_{k=1}^R (\Phi_{\mu, \theta_{h_k, \xi_k}}^* \cdot \Phi_{\mu, \theta_{h_k, \xi_k}})^{(n)}(u) \geq 0.
\end{align*}
By linearity and boundedness, it suffices to check (\ref{eq_capPhi}) for the elements $u$ of the form $f^* \otimes g$ with $f, g \in C(\Omega)$.
A direct computation yields
\begin{align*}
(\Phi_{\mu, W}^* \cdot \Phi_{\mu, W})(f^* \otimes g)
& = \\
& \hspace{-2.5cm} =
\pi_\mu(f)^* W^*W \pi_\mu(g) 
=
\pi_\mu(f)^*\left( \sum_{k, \ell=1}^R \theta_{\xi_k, \sca{h_\ell, h_k} \xi_\ell} \right) \pi_\mu(g) \\
& \hspace{-2.5cm} =
\pi_\mu(f)^*\left( \sum_{k=1}^R \theta_{\xi_k, \sca{h_k, h_k} \xi_k} \right) \pi_\mu(g) 
=
\sum_{k=1}^R \pi_\mu(f)^* \theta_{\xi_k, \sca{h_k, h_k} \xi_k} \pi_\mu(g) \\
& \hspace{-2.5cm} =
\sum_{k=1}^R (\Phi_{\mu, \theta_{h_k, \xi_k}}^* \cdot \Phi_{\mu, \theta_{h_k, \xi_k}}) (f^* \otimes g),
\end{align*}
as required.

Finally, the implication [(iii)$\Rightarrow$(ii)] is trivial.
\end{proof}

We proceed to the description of the 
matricial cones of $C(\Om)^* \otimes_{\rm s} C(\Om)$.
Towards this end, we can restrict our attention to maps arising from finite sets ${\sf M}$ and finite rank contractions.

%%%%%%%%%%%%%%%%%%%%%%%%%%%%%%%%%%%%%%%%
\begin{lemma}\label{l_finred}
Let $u\in M_n(C(\Om)^* \otimes_{\rm s} C(\Om))$. 
The following are equivalent:
\begin{enumerate}
\item
$u\in M_n(C(\Om)^* \otimes_{\rm s} C(\Om))^+$;

\item
$(\Phi_{{\sf M},W,V}^*\cdot \Phi_{{\sf M},W,V})^{(n)}(u)\in M_n(\cl B(H))^+$ for all sets $\sf M$ of Borel probability measures on $\Omega$ and contractions $W \in \cl B(H_{\sf M}, K)$ and $V \in \cl B(H_{\sf M}, H)$;

\item
$(\Phi_{{\sf M},W}^*\cdot \Phi_{{\sf M},W})^{(n)}(u)\in M_n(\cl B(H_{\sf M}))^+$ for all sets $\sf M$ of Borel probability measures on $\Omega$ and contractions $W \in \cl B(H_{\sf M}, K)$;

\item
$(\Phi_{{\sf M},W}^*\cdot \Phi_{{\sf M},W})^{(n)}(u)\in M_n(\cl B(H_{\sf M}))^+$ for all finite sets $\sf M$ of Borel probability measures on $\Omega$ and contractions $W \in \cl B(H_{\sf M}, K)$;

\item
$(\Phi_{{\sf M},W}^*\cdot \Phi_{{\sf M},W})^{(n)}(u)\in M_n(\cl B(H_{\sf M}))^+$ for all finite sets $\sf M$ of Borel probability measures on $\Omega$ and contractions $W \in \cl B(H_{\sf M},K)$ of finite rank;

\item
$(\Phi_{{\sf M},W}^*\cdot \Phi_{{\sf M},W})^{(n)}(u)\in M_n(\cl B(H_{\sf M}))^+$ for all finite sets $\sf M$ of Borel probability measures on $\Omega$ and contractions $W \in \cl B(H_{\sf M})$ of finite rank. 
\end{enumerate}
\end{lemma}

\begin{proof}
\noindent
[(i)$\Leftrightarrow$(ii)]. 
The equivalence is immediate in view of Theorem \ref{t_symuni}, since every $\Phi$ can be written in the form of $\Phi_{{\sf M}, W, V}$ for an appropriate choice of ${\sf M}$, $W$ and $V$.

\smallskip

\noindent
[(ii)$\Leftrightarrow$(iii)]. 
The forward implication is immediate by considering $V$ to be the identity in $\cl B(H_{\sf M})$.
Conversely, given $V$ and $W$, we see that
\begin{align*}
(\Phi_{{\sf M},W,V}^*\cdot \Phi_{{\sf M},W,V})^{(n)}(u) 
& = 
(V\otimes I_n)
(\Phi_{{\sf M},W'}^*\cdot \Phi_{{\sf M},W})^{(n)}(u)(V\otimes I_n)^* \geq 0,
\end{align*}
since $(\Phi_{{\sf M},W'}^*\cdot \Phi_{{\sf M},W'})^{(n)}(u) \geq 0$ by assumption.

\smallskip

\noindent
[(iii)$\Leftrightarrow$(iv)]. 
The forward direction is immediate.
For the converse, for every $n \in \bb N$ we have that $(\Phi_{{\sf M},W}^*\cdot \Phi_{{\sf M},W})^{(n)}$ is the point-weak* limit of the net $((\Phi_{{\sf M}_F,W}^*\cdot \Phi_{{\sf M}_F,W})^{(n)})_{{\sf M}_F}$ where ${\sf M}_F = \{\mu_{\al_i}\}_{i \in F}$ ranges over finite subsets $F \subseteq {\sf M}$ with the usual inclusion.
Then item (iii) follows from item (iv) as positivity is preserved under weak* limits.

\smallskip

\noindent
[(iv)$\Leftrightarrow$(v)]. 
The forward implication is trivial. 
The converse follows from Lemma \ref{l_finen} by taking $\phi = \Phi_{{\sf M}, W}$.

\smallskip

\noindent
[(v)$\Leftrightarrow$(vi)].
The forward implication is immediate.
For the converse, let a contraction $W \colon H_{\sf M} \to H$.
Then $W^*W$ is a positive contraction in $\cl B(H_{\sf M})$ and thus there is a positive finite rank contraction $W' \in \cl B( H_{\sf M})$ such that 
\[
W^* W = (W')^* W'.
\]
The conclusion then follows by noting that 
\[
\Phi_{{\sf M}, W}^* \cdot \Phi_{{\sf M}, W} = \Phi_{{\sf M}, W'}^* \cdot \Phi_{{\sf M}, W'},
\]  
and the proof is complete.
\end{proof}

For $N \in \bb N$, we write $\Om^{(\sqcup, N)}$ for the disjoint union of $N$ copies of $\Om$; we will write $\Om^{(k)}$ for the $k$-th copy of $\Om$ in $\Om^{(\sqcup, N)}$, and write $x^{(k)}$ for the element $x$ of $\Om$, when it is viewed as an element of $\Om^{(k)}$.
For a function $F \colon \Om \to M_{m,n}$, we define
\[
F^{(N)} \colon \Om^{(\sqcup, N)} \to M_{m,n} \text{ by } F^{(N)}|_{\Om^{(k)}} := F.
\]
For $F, G \colon \Om \to M_{m,n}$ we have that
\[
(F + G)^{(N)} = F^{(N)} + G^{(N)}.
\]
Likewise, for $u \colon \Om \times \Om \to M_n$, we define
\[
u^{(N^2)} \colon (\Om \times \Om)^{(\sqcup, N^2)} \to M_n \text{ by } u^{(N^2)}|_{(\Om \times \Om)^{(k, \ell)}} := u, \text{ for } k,\ell=1, \dots, N.
\]
Since the product distributes over the disjoint union we can make the identification
\[
(\Om \times \Om)^{(\sqcup, N^2)} =  \Om^{(\sqcup, N)} \times \Om^{(\sqcup, N)},
\]
and view $u^{(N^2)}$ as a function over $\Om^{(\sqcup, N)} \times \Om^{(\sqcup, N)}$.
With this identification, for $F, G \colon \Om \to M_{m,n}$ and the induced $F^* \otimes G \colon \Om \times \Om \to M_{m,n}$ we have
\[
(F^*)^{(N)} \otimes G^{(N)} = (F^* \otimes G)^{(N^2)}.
\]

%%%%%%%%%%%%%%%%%%%%%%%%%%%%%%%%%%%%%%%%
\begin{lemma}\label{l_uamp}
Let $u \colon \Om \times \Om \to M_n$ and $N \in \bb N$.
Then $u$ is positive semi-definite if and only $u^{(N^2)} \colon \Om^{(\sqcup, N)} \times \Om^{(\sqcup, N)} \to M_n$ is positive semi-definite.
\end{lemma}

\begin{proof}
Suppose that $u$ is positive semi-definite.
Let $x_1, \dots, x_r \in \Om^{(\sqcup, N)}$, and write $x_i = (x_{i,1},\dots,x_{i,N})$, where $x_{i,k}$ is the element of $\Omega^{(k)}$, corresponding to $x_i$, for $k = 1,\dots,N$, $i = 1,\dots,r$. 
We have 
\[
(u^{(N^2)}(x_{i}, x_{j}))_{i,j}
=
(u(x_{i,k}, x_{j,l}))_{(i,j), (k,l)} \geq 0.
\]
Conversely, if $x_1, \dots, x_r \in \Om$, then 
\[
(u(x_i, x_j))_{i,j}
=
\left(u^{(N^2)}\left(x_i^{(1)}, x_j^{(1)}\right)\right)_{i,j} \geq 0,
\]
and the proof is complete.
\end{proof}

We introduce some notation we will use shortly.
Let ${\sf M} = \{\mu_1, \dots, \mu_N\}$ be a (finite) set of Borel probability measures, and let the measure $\mu_{\sf M}$ on $\Om^{(\sqcup, N)}$ be given by 
\begin{equation}\label{eq_Bosp}
\mu_{\sf M}(A) := \frac{1}{N} \sum_{k=1}^N \mu_k\left(A \cap \Om^{(k)}\right) 
\ \textup{ for all Borel sets } A \subseteq \Om^{(\sqcup, N)}.
\end{equation}
Then the operator 
$U_{\sf M} \colon \oplus_{k=1}^N H_{\mu_k} \to H_{\mu_{\sf M}}$, given by
\[
U_{\sf M}\left((\eta_k)_{k=1}^N\right) = N^{1/2} \sum_{k=1}^N \eta_k \chi_{\Om^{(k)}}, 
\ \ \eta_k \in H_{\mu_k}, \ k=1, \dots, N,
\]
is unitary.
Indeed, on one hand, we have 
\begin{align*}
\sca{U_{\sf M}\left((\eta_k)_{k=1}^N\right), 
U_{\sf M}\left((\eta_k')_{k=1}^N\right)}_{H_{\mu_{\sf M}}}
& = \\
& \hspace{-4.5cm} =
\sum_{k=1}^N N \int_{\Om^{(\sqcup, N)}} \eta_k(x) \overline{\eta'_k(x)} \chi_{\Om^{(k)}} d \mu_{\sf M}(x) 
 = 
\sum_{k=1}^N \int_\Om \eta_k(x) \overline{\eta'_k(x)} d \mu_k(x) \\
& \hspace{-4.5cm} =
\sum_{k=1}^N \sca{\eta_k, \eta_k'}_{H_{\mu_k}} 
 =
\sca{(\eta_k)_{k=1}^N, (\eta_k')_{k=1}^N}_{\oplus_{k=1}^N H_{\mu_k}},
\end{align*}
and thus $U_{\sf M}$ is an isometry.
On the other hand, for $\eta \in H_{\mu_{\sf M}}$ let 
\[
\eta_k := N^{-1/2} \eta \chi_{\Om^{(k)}}, \ \ \ k=1, \dots, N;
\]
then $U\left((\eta_k)_{k=1}^N\right) = \eta$, and thus $U_{\sf M}$ is a surjection.
We now note that 
\begin{align*}
U_{\sf M} \pi_{\sf M}(f) (\eta_k)_{k=1}^N
& =
N^{1/2} \sum_{k=1}^N  f \eta_k \chi_{\Om^{(k)}}
=
\pi_{\mu_{\sf M}}(f^{(N)}) U_{\sf M}(\eta_k)_{k=1}^N,
\end{align*}
that is, 
\begin{equation}\label{eq_sfM2}
\pi_{\sf M}(f) = U_{\sf M}^* \pi_{\mu_{\sf M}}(f^{(N)}) U_{\sf M}, \ \ \ f \in C(\Om).
\end{equation}

%%%%%%%%%%%%%%%%%%%%%%%%%%%%%%%%%%%%%%%%
\begin{lemma}\label{l_Mamp}
Let ${\sf M} = \{\mu_1, \dots, \mu_N\}$ be a (finite) set of Borel probability measures, $H$ be a Hilbert space, and $W \colon \oplus_{k=1}^N H_{\mu_k} \to H$ be a bounded operator.
Then 
\[
\pi_{\sf M}^{(m,n)}(F) = (U_{\sf M} \otimes I_m)^* \pi_{\mu_{\sf M}}^{(m,n)}(F^{(N)}) (U_{\sf M} \otimes I_n)
\]
for all $F \in C(\Om, M_{m,n})$, and
\begin{align*}
(\Phi_{{\sf M}, W}^* \cdot \Phi_{{\sf M}, W})^{(n)}(u) 
& = \\
& \hspace{-2.5cm} =
(U_{\sf M} \otimes I_n)^* (\Phi_{\mu_{\sf M}, W U_{\sf M} \otimes I_m}^* \cdot \Phi_{\mu_{\sf M}, W U_{\sf M} \otimes I_m})^{(n)}(u^{(N^2)}) (U_{\sf M} \otimes I_n)
\end{align*}
for all $u \in M_n(C(\Om)^* \otimes_{\rm s} C(\Om))$.
\end{lemma}

\begin{proof}
For $F \in C(\Om, M_{m,n})$ we can write $F = (F_{i,j})_{i,j}$ for $F_{i,j} = E_{i,i} F E_{j,j} \in C(\Om)$.
Then $F^{(N)} = ( F_{i,j}^{(N)} )_{i,j}$ and thus, by 
(\ref{eq_sfM2}), 
\begin{align*}
\pi_{\sf M}^{(m,n)}(F)
& = 
( \pi_{\sf M}(F_{i,j}) )_{i,j} \\
& =
( U_{\sf M}^* \pi_{\mu_{\sf M}}(F^{(N)}_{i,j}) U_{\sf M} )_{i,j}\\
& =
(U_{\sf M} \otimes I_m)^* ( \pi_{\mu_{\sf M}}(F^{(N)}_{i,j}) )_{i,j} (U_{\sf M} \otimes I_n) \\
& =
(U_{\sf M} \otimes I_m)^* \pi_{\mu_{\sf M}}^{(m,n)}(F^{(N)}) (U_{\sf M} \otimes I_n).
\end{align*}

For the second claim, it suffices to assume that $u = F^* \odot G$ for some $F, G \in C(\Om, M_{m,n})$.
We note that
\begin{align*}
\left(F^{(N)}\right)^* \odot G^{(N)}
& =
\left( \sum_{k=1}^m (F_{ki}^*)^{(N)} \otimes G_{kj}^{(N)} \right)_{i,j} 
=
\left( \sum_{k=1}^m \left( F_{ki}^* \otimes G_{kj} \right)^{(N^2)} \right)_{i,j} \\
& =
\left( \left(\sum_{k=1}^m F_{ki}^* \otimes G_{kj}\right)_{i,j} \right)^{(N^2)} 
=
(F^* \odot G)^{(N^2)}.
\end{align*}
Hence 
\begin{align*}
(\Phi_{{\sf M}, W}^* \cdot \Phi_{{\sf M}, W})^{(n)}(F^* \odot G)
& = \\
& \hspace{-4cm} =
\pi_{\sf M}^{(m,n)}(F)^* (W^*W \otimes I_m) \pi_{\sf M}^{(m,n)}(G) \\
& \hspace{-4cm} =
(U_{\sf M} \otimes I_n) \pi_{\mu_{\sf M}}^{(m,n)}(F^{(N)})^* (W U_{\sf M} \otimes I_m)^* \cdot \\
& \hspace{1cm} \cdot
(W U_{\sf M} \otimes I_m) \pi_{\mu_{\sf M}}^{(m,n)}(G^{(N)}) (U_{\sf M} \otimes I_n) \\
& \hspace{-4cm} =
(U_{\sf M} \otimes I_n)^* (\Phi_{\mu_{\sf M}, W U_{\sf M} \otimes I_m}^* \cdot \Phi_{\mu_{\sf M}, W U_{\sf M} \otimes I_m})^{(n)}((F^{(N)})^* \odot G^{(N)}) (U_{\sf M} \otimes I_n) \\
& \hspace{-4cm} =
(U_{\sf M} \otimes I_n)^* (\Phi_{\mu_{\sf M}, W U_{\sf M} \otimes I_m}^* \cdot \Phi_{\mu_{\sf M}, W U_{\sf M} \otimes I_m})^{(n)}((F^* \odot G)^{(N^2)}) (U_{\sf M} \otimes I_n),
\end{align*}
and the proof is complete.
\end{proof}

%%%%%%%%%%%%%%%%%%%%%%%%%%%%%%%%%%%%%%%%
\begin{theorem}\label{th_fspa}
Let $u\in M_n(C(\Om)^* \otimes_{\rm s} C(\Om))$. 
The following are equivalent:
\begin{enumerate}
\item
$u\in M_n(C(\Om)^* \otimes_{\rm s} C(\Om))^+$;

\item
$u \colon \Omega \times \Omega \to M_n$ is positive semi-definite. 
\end{enumerate}
\end{theorem}

\begin{proof}
\noindent
[(i)$\Rightarrow$(ii)]. 
Suppose that $u\in M_n(C(\Om)^* \otimes_{\rm s} C(\Om))^+$.
By Lemma \ref{l_finred}, $(\Phi_{\mu, W}^* \cdot \Phi_{\mu, W})^{(n)}(u) \geq 0$ for every Borel probability measure $\mu$ and a finite rank contraction $W$.
By Proposition \ref{p_rank1}, $u$ is positive semi-definite.

\smallskip

\noindent
[(ii)$\Rightarrow$(i)].
Suppose that $u$ is positive semi-definite.
By Lemma \ref{l_finred}, it suffices to show that $(\Phi_{{\sf M}, W}^* \cdot \Phi_{{\sf M}, W})^{(n)}(u) \geq 0$ for every finite set ${\sf M} = \{\mu_1, \dots, \mu_N\}$ of Borel probability measures and a finite rank contraction $W$.
Let $\mu_{\sf M}$ be the measure defined in (\ref{eq_Bosp}), and $U_{\sf M}$ be the induced unitary so that $W U_{\sf M} \otimes I_m$ is also a finite rank contraction.
Since $u$ is positive semi-definite we have that $u^{(N^2)}$ is positive semi-definite by Lemma \ref{l_uamp}.
By Proposition \ref{p_rank1}, 
\[
\left(\Phi_{\mu_{\sf M}, W U_{\sf M} \otimes I_m}^* \cdot \Phi_{\mu_{\sf M}, W U_{\sf M} \otimes I_m}\right)^{(n)}(u^{(N^2)}) \geq 0.
\]
Therefore, using Lemma \ref{l_Mamp},
\begin{align*}
\left(\Phi_{\sf M, W}^* \cdot \Phi_{\sf M, W}\right)^{(n)}(u)
& = \\
& \hspace{-3cm} = 
(U_{\sf M} \otimes I_n)^* \left(\Phi_{\mu_{\sf M}, W U_{\sf M} \otimes I_m}^* \cdot \Phi_{\mu_{\sf M}, W U_{\sf M} \otimes I_m}\right)^{(n)}(u^{(N^2)}) (U_{\sf M} \otimes I_n) \geq 0,
\end{align*}
and the proof is complete.
\end{proof}

Note that the involution on $C(\Om)^* \otimes_{\rm s} C(\Om)$ 
translates into the involution on the space of 
functions $u \colon \Om \times \Om \to \bb C$, given by 
\[
u^*(x,y):= \ol{u(y,x)},
\]
extending canonically to functions with range in the matrix algebras.
Therefore, if $u \in C(\Om)^* \otimes_{\rm s} C(\Om)$ is selfadjoint, then $(u(x_p,x_q))_{p,q=1}^r \in M_{rn}$ is also selfadjoint for every $x_1, \dots, x_r \in \Om$.

%%%%%%%%%%%%%%%%%%%%%%%%%%%%%%%%%%%%%%%%
\begin{proposition}
Let $\Om$ be a compact space and let $\mathbf{1}$ be the characteristic function on the diagonal of $\Om \times \Om$, that is,
\[
\mathbf{1}(x,y)
=
\begin{cases}
1 & \text{ if } x=y, \\
0 & \text{ if } x \neq y.
\end{cases}
\]
The following are equivalent:
\begin{enumerate}
\item $\Om$ is a finite set;
\item $\Om$ is a discrete topological space;
\item $\mathbf{1}$ is a continuous function on $\Om \times \Om$.
\end{enumerate}
If any of the above holds, then $\mathbf{1}$ is an Archimedean matrix order unit in $C(\Om) \otimes_{\rm s} C(\Om)$.
\end{proposition}

\begin{proof}
The equivalence of items (i)-(iii) is straightforward.
Suppose that $\Om$ is finite, say $\Om = \{x_1, \dots, x_k\}$. 
Then $\mathbf{1} = \sum_{i=1}^k \de_{x_i}^* \otimes \de_{x_i}$ and thus $\mathbf{1} \in C(\Om)^* \odot C(\Om)$.
Therefore $\mathbf{1} \in C(\Om)^* \otimes_{\rm s} C(\Om)$, and it remains to show that $\mathbf{1}$ is an Archimedean matrix order unit in $C(\Om)^* \otimes_{\rm s} C(\Om)$.

First we show that $\mathbf{1}$ is a matrix order unit.
Towards this end, let $u$ be a selfadjoint element in $M_n(C(\Om)^* \otimes_{\rm s} C(\Om))$.
Let $x_1, \dots, x_r \in \Om$ be distinct points, and so $r \leq k = |\Om|$.
Note that $( u(x_p,x_q) )_{p,q=1}^r$ is a selfadjoint matrix in $M_{rn}$, and a compression of $( u(x_p,x_q) )_{p,q=1}^k$.
Thus
\begin{align*}
\left\| ( u(x_p,x_q) )_{p,q=1}^r \right\|_{M_{rn}}
\leq
\left\| ( u(x_p,x_q) )_{p,q=1}^k \right\|_{M_{kn}}
=: m_u,
\end{align*}
and we have 
\begin{align*}
\left( ( m_u \cdot \mathbf{1} \otimes I_n + u)(x_p,x_q) \right)_{p,q=1}^r
& = \\
& \hspace{-2.5cm} =
\left( ( m_u \cdot \mathbf{1} \otimes I_n)(x_p,x_q) + u(x_p,x_q) \right)_{p,q=1}^r \\
& \hspace{-2.5cm} =
m_u 1_{rn} + (u(x_p,x_q))_{p,q=1}^r \\
& \hspace{-2.5cm} \geq
\| ( u(x_p,x_q) )_{p,q=1}^r \|_{M_{rn}} \cdot 1_{rn} +  (u(x_p,x_q))_{p,q=1}^r
\geq 0,
\end{align*}
where we used that $1_{rn}$ is an order unit in $M_{rn}$.
Thus $m_u \cdot \mathbf{1} + u$ is positive semi-definite, and therefore it is in $M_n(C(\Om)^* \otimes_{\rm s} C(\Om))^+$ by Theorem \ref{th_fspa}.

For the Archimedean property, we proceed likewise.
Let $u$ be a selfadjoint element in $M_n(C(\Om)^* \otimes_{\rm s} C(\Om))$ such that $\eps \mathbf{1} + u \geq 0$ for every $\eps >0$.
Then $\eps \mathbf{1} + u$ is positive semi-definite.
Thus for every $x_1, \dots, x_r \in \Om$ we have 
\[
\eps 1_{rn} + (u(x_p,x_q))_{p,q=1}^r = \left( (\eps \cdot \mathbf{1} \otimes I_n + u)(x_p,x_q) \right)_{p,q=1}^r \geq 0
\]
Since $(u(x_p,x_q))_{p,q=1}^r$ is selfadjoint and this holds for every $\eps>0$ we conlcude that $(u(x_p,x_q))_{p,q=1}^r \geq 0$, and thus $u \in M_n(C(\Om) \otimes_{\rm s} C(\Om))^+$ by Theorem \ref{th_fspa}.
\end{proof}

%%%%%%%%%%%%%%%%%%%%%%%%%%%%
\begin{remark}
As we have noted the symmetrisation $C(\Om) \otimes_{\rm s} C(\Om)$ is not an operator system when $\Om$ is finite.
Indeed, if $|\Om| = n$, then 
\[
C(\Om)^* \otimes_{\rm s} C(\Om) = D_n^* \otimes_{\rm s} D_n,
\]
and the conclusion follows from Remark \ref{r_appnotos}.
This shows that the extended norm arising from 
the element $\mathbf{1}$ does not coincide with the symmetrisation norm. 
\end{remark}

%%%%%%%%%%%%%%%%%%%%%%%%%%%%
\section{Duals}\label{s_selfduality}
%%%%%%%%%%%%%%%%%%%%%%%%%%%%

In this section, we examine the symmetrisation of the dual of a given operator space. 
We describe explicitly the positivity in the matrix order dual of the symmetrisation of a given operator space; as a consequence, we see that the symmetrisation enjoys a property akin to self-duality. 
As an application we show that the completely contractive completely positive functionals on $\cl E^* \otimes_{\rm s} \cl E$ correspond to completely bounded maps from $\cl E$ to $\cl E^{* {\rm d}}$ that factor symmetrically through a Hilbert space.

%%%%%%%%%%%%%%%%%%%%%%%%%%%%
\subsection{The operator space dual: preliminaries}
%%%%%%%%%%%%%%%%%%%%%%%%%%%%

We fix an operator space $\cl E$ and let $\cl E^{\rm d}$ be its dual Banach space. 
Every element $\Phi = (\Phi_{i,j})_{i,j} \in M_{m,n}(\cl E^{\rm d})$ gives rise to a linear map 
$F_{\Phi} \colon \cl E\to M_{m,n}$, given by
\[
F_{\Phi}(x) = \left(\Phi_{i,j}(x)\right)_{i,j}, \ \ \ x\in \cl E.
\]
By letting 
\[
\|\Phi\|^{(n)} := \|F_{\Phi}\|_{\rm cb}, \ \ \ \Phi\in M_{n}(\cl E^{\rm d}),
\]
we obtain a family $\{\|\cdot\|^{(n)}\}_{n\in \bb{N}}$ of matricial norms, which turn $\cl E^{\rm d}$ into an operator space, called the \emph{dual operator space} of $\cl E$, see for example \cite[Paragraphs 1.2.19 and 1.2.20]{blm}. 

%%%%%%%%%%%%%%%%%%%%%%%%%%%%
\begin{lemma}\label{l_dstar}
Let $\cl E$ be an operator space.
Then the map
\[
\phi \colon (\cl E^{\rm d})^* \to (\cl E^*)^{\rm d}; \ \phi(\Phi^*)(x^*) := \ol{\Phi(x)},
\]
is a completely isometric isomorphism.
\end{lemma}

\begin{proof}
It is routine to check that the map $\phi$ is well defined.
We next show that it is completely isometric.
Towards this end, let $\Phi = (\Phi_{i,j})_{i,j} \in M_n(\cl E^{\rm d})$; for every $x = (x_{q,w})_{q,w} \in M_m(\cl E)$ we have 
\begin{align*}
F_{\phi^{(n)}(\Phi^*)}^{(m)}(x^*)
& =
\big( \phi(\Phi^*_{j,i})(x_{w,q}^*) \big)_{i,j,q,w} \\
& =
\big( \ol{\Phi_{j,i}(x_{w,q})} )_{i,j,q,w}
=
\big( \Phi_{i,j}(x_{q,w}) )_{i,j,q,w}^*,
\end{align*}
and therefore, 
\begin{align*}
\| F_{\phi^{(n)}(\Phi^*)}^{(m)}(x^*) \|_{M_{mn}}
& =
\| \big( \Phi_{i,j}(x_{q,w}) )_{i,j,q,w}^* \|_{M_{mn}} \\
& =
\| \big( \Phi_{i,j}(x_{q,w}) )_{i,j,q,w} \|_{M_{mn}}
=
\|F_\Phi^{(m)}(x)\|_{M_{mn}}.
\end{align*}
Taking suprema over $\nor{x^*} \leq 1$ (which is equivalent to $\nor{x} \leq 1$) and $m \in \bb N$ on both sides gives that
\[
\|\phi^{(n)}(\Phi^*) \|^{(n)} = \|F_{\phi^{(n)}(\Phi^*)}\|_{\rm cb} = \nor{F_\Phi}_{\rm cb} = \nor{\Phi}^{(n)} = \nor{\Phi^*}^{(n)},
\]
and thus the map $\phi$ is completely isometric.

In order to show surjectivity we define the linear map 
$\psi \colon (\cl E^*)^{\rm d} \to (\cl E^{\rm d})^*$ by letting 
\[
\psi(\Psi)^*(x) := \ol{\Psi(x^*)}, \ \ \ \Psi \in (\cl E^*)^{\rm d}, \ x \in \cl E.
\]
Then
\begin{align*}
(\phi \circ \psi)(\Psi) (x^*)
& =
\phi(\psi(\Psi))(x^*)
=
\ol{ (\psi(\Psi)^*)(x) }
=
\ol{ \ol{\Psi(x^*)} }
=
\Psi(x^*),
\end{align*}
for all $x \in \cl E$.
Hence $\psi$ is a right inverse of $\phi$ which completes the proof.
\end{proof}

%%%%%%%%%%%%%%%%%%%%%%%%%%%%
\subsection{Embedding into symmetrisation's dual}
%%%%%%%%%%%%%%%%%%%%%%%%%%%%

Henceforth we suppress the explicit use of the
notation $\phi$ from Lemma \ref{l_dstar}, and set $\cl E^{{\rm d}*} := \left(\cl E^{\rm d}\right)^*$.
By \cite[Theorem 9.4.7]{er} (see also \cite[Paragraph 1.6.9]{blm}), if $\cl E$ and $\cl F$ are operator spaces, there exists a (not necessarily surjective) complete isometry 
\[
\iota \colon \cl F^{{\rm d}}\otimes_{\rm h} \cl E^{\rm d} \to \left(\cl F\otimes_{\rm h} \cl E\right)^{\rm d}, 
\text{ such that }
\iota(\Psi \otimes \Phi)(y \otimes x) = \Psi(y) \Phi(x).
\]
If either $\cl E$ or $\cl F$ is finite dimensional, then this map is surjective (see for example \cite[Corollary 9.4.8]{er}).

%%%%%%%%%%%%%%%%%%%%%%%%%%%%
\begin{lemma}\label{l_FPsi}
Let $\cl E$ and $\cl F$ be operator spaces.
Let $\Psi\in M_{m,k}(\cl F^{\rm d})$, $\Phi \in M_{k,n}(\cl E^{\rm d})$ so that $\Psi \odot\Phi\in M_{m,n}(\cl F^{\rm d}\odot \cl E^{\rm d})$. 
Then 
\[
F_{\iota^{(m,n)}(\Psi \odot\Phi)}^{(p,t)}(y \odot x) 
= 
F_{\Psi}^{(p,r)}(y) F_{\Phi}^{(r,t)}(x), \ \ \ y \in M_{p,r}(\cl F), x \in M_{r,t}(\cl E).
\]
\end{lemma}

\begin{proof}
Let us write $\Psi = (\Psi_{i,\ell})_{i,\ell}$, $\Phi = (\Phi_{\ell,j})_{\ell,j}$, and $y = (y_{z,q})_{z,q}$, $x = (x_{q,w})_{q,w}$.
By definition, for every $(i,j)$ we have
\begin{align*}
\iota^{(m,n)}(\Psi \odot \Phi)_{i,j}
=
\iota( (\Psi \odot \Phi)_{i,j} )
=
\sum_{\ell=1}^k \iota( \Psi_{i,\ell} \otimes \Phi_{\ell,k} ).
\end{align*}
We thus have 
\begin{align*}
F_{\iota^{(m,n)}(\Psi \odot\Phi)}^{(p,t)}(y \odot x)
& = 
\sum_{q=1}^r F_{\iota^{(m,n)}(\Psi \odot\Phi)}^{(p,t)}( (y_{z,q} \otimes x_{q,w})_{z,w} ) \\
& =
\sum_{q=1}^r ( F_{\iota^{(m,n)}(\Psi \odot\Phi)}( y_{z,q} \otimes x_{q,w} ) )_{z,w} \\
& =
\sum_{q=1}^r ( \iota^{(m,n)}(\Psi \odot\Phi)_{i,j} ( y_{z,q} \otimes x_{q,w} ) )_{i,j,z,w} \\
& =
\sum_{q=1}^r \sum_{\ell=1}^{k} ( \iota(\Psi_{i,\ell} \otimes \Phi_{\ell,j}) (y_{z,q} \otimes x_{q,w}) )_{i, j, z, w} \\ 
& =
\sum_{q=1}^r \sum_{\ell=1}^{k} ( \Psi_{i,\ell}(y_{z,q}) \Phi_{\ell,j}(x_{q,w}) )_{i, j, z, w}.
\end{align*}
On the other hand, 
\begin{align*}
F_{\Psi}^{(p,r)}(y) F_{\Phi}^{(r,t)}(x)
& =
(F_{\Psi}(y_{z,q}) )_{z,q} \cdot ( F_\Phi(x_{q,w}) )_{q,w}\\ 
& =
\sum_{q=1}^r \big( F_{\Psi}(y_{z,q}) \cdot F_\Phi(x_{q,w}) \big)_{z,w} \\
& =
\sum_{q=1}^r \big( (\Psi_{i,\ell}(y_{z,q}))_{i, \ell} \cdot (\Phi_{\ell,j}(x_{q,w}))_{\ell,j} \big)_{z,w} \\
& =
\sum_{q=1}^r \sum_{\ell=1}^{k} \big( \Psi_{i,\ell}(y_{z,q}) \Phi_{\ell, j}(x_{q,w}) \big)_{i,j,z,w}.
\end{align*}
Since the two expressions match at every entry, the proof is complete.
\end{proof}

By considering $\cl F = \cl E^*$, Lemma \ref{l_dstar} implies the existence of a (not necessarily surjective) complete isometry
\[
\iota \colon \cl E^{{\rm d} *}\otimes_{\rm h} \cl E^{\rm d} \to \left(\cl E^*\otimes_{\rm h} \cl E\right)^{\rm d};
\iota(\Psi^*\otimes \Phi)(x^*\otimes y) = \Psi^*(x^*) \Phi(y), \ \ \ x,y\in \cl E.
\]
By Lemma \ref{l_norm}, the mapping $\iota$ gives rise to a completely bounded isomorphism onto its range (denoted in the same way)
\begin{equation}\label{eq_ds}
\iota \colon \cl E^{{\rm d} *}\otimes_{\rm s} \cl E^{\rm d} \to \left(\cl E^*\otimes_{\rm s} \cl E\right)^{\rm d}.
\end{equation}
If $\cl E$ is finite dimensional, then this map is surjective.
Our aim is to show that $\iota$ is a complete order topological monomorphism, although it need not be a complete isometry.
We break the proof into several steps.

%%%%%%%%%%%%%%%%%%%%%%%%%%%%
\begin{lemma}\label{l_Fpos}
Let $\cl E$ be an operator space and let $\Phi \in M_{m,n}(\cl E^{\rm d})$. 
If $u \in M_p(\cl E^*\otimes_{\rm s} \cl E)^+$, then $F_{\iota^{(n)}(\Phi^* \odot \Phi)}^{(p)}(u) \in M_{pn}^+$. 
\end{lemma}

\begin{proof}
By Proposition \ref{p_synth}, it suffices to establish the 
statement in the case where $u = x^*\odot x$ for some $x \in M_{r,p}(\cl E)$, $r \in \bb{N}$.  
Let us write $\Phi = (\Phi_{i,j})_{i,j}$ and $x = (x_{q,w})_{q,w}$.
Then
\begin{align*}
F^{(p,r)}_{\Phi^*}(x^*)
& =
( F_{\Phi^*}(x_{w,q}^*) )_{q,w} 
=
(\Phi^*_{j,i}(x_{w,q}^*) )_{i, j, q, w} 
=
( \ol{\Phi_{j,i}(x_{w,q})} )_{i,j, q, w} \\
& =
( (\Phi_{i,j}(x_{q,w}) )_{i,j,q,w} )^* 
=
\left( ( F_{\Phi}(x_{q,w}) )_{q,w} \right)^*
 =
\left( F_{\Phi}^{(r,p)}(x) \right)^*.
\end{align*}
By Lemma \ref{l_FPsi}, 
\begin{align*}
F_{\iota^{(p)}(\Phi^* \odot \Phi)}^{(n)}(x^* \odot x)
& =
F^{(p,r)}_{\Phi^*}(x^*) F^{(r,p)}_{\Phi}(x) 
=
\left( F^{(r,p)}_{\Phi}(x) \right)^* F^{(r,p)}_{\Phi}(x) \geq 0,
\end{align*}
and the proof is complete.
\end{proof}

Henceforth we denote by $J \colon \cl E \to \cl E^{\rm dd}$ the canonical embedding of $\cl E$ in its second dual.
By \cite[Proposition 1.4.1]{blm}, the map $J$ is completely isometric.
For $x = (x_{q,w})_{q,w} \in M_{r,p}(\cl E)$, we have $J^{(r,p)}(x) \in M_{r,p}(\cl E^{\rm dd})$, and therefore we can write
\[
F_{J^{(r,p)}(x)} \colon \cl E^{\rm d} \to M_{r,p}
\]
for the induced completely bounded map.
Thus we obtain a completely bounded completely positive map
\[
\left( F_{J^{(r,p)}(x)} \right)^* \cdot F_{J^{(r,p)}(x)} \colon \cl E^{\rm d *} \otimes_{\rm s} \cl E^{\rm d} \to M_p.
\]
In the next lemma we give an elementwise description for this map.

%%%%%%%%%%%%%%%%%%%%%%%%%%%%
\begin{lemma}\label{l_JGa}
Let $\cl E$ be an operator space and $x \in M_{r,p}(\cl E)$.
We have
\[
\left( \left( F_{J^{(r,p)}(x)} \right)^* \cdot F_{J^{(r,p)}(x)} \right)^{(n)} 
\hspace{-0.15cm}(\Ga)
=
F^{(p)}_{\iota^{(n)}(\Ga)} (x^* \odot x), 
\foral
\Ga \in M_n(\cl E^{\rm d *} \otimes_{\rm s} \cl E^{\rm d}).
\]
\end{lemma}

\begin{proof}
First assume that $\Ga = \Psi^* \odot \Phi$ for $\Psi, \Phi \in M_{k,n}(\cl E^{\rm d})$.
We have
\begin{align*}
\left( \left( F_{J^{(r,p)}(x)} \right)^* \cdot F_{J^{(r,p)}(x)} \right)^{(n)} (\Psi^* \odot \Phi)
& = \\
& \hspace{-2.5cm} = 
\left( F^{(k,n)}_{J^{(r,p)}(x)} (\Psi) \right)^* \cdot \left( F^{(k,n)}_{J^{(r,p)}(x)} (\Phi) \right) \\
& \hspace{-2.5cm} = \left( F_{J^{(r,p)}(x)} (\Psi_{i,j}) \right)_{i,j}^* \cdot \left( F_{J^{(r,p)}(x)} (\Phi_{i,j}) \right)_{i,j} \\
& \hspace{-2.5cm} = 
\left( J^{(r,p)}(x) (\Psi_{i,j}) \right)_{i,j}^* \cdot \left( J^{(r,p)}(x) (\Phi_{i,j}) \right)_{i,j} \\
& \hspace{-2.5cm} = 
\big( J(x_{q,w}) (\Psi_{i,j}) \big)_{i,j,q,w}^* \cdot \big( J(x_{q,w}) (\Phi_{i,j}) \big)_{i,j,q,w} \\
& \hspace{-2.5cm} = 
\left( \Psi_{i,j}(x_{q,w}) \right)_{i,j,q,w}^* \cdot \left( \Phi_{i,j}(x_{q,w}) \right)_{i,j,q,w}.
\end{align*}
On the other hand, by Lemma \ref{l_FPsi}, 
\begin{align*}
F^{(p)}_{\iota^{(n)}(\Ga)} (x^* \odot x)
& =
\left(F^{(r,p)}_\Psi(x) \right)^* \cdot F^{(r,p)}_{\Phi}(x) \\
& =
(F_\Psi(x_{q,w})_{q,w})^* \cdot (F_\Phi(x_{q,w}))_{q,w} \\
& =
\left( \Psi_{i,j}(x_{q,w}) \right)_{i,j,q,w}^* \cdot \left( \Phi_{i,j}(x_{q,w}) \right)_{i,j,q,w}.
\end{align*}
This shows that the statement holds for $\Ga = \Psi^* \odot \Phi$.

Next suppose that $(\Ga_\al)_\al$ is a net in $M_n(\cl E^{\rm d *} \otimes_{\rm s} \cl E^{\rm d})$, such that $\Ga_\al = \Psi_\al \odot \Phi_\al$, converging to $\Ga \in M_n(\cl E^{\rm d *} \otimes_{\rm s} \cl E^{\rm d})$ in the symmetrisation norm.
By the previous paragraph, 
\[
\left( \left( F_{J^{(r,p)}(x)} \right)^* \cdot F_{J^{(r,p)}(x)} \right)^{(n)} (\Ga_\al) 
= F_{\iota^{(n)}(\Ga_\al)}(x^* \odot x)
\]
for all $\al$.
By the definition of the symmetrisation norm we have
\[
\lim_\al \left( \left( F_{J^{(r,p)}(x)} \right)^* \cdot F_{J^{(r,p)}(x)} \right)^{(n)} (\Ga_\al)
=
\left( \left( F_{J^{(r,p)}(x)} \right)^* \cdot F_{J^{(r,p)}(x)} \right)^{(n)} (\Ga).
\]
On the other hand we have that the map $\iota \colon \cl E^{{\rm d} *}\otimes_{\rm h} \cl E^{\rm d} \to \left(\cl E^*\otimes_{\rm h} \cl E\right)^{\rm d}$ is completely bounded.
By Lemma \ref{l_norm} we thus get $\lim_\al \iota^{(n)}(\Ga_\al) = \iota^{(n)}(\Ga)$ in $M_n(\left(\cl E^*\otimes_{\rm h} \cl E\right)^{\rm d})$.
By the definition of the operator space structure of the dual we have
\[
\lim_\al F_{\iota^{(n)}(\Ga_\al)} = F_{\iota^{(n)}(\Ga)}
\]
in the cb-norm, and therefore
\[
\lim_\al F_{\iota^{(n)}(\Ga_\al)}(x^* \odot x) = F_{\iota^{(n)}(\Ga)}(x^* \odot x)
\]
in $M_{pn}$.
Consequently we have 
\begin{align*}
\left( \left( F_{J^{(r,p)}(x)} \right)^* \cdot F_{J^{(r,p)}(x)} \right)^{(n)} (\Ga)
& =
\lim_\al \left( \left( F_{J^{(r,p)}(x)} \right)^* \cdot F_{J^{(r,p)}(x)} \right)^{(n)} (\Ga_\al) \\
& =
\lim_\al F_{\iota^{(n)}(\Ga_\al)}(x^* \odot x) \\
& =
F_{\iota^{(n)}(\Ga)}(x^* \odot x),
\end{align*}
and the proof is complete.
\end{proof}

%%%%%%%%%%%%%%%%%%%%%%%%%%%%
\begin{lemma} \label{l_unbdd}
Let $\cl E$ be an operator space and let $f \in M_{r,p}(\cl E^{\rm dd})$.
Let $(x^\al)_\al$ be a uniformly bounded net in $M_{r,p}(\cl E)$ such that $\textup{w*-}\lim_\al J(x^\al) = f$.
Then
\[
\lim_\al \left( \left( F_{J^{(r,p)}(x^\al)} \right)^* \cdot F_{J^{(r,p)}(x^\al)} \right)^{(n)} (\Ga)
=
(f^* \cdot f)^{(n)} (\Ga)
\]
for all $\Ga \in M_n(\cl E^{\rm d *} \otimes_{\rm s} \cl E^{\rm d})$, where the limit is taken in $M_{rn,pn}$.
\end{lemma}

\begin{proof}
Without loss of generality we may assume that $\nor{x^\al} \leq 1$ for all $\al$.
First assume that $\Ga = \Psi^* \odot \Phi$ for $\Psi, \Phi \in M_{k,n}(\cl E^{\rm d})$.
By the definition of the weak* topology we have 
\[
\lim_\al \Psi_{i,j}(x^\al_{q,w}) = f_{q,w}(\Psi_{i,j})
\qand
\lim_\al \Phi_{i,j}(x^\al_{q,w}) = f_{q,w}(\Phi_{i,j}),
\]
for all $q \in \{1, \dots, r\}$, $w \in \{1, \dots, p\}$, $i \in \{1, \dots, k\}$ and $j \in \{1, \dots, n\}$.
Since the matrix multiplication is jointly continuous, we get
\begin{align*}
\lim_\al \left( \left( F_{J^{(r,p)}(x^\al)} \right)^* \cdot F_{J^{(r,p)}(x^\al)} \right)^{(n)} (\Psi^* \odot \Phi)
& = \\
& \hspace{-3cm} = 
\lim_\al \left( \Psi_{i,j}(x_{q,w}^\al) \right)_{i,j,q,w}^* \cdot \left( \Phi_{i,j}(x_{q,w}^\al) \right)_{i,j,q,w} \\
& \hspace{-3cm} = 
( f_{q,w}(\Psi_{i,j}) )_{i,j,q,w}^* \cdot (f_{q,w}(\Phi_{i,j}))_{i,j,q,w} \\
& \hspace{-3cm} = 
(f^* \cdot f)^{(n)}(\Psi^* \odot \Phi),
\end{align*}
and thus the statement holds for $\Ga = \Psi^* \odot \Phi$.

Next let $\Ga \in M_n(\cl E^{\rm d *} \otimes_{\rm s} \cl E^{\rm d})$.
For a given $\eps>0$, there exists $\Psi, \Phi \in M_{k,n}(\cl E^{\rm d})$ such that 
\[
\nor{\Ga - \Psi^* \odot \Phi}_{\rm s}^{(n)} < \eps.
\]
By the previous paragraph, there exists $\al_0$ such that
\[
\left\| \left( \left( F_{J^{(r,p)}(x^{\al})} \right)^* \cdot F_{J^{(r,p)}(x^{\al})} \right)^{(n)} (\Psi^*\odot \Phi) - (f^* \cdot f)^{(n)}(\Psi^* \odot \Phi) \right\|_{M_{rn, pn}} \hspace{-0.3cm} < \eps,
\]
for all $\al \geq \al_0$.
Note here that 
\begin{align*}
\nor{x^{\al}} \leq 1
& \Rightarrow
\nor{J^{(r,p)}(x^{\al})} \leq 1 \\
& \Rightarrow
\|F_{J^{(r,p)}(x^{\al})} \|_{\rm cb} \leq 1 \\
& \Rightarrow
\| (F_{J^{(r,p)}(x^{\al})})^* \cdot F_{J^{(r,p)}(x^{\al})} \|_{\rm cb} \leq 1 \\
& \Rightarrow
\| \left( \left( F_{J^{(r,p)}(x^{\al})} \right)^* \cdot F_{J^{(r,p)}(x^{\al})} \right)^{(n)} \| \leq 1.
\end{align*}
Therefore, for $\al \geq \al_0$, we conclude
\begin{align*}
\left\| (f^* \cdot f)^{(n)} (\Ga) - \left( \left( F_{J^{(r,p)}(x^\al)} \right)^* \cdot F_{J^{(r,p)}(x^\al)} \right)^{(n)} (\Ga) \right\|_{M_{rn, pn}} \leq \\
& \hspace{-10cm} \leq
\left\| (f^* \cdot f)^{(n)} (\Ga - \Psi^* \odot \Phi) \right\|_{M_{rn, pn}} \\
& \hspace{-9.5cm} +
\left\|(f^* \cdot f)^{(n)} (\Psi^* \odot \Phi) - \left( \left( F_{J^{(r,p)}(x^{\al})} \right)^* \cdot F_{J^{(r,p)}(x^{\al})} \right)^{(n)} (\Psi^* \odot \Phi) \right\|_{M_{rn, pn}} \\
& \hspace{-9.5cm} +
\left\|\left( \left( F_{J^{(r,p)}(x^{\al})} \right)^* \cdot F_{J^{(r,p)}(x^{\al})} \right)^{(n)} (\Psi^* \odot \Phi - \Ga) \right\|_{M_{rn, pn}} \\
& \hspace{-10cm} \leq
\left\| (f^* \cdot f) \right\|_{\rm cb} 
\left\|\Ga - \Psi^* \odot \Phi \right\|_{\rm s}^{(n)}
+ \eps \\
& \hspace{-9.5cm} +
\left\| \left( \left( F_{J^{(r,p)}(x^{\al})} \right)^* \cdot F_{J^{(r,p)}(x^{\al})} \right)^{(n)} \right\| \left\|\Ga - \Psi^* \odot \Phi \right\|_{\rm s}^{(n)} \\
& \hspace{-10cm} \leq
( \|  (f^* \cdot f) \|_{\rm cb} + 2) \eps,
\end{align*}
and the proof is complete.
\end{proof}

Since $\cl E^* \otimes_{\rm s} \cl E$ is a matrix ordered $*$-vector space, we can endow $\left(\cl E^*\otimes_{\rm s} \cl E\right)^{\rm d}$ with a matrix cone structure in the following way.
We say that $\fr{X} \in M_n(\left(\cl E^*\otimes_{\rm s} \cl E\right)^{\rm d})$ is \emph{positive} if the induced map $F_{\fr{X}} \colon \cl E^*\otimes_{\rm s} \cl E \to M_n$ is completely positive.

%%%%%%%%%%%%%%%%%%%%%%%%%%%%
\begin{theorem}\label{th_findimsd}
Let $\cl E$ be an operator space. 
Then the mapping 
\[
\iota \colon \cl E^{{\rm d} *}\otimes_{\rm s} \cl E^{\rm d} 
\to \left(\cl E^*\otimes_{\rm s} \cl E\right)^{\rm d}, 
\text{ given by }
\iota(\Psi^*\otimes \Phi)(y^*\otimes x) = \Psi^*(y^*) \Phi(x),
\]
is a complete order isomorphism onto its range. 
If $\cl E$ is finite dimensional, then this map is surjective.
\end{theorem}

\begin{proof}
The fact that $\iota$ is a well-defined topological monomorphism follows by Lemma \ref{l_norm} and the self-duality of the Haagerup tensor product \cite{bs}. 
We have also noted that surjectivity of the map follows if $\cl E$ is finite dimentional.

We show that $\iota$ is completely positive.
Towards this end, it suffices to show that, if $\Phi \in M_{k,n}(\cl E^{\rm d *} \otimes_{\rm s} \cl E^{\rm d})^+$, then $F_{\iota^{(n)}(\Phi^* \odot \Phi)}$ is a completely positive map.
However this follows directly from Lemma \ref{l_Fpos}.

Conversely, let $\Ga \in M_n(\cl E^{\rm d *} \otimes_{\rm s} \cl E^{\rm d})$ be such that $\iota^{(n)}(\Ga) \geq 0$, that is, such that the map $F_{\iota^{(n)}(\Ga)}$ is completely positive.
By invoking Proposition \ref{p_finen}, it suffices to show that $(f^* \cdot f)^{(n)}(\Ga) \geq 0$ for every completely contractive map $f \colon \cl E^{\rm d} \to M_{r,p}$.
By definition, $f \in M_{r,p}(\cl E^{\rm dd})$; by Goldstine's Theorem, there exists a uniformly bounded net $(x^\al)$ in $M_{r,p}(\cl E)$ such that 
\[
\text{w*-}\lim_\al J^{(r,p)}(x^\al) = f
\]
in the weak* topology of the second dual.
By Lemmas \ref{l_JGa} and \ref{l_unbdd}, 
\begin{align*}
(f^* \cdot f)^{(n)} (\Ga)
& =
\lim_\al \left( \left( F_{J^{(r,p)}(x^\al)} \right)^* \cdot F_{J^{(r,p)}(x^\al)} \right)^{(n)} (\Ga) \\
& =
\lim_\al F^{(p)}_{\iota^{(n)}(\Ga)} ((x^\al)^* \odot x^\al),
\end{align*}
where the limit is taken in $M_{rn, pn}$.
However $F_{\iota^{(n)}(\Ga)}$ is assumed to be completely positive, and thus $F^{(p)}_{\iota^{(n)}(\Ga)}$ is a positive map; therefore we obtain $F^{(p)}_{\iota^{(n)}(\Ga)} ((x^\al)^* \odot x^\al) \geq 0$.
Since positivity is preserved by limits in $M_{rn,pn}$ we get $(f^* \cdot f) (\Ga) \geq 0$ as required.
\end{proof}

%%%%%%%%%%%%%%%%%%%%%%%%%%%%
\begin{remark}\label{r_notisom} 
The map $\iota \colon \cl E^{{\rm d} *}\otimes_{\rm s} \cl E^{\rm d} \to \left(\cl E^*\otimes_{\rm s} \cl E\right)^{\rm d}$ of Theorem \ref{th_findimsd} need not be a complete isometry.
For an example, consider $\cl E = R_n$, in which case we have $\cl E^{\rm d} = \cl E^* = C_n$.
Then we have the following sequence of completely contractive maps
\begin{align*}
(M_n)^{\rm d}
\simeq
R_n \otimes_{\rm h} C_n
& \stackrel{\id}{\longrightarrow}
R_n \otimes_{\rm s} C_n
=
\cl E^{\rm d *} \otimes_{\rm s} \cl E^{\rm d} 
\longrightarrow \\
& \stackrel{\iota}{\longrightarrow}
\left(\cl E^*\otimes_{\rm s} \cl E\right)^{\rm d}
\simeq
\left( C_n \otimes_{\rm s} R_n \right)^{\rm d}
\simeq
(M_n)^{\rm d}
\end{align*}
which give the identity map on $(M_n)^{\rm d}$.
If $\iota$ were completely isometric, then the map $\id \colon R_n \otimes_{\rm h} C_n \to R_n \otimes_{\rm s} C_n$ would be completely isometric.
However since $R_n \otimes_{\rm h} C_n \simeq (M_n)^{\rm d}$ is an operator system, then $R_n \otimes_{\rm s} C_n$ would be an operator system, which contradicts Remark \ref{r_appnotos}.
\end{remark}

We finish this section with a note on the quasi-state space of $\cl E^*\otimes_{\rm s}\cl E$, that is, the convex set 
\[
{\rm CCP}(\cl E^*\otimes_{\rm s}\cl E, \bb C) := \{\vphi \colon \cl E^*\otimes_{\rm s}\cl E \to \bb{C} 
\ : \ \vphi \mbox{ is contractive and positive}\}.
\]
We will show that there is a canonical identification with mappings that factor through Hilbert spaces.

Recall that, if $\cl E$ and $\cl F$ are operator spaces, we say that a completely bounded map $f \colon \cl E\to \cl F^{\rm d}$ \emph{factors through a Hilbert space} if there exist a Hilbert space $K$ and complete bounded maps $r \colon \cl E\to K_{\rm c}$ and $s \colon K_{\rm c} \to \cl F^{\rm d}$ such that $f = s \circ r$ (see for example \cite[Section 13.3]{er}). 
We write
\[
\Gamma_{2}^{\rm c}(\cl E,\cl F^{\rm d})
= \{f \colon \cl E\to \cl F^{\rm d} \ : \
f \mbox{ factors through a Hilbert space}\}.
\]
By identifying a matrix $f \in M_n(\Gamma_{2}^{\rm c}(\cl E,\cl F^{\rm d}))$ with the mapping $f \colon \cl E \to M_n(\cl F^{\rm d})$ we see that there are completely bounded maps
\[
r \colon \cl E \to M_{1,n}(K_{\rm c})
\qand 
s \colon K_{\rm c} \to M_{n,1}(\cl F^{\rm d})
\]
such that $f = s^{(1,n)} \circ r$
(see \cite[Section 13.3]{er}).
The space $\Gamma_{2}^{\rm c}(\cl E,\cl F^{\rm d})$ endowed with
\[
\ga_{2,n}^{\rm c}(f) := \inf\{ \|s\|_{\rm cb} \cdot \|r\|_{\rm cb} \ : \ \textup{ $f = s \circ r$ factors through a Hilbert space} \}
\]
becomes an operator space.
In particular, by \cite[Lemma 13.3.1]{er} there is a completely isometric isomorphism
\begin{equation} \label{eq_lam}
\La_{\rm h} \colon (\cl F \otimes_{\rm h} \cl E)^{\rm d} \to \Gamma_{2}^{\rm c}(\cl E,\cl F^{\rm d})
\end{equation}
given by
\begin{equation*}
\La_{\rm h}(\vphi)(x)(y) = \vphi(y \otimes x), \ \ \ x \in \cl E, y \in \cl F.
\end{equation*}
By considering $\cl F \subseteq \cl F^{\rm dd}$, this gives a description of the operator space structure on $\Gamma_{2}^{\rm c}(\cl E,\cl F^{\rm d})$.

We will use (\ref{eq_lam}) in order to describe ${\rm CCP}(\cl E^*\otimes_{\rm s}\cl E, \bb C)$.
Towards this end, we include a definition.
Recall the identification $K_{\rm c}^* \simeq K_{\rm r}$, and thus, if $r \colon \cl E \to K_{\rm c}$ is completely bounded, then we can consider the completely bounded map $r^* \colon \cl E^* \to K_{\rm r}$.
Moreover, we also have the identification $K_{\rm r}^{\rm d} \simeq K_{\rm c}$, and thus we obtain the completely bounded map $r^{* \rm d} \colon K_{\rm c} \to \cl E^{* \rm d}$.

%%%%%%%%%%%%%%%%%%%%%%%%%%%%
\begin{definition}
Let $\cl E$ be an operator space.
We say that a completely bounded map $\vphi \colon \cl E \to \cl E^{* \rm d}$ \emph{factors symmetrically through a Hilbert space} if there exists a Hilbert space $K$ and a completely bounded map $r \colon \cl E \to K_{\rm c}$ such that $\vphi = r^{* \rm d} \circ r$.
We write
\[
\Gamma_{2 +}^{\rm c}(\cl E, \cl E^{* {\rm d}})
:= \{\vphi \colon \cl E\to \cl E^{* {\rm d}} \ : \
\vphi \mbox{ factors sym. through a Hilbert space}\}.
\]
\end{definition}

It is easy to see that $\Gamma_{2 +}^{\rm c}(\cl E, \cl E^{* {\rm d}})$ is a cone in $\Gamma_{2}^{\rm c}(\cl E, \cl E^{* {\rm d}})$.
Since $\cl E^* \otimes_{\rm s} \cl E$ is completely boundedly isomorphic to $\cl E^* \otimes_{\rm h} \cl E$ by Lemma \ref{l_norm}, (\ref{eq_lam}) defines a completely bounded isomorphism
\begin{equation} \label{eq_lam2}
\La_{\rm s} \colon (\cl E^* \otimes_{\rm s} \cl E)^{\rm d} \to \Gamma_{2}^{\rm c}(\cl E,\cl E^{* \rm d})
\end{equation}
given by
\begin{equation*}
\La_{\rm s}(\vphi)(x)(y^*) = \vphi(y^* \otimes x), \ \ \ x,y \in \cl E.
\end{equation*}
This map will give the required identification of ${\rm CCP}(\cl E^*\otimes_{\rm s}\cl E, \bb C)$ with the ball of $\Gamma_{2 +}^{\rm c}(\cl E, \cl E^{* {\rm d}})$.

%%%%%%%%%%%%%%%%%%%%%%%%%%%%
\begin{corollary}\label{c_states}
Let $\cl E$ be an operator space.
Then the map
\[
\La_{\rm s +} \colon {\rm CCP}(\cl E^*\otimes_{\rm s}\cl E, \bb C) \to \{f \in \Gamma_{2 +}^{\rm c}(\cl E,\cl E^{* \rm d}) \ : \ \ga_{2,1}^{\rm c}(f) \leq 1 \}
\]
given by
\[
\La_{\rm s +}(\vphi)(x)(y^*) = \vphi(y^* \otimes x), \ \ \ x,y \in \cl E,
\]
is an isometric affine surjection.
\end{corollary}

\begin{proof}
The map $\La_{\rm s +}$ is given by the restriction of the completely bounded isomorphism $\La_{\rm s}$, defined in (\ref{eq_lam2}), and thus is a well-defined convex map.
We need to show that it is an isometric bijection.

First note that, if $\vphi \in {\rm CCP}(\cl E^*\otimes_{\rm s}\cl E, \bb C)$ then, by Lemma \ref{l_ssgen}, 
there exist a Hilbert space $K$ and a completely contractive map
\[
\phi \colon \cl E \to \cl B(\bb C, K) \simeq K_{\rm c}
\]
such that $\vphi  = \phi^* \cdot \phi$ and $\|\phi\|_{\rm cb}^2 = \|\vphi\|_{\rm cb}$.
For all $x,y \in \cl E$ we then have
\begin{equation} \label{eq_lamplus}
\begin{split}
\La_{\rm s +}(\vphi)(x)(y^*)
& =
\vphi(y^* \otimes x)
=
\phi(y)^* \phi(x) \\
& =
\sca{\phi(x), \phi(y)}_{K}
=
(\phi^{* \rm d} \circ \phi)(x)(y^*).    
\end{split}
\end{equation}
Therefore $\La_{\rm s +}(\vphi) = \phi^{* \rm d} \circ \phi$ with
\[
\ga_{2,1}^{\rm c}(\La_{\rm s +}(\vphi)) \leq \|\phi^{* \rm d}\|_{\rm cb} \cdot \|\phi\|_{\rm cb} = \|\phi\|_{\rm cb}^2 = \|\vphi\|_{\rm cb} \leq 1.
\]
Thus $\La_{\rm s +}$ takes values inside the unit ball of $\Gamma_{2 +}^{\rm c}(\cl E,\cl E^{* \rm d})$.

For $\vphi$ and $\phi$ as above, let us write $\vphi'$ for the induced map on $\cl E^* \otimes_{\rm h} \cl E$; by the definition of $\La_{\rm s +}$, we have
\[
\La_{\rm h}(\vphi')
=
\La_{\rm s +}(\vphi)
=
\phi^{* \rm d} \circ \phi.
\]
Using the fact that $\La_{\rm h}$ is completely isometric and that $\|\vphi\|_{\rm cb} = \|\vphi'\|_{\rm cb}$ by Remark \ref{r_norest} we have
\begin{align*}
\|\vphi\|_{\rm cb}
& =
\|\vphi'\|_{\rm cb}
=
\ga_{2,1}^{\rm c}(\La_{\rm h}(\vphi'))
\leq
\|\phi\|_{\rm cb}^2
=
\|\vphi\|_{\rm cb}.
\end{align*}
This shows that $\La_{\rm s +}$ is an isometry.

For surjectivity, let $f \in \Gamma_{2 +}^{\rm c}(\cl E,\cl E^{* \rm d})$ with $\ga_{2,1}^{\rm c}(f) \leq 1$.
Let $\phi \colon \cl E \to \cl B(\bb C, K)$ be a completely bounded map such that $f = \phi^{* \rm d} \circ \phi$.
Let us write $\vphi := \phi^* \cdot \phi$ on $\cl E^* \otimes_{\rm s} \cl E$ and let $\vphi'$ be the induced map defined on $\cl E^* \otimes_{\rm h} \cl E$.
Then $\vphi$ is completely bounded and completely positive, and by (\ref{eq_lamplus}) we get  
\[
\La_{\rm s +}(\vphi) = \La_{\rm h}(\vphi') = f.
\]
By using that $\La_{\rm h}$ is completely isometric and that $\|\vphi\|_{\rm cb} = \|\vphi'\|_{\rm cb}$ from Remark \ref{r_norest} we have
\[
\|\vphi\|_{\rm cb} = \|\vphi'\|_{\rm cb} = \ga_{2,1}^{\rm c}(f) \leq 1,
\]
and thus $\vphi \in {\rm CCP}(\cl E^*\otimes_{\rm s}\cl E, \bb C)$.
\end{proof}

%%%%%%%%%%%%%%%%%%%%%%%%%%%%
\begin{remark}
By Corollary \ref{c_states}, the set 
\[
\{f \in \Gamma_{2 +}^{\rm c}(\cl E,\cl E^{* \rm d}) \ : \ \ga_{2,1}^{\rm c}(f) \leq 1 \}
\]
is a convex subset of $\Gamma_{2}^{\rm c}(\cl E,\cl E^{* \rm d})$, being isometrically isomorphic to the convex set ${\rm CCP}(\cl E^*\otimes_{\rm s}\cl E, \bb C)$.
Moreover by the proof of Corollary \ref{c_states} we have
\begin{align*}
\ga_{2,1}^{\rm c}(f)
& =
\inf\{ \|\phi\|_{\rm cb}^2 \ : \ \phi \colon \cl E \to K_{\rm c}, f = \phi^{* \rm d} \circ \phi\}, 
\end{align*}
for any $f \in \Gamma_{2 +}^{\rm c}(\cl E,\cl E^{* \rm d})$.
In particular, the infimum is a minimum in virtue of Lemma \ref{l_ssgen} and (\ref{eq_lamplus}).
\end{remark}

%%%%%%%%%%%%%%%%%%%%%%%%%%%%
\section{Balanced symmetrisation}\label{s_balsym}
%%%%%%%%%%%%%%%%%%%%%%%%%%%%

Let $\cl A$ be a C*-algebra, $\cl S$ be an operator $\cl A$-system and $\cl E$ be an operator $\cl A$-space.
In this section we will define an $\cl A$-balanced version of the $\cl S$-symmetrisation.
We will show how this can be derived as a quotient of the $\cl S$-symmetrisation $\cl E^*\otimes_{\rm s}\cl S \otimes_{\rm s} \cl E$.

%%%%%%%%%%%%%%%%%%%%%%%%%%%%
\subsection{Construction and properties}\label{ss_constpro}
%%%%%%%%%%%%%%%%%%%%%%%%%%%%

Let $\cl A$ be a unital C*-algebra, $\cl S$ be an operator $\cl A$-system and $\cl E$ be an operator $\cl A$-space.
Note that, in this case, $\cl E^*$ is a right operator $\cl A$-space with respect to the action $x^*\cdot a := (a^*\cdot x)^*$.
Recall that an admissible pair $(\phi,\psi)$ of maps is called \emph{$\cl A$-admissible} if 
\begin{align*}
\psi(s \cdot a) \phi(x) = \psi(s) \phi(a \cdot x), \ \ s \in \cl S, a \in \cl A, x \in \cl E.
\end{align*}
We will further say that a map $\wt{\theta}_{\rm s} \colon \cl E^* \otimes_{\rm s} \cl S \otimes_{\rm s} \cl E\to \cl B(H)$ is \emph{$\cl A$-balanced} if the admissible pair, corresponding to $\wt{\theta}_{\rm s}$ via Theorem \ref{t_symuni}, is an $\cl A$-admissible pair; note that this happens if and only if the delinearisation $\theta$ of $\wt{\theta}_{\rm s}$ is an $\cl A$-balanced trilinear map.
It follows from Lemma \ref{l_ssgen} that every \emph{$\cl A$}-balanced completely contractive completely positive map can be factored through an $\cl A$-admissible pair, and in the other direction every $\cl A$-admissible pair defines an \emph{$\cl A$}-balanced completely contractive completely positive map.

The balanced symmetrisation can be produced in two ways.
First we see how it is given as a quotient of the symmetrisation.
For $n\in \bb{N}$, let 
\begin{align*}
D_n^{\cl A} := \{ u\in M_n(\cl E^*\otimes_{\rm s}\cl S \otimes_{\rm s} \cl E)_h \ : \ 
& (\phi^*\cdot\psi \cdot \phi)^{(n)}(u)\in M_n(\cl B(H))^+, \\
& \; \text{ for every $\cl A$-admissible pair } (\phi,\psi) \},
\end{align*}
and let
\[
\cl J_{\cl E,\cl S}^{\otimes_{\rm s}^{\cl A}} :=  {\rm span} \{ u \in (\cl E^* \otimes_{\rm s} \cl S \otimes_{\rm s} \cl E)_{h} \ : \ u \in D_1^{\cl A}\cap (-D_1^{\cl A}) \};
\]
by definition, $\cl J_{\cl E,\cl S}^{\otimes_{\rm s}^{\cl A}}$ is selfadjoint subspace of $\cl E \otimes_{\rm s} \cl S \otimes_{\rm s} \cl E$.

We will keep track of the subspace related to the balanced relations as well; towards this end we define
\begin{align}
\cl J_{\cl E,\cl S}^{\cl A}
& := 
\ol{\rm span}\{ y^*\otimes (b \cdot s \cdot a)\otimes x -  (y^*\cdot b) \otimes  s \otimes (a \cdot x)  \ : \  \label{eq_J} \\
& \hspace{7cm} a,b\in \cl A, x,y\in \cl E\}, \nonumber
\end{align}
which is a (closed) subspace of $\cl E^*\otimes_{\rm s} \cl S \otimes_{\rm s} \cl E$.
By the definition of the involution, it is clear that $\cl J_{\cl E,\cl S}^{\cl A}$ is selfadjoint and thus spanned by its hermitian elements.

%%%%%%%%%%%%%%%%%%%%%%%%%%%%
\begin{lemma}\label{l_kerne}
Let $\cl A$ be a unital C*-algebra, $\cl S$ be an operator $\cl A$-system and $\cl E$ be an operator $\cl A$-space.
If $\theta \colon \cl E^* \times \cl S \times \cl E \to \cl B(H)$ is an $\cl A$-balanced completely contractive completely positive map, and $\wt{\theta}_{\rm s}$ is the induced map on the symmetrisation, then $\wt{\theta}_{\rm s}$ annihilates $\cl J_{\cl E, \cl S}^{\otimes_{\rm s}^{\cl A}}$.
Moreover $\cl J_{\cl E, \cl S}^{\otimes_{\rm s}^{\cl A}}$ is a kernel of $\cl E \otimes_{\rm s} \cl S \otimes_{\rm s} \cl E$ that contains $\cl J_{\cl E,\cl S}^{\cl A}$.
\end{lemma}

\begin{proof}
Let $\theta \colon \cl E^* \times \cl S \times \cl E$ be an $\cl A$-balanced completely contractive completely positive map, and $\wt{\theta}_{\rm s}$ be the induced map on the symmetrisation.
Let $(\phi, \psi)$ be an $\cl A$-admissible pair such that $\wt{\theta}_{\rm s} = \phi^* \cdot \psi \cdot \phi$ obtained by Lemma \ref{l_ssgen}.
If $u$ is hermitian with $u \in D_1^{\cl A}\cap (-D_1^{\cl A})$, then we have that 
\[
\wt{\theta}_{\rm s}(\pm u) = (\phi^* \cdot \psi \cdot \phi)(\pm u) \geq 0,
\]
and thus $u \in \ker \wt{\theta}_{\rm s}$.
Therefore $\wt{\theta}_{\rm s}$ annihilates $\cl J_{\cl E, \cl S}^{\otimes_{\rm s}^{\cl A}}$.

We next show that $\cl J_{\cl E, \cl S}^{\otimes_{\rm s}^{\cl A}}$ is a kernel.
We have already shown that $\cl J_{\cl E, \cl S}^{\otimes_{\rm s}^{\cl A}} \subseteq \ker f$ for every $\cl A$-balanced completely contractive completely positive functional $f$.
Let $u \in  \cl E^* \otimes_{\rm s} \cl S \otimes_{\rm s} \cl E$ be such that $f(u) = 0$ for every $\cl A$-balanced completely contractive completely positive functional on the symmetrisation.
Write
\[
u = u_1 +i u_2
\text{ for }
u_1 := \frac{u + u^*}{2}, u_2:= \frac{u - u^*}{2i}.
\]
Let $(\phi, \psi)$ be an $\cl A$-balanced admissible pair acting on $(H, K)$.
Then we get $f \circ (\phi^* \cdot \psi \cdot \phi) (u) = 0$ for every completely contractive completely positive functional $f$ on $\cl B(H)$, since $f \circ (\phi^* \cdot \psi \cdot \phi)$ is an $\cl A$-balanced completely contractive completely positive functional on the symmetrisation.
In particular,
\[
f \circ (\phi^* \cdot \psi \cdot \phi) (\pm u_1) = 0 = f \circ (\phi^* \cdot \psi \cdot \phi) (\pm u_2),
\]
and thus $(\phi^* \cdot \psi \cdot \phi) (\pm u_1) = 0 = (\phi^* \cdot \psi \cdot \phi) (\pm u_2)$.
Since the pair $(\phi,\psi)$ is arbitrary, we derive that 
\[
u_1, u_2 \in D_1^{\cl A}\cap (-D_1^{\cl A}),
\]
and thus $u \in \cl J_{\cl E, \cl S}^{\otimes_{\rm s}^{\cl A}}$.
Hence we have shown that
\[
\cl J_{\cl E, \cl S}^{\otimes_{\rm s}^{\cl A}} 
= 
\bigcap \{\ker f \ : \ f \text{ is an $\cl A$-balanced c.c.p. functional of $\cl E^* \otimes_{\rm s} \cl S \otimes_{\rm s} \cl E$} \},
\]
and thus $\cl J_{\cl E, \cl S}^{\otimes_{\rm s}^{\cl A}}$ is a kernel by Remark \ref{r_kernel}.

To finish the proof, it is clear that $\cl J_{\cl E, \cl S}^{\cl A}$ is in the kernel of every completely contractive completely positive $\cl A$-balanced map $\wt{\theta}_{\rm s} \colon \cl E^*\otimes_{\rm s}\cl S\otimes_{\rm s} \cl E\to \cl B(H)$.
Therefore, if $(\phi, \psi)$ is an $\cl A$-admissible pair and $u \in \cl J_{\cl E, \cl S}^{\cl A}$ with $u = u^*$, then $(\phi^* \cdot \psi \cdot \phi) (\pm u) = 0$ giving that $u \in D_1^{\cl A}\cap (-D_1^{\cl A})$.
Since $\cl J_{\cl E, \cl S}^{\cl A}$ is spanned by its hermitian elements we have
%\begin{equation}\label{eq_JES}
$\cl J_{\cl E,\cl S}^{\cl A} \subseteq \cl J_{\cl E,\cl S}^{\otimes_{\rm s}^{\cl A}}$,
%\end{equation}
and the proof is complete.
\end{proof}

By Lemma \ref{l_kerne} and \cite[Theorem 8.7]{kkm}, the quotient $\cl E^* \otimes_{\rm s} \cl S \otimes_{\rm s} \cl E / \cl J_{\cl E, \cl S}^{\otimes_{\rm s}^{\cl A}}$ is a selfadjoint operator space in a canonical fashion. 
We let 
\[
\wt{q} \colon \cl E^* \otimes_{\rm s} \cl S \otimes_{\rm s} \cl E 
\to \cl E^* \otimes_{\rm s} \cl S \otimes_{\rm s} \cl E / \cl J_{\cl E, \cl S}^{\otimes_{\rm s}^{\cl A}}
\]
be the quotient map; then
$\wt{q}$ is completely contractive and completely positive.

We will see how the quotient relates in a canonical way to the symmetrisation associated with admissible $\cl A$-balanced pairs, analogous to the one presented in Section \ref{s_univco}. 
Namely, for an element $u\in M_n(\cl E^* \odot^{\cl A} \cl S \odot^{\cl A} \cl E)$, where $\odot^{\cl A}$ designates the balanced algebraic tensor product, we define
\[
\left\|u\right\|^{(n)}_{\rm s \cl A}
:= 
\sup\{\|\wt{\theta}_{\cl A}^{(n)}(u)\| \ : \ \theta \mbox{ is an $\cl A$-balanced trilinear map of $\cl E^* \times \cl S \times \cl E$}\}.
\]
The following is the analogue of Lemma \ref{l_norm} for the $\cl A$-balanced Haagerup norm $\|\cdot\|_{{\rm h} \cl A}^{(n)}$; see for example \cite[Section 3.4]{blm} for the balanced Haagerup norm. 

%%%%%%%%%%%%%%%%%%%%%%%%%%%%
\begin{lemma}\label{l_balnorm}
Let $\cl E$ be an operator space and $\cl S$ be an operator system. 
Then the map $\|\cdot\|_{\rm s \cl A}^{(n)} \colon M_n(\cl E^*\odot^{\cl A} \cl S \odot^{\cl A} \cl E) \to \bb{R}^+$ is a semi-norm satisfying
\begin{equation}\label{eq_esymhbal}
\|u\|_{\rm s \cl A}^{(n)} \leq \|u\|_{{\rm h} \cl A}^{(n)}, \ \ \ \ u\in M_n(\cl E^*\odot\cl S \odot\cl E),
\end{equation}
and
\begin{align*}
\|u\|_{\rm s\cl A}^{(n)} 
& = 
\sup\{\|(\phi^*\cdot\psi\cdot \phi)^{(n)}(u)\| \ : \ (\phi,\psi) \mbox{ $\cl A$-adm. pair}\} \\
& = 
\sup\{\|(\phi^*\cdot\psi\cdot \phi)^{(n)}(u)\| \ : \ (\phi,\psi) \mbox{ non-degenerate $\cl A$-adm. pair}\}.
\end{align*}
\end{lemma}

\begin{proof}
The proof follows the arguments from the proof of Lemma \ref{l_norm}, and the proof of Proposition \ref{p_nd}, where we also note that the non-degenerate compression of an $\cl A$-admissible pair is $\cl A$-admissible.
\end{proof}

%%%%%%%%%%%%%%%%%%%%%%%%%%%%
\begin{remark}
The injective envelope provides non-trivial $\cl A$-admissible pairs when $\cl A \subseteq \cl S$.
Indeed, let the embedding 
$\phi \colon \cl E \to \cl I(\cl S(\cl E))$.
Since $\cl E$ is a right operator $\cl A$-space, there exists a $*$-representation $\pi \colon \cl A \to \cl I_{11}(\cl E)$ such that 
\[
\pi(a) \phi(x) = \phi(a \cdot x), \ \ \ a \in \cl A, x \in \cl E.
\]
By injectivity of $\cl I(\cl S(\cl E))$ we can extend $\pi$ to a unital completely positive map $\psi \colon \cl S \to \cl I(\cl S(\cl E))$.
By definition, we have that $(\phi, \psi)$ is an $\cl A$-admissible pair, and thus is any pair that arises by composing with a $*$-representation of $\cl I(\cl S(\cl E))$.
\end{remark}

Similarly to Proposition \ref{p_estareos}, we have that $\left\{\| \cdot \|_{\rm s\cl A}^{(n)}\right\}_{n\in \bb{N}}$ is a family of matricial seminorms on $\cl E^*\odot^{\cl A} \cl S \odot^{\cl A} \cl E$.
We will write $\cl E^*\otimes_{\rm s}^{\cl A} \cl S \otimes_{\rm s}^{\cl A} \cl E$ for the Hausdorff completion of $\cl E^*\odot^{\cl A} \cl S \odot^{\cl A} \cl E$ with respect to $\nor{\cdot}_{\rm s \cl A}$, and we will refer to it as \emph{the balanced symmetrisation}.
Since $\cl J_{\cl E,\cl S}^{\cl A}$ is annihilated by every completely contractive completely positive $\cl A$-balanced map, the balanced symmetrisation can be obtained also by the Hausdorff completion of $\cl E^* \odot \cl S \odot \cl E$ by the seminorms
\[ 
\sup\{\|\wt{\theta}^{(n)}(u)\| \ : \ \theta \mbox{ is an $\cl A$-balanced trilinear map of $\cl E^* \times \cl S \times \cl E$}\}
\]
for $u \in \cl E^* \odot \cl S \odot \cl E$, which coincide 
with $\left\|u\right\|^{(n)}_{\rm s \cl A}$.
We will not differentiate between these two viewpoints.

We next examine the matricial cones in the balanced symmetrisation.
For $n\in \bb{N}$, let 
\begin{align*}
C_n^{\cl A} 
& := 
\{u\in M_n(\cl E^*\otimes_{\rm s}^{\cl A} \cl S \otimes_{\rm s}^{\cl A} \cl E)_h \ : \ 
(\phi^*\cdot\psi \cdot \phi)^{(n)}(u)\in M_n(\cl B(H))^+, \\
& \hspace{3.5cm} \text{ if } (\phi,\psi) \text{ is an $\cl A$-admissible pair for some } (H,K)\}.
\end{align*}

Following similar steps to those in Lemma \ref{l_mos}, we can verify that $\{C_n^{\cl A}\}_{n\in \bb{N}}$ is a matrix order structure.
Moreover it is straightforward to check that an $\cl A$-modular version of Theorem \ref{th_estareos} holds true for the $\cl A$-balanced tensor product $\cl E^*\otimes_{\rm s}^{\cl A} \cl S \otimes_{\rm s}^{\cl A} \cl E$.

%%%%%%%%%%%%%%%%%%%%%%%%%%%%
\begin{theorem}\label{t_Amodv}
Let $\cl A$ be a unital C*-algebra, $\cl S$ be an operator $\cl A$-system and $\cl E$ be an operator $\cl A$-space.
Then the following hold:
\begin{enumerate}
\item the family $\{C_n^{\cl A}\}_{n\in \bb{N}}$ is a matrix ordered structure on $\cl E^* \otimes_{\rm s}^{\cl A} \cl S \otimes_{\rm s}^{\cl A} \cl E$;
\item the operator space $\cl E^* \otimes_{\rm s}^{\cl A} \cl S \otimes_{\rm s}^{\cl A} \cl E$ equipped with the family $\{C_n^{\cl A}\}_{n\in \bb{N}}$ is a selfadjoint operator space.
\end{enumerate}

Moreover, if $\{D_n\}_{n \in \bb N}$ is a family of matricial cones such that $D_n \subseteq C_n^{\cl A}$ for every $n \in \bb N$ then the operator space $\cl E^* \otimes_{\rm s}^{\cl A} \cl S \otimes_{\rm s}^{\cl A} \cl E$ equipped with the family $\{D_n\}_{n \in \bb N}$ is a selfadjoint operator space. 
\end{theorem}

\begin{proof}
Item (i) is obtained by similar arguments as to those in Lemma \ref{l_mos}.
Item (ii) and the last part of the statement are obtained by an $\cl A$-modular version of the arguments in Theorem \ref{th_estareos}.
\end{proof}

Next we show that $\cl E^* \otimes_{\rm s}^{\cl A} \cl S \otimes_{\rm s}^{\cl A} \cl E$ is characterised by the following universal property.

%%%%%%%%%%%%%%%%%%%%%%%%%%%%
\begin{theorem}\label{t_univpro}
Let $\cl A$ be a unital C*-algebra, $\cl S$ be an operator $\cl A$-system and $\cl E$ be an operator $\cl A$-space.
Then the balanced symmetrisation $\cl E^* \otimes_{\rm s}^{\cl A} \cl S \otimes_{\rm s}^{\cl A} \cl E$ has the following universal property:
\begin{enumerate}
\item the trilinear map
\[
\iota \colon \cl E^* \times \cl S \times \cl E \to \cl E^* \otimes_{\rm s}^{\cl A} \cl S \otimes_{\rm s}^{\cl A} \cl E; (y^*, s, x) \mapsto y^* \odot^{\cl A} s \odot^{\cl A} x
\]
is completely contractive completely positive with dense range, and
\item if $\theta \colon \cl E^* \times \cl S \times \cl E \to \cl B(H)$ is a completely contractive completely positive $\cl A$-balanced trilinear map, then there exists a completely contractive completely positive map $\wt{\theta}_{\rm s \cl A} \colon \cl E^* \otimes_{\rm s}^{\cl A} \cl S \otimes_{\rm s}^{\cl A} \cl E \to \cl B(H)$ that makes the following diagram
\[
\xymatrix{
\cl E^* \times \cl S \times \cl E \ar[rr]^{\theta} \ar[d]^{\iota} & & \cl B(H) \\
\cl E^* \otimes_{\rm s}^{\cl A} \cl S \otimes_{\rm s}^{\cl A} \cl E \ar[urr]_{\wt{\theta}_{\rm s \cl A}} & &
}
\]
commutative.
\end{enumerate}
\end{theorem}

\begin{proof}
Item (i) follows by the definition of the symmetrisation norm since
\[
\| y^* \odot^{\cl A} s \odot^{\cl A} x \|_{\rm s \cl A}^{(n)} \leq \|y^*\|_{\cl E^*}^{(n,m)} \cdot \|s\|_{\cl S}^{(m,k)} \cdot \|x\|_{\cl E}^{(k,n)}
\]
for all $y \in M_{m,n}(\cl E)$, $s \in M_{m,k}(\cl S)$ and $x \in M_{k,n}(\cl E)$.
For item (ii), let $\theta \colon \cl E^* \times \cl S \times \cl E \to \cl B(H)$ be a completely contractive completely positive $\cl A$-balanced trilinear map.
By applying Lemma \ref{l_ssgen} we obtain an $\cl A$-admissible pair $(\phi, \psi)$ such that $\wt{\theta}_{\rm s \cl A} = \phi^* \cdot \psi \cdot \phi$, where $\phi$ is completely contractive since $\|\phi\|_{\rm cb}^2 = \|\theta\|_{\rm cb} \leq 1$.
By the definition of the balanced symmetrisation norm we have $\|\wt{\theta}_{\rm s \cl A}\|_{\rm cb} \leq 1$, and by definition $\wt{\theta}_{\rm s \cl A}$ is completely positive.
Universality follows by a standard argument, and the proof is complete.
\end{proof}

%%%%%%%%%%%%%%%%%%%%%%%%%%%%
\begin{remark}
Similarly to Remark \ref{r_norest}, let $\theta \colon \cl E^* \times \cl S \times \cl E \to \cl B(H)$ be a completely bounded completely positive $\cl A$-balanced trilinear map.
Then we can induce two completely bounded maps; one on the balanced symmetrisation,
\[
\wt{\theta}_{\rm s \cl A} \colon \cl E^* \otimes_{\rm s}^{\cl A} \otimes \cl S \otimes_{\rm s}^{\cl A} \cl E \to \cl B(H),
\]
and one on the balanced Haagerup tensor product,
\[
\wt{\theta}_{\rm h \cl A} \colon \cl E^* \otimes_{\rm h}^{\cl A} \cl S \otimes_{\rm h}^{\cl A} \cl E \to \cl B(H).
\]
It follows that 
\[
\|\wt{\theta}_{\rm s \cl A}\|_{\rm cb} = \|\wt{\theta}_{\rm h \cl A}\|_{\rm cb}.
\]
Indeed, by Lemma \ref{l_balnorm} we have that $\|\wt{\theta}_{\rm h \cl A}\|_{\rm cb} \leq \|\wt{\theta}_{\rm s \cl A}\|_{\rm cb}$.
On the other hand assume without loss of generality that $\|\wt{\theta}_{\rm h \cl A}\|_{\rm cb} = 1$.
Then the delinearisation map is completely contractive (and completely positive) and by the definition of the balanced symmetrisation norm we get 
\[
\|\wt{\theta}_{\rm s \cl A}\|_{\rm cb} \leq 1 = \|\wt{\theta}_{\rm h \cl A}\|_{\rm cb}.
\]
\end{remark}

%%%%%%%%%%%%%%%%%%%%%%%%%%%%
\begin{corollary}\label{c_iden}
Let $\cl A$ be a unital C*-algebra, $\cl S$ be an operator $\cl A$-system and $\cl E$ be an operator $\cl A$-space.
Then 
\[
\cl E^* \otimes_{\rm s} \cl S \otimes_{\rm s} \cl E / \cl J_{\cl E, \cl S}^{\otimes_{\rm s}^{\cl A}} \simeq
\cl E^* \otimes_{\rm s}^{\cl A} \cl S \otimes_{\rm s}^{\cl A} \cl E,
\]
via a canonical completely isometric complete order isomorphism.
\end{corollary}

\begin{proof}
First consider the trilinear map
\[
q \colon \cl E^* \times \cl S \times \cl E \to \cl E^* \otimes_{\rm s} \cl S \otimes_{\rm s} \cl E / \cl J_{\cl E, \cl S}^{\otimes_{\rm s}^{\cl A}} ; (y^*, s, x) \mapsto y^* \otimes s \otimes x + \cl J_{\cl E, \cl S}^{\otimes_{\rm s}^{\cl A}}.
\]
The map $q$ is completely contractive and completely positive as the composition of the completely contractive completely positive maps
\[
\cl E^* \times \cl S \times \cl E \to \cl E^* \otimes_{\rm s} \cl S \otimes_{\rm s} \cl E \to \cl E^* \otimes_{\rm s} \cl S \otimes_{\rm s} \cl E / \cl J_{\cl E, \cl S}^{\otimes_{\rm s}^{\cl A}}.
\]

Next let $\theta \colon \cl E^* \times \cl S \times \cl E \to \cl B(H)$ be an $\cl A$-balanced completely contractive completely positive map. 
Let $\wt{\theta}_{\rm s} \colon \cl E^* \otimes_{\rm s} \cl S \otimes_{\rm s} \cl E \to \cl B(H)$ be the map, arising from $\theta$ via the universal property of the symmetrisation.
By Lemma \ref{l_kerne} we have that $\wt{\theta}_{\rm s}$ annihilates $\cl J_{\cl E, \cl S}^{\otimes_{\rm s}^{\cl A}}$, and now the universal property of the selfadjoint operator space quotient (see \cite[Section 8]{kkm}) implies the existence of a canonical map 
\[
\wt{\theta}' \colon \cl E^* \otimes_{\rm s} \cl S \otimes_{\rm s} \cl E / \cl J_{\cl E, \cl S}^{\otimes_{\rm s}^{\cl A}} \to \cl B(H),
\] 
such that  $\theta = \wt{\theta}' \circ q$. 
We have thus shown that the quotient $\cl E^* \otimes_{\rm s} \cl S \otimes_{\rm s} \cl E / \cl J_{\cl E, \cl S}^{\otimes_{\rm s}^{\cl A}}$ satisfies the universal property in item (iii) of Theorem \ref{t_univpro}, and the claim follows. 
\end{proof}

%%%%%%%%%%%%%%%%%%%%%%%%%%%%
\begin{corollary}\label{c_consym}
Let $\cl A$ be a unital C*-algebra, $\cl S$ be an operator $\cl A$-system and $\cl E$ be an operator $\cl A$-space.
Let $C_n = M_n(\cl E^* \otimes_{\rm s} \cl S \otimes_{\rm s} \cl E)^+$, $n\in \bb{N}$, and 
\[
\wt{q}_{\rm s} \colon \cl E^* \otimes_{\rm s} \cl S \otimes_{\rm s} \cl E 
\to
\cl E^* \otimes_{\rm s}^{\cl A} \cl S \otimes_{\rm s}^{\cl A} \cl E
\]
be the canonical quotient map.
Then, for every $n\in \bb{N}$, we have
\begin{align*}
C_n^{\cl A} 
& = 
\ol{\wt{q}_{\rm s}(C_n)}^{\|\cdot\|_{\rm s \cl A}} \\
& = 
\{x^* \odot^{\cl A} s \odot^{\cl A} x \ : \ x \in M_{k,n}(\cl E), s \in M_k(\cl S)^+, k \in \bb N\}^{-\|\cdot\|_{\rm s \cl A}}.
\end{align*}
\end{corollary}

\begin{proof}
By the definition of $C_n^{\cl A}$ we have
\[
\ol{\wt{q}_{\rm s}(C_n)}^{\|\cdot\|_{\rm s \cl A}} \subseteq C_n^{\cl A}.
\]
Consider the completely contractive completely positive map
\[
q \colon \cl E^* \times \cl S \times \cl E \to 
(\cl E^* \otimes_{\rm s}^{\cl A} \cl S \otimes_{\rm s}^{\cl A} \cl E, \{C_n^{\cl A}\}_{n\in \bb{N}}).
\]
Note that, by Theorem \ref{t_Amodv}, $(\cl E^* \otimes_{\rm s}^{\cl A} \cl S \otimes_{\rm s}^{\cl A} \cl E, \{\ol{\wt{q}(C_n)}^{\|\cdot\|_{\rm s \cl A}}\}_{n\in \bb{N}})$ is a selfadjoint operaror space, and let the completely contractive map
\[
q' \colon \cl E^* \times \cl S \times \cl E \to 
(\cl E^* \otimes_{\rm s}^{\cl A} \cl S \otimes_{\rm s}^{\cl A} \cl E, \{\ol{\wt{q}(C_n)}^{\|\cdot\|_{\rm s \cl A}}\}_{n\in \bb{N}}),
\]
so that
\[
q'(y^*, s, x) = y^* \otimes^{\cl A} s \otimes^{\cl A} x, \ \ x,y \in \cl E, s \in \cl S.
\]
Then, by definition, the $\cl A$-balanced map $q'$ is completely contractive and completely positive, and thus by item (ii) of Theorem \ref{t_univpro} there exists a unique completely contractive completely positive map
\[
\wt{q'}_{\rm s \cl A} \colon 
(\cl E^* \otimes_{\rm s}^{\cl A} \cl S \otimes_{\rm s}^{\cl A} \cl E, \{C_n^{\cl A}\}_{n\in \bb{N}}) 
\to (\cl E^* \otimes_{\rm s}^{\cl A} 
\cl S \otimes_{\rm s}^{\cl A} \cl E, \{\ol{\wt{q}_{\rm s}(C_n)}^{\|\cdot\|_{\rm s \cl A}}\}_{n\in \bb{N}})
\]
such that $q' = \wt{q'}_{\rm s \cl A} \circ q$; hence $\wt{q'}_{\rm s \cl A}$ is the identity map on $\cl E^* \otimes_{\rm s}^{\cl A} \cl S \otimes_{\rm s}^{\cl A} \cl E$.
In particular complete positivity of $\wt{q'}_{\rm s \cl A}$ shows that 
\[
C_n^{\cl A} \subseteq \ol{\wt{q}_{\rm s}(C_n)}^{\|\cdot\|_{\rm s \cl A}},
\]
and the proof is complete.
\end{proof}

The ground cone spans the balanced symmetrisation.
This can be seen by employing similar arguments to those in Remark \ref{r_span} in relation to the description of the cones of Corollary \ref{c_consym}, or by using the quotient map.

The balanced symmetrisation has the following properties analogous to Theorem \ref{t_symuni}.

%%%%%%%%%%%%%%%%%%%%%%%%%%%%
\begin{theorem}\label{t_symunibal}
Let $\cl A$ be a unital C*-algebra, $\cl S$ be an operator $\cl A$-system and $\cl E$ be an operator $\cl A$-space.
Then the following hold:
\begin{enumerate}
\item If $\phi \colon \cl E\to \cl B(H,K)$ is a completely contractive map and $\psi \colon \cl S\to \cl B(K)$ is a completely positive map with the property that $(\phi, \psi)$ is an $\cl A$-admissible pair, then the map
\[
\phi^*\cdot\psi\cdot \phi \colon \cl E^* \otimes_{\rm s}^{\cl A} \cl S \otimes_{\rm s}^{\cl A} \cl E \to \cl B(H)
\]
is a completely contractive completely positive map. 

\item If $\wt{\theta}_{\rm s \cl A} \colon \cl E^*\otimes_{\rm s}^{\cl A} \cl S \otimes_{\rm s}^{\cl A} \cl E\to \cl B(H)$ is a completely contractive completely positive map, then there exist a completely contractive map $\phi \colon \cl E \to \cl B(H,K)$ and a unital completely positive map $\psi \colon \cl S\to \cl B(K)$ such that the pair $(\phi, \psi)$ is $\cl A$-admissible and 
\[
\wt{\theta}_{\rm s \cl A} = \phi^* \cdot \psi \cdot \phi.
\]

\item If $\cl A \subseteq \cl S$ and $\wt{\theta}_{\rm s \cl A} \colon \cl E^* \otimes_{\rm s}^{\cl A} \cl S \otimes_{\rm s}^{\cl A} \cl E\to \cl B(H)$ is a completely contractive completely positive map, then we can choose $\psi$ so that $\pi := \psi|_{\cl A}$ is a $*$-representation, $(\phi, \pi)$ is an $\cl A$-representation of $\cl E$ and $(\psi, \pi)$ is an $\cl A$-represen\-tation of $\cl S$.

\item If $\cl A \subseteq \cl S$, $\wt{\theta}_{\rm s \cl A} \colon \cl E^* \otimes_{\rm s}^{\cl A} \cl S \otimes_{\rm s}^{\cl A} \cl E\to \cl B(H)$ is a completely isometric completely positive map, and $(\phi, \psi)$ is as in (ii), then $\phi$ is a complete isometry. 
\end{enumerate}
\end{theorem}

\begin{proof}
Items (i) and (ii) follow in a similar way as in Theorem \ref{t_symuni}.
Likewise item (iii) follows from item (ii) of Lemma \ref{l_ssgen}.

For item (iv), let $(\phi_0, \pi_0)$ be a left $\cl A$-module pair for $\cl E$ over $\cl A$ such that $\phi_0$ is a complete isometry (for example by taking the representation in the injective envelope of $\cl E$).
Since $\cl A \subseteq \cl S$, there exists a unital completely positive extension $\psi_0$ of the unital $*$-homomorphism $\pi_0$ on $\cl S$.
Then the pair $(\phi_0, \psi_0)$ is $\cl A$-admissible and thus defines a completely contractive completely positive map $\phi_0^* \cdot \psi_0 \cdot \phi_0$ on $\cl E^*\otimes_{\rm s}^{\cl A} \cl S \otimes_{\rm s}^{\cl A} \cl E$.
By using a computation as in the proof of item (iii) of Theorem \ref{t_symunibal} for $x \in M_n(\cl E)$ we obtain
\begin{align*}
\|x\|^2 
& =
\| \phi_0^{(n)}(x) \phi_0^{(n)}(x)^* \|
\leq
\| x \odot (1_{\cl S} \otimes I_n) \odot x^* \|^{(n)}_{\rm s \cl A} \\
& =
\| \wt{\theta}_{\rm s \cl A}^{(n)} (x \odot (1_{\cl S} \otimes I_n) \odot x^*) \|^{(n)}
=
\| \phi^{(n)}(x) \phi^{(n)}(x)^* \| \leq \|x\|^2,
\end{align*}
and therefore $\phi$ is a complete isometry, as required.
\end{proof}

%%%%%%%%%%%%%%%%%%%%%%%%%%%%
\begin{proposition}\label{p_tensmaps}
For $i=1, 2$, let $\cl A_i$ be a unital C*-algebra, $\cl E_i$ be an operator $\cl A_i$-space and $\cl S_i$ be an operator $\cl A_i$-system with $\cl A_i \subseteq \cl S_i$. 
Let $\pi \colon \cl A_1\to \cl A_2$ be a unital $*$-homomorphism, $\phi \colon \cl E_1\to \cl E_2$ be a completely contractive map and $\psi \colon \cl S_1\to \cl S_2$ be a unital completely positive map, such that
\[
\phi(a\cdot x) = \pi(a)\cdot \phi(x) \ \ \mbox{ and } \ \ \psi(b \cdot s\cdot a) = \pi(b) \cdot \psi(s)\cdot\pi(a)
\]
for all $x\in \cl E_1, s\in \cl S_1, a,b\in \cl A_1$.
Then $\phi^* \otimes \psi \otimes \phi$ induces a completely contractive completely positive map (denoted in the same way) 
\begin{equation}\label{eq_tenma34}
\phi^* \otimes \psi \otimes \phi \colon \cl E_1^*\otimes_{\rm s}^{\cl A_1} \cl S_1\otimes_{\rm s}^{\cl A_1} \cl E_1 \to \cl E_2^*\otimes_{\rm s}^{\cl A_2} \cl S_2\otimes_{\rm s}^{\cl A_2} \cl E_2.
\end{equation}

Moreover, if $\cl A_i \subseteq \cl S_i$, $\pi$ is a $*$-isomorphism, $\phi$ is a complete isometry and $\psi$ is a complete order embedding, then $\phi^*\otimes \psi \otimes \phi$ is a completely isometric complete order embedding.
\end{proposition}

\begin{proof}
Let $\wt{\theta}_{\rm s \cl A_2} \colon \cl E_2^*\otimes_{\rm s}^{\cl A_2} \cl S_2\otimes_{\rm s}^{\cl A_2} \cl E_2 \to \cl B(H)$ be a completely isometric complete order embedding.
Using Theorem \ref{t_symuni}, we can write $\theta = \phi_0^* \cdot \psi_0\cdot\phi_0$, where $(\phi_0, \psi_0)$ is an $\cl A_2$-admissible pair for $(\cl E_2, \cl S_2)$ with $\phi_0$ a complete isometry.
By assumption, we then have that $(\phi_0 \circ \phi, \psi_0 \circ \psi)$ is an $\cl A_1$-admissible pair for $(\cl E_1, \cl S_1)$, and thus implements a completely contractive completely positive map 
\[
\wt{\theta}_{\rm s \cl A_1} \colon \cl E_1^*\otimes_{\rm s}^{\cl A_1} \cl S_1\otimes_{\rm s}^{\cl A_1} \cl E_1 \to \cl B(H)
\]
with (closed) range lying inside the (closed) range of $\wt{\theta}_{\rm s \cl A_2}$.
Then $\wt{\theta}^{-1}_{\rm s \cl A_2} \circ \wt{\theta}_{\rm s \cl A_1}$ is the required map.

For the second part, without loss of generality we may assume that $\cl A_1 = \cl A_2 = : \cl A$, that $\cl E_1 \subseteq \cl E_2$, and that $\cl S_1 \subseteq \cl S_2$.
Let $(\phi_1, \psi_1)$ be an admissible pair of $(\cl E_1, \cl S_1)$ and in addition assume that $\phi_1$ is non-degenerate.
Let $\pi:= \psi_1|_{\cl A}$ as $\cl A \subseteq \cl S_1$; since $\phi_1$ is non-degenerate it follows that $\pi$ is a $*$-homomorphism of $\cl A$.
Indeed, for $a,b \in \cl A$, $x \in \cl E$ and $h \in H$, we have
\begin{align*}
\pi(a)\pi(b) \phi(x)h 
& = 
\psi_1(1 \cdot a) \psi_1(1 \cdot b) \phi_1(x)h 
= 
\phi_1(ab \cdot x) h \\
& = 
\psi_1(1 \cdot ab) \phi_1(x)h 
= 
\pi(ab) \phi_1(x) h.
\end{align*}
By the module version of Arveson's Extension Theorem by Wittstock (see for example \cite[Theorem 3.6.2]{blm}) applied to the C*-representations $(\psi_1, \pi)$ and $(\phi_1, \pi)$ we can extend $\psi_1$ and $\phi_1$ to completely contractive maps $\psi_2$ and $\phi_2$ respectively so that $(\psi_2, \pi)$ and $(\phi_2, \pi)$ are C*-representations.
Note here that since $\cl S_1 \subseteq \cl S_2$ and they share the same unit, $\psi_2$ is unital, and thus completely positive as well.
This shows that
\[
\| (\phi_1^* \cdot \psi_1 \cdot \phi_1)^{(n)}(u) \| \leq \|\iota(u)\|^{(n)}_{\rm s \cl A}
\]
for the inclusion map $\iota \colon \cl E^*_1 \odot^{\cl A} \cl S_1 \odot^{\cl A} \cl E_1$.
Taking the supremum over the non-degenerate $\cl A$-admissible pairs and using Lemma \ref{l_balnorm} shows that the inclusion map is completely isometric.
\end{proof}

The second part of Proposition \ref{p_tensmaps} shows that the balanced symmetrisation is injective, when the spaces are modules over the same C*-algebra.
In Example \ref{e_ccnotcis} we will show that the map
\[
\cl E^* \otimes_{\rm s}^{\cl A_1} \cl S \otimes_{\rm s}^{\cl A_1} \cl E
\to 
\cl E^* \otimes_{\rm s}^{\cl A_2} \cl S \otimes_{\rm s}^{\cl A_2} \cl E
\]
may not be completely isometric when $\cl A_1 \subsetneq \cl A_2$.

Similarly to the non-balanced case, we say that $\cl E$ is \emph{$\cl A$-balanced $\cl S$-semi-unital} if 
there exists a $(\phi,\psi)$-semi-unit for some $\cl A$-admissible pair $(\phi,\psi)$ that gives rise to a completely isometric complete order embedding map $\phi^* \cdot \psi \cdot \phi$ of $\cl E^* \otimes_{\rm s}^{\cl A} \cl S \otimes_{\rm s}^{\cl A} \cl E$.
The following is obtained in the same way as in Proposition \ref{p_semi-unit}.

%%%%%%%%%%%%%%%%%%%%%%%%%%%%
\begin{proposition}\label{p_balsemi-unit}
Let $\cl A$ be a unital C*-algebra, $\cl S$ be an operator $\cl A$-system and $\cl E$ be an operator $\cl A$-space.
The following are equivalent:
\begin{enumerate}
\item $\cl E$ is $\cl A$-balanced $\cl S$-semi-unital;
\item the canonical balanced embedding $\cl E^* \otimes_{\rm s}^{\cl A} \cl A \otimes_{\rm s}^{\cl A} \cl E \to \cl E^* \otimes_{\rm s}^{\cl A} \cl S \otimes_{\rm s}^{\cl A} \cl E$ is unital.
\end{enumerate}
In particular, if $\cl E$ is $\cl A$-balanced $\cl S$-semi-unital, then it admits a symmetric $\cl A$-balanced $\cl S$-semi-unit.
\end{proposition}

The next proposition shows that the symmetrisation naturally captures the bimodule structure.

%%%%%%%%%%%%%%%%%%%%%%%%%%%%
\begin{proposition}\label{l_scstar}
Let $\cl A$ be a unital C*-algebra and $\cl S$ be an operator $\cl A$-system. 
Then the multiplication map $\cl A\otimes_{\rm s}^{\cl A}\cl S \otimes_{\rm s}^{\cl A}\cl A \to \cl S$ is a complete order isomorphism, and therefore $\cl A\otimes_{\rm s}^{\cl A}\cl S \otimes_{\rm s}^{\cl A}\cl A$ is an operator system with unit $1_{\cl A} \otimes 1_{\cl S} \otimes 1_{\cl A}$.
\end{proposition}

\begin{proof}
Let $\theta \colon \cl A \times \cl S \times \cl A \to \cl S$ be the bounded trilinear map, given by 
\[
\theta(a, s, b) = a\cdot s\cdot b, \ \ \ a,b \in \cl A, s \in \cl S.
\] 
Since $\cl S$ is an operator $\cl A$-system the map $\theta$ is completely contractive, completely positive and $\cl A$-balanced.
Hence by theorem \ref{t_univpro} it induces a completely contractive completely positive map
\[
\wt{\theta} \colon \cl A\otimes_s^{\cl A} \cl S \otimes_s^{\cl A} \cl A \to \cl S.
\]

On the other hand, let $\wt{\theta}' \colon \cl S\to \cl A\otimes_{\rm s}\cl S \otimes_{\rm s} \cl A$ be the map, given by 
\[
\wt{\theta}'(s) = 1_{\cl A}\otimes s \otimes 1_{\cl A}, \ \ \ s\in \cl S.
\]
By definition, $\wt{\theta}'$ is completely contractive and completely positive. 

Finally, let $\wt{q}_{\rm s} \colon \cl A\otimes_{\rm s}\cl S \otimes_{\rm s} \cl A \to \cl A\otimes_{\rm s}^{\cl A}\cl S \otimes_{\rm s}^{\cl A} \cl A$ be the (completely contractive completely positive) quotient map.
Consider the diagram
\[
\cl A\otimes_s^{\cl A} \cl S \otimes_s^{\cl A} \cl A \stackrel{\wt{\theta}}{\to} 
\cl S 
\stackrel{\wt{\theta}'}{\to} 
\cl A\otimes_s \cl S \otimes_s \cl A 
\stackrel{\wt{q}_{\rm s}}{\to} 
\cl A\otimes_s^{\cl A} \cl S \otimes_s^{\cl A} \cl A,
\]
and note that $\wt{q}_{\rm s}\circ \wt{\theta}' \circ \wt{\theta} = {\rm id}$.
It follows that each of the intermediate maps is a complete order isomorphism onto its range. 
However, $\wt{\theta}$ is clearly surjective; it follows that $\wt{\theta}$ is a complete order isomorphism.
\end{proof}

More generally we have the following proposition.

%%%%%%%%%%%%%%%%%%%%%%%%%%%%
\begin{theorem}\label{t_opBsy}
Let $\cl A$ and $\cl B$ be unital C*-algebras, $\cl S$ be an operator $\cl A$-system and $\cl E$ be a left operator $\cl A$-space and a right operator $\cl B$-space. 
If 
\[
\cl B \cdot \cl J_{\cl E, \cl S}^{\otimes_{\rm s}^{\cl A}} \subseteq \cl J_{\cl E, \cl S}^{\otimes_{\rm s}^{\cl A}},
\]
then $\cl E^*\otimes_{\rm s}^{\cl A} \cl S\otimes_{\rm s}^{\cl A} \cl E$ is a selfadjoint operator $\cl B$-space in a canonical fashion.
\end{theorem}

\begin{proof}
We will show that the symmetrisation $\cl E^*\otimes_{\rm s} \cl S\otimes_{\rm s} \cl E$ becomes a selfadjoint operator $\cl B$-space in a canonical fashion.
Then Corollary \ref{c_iden} and Corollary \ref{c_qasosbm} apply to pass to the quotient.

First, note that the mapping
\[
\cl B\times (\cl E^*\odot\cl S\odot \cl E)\times \cl B \to \cl E^*\odot\cl S\odot \cl E,
\]
\[
(b_1, y^*\otimes s\otimes x, b_2)\to (y\cdot b_1)^*\otimes s\otimes (x\cdot b_2) \ \ \ x,y\in \cl E, s\in \cl S, b_1, b_2\in \cl B,
\]
induces a (well-defined) $\cl B$-bimodule unital action on the algebraic tensor product $\cl E^*\odot\cl S\odot \cl E$. 

Next observe that, if $b\in M_{n,m}(\cl B)$ and $u\in C_n$, then $b^*\cdot u\cdot b\in C_m$. 
Indeed, first suppose that $u = x^*\odot s\odot x$ for some $s\in M_k(\cl S)^+$ and some $x\in M_{k,n}(\cl E)$. 
By Lemma \ref{l_mos}, we have
\begin{equation}\label{eq_conjg}
b^*\cdot (x^*\odot s\odot x) \cdot b = (x\cdot b)^*\odot s\odot (x\cdot b)\in M_n(\cl E^*\otimes_{\rm s} \cl S\otimes_{\rm s} \cl E)^+.
\end{equation}
By the characterisation of the cones from Proposition \ref{p_synth} we then get that $\cl E^* \otimes_{\rm s} \cl S \otimes_{\rm s} \cl E$ satisfies axiom (i) of Definition \ref{d_sosbim}.

It also follows by the definition of the involution that $(b \cdot u)^* = u^* \cdot b^*$ for all $u \in \cl E^*\odot\cl S\odot \cl E$.
Indeed, it suffices to show this for elementary tensors where we get
\[
(b \cdot (y^* \otimes s \otimes x))^*
=
( (yb^*)^* \otimes s \otimes x)^*
=
x^* \otimes s \otimes (yb^*)
=
(y^* \otimes s \otimes x)^* \cdot b^*
\]
for all $b \in \cl B$, and $x, y \in \cl E$ and $s \in \cl S$.
Thus $\cl E^* \otimes_{\rm s} \cl S \otimes_{\rm s} \cl E$ satisfies axiom (iii) of Definition \ref{d_sosbim}.

It is left to show that the $\cl B$-actions are completely contractive on the algebraic tensor product with respect to the symmetrisation norm, and by extension $\cl E^*\otimes_{\rm s} \cl S\otimes_{\rm s} \cl E$ is an operator $\cl B$-bimodule (see \cite[Theorem 4.6.7]{blm}). 
Then $\cl E^* \otimes_{\rm s} \cl S \otimes_{\rm s} \cl E$ satisfies axiom (ii) of Definition \ref{d_sosbim}, and thus $\cl E^* \otimes_{\rm s} \cl S \otimes_{\rm s} \cl E$ is an operator $\cl B$-space.
That is, we need to show that
\[
\| b \cdot (y^* \odot s \odot x) \|_{M_n(\cl E^* \otimes_{\rm s} \cl S \otimes_{\rm s} \cl E)}
\leq
\| b \|_{M_n(\cl B)} \cdot \| y^* \odot s \odot x \|_{M_n(\cl E^* \otimes_{\rm s} \cl S \otimes_{\rm s} \cl E)},
\]
and that
\[
\| (y^* \odot s \odot x) \cdot b \|_{M_n(\cl E^* \otimes_{\rm s} \cl S \otimes_{\rm s} \cl E)}
\leq
\| y^* \odot s \odot x \|_{M_n(\cl E^* \otimes_{\rm s} \cl S \otimes_{\rm s} \cl E)} \cdot \| b\|_{M_n(\cl B)},
\]
for every $y^* \odot s \odot x \in M_n(\cl E^* \otimes_{\rm s} \cl S \otimes_{\rm s} \cl E)$ and $b \in M_n(\cl B)$.
Towards this end, let $(\phi, \psi)$ be an admissible pair so that
\[
\phi \colon \cl E \to \cl B(H, K) \qand \psi \colon \cl S \to \cl B(K).
\]
Since $\phi$ is completely contractive and $\cl E$ is a left operator $\cl B$-module, the bilinear map
\[
\cl E \times \cl B \rightarrow \cl B(H, K); \ (x, b) \mapsto \phi(x b)
\]
is completely contractive and thus, by the CSPS Theorem \cite[Theorem 1.5.7]{blm}, there exist completely contractive maps $\wt{\phi}$ of $\cl E$ and $\wt{\pi}$ of $\cl B$ such that 
\[
\phi^{(k,n)}(x \cdot b)= \wt{\phi}^{(k,n)}(x) \wt{\pi}^{(n)}(b), \ \ x\in M_{k,n}(\cl E), b \in M_n(\cl B).
\]
In particular, we have
\[
\phi^{(k,n)} ( x ) = \wt{\phi}^{(k,n)} (x) \wt{\pi}^{(n)} (1_{\cl B} \otimes I_n), \ \ x\in M_{k,n}(\cl E).
\]
Trivially $(\wt{\phi}, \psi)$ is an admissible pair.
Let $x,y \in M_{k,n}(\cl E)$, $s \in M_k(\cl S)$ and $b \in M_n(\cl B)$; we then have
\begin{align*}
\| (\phi^* \cdot \psi \cdot \phi)^{(n)}( (y^* \odot s \odot x) \cdot b) \| 
& = \\
& \hspace{-3cm} =
\| \phi^{(k,n)}(y)^* \psi^{(k)}(s) \phi^{(k,n)}(x b) \| \\
& \hspace{-3cm} =
\| \wt{\pi}^{(n)}(1_{\cl B} \otimes I_n)^* \wt{\phi}^{(k,n)}(y)^* \psi^{(k)}(s) \wt{\phi}^{(k,n)}(x) \wt{\pi}^{(n)}(b) \| \\
& \hspace{-3cm} \leq
\| \wt{\pi}^{(n)}(1_{\cl B} \otimes I_n)^*\| \cdot \| \wt{\phi}^{(k,n)}(y)^* \psi^{(k)}(s) \wt{\phi}^{(k,n)}(x)\| \cdot \|\wt{\pi}^{(n)}(b) \| \\
& \hspace{-3cm} \leq
\| y^* \odot s \odot x \|_{M_n(\cl E^* \otimes_{\rm s} \cl S \otimes_{\rm s} \cl E)} \cdot \|b\|_{M_n(\cl B)}.
\end{align*}
Thus by taking the supremum over all admissible pairs $(\phi, \psi)$ we derive that
\[
\| (y^* \odot s \odot x) \cdot b \|_{M_n(\cl E^* \otimes_{\rm s} \cl S \otimes_{\rm s} \cl E)} \leq 
\| y^* \odot s \odot x \|_{M_n(\cl E^* \otimes_{\rm s} \cl S \otimes_{\rm s} \cl E)} \cdot \|b\|_{M_n(\cl B)}.
\]
Since the set
\[
\{ y^* \odot s \odot x \ : \ x,y \in M_{k,n}(\cl E), s \in M_k(\cl S)\}
\]
is dense in $M_n(\cl E^* \otimes_{\rm s} \cl S \otimes_{\rm s} \cl E)$, and $k, n \in \bb N$ were arbitrary, we deduce that $\cl E^* \otimes_{\rm s} \cl S \otimes_{\rm s} \cl E$ becomes a right operator $\cl B$-space.
The symmetric arguments for the left action imply that $\cl E^* \otimes_{\rm s} \cl S \otimes_{\rm s} \cl E$ becomes an operator $\cl B$-space, and the proof is complete.
\end{proof}

%%%%%%%%%%%%%%%%%%%%%%%%%%%%
\begin{remark} \label{r_bal}
Suppose that $\cl E$ is an $\cl A$-$\cl B$-bimodule over C*-algebras $\cl A$ and $\cl B$, and suppose that $\cl S$ is an operator space that is an $\cl A$-bimodule, so that $\cl E^* \otimes_{\rm s}^{\cl A} \cl S \otimes_{\rm s}^{\cl A} \cl E$ becomes a $\cl B$-bimodule.
Let $(\rho, \wt{\theta}_{\rm s \cl A})$ be a $\cl B$-bimodule map of the symmetrisation acting on a Hilbert space $H$.
Write $\wt{\theta}_{\rm s \cl A} = \phi^* \cdot \psi \cdot \phi$ for the admissible $\cl A$-pair constructed in the proof of Theorem \ref{t_symuni}.
We wish to show that $(\phi, \rho)$ is also a right $\cl B$-module map.

Towards this end, let $\theta$ be the delinearisation of $\wt{\theta}_{\rm s \cl A}$ so that 
\[
\theta(y^*, s, x) \rho(b) = \wt{\theta}_{\rm s \cl A}(y^* \otimes s \otimes x) \rho(b)
= \wt{\theta}_{\rm s \cl A}(y^* \otimes s \otimes (xb)) = \theta(y^*, s, xb)
\]
for all $x, y \in \cl E, s \in \cl S$ and $b \in \cl B$.
For $x \in \cl E$ and $b \in \cl B$, the construction of $K$ and $\phi$ of Lemma \ref{l_ssgen} yields
\begin{align*}
\sca{\phi(x) \rho(b) \xi, y \otimes \eta + N}_K
& =
\sca{\theta(y^*, 1_{\cl S}, x) \rho(b) \xi, \eta}_H \\
& =
\sca{\theta(y^*, 1_{\cl S}, x \cdot b) \xi, \eta}_H 
=
\sca{\phi(x b) \xi, y \otimes \eta + N}_K
\end{align*}
for all $y \in \cl E$, $\xi, \eta \in H$.
Thus $\phi(x) \rho(b) = \phi(x b)$ as required.
\end{remark}

%%%%%%%%%%%%%%%%%%%%%%%%%%%%
\subsection{Symmetrisation and ternary rings of operators}\label{ss_env}
%%%%%%%%%%%%%%%%%%%%%%%%%%%%

In this section we explore the properties of the symmetrisation when a ternary ring of operators is involved.
We further give some examples for showing that the tensor map 
(\ref{eq_tenma34}) of Proposition \ref{p_tensmaps} may not be an order embedding.

By Theorem \ref{t_symuni}, we have that every completely isometric completely positive map $\wt{\theta}_{\rm s} \colon \cl E^* \otimes_{\rm s} \cl S \otimes_{\rm s} \cl E \to \cl B(H)$ has the form $\wt{\theta}_{\rm s} = \phi^* \cdot \psi \cdot \phi$ for some complete isometry  $\phi \colon \cl E \to \cl B(H, K)$ and a unital completely positive map $\psi \colon \cl S \to \cl B(K)$.
The injectivity of the symmetrisation will allow us to choose $\wt{\theta}_{\rm s}$ in a way that $\psi$ is also completely isometric, provided that $\cl E^*$ is semi-unital.

%%%%%%%%%%%%%%%%%%%%%%%%%%%%
\begin{proposition}\label{l_ci}
Let $\cl M \subseteq \cl B(H, K)$ be a closed TRO and $\cl S \subseteq \cl B(K)$ be an operator system such that $\cl S = [\cl M \cl M^* \cl S]$.
Let $\wt{\theta}_{\rm s} \colon \cl M^* \otimes_{\rm s} \cl S \otimes_{\rm s} \cl M\to \cl B(\tilde{H})$ be a completely isometric completely positive map for some Hilbert space $\tilde{H}$, and let $(\phi, \psi)$ be the admissible pair with $\wt{\theta}_{\rm s} = \phi^* \cdot \psi \cdot \phi$.
Then $\psi$ is completely isometric.
\end{proposition}

\begin{proof}
Let $(\un{m}^\mu)_\mu \subseteq R(\cl M)$ be a net of row contractions with finite support 
so that
\[
\lim_\mu \un{m}^\mu (\un{m}^\mu)^* \cdot m = m, \ \ m \in \cl M;
\]
see for example \cite[Paragraph 8.1.23]{blm}.
Since $\cl S = [\cl M \cl M^* \cl S]$, we obtain
\[
\lim_\mu \un{m}^\mu (\un{m}^\mu)^* \cdot s = \lim_\mu s \cdot \un{m}^\mu (\un{m}^\mu)^* = s, \ \ s \in \cl S.
\]
Applying for $s = 1_{\cl S} = I_K$, we get that $[\cl M \cl M^*]$ is unital. 
By an application of Kasparov's Stabilisation Theorem \cite[Lemma 7.3]{lance} we can choose a row contraction $\un{m} \in R(\cl M)$ over $\bb N$ such that $\sum_n m_n m_n^* = 1_{\cl S}$.
Setting 
\[
\un{m}_n := [m_1, \cdots, m_n, 0, \cdots],
\]
we obtain
\begin{align}\label{eq_mnooo}
\| (s_{i,j})_{i,j} \|_{M_k(\cl S)} 
& = \lim_{n\to\infty} \| (\un{m}_n \un{m}^*_n \otimes I_k) (s_{i,j})_{i,j} (\un{m}_n \un{m}^*_n \otimes I_k)\| \nonumber\\ 
& \leq \sup_{n\in \bb{N}} \| (\un{m}^*_n \otimes I_k) (s_{i,j})_{i,j} (\un{m}_n \otimes I_k) \|_{M_{kn}(\cl M^* \cl S \cl M)}
\end{align}
for every $(s_{i,j})_{i,j} \in M_k(\cl S)$, $k \in \bb N$.
The multiplication map
\[
\cl M^* \otimes_{\rm s} \cl S \otimes_{\rm s} \cl M \to [\cl M^* \cl S \cl M],
\]
together with the isomorphism 
\[
[\phi(\cl M)^* \psi(\cl S) \phi(\cl M)] \simeq \cl M^* \otimes_{\rm s} \cl S \otimes_{\rm s} \cl M,
\]
yield a completely positive map 
\[
[\phi(\cl M)^* \psi(\cl S) \phi(\cl M)] \to [\cl M^* \cl S \cl M]; \ \phi(m_1)^* \psi(s) \phi(m_2) \mapsto m^*_1 s m_2.
\]
Using (\ref{eq_mnooo}), we derive
\begin{align*}
\| (s_{i,j})_{i,j} \|_{M_k(\cl S)}
& \leq \sup_{n\in \bb{N}} \| (\un{m}^*_n \otimes I_k) (s_{i,j})_{i,j} (\un{m}_n \otimes I_k) \|_{M_{kn}(\cl M^* \cl S \cl M)} \\
& \leq \sup_{n\in \bb{N}} \| (\phi^{(k,kn)}(\un{m}_n) \otimes I_k)^* \psi^{(k)}( (s_{i,j})_{i,j} ) (\phi^{(k,kn)}(\un{m}_n) \otimes I_k) \| \\
& \leq \|\psi^{(k)}( (s_{i,j})_{i,j} )\|,
\end{align*}
giving that $\psi^{(k)}$ is isometric.
Since $k$ was arbitrary we have that $\psi$ is completely isometric. 
\end{proof}

%%%%%%%%%%%%%%%%%%%%%%%%%%%%
\begin{corollary}
Let $\cl E$ be an operator space and $\cl S$ be an operator system.
If $\cl E^*$ is semi-unital, then there exists an admissible pair $(\phi, \psi)$ such that $\psi$ is a complete isometry and $\phi^* \cdot \psi \cdot \phi$ is a completely isometric complete order embedding.
\end{corollary}

\begin{proof}
Without loss of generality, let $\cl E \subseteq \cl B(H, K)$ be a concrete operator space with 
$I_K \in [\cl E \cl E^*]$, and let $\cl M$ be the closed TRO generated by $\cl E$.
Assume that $\cl S \subseteq \cl B(L)$ as an operator subsystem, and set
\[
\fr M := \cl M \otimes \bb{C}I_L \subseteq \cl B(H \otimes L, K \otimes L)
\qand
\fr S := [\cl M \cl M^*] \otimes \cl S \subseteq \cl B(K \otimes L).
\]
Note that $I_K \in [\cl M \cl M^*]$, since $I_K \in [\cl E \cl E^*]$, and so $\fr S$ is an operator system.
Consider the completely isometric map
\[
\cl E \to \fr M; \ x \mapsto x \otimes I_L,
\]
and the complete order embedding
\[
\cl S \to \fr S; \ s \mapsto I_K \otimes s.
\]
By the injectivity of the symmetrisation (Corollary \ref{t_injective}), these maps give rise to a completely isometric complete order embedding
\[
\cl E^* \otimes_{\rm s} \cl S \otimes_{\rm s} \cl E \hookrightarrow \fr M^* \otimes_{\rm s} \fr S \otimes_{\rm s} \fr M.
\]
Since $[\fr M \fr M^*]$ is unital, we have 
\[
\fr S \subseteq [\fr M \fr M^* \fr S] \subseteq [\cl M \cl M^* \cl M \cl M^*] \otimes \cl S = [\cl M \cl M^*] \otimes \cl S = \fr S,
\]
and thus $\fr M$ and $\fr S$ satisfy the assumptions of Proposition \ref{l_ci}.
Composing with a completely isometric complete order embedding of $\fr M^* \otimes_{\rm s} \fr S \otimes_{\rm s} \fr M$ 
and using Proposition \ref{l_ci} completes the proof.
\end{proof}

The following lemma gives a description of the Haagerup tensor product when involving TRO's; while the statement seems to be known, we have not been able to locate a precise reference and therefore include a proof for the convenience of the reader. 

%%%%%%%%%%%%%%%%%%%%%%%%%%%%
\begin{lemma}\label{l_trohaag}
Let $\cl M_1 \subseteq \cl B(H,K)$ and $\cl M_2 \subseteq \cl B(H,K)$ be closed TRO's, and $\cl E \subseteq \cl B(K)$ be an operator space such that
\[
[\cl M_1 \cl M_1^*] \cdot \cl E \subseteq \cl E
\quad \textup{and} \quad
\cl E \cdot [\cl M_2 \cl M_2^*] \subseteq \cl E.
\]
Then the canonical map
\[
\cl M_1^* \otimes_{\rm h}^{[\cl M_1 \cl M_1^*]} \cl E \otimes_{\rm h}^{[\cl M_2 \cl M_2^*]} \cl M_2
\to [\cl M_1^* \cl E \cl M_2]; m^*_1 \otimes x \otimes m_2 \mapsto m^*_1 x m_2,
\]
is a complete isometry.
\end{lemma}

\begin{proof}
The map is well-defined and completely contractive due to the universal property of the Haagerup tensor product.
In order to show that it is a complete isometry, let $X_i \in M_{k,n}(\cl M_1)$, $E_i \in M_{k,m}(\cl E)$ and $Y_i \in M_{m,n}(\cl M_2)$ for $i=1, \dots, N$; we will show that
\[
\left\| \sum_{i=1}^N X_i^* \odot E_i \odot Y_i \right\|_{\rm h} 
\leq \left\| \sum_{i=1}^N X_i^* E_i Y_i \right\|.
\]
Without loss of generality we may assume that $k=m$ by 
completing with zeroes.
Using the fact that $\cl M_1$ and $\cl M_2$ are TRO's, let $(\un{m}_{1}^\mu) \subseteq C(\cl M_1)$ and $(\un{m}_{2}^\la) \subseteq C(\cl M_2)$ be nets of column contractions such that
\[
\lim_\mu (\un{m}_{1}^\mu)^* \un{m}_{1}^\mu \cdot m_1^* = m_1^*
\qand
\lim_\la m_2 \cdot (\un{m}_{2}^\la)^* \un{m}_{2}^\la = m_2
\]
for all $m_1 \in \cl M_1$ and $m_2 \in \M_2$.
For $\eps>0$, let $\mu_0$ and $\la_0$ be such that
\begin{align*}
\left\| \sum_{i=1}^N X_i^* \odot E_i \odot Y_i \right\|_{\rm h} - \eps \leq
\\
& \hspace{-4cm} \leq
\left\| \sum_{i=1}^N ( \big[ (\un{m}_{1}^{\mu_0})^* \un{m}_{1}^{\mu_0} \otimes I_n) X_i^* \big] \odot E_i \odot \big[ Y_i ( (\un{m}_{2}^{\la_0})^* \un{m}_{2}^{\la_0} \otimes I_n) \big] \right\|_{\rm h} \\
& \hspace{-4cm} = 
\left\| (\un{m}_{1}^{\mu_0} \otimes I_n)^* \odot \left( \sum_{i=1}^N (\un{m}_{1}^{\mu_0} \otimes I_n) X_i^* E_i Y_i (\un{m}_{2}^{\la_0} \otimes I_n)^* \right) \odot (\un{m}_{2}^{\la_0} \otimes I_n) \right\|_{\rm h} \\
& \hspace{-4cm} \leq
\left\| \sum_{i=1}^N (\un{m}_{1}^{\mu_0} \otimes I_n) X_i^* E_i Y_i (\un{m}_{2}^{\la_0} \otimes I_n)^* \right\| \\
& \hspace{-4cm} =
\left\| (\un{m}_{1}^{\mu_0} \otimes I_n) \left( \sum_{i=1}^N X_i^* E_i Y_i \right) (\un{m}_{2}^{\la_0} \otimes I_n)^* \right\|
\leq
\left\| \sum_{i=1}^N X_i^* E_i Y_i \right\|,
\end{align*} 
where we have used the bimodule property of the balanced Haagerup tensor product and the fact that
\[
(\un{m}_{1}^{\mu_0} \otimes I_n) X_i^* \in M_{n,k}([\cl M_1 \cl M_1^*])
\qand
Y_i (\un{m}_{2}^{\la_0} \otimes I_n)^* \in M_{k,n}([\cl M_2 \cl M_2^*]).
\]
Taking $\eps \to 0$ gives the required inequality.
\end{proof}

%%%%%%%%%%%%%%%%%%%%%%%%%%%%
\begin{theorem} \label{t_trosym}
Let $\cl S$ be an operator system.
Suppose that $\psi \colon \cl S \to \cl B(K)$ is a unital complete order embedding and $\cl M \subseteq \cl B(H,K)$ is a closed TRO such that $\cl M \cl M^* \psi(\cl S) \subseteq \psi(\cl S)$.
Then $[\cl M \cl M^*]$ embeds canonically and completely isometrically in $\cl A_{\cl S}$ and the completely contractive completely positive map
\begin{align*}
\cl M^* \otimes_{\rm s}^{[\cl M \cl M^*]} \cl S \otimes_{\rm s}^{[\cl M \cl M^*]} \cl M 
& \to 
[\cl M^* \psi(\cl S) \cl M]; \\
m_1^* \otimes^{[\cl M \cl M^*]} s \otimes^{[\cl M \cl M^*]} m_2 
& \mapsto 
m_1^* \psi(s) m_2,
\end{align*}
is a completely isometric complete order isomorphism.
In particular, the canonical map
\[
\cl M^* \otimes_{\rm h}^{[\cl M \cl M^*]} \cl S \otimes_{\rm h}^{[\cl M \cl M^*]} \cl M
\to
\cl M^* \otimes_{\rm s}^{[\cl M \cl M^*]} \cl S \otimes_{\rm s}^{[\cl M \cl M^*]} \cl M
\]
is a complete isometry.
\end{theorem}

\begin{proof} 
Let the map $m \colon [\cl M^* \cl M] \to \cl A_{\cl S}$ such that $m_a(s) = \psi^{-1}(a \psi(s))$;
by \cite[Theorem 4.6.2]{blm} this map is a well-defined $*$-homomorphism.
Suppose that $m_a = 0$, so that $\psi^{-1}(a \psi(s)) = 0$ for all $s \in \cl S$.
Then, in particular, $a = a \psi(1_{\cl S}) = 0$ and thus $m$ is injective.

By the definition of the symmetrisation, we have canonical completely contractive completely positive maps
\[
\cl M^* \hspace{-0.1cm}\otimes_{\rm s}^{[\cl M \cl M^*]} \cl S \otimes_{\rm s}^{[\cl M \cl M^*]} 
\hspace{-0.1cm}\cl M
\to \cl M^* \hspace{-0.1cm}\otimes_{\rm s}^{[\cl M \cl M^*]} \psi(\cl S) \otimes_{\rm s}^{[\cl M \cl M^*]} \hspace{-0.1cm}\cl M
\to [\cl M^* \psi(\cl S) \cl M],
\]
where the first arrow is a complete order isomorphism.
Hence without loss of generality we may assume that $\cl S \subseteq \cl B(K)$, that is, that $\psi = \id$.
By Lemma \ref{l_trohaag}, the composition
\[
\cl M^* \otimes_{\rm h}^{[\cl M \cl M^*]} \cl S \otimes_{\rm h}^{[\cl M \cl M^*]} \cl M
\to
\cl M^* \otimes_{\rm s}^{[\cl M \cl M^*]} \cl S \otimes_{\rm s}^{[\cl M \cl M^*]} \cl M
\to
[\cl M^* \cl S \cl M]
\]
is completely isometric.
Hence it suffices to show that the canonical completely isometric completely positive map
\[
\Phi \colon \cl M^* \otimes_{\rm s}^{[\cl M \cl M^*]} \cl S \otimes_{\rm s}^{[\cl M \cl M^*]} \cl M
\to
[\cl M^* \cl S \cl M]
\]
has a completely positive inverse.

By Corollary \ref{c_consym}, it suffices to show that if $u \in M_n([\cl M^* \cl S \cl M])$ is positive, then it is in the closure of
\begin{align*}
\Phi(\{m^* \odot^{\cl A} s \odot^{\cl A} m \ : \ m \in M_{k,n}(\cl M), s \in M_k(\cl S)^+, k \in \bb N\})
& = \\
& \hspace{-6cm} = 
\{m^* s m \ : \ m \in M_{k,n}(\cl M), s \in M_k(\cl S)^+, k \in \bb N\}.
\end{align*}
Towards this end, recall that since $\cl M$ is a closed TRO, there exists a net of finitely supported columns $\un{m}_\al \in C_{k_\al}(\cl M)$, such that
\[
\lim_\al \un{m}_\al^* \un{m}_\al m^* = m^*, \ \ m \in \cl M.
\]
and therefore
\[
\lim_\al \un{m}_\al^* \un{m}_\al u = u, \ \ u \in [\cl M^* \cl S \cl M].
\]
Fix $n \in \bb N$ and a positive element $u \in M_n([\cl M^* \cl S \cl M])$, and set
\[
s_\al := 
(\un{m}_\al \otimes I_n) 
\cdot
u
\cdot
(\un{m}_\al \otimes I_n) ^*.
\]
Note that
\[
s_{\alpha} \in M_{k_\al \cdot n, n}(\cl M) \cdot M_n([\cl M^* \cl S \cl M]) \cdot M_{n, k_\al \cdot n}(\cl M^*) \subseteq M_{k_\al \cdot n}(\cl S),
\]
since $\cl M \cl M^* \cl S \subseteq \cl S$.
Moreover, each $s_\al$ is positive since $u$ is positive.
We now observe that 
\begin{align*}
u
& =
\lim_\al \un{m}_\al^* \un{m}_\al \cdot u \cdot \un{m}_\al^* \un{m}_\al 
=
\lim_\al 
(\un{m}_\al \otimes I_n)^*
\cdot
s_\al
\cdot
(\un{m}_\al \otimes I_n),
\end{align*}
and the proof is complete.
\end{proof}

%%%%%%%%%%%%%%%%%%%%%%%%%%%%
\begin{corollary} \label{c_trosymbim}
Under the conditions of Theorem \ref{t_trosym}, the balanced symmetrisation $\cl M^* \otimes_{\rm s}^{[\cl M \cl M^*]} \cl S \otimes_{\rm s}^{[\cl M \cl M^*]} \cl M$ becomes a selfadjoint 
operator $[\cl M^* \cl M]$-space in a canonical fashion.
\end{corollary}

\begin{proof}
Recall that the balanced Haagerup tensor product is the quotient of Haagerup tensor product by the balanced relations.
Since the canonical map
\begin{align*}
\cl M^* \otimes_{\rm h}^{[\cl M \cl M^*]} \cl S \otimes_{\rm h}^{[\cl M \cl M^*]} \cl M 
&\to 
\cl M^* \otimes_{\rm s}^{[\cl M \cl M^*]} \cl S \otimes_{\rm s}^{[\cl M \cl M^*]} \cl M 
\end{align*}
is a complete isometry by Theorem \ref{t_trosym}, we have that
\[
\cl J_{\cl E, \cl S}^{[\cl M^* \cl M]} = \cl J_{\cl E, \cl S}^{\otimes_{\rm s}^{[\cl M^* \cl M]}}.
\]
 By definition, 
\[
[\cl M^* \cl M] \cdot \cl J_{\cl E, \cl S}^{[\cl M^* \cl M]} \subseteq \cl J_{\cl E, \cl S}^{[\cl M^* \cl M]}.
\]
Therefore the condition of Theorem \ref{t_opBsy} is satisfied and the bimodule structure of $\cl M^* \otimes_{\rm s} \cl S \otimes_{\rm s} \cl M$ over $[\cl M^* \cl M]$ descends to the balanced symmetrisation.
\end{proof}

We record here the following remark on a well-known result for TRO's. 
We will use it in this setup in the sequel.

%%%%%%%%%%%%%%%%%%%%%%%%%%%%
\begin{remark} \label{r_unit}
Assume the setup of Theorem \ref{t_trosym}, and suppose in addition that $\cl T := [\cl M^* \psi(\cl S) \cl M]$ is unital.
Let $\un{m}_\al \in C_{k_\al}(\cl M)$ be a net of finitely supported columns such that 
\[
\lim_\al \un{m}_\al^* \un{m}_\al m^* = m^*, \ \ \ m \in \cl M.
\]
Then, in particular, 
\[
\lim_\al \un{m}_\al^* \un{m}_\al t = t, \ \ t \in [\cl M^* \psi(\cl S) \cl M],
\]
and thus $\lim_\al \un{m}_\al^* \un{m}_\al = 1_{\cl T}$.
By the complete order isomorphism of Theorem \ref{t_trosym}, we have that $\cl M$ is $[\cl M^* \cl M]$-balanced $\cl S$-semi-unital, that is, the balanced symmetrisation $\cl M^* \otimes_{\rm s}^{[\cl M \cl M^*]} \cl S \otimes_{\rm s}^{[\cl M \cl M^*]} \cl M$ is unital, and if $e$ is its unit, then
\[
\lim_\al \un{m}_{\al}^* \odot (1_{\cl S} \otimes 1_{k_\al}) \odot \un{m}_\al = e.
\]
It moreover follows that $\cl M$ is finitely generated as a module over $[\cl M^* \cl M]$.
Indeed, since $\psi(\cl S)$ contains the unit, we have that $[\cl M^* \cl M] \subseteq \cl T$.
Since $\T$ is unital, then $\lim_\al \un{m}_\al^* \un{m}_\al = 1_{\cl T}$, and thus $[\cl M^* \cl M]$ is unital.
By an application of Kasparov's Stabilisation Theorem \cite[Lemma 7.3]{lance} there is a $(y_i)_{i \in \bb N} \subseteq \cl M$ such that $\sum_{i} y_i^* y_i = 1_{\cl M}$.
Let $n \in \bb N$ such that
\[
\left\|1 - \sum_{i=1}^n y_i^* y_i \right\| \leq 1/2.
\]
Then $c := \sum_{i=1}^n y_i^* y_i \in [\cl M^* \cl M]$ is a selfadjoint invertible element.
Set 
\[
x_i := y_i c^{-1/2} \in [\cl M \cl M^* \cl M] = \cl M
\]
and compute
\begin{align*}
\sum_{i=1}^n x_i^* x_i = c^{-1/2} \left( \sum_{i=1}^n y_i^* y_i \right) c^{-1/2} = c^{-1/2} c c^{-1/2} = 1.
\end{align*}
Hence $m = \sum_{i=1}^n (m x_i^*) x_i$ for every $m \in \cl M$.
\end{remark}

%%%%%%%%%%%%%%%%%%%%%%%%%%%%
\begin{corollary}\label{c_troscand}
Let $\cl M$ be a closed TRO.
Then the maps
\[
\cl M^* \otimes_{\rm s}^{[\cl M \cl M^*]} \cl M \to [\cl M^* \cl M]; \ 
y^* \otimes x \mapsto y^*x
\]
and
\[
\cl M \otimes_{\rm s}^{[\cl M^* \cl M]} \cl M^* \to [\cl M \cl M^*]; \ 
x \otimes y^* \mapsto xy^*
\]
are completely isometric complete order isomorphisms.
\end{corollary}

%%%%%%%%%%%%%%%%%%%%%%%%%%%%
\begin{remark}\label{r_troscand}
An application of Corollary \ref{c_troscand} for $\cl M = R_n$ gives that
\[
R_n \otimes_{\rm s}^{M_n} C_n \simeq \bb C
\qand
C_n \otimes_{\rm s} R_n \equiv C_n \otimes_{\rm s}^{\bb C} R_n \simeq M_n,
\]
up to completely isometric complete order isomorphisms.
In particular, $C_n \otimes_{\rm s} R_n$ is an operator system.

More generally, write $R_I$ for the row space over a set $I$, and let $\cl K_I$ be the C*-algebra of all compact operators on $\ell^2(I)$.
Then
\[
C_I \otimes_{\rm s} R_I \equiv C_I \otimes_{\rm s}^{\bb C} R_I \simeq \cl K_I.
\]
Hence $C_I \otimes_{\rm s} R_I$ is unital if and only if $\cl K_I$ is unital if and only if $I$ is finite.
Indeed, if $I$ is finite, then $\cl K_I = M_{|I|}$ which is an operator system.
Conversely, if $\cl K_I$ is an operator system, then $\cl K_I$ is a unital C*-algebra (coinciding with its own C*-envelope), and thus $I$ is finite.
Hence even separable operator spaces may provide non-unital selfadjoint operator spaces through their symmetrisation, cf., Theorem \ref{t_nonsepnonun}.
\end{remark}

%%%%%%%%%%%%%%%%%%%%%%%%%%%%
\begin{example} \label{e_ccnotcis}
For an operator space $\cl E$, the quotient map 
\[
\cl E^* \otimes_{\rm s} \cl E \to \cl E^* \otimes_{\rm s}^{\cl A_{\ell}(\cl E)} \cl E
\]
is completely contractive and completely positive.
We now give an example for which this map is not isometric.

Towards this end, consider the non-degenerate TRO $\cl M = M_{k,n} \oplus M_{m, \ell}$ for integers $k, n, m, \ell \neq 1$, so that 
\[
\cl A_{\ell}(\cl M) = [\cl M \cl M^*] = M_k \oplus M_m.
\]
Let $y := 0 \oplus E_{m,\ell}$, $x := E_{1,1} \oplus 0$ in $\cl \M$, and $a := E_{1, 1} \oplus 0$ in $\cl A_{\ell}(\cl M)$, so that $y^* a = 0$ and $ax = x$. 
On one hand, by Lemma \ref{l_norm}, we have
\[
\|y^* \otimes x\|_{\rm s} \geq \frac{1}{4} \|y^* \otimes x\|_{\rm h} = \frac{1}{4} \|y^*\| \cdot \|x\| = \frac{1}{4};
\]
on the other hand,
\[
\|y^* \otimes x\|_{{\rm s} \cl A_{\ell}(\cl M)} = 
\|y^* \otimes a x\|_{{\rm s} \cl A_{\ell}(\cl M)} =
\|y^* a \otimes x\|_{{\rm s} \cl A_{\ell}(\cl M)} = 0.
\]
Hence the map
\[
\cl M^* \otimes_{\rm s} \cl M \to \cl M^* \otimes_{\rm s}^{\cl A_{\ell}(\cl M)} \cl M; y^* \otimes x \mapsto y^* \otimes x
\]
is not isometric.
\end{example}

%%%%%%%%%%%%%%%%%%%%%%%%%%%%
\begin{example}\label{e_env}
Let $\cl E$ be an operator space. Since $\cl A_{\ell}(\cl E)$ is a unital C*-subalgebra of $\cl A_{\ell}(\tenv(\cl E))$ in a canonical fashion, we have that the map 
\[
q \colon \cl E^* \otimes_{\rm s}^{\cl A_{\ell}(\cl E)} \cl E \to \tenv(\cl E)^* \otimes_{\rm s}^{\cl A_{\ell}(\tenv(\cl E))} \tenv(\cl E); \ y^* \otimes x \mapsto y^* \otimes x
\]
is well defined, completely contractive and completely positive.
We now provide an example of an operator space $\cl E$ for which this map is not completely isometric.

Towards this end, consider the space
\[
{\rm C}(S^1)^{(2)} := \left\{ \begin{pmatrix} \al & \be \\ \ga & \al \end{pmatrix} \ : \ \al, \be, \ga \in \bb C \right\}.
\]
A more general class of operator spaces (systems), including ${\rm C}(S^1)^{(2)}$, was considered in \cite{cvs}, where it was shown that 
\[
\cenv({\rm C}(S^1)^{(2)}) = M_2.
\]
In particular, ${\rm C}(S^1)^{(2)}$ is rigid, that is, 
\[
\cl A_{\ell}({\rm C}(S^1)^{(2)}) = \bb C.
\]
Note here that ${\rm C}(S^1)^{(2)}$ is a unital operator space, and thus its C*-envelope coincides with its TRO-envelope, see \cite[Section 4.4.7]{blm}.
In this case the map $q$ yields a canonical map
\[
({\rm C}(S^1)^{(2)})^* \otimes_{\rm s} {\rm C}(S^1)^{(2)} \to M_2^* \otimes_{\rm  s}^{M_2} M_2 \simeq [M_2^* M_2] = M_2,
\]
where we have applied Lemma \ref{c_troscand}.
Letting $y := E_{2,1}$ and $x := E_{1,2}$ as elements of ${\rm C}(S^1)^{(2)}$, 
by Lemma \ref{l_norm} we have 
\[
\| y^* \otimes x \|_{\rm s} \geq \frac{1}{4} \|y^* \otimes x\|_{\rm h} = \frac{1}{4} \cdot \|y^*\| \cdot \|x\| = \frac{1}{4}.
\]
On the other hand,
\[
\|q(y^* \otimes x)\|_{M_2} = \|y^* x\|_{M_2} = \| E_{2,1} \cdot E_{2,1} \|_{M_2} = 0,
\]
and thus the map $q$ is not (completely) isometric.
\end{example}

%%%%%%%%%%%%%%%%%%%%%%%%%%%%
\section{Morita equivalence}\label{s_morita}
%%%%%%%%%%%%%%%%%%%%%%%%%%%%

This section is devoted to one of our main applications of the symmetrisation construction, namely, 
a factorisation result for $\Delta$-equivalent operator systems that can be viewed, in the terminology of \cite{Bas62}, as Morita Theorem II for the category of operator systems. 

%%%%%%%%%%%%%%%%%%%%%%%%%%%%
\subsection{$\Delta$-equivalence}
%%%%%%%%%%%%%%%%%%%%%%%%%%%%

We recall from \cite{ekt} the notion of $\Delta$-equivalence in the operator system category, 
and its abstract characterisations obtained therein. 
Two concrete closed operator systems $\cl S\subseteq \cl B(H)$ and $\cl T \subseteq \cl B(K)$ are called \emph{TRO-equivalent} if there exists a closed TRO $\cl M\subseteq \cl B(H,K)$ such that 
\[
[\cl M^*\cl T\cl M] = \cl S \ \ \mbox{ and } \ \ [\cl M\cl S\cl M^*] = \cl T;
\]
in this case, we write $\cl S\sim_{\rm TRO}\cl T$.
If $\cl S$ and $\cl T$ are abstract operator systems, they are called \emph{$\Delta$-equivalent} if there exist 
(unital) complete order embeddings $\psi_{\cl S} \colon \cl S\to \cl B(H)$ and $\psi_{\cl T} \colon \cl T\to \cl B(K)$ such that $\psi_{\cl S}(\cl S) \sim_{\rm TRO} \psi_{\cl T}(\cl T)$; in this case we write $\cl S\sim_{\Delta} \cl T$. 

Let $\cl S$ and $\cl T$ be abstract operator systems and $\M$ be a TRO.
We say that the quintuple $\big( \cl S, \cl T, \M, [\cdot, \cdot, \cdot], (\cdot, \cdot, \cdot) \big)$ is a \emph{$\Delta$-context} if:

\begin{itemize}
\item[(i)] the C*-algebras $[\M^* \M]$ and $[\M \M^*]$ are unital;

\item[(ii)] $\cl S$ is an operator bimodule over the C*-algebra $[\M^* \M]$ and 
$\cl T$ is a operator bimodule over the C*-algebra $[\M \M^*]$; 

\item[(iii)] $[\cdot,\cdot,\cdot] \colon \M^* \times \cl T \times \M \longrightarrow \cl S$ and $(\cdot,\cdot,\cdot) \colon \M \times \cl S \times \M^* \longrightarrow \cl T$ are completely contractive completely positive maps, modular over $[\M^*\M]$ and $[\M \M^*]$ on the outer variables (with unital module actions);

\item[(iv)] the associativity relations
\[
(m_1, [m_2^*, t, m_3], m_4^*) = (m_1m_2^*) \cdot t \cdot (m_3 m_4^*)
\]
and
\[
[m_1^*, (m_2,  s, m_3^*), m_4] = (m_1^* m_2) \cdot s \cdot (m_3^* m_4)
\]
hold for all $s \in \cl S$, $t \in \cl T$ and all $m_1, m_2, m_3, m_4 \in \M$; and

\item[(v)] 
the trilinear maps $[\cdot, \cdot, \cdot]$ and $(\cdot, \cdot, \cdot)$ satisfy the relations
\begin{equation}\label{eq_move}
(m_1, 1_\cl S, m_2^*) = (m_1 m_2^*) \cdot 1_\cl T
\qand
[m_1^*, 1_\cl S, m_2] = (m_1^* m_2) \cdot 1_\cl S
\end{equation}
for all $m_1, m_2 \in \M$.
\end{itemize}
By items (i) and (iv) we see that the trilinear maps $[\cdot,\cdot,\cdot]$ and $(\cdot,\cdot,\cdot)$ are surjective.

More generally, let $\cl X$ be an abstract operator space.
We say that the quintuple $\big( \cl S, \cl T, \cl X, [\cdot, \cdot, \cdot], (\cdot, \cdot, \cdot) \big)$ is a \emph{bihomomorphism context} if:

\begin{itemize}
\item[(i)] $\cl X$ is non-degenerate;

\item[(ii)] 
$[\cdot,\cdot,\cdot] \colon \cl X^* \times \cl T \times \cl X \longrightarrow \cl S$ and 
$(\cdot,\cdot,\cdot) \colon \cl X \times \cl S \times \cl X^* \longrightarrow \cl T$
are completely contractive completely positive maps such that
\[
[\cl X^*, 1_\cl T, \cl X] \subseteq \cl A_{\cl S} \qand (\cl X, 1_\cl S, \cl X^*) \subseteq \cl A_{\cl T};
\]

\item[(iii)] the associativity relations
\[
[x_1^*, (x_2, s, x_3^*), x_4] = [x_1^*, 1_\cl T, x_2] \cdot s \cdot [x_3^*, 1_\cl T, x_4]
\]
and
\[
(x_1, [x_2^*, t, x_3], x_4^*) = (x_1, 1_\cl S, x_2^*) \cdot t \cdot (x_3, 1_\cl S, x_4^*)
\]
hold for all $s \in\cl S, t \in \cl T$ and all $x_1, x_2, x_3, x_4 \in \cl X$; and

\item[(iv)] there exist semi-units $((\un{y}_i)_i, (\un{x}_i)_i)$ and $((\un{w}_i)_i, (\un{z}_i)_i)$ over $\cl X$ and $\cl X^*$, respectively, such that 
\begin{equation}\label{eq_adjxi}
\lim_i [\un{y}_i^*, 1_\cl T \otimes I, \un{x}_i] =  1_\cl S
\ \ \mbox{ and } \ \ 
\lim_i (\un{w}_i, 1_\cl S \otimes I, \un{z}_i^*) = 1_\cl T.
\end{equation}
\end{itemize}
By items (iii) and (iv) we see that the trilinear maps $[\cdot,\cdot,\cdot]$ and $(\cdot,\cdot,\cdot)$ are surjective.
In \cite{ekt} we studied the connection between $\Delta$-contexts, bihomomorphism contexts and $\Delta$-equivalence.
We write $\cl K$ for the set of compact operators in $\ell^2$.

%%%%%%%%%%%%%%%%%%%%%%%%%%%%%%%%
\begin{theorem} \cite[Theorem 3.17]{ekt}\label{th_ekt}
Let $\cl S$ and $\cl T$ be (abstract) operator systems. 
The following are equivalent:
\begin{enumerate}
\item $\cl S\sim_{\Delta}\cl T$;
\item $\cl S \otimes \cl K \simeq \cl T \otimes \cl K$ by a complete order isomorphism;
\item there exists a $\Delta$-context for $\cl S$ and $\cl T$;
\item there exists a bihomomorphism context for $\cl S$ and $\cl T$.
\end{enumerate}
\end{theorem}

In \cite{bmp} Blecher-Muhly-Paulsen defined the notion of a Morita context for a pair of (nonselfadjoint) operator algebras, and in \cite[Theorem 3.5]{bmp} they showed that a Morita context for approximately unital algebras results in a factorisation via Haagerup tensor products, in analogy to the characterisation of Morita equivalence of rings.
The additional step for the category of operator spaces in comparison to rings is that there are more than one topological tensor product, and the choice of the right norm plays an important role.
The following remark is analogous to \cite[Theorem 3.5]{bmp} for $\Delta$-equivalence of operator systems, taking into account the fact that $\Delta$-contexts are the analogues of Morita contexts in the operator system category.
We note that the notion of a bihomomorphism context does not have a counterpart in the Morita theory of operator algebras. 

%%%%%%%%%%%%%%%%%%%%%%%%%%%%
\begin{remark}\label{r_deltahaag}
Following \cite{ek}, in \cite{ekt} we established that, if $\cl S \sim_{\Delta}\cl T$, then there exist complete order embeddings 
\[
\psi_{\cl S} \colon \cl S \to \cl B(H)
\qand 
\psi_{\cl T} \colon \cl T \to \cl B(K),
\]
and a non-degenerate TRO $\cl M \subseteq \cl B(H, K)$, such that $\psi_{\cl S}(\cl S) \sim_{\rm TRO} \psi_{\cl T}(\cl T)$ via $\cl M$ and, in addition, 
\[
\psi_{\cl S}(\cl A_{\cl S}) = [\cl M^* \cl M] \qand \psi_{\cl T}(\cl A_{\cl T}) = [\cl M \cl M^*].
\]
Furthermore, 
\[
\cenv(\cl S) \simeq \ca(\psi_{\cl S}(\cl S)) \qand \cenv(\cl T) \simeq \ca(\psi_{\cl T}(\cl T)),
\]
and consequently $\cenv(\cl S) \sim_{\Delta} \cenv(\cl T)$.
By using the properties of the symmetrisation we obtain canonical completely contractive completely positive maps
\[
\cl M^* \otimes_{s}^{\cl A_{\cl T}} \cl T \otimes_{s}^{\cl A_{\cl T}} \cl M 
\to [\cl M^* \psi_{\cl T}(\cl T) \cl M] 
\simeq \cl S
\]
and
\[
\cl M \otimes_{s}^{\cl A_{\cl S}} \cl S \otimes_{s}^{\cl A_{\cl S}} \cl M^* 
\to [\cl M \psi_{\cl S}(\cl S) \cl M^*] 
\simeq \cl T,
\]
induced by the multiplication maps.
An application of Theorem \ref{t_trosym} shows that these maps are completely isometric complete order isomorphisms.

An application of Remark \ref{r_unit} implies that $\cl M$ and $\cl M^*$ are necessarily finitely generated as right Hilbert modules.
Moreover, the symmetrisations are unital, and a fortiori $\cl M$ is $\cl A_{\cl T}$-balanced $\cl T$-semi-unital, and $\cl M^*$ is $\cl A_{\cl S}$-balanced $\cl S$-semi-unital.

Furthermore, an application of Corollary \ref{c_consym} yields a canonical association between the positive elements.
In particular, we have
\[
\psi_{\cl S}^{(n)}(M_n(\cl S)^+) = 
\{x^* \psi_{\cl T}^{(k)}(M_k(\cl S)^+) x \ : \ x \in M_{k,n}(\cl M), k \in \bb N\}^{-\|\cdot\|},
\]
and 
\[
\psi_{\cl T}^{(n)}(M_n(\cl T)^+) = 
\{x \psi_{\cl S}^{(k)}(M_k(\cl S)^+) x^* \ : \ x \in M_{n,k}(\cl M), k \in \bb N\}^{-\|\cdot\|}
\]
for the positive cones $\{C_n({\cl S})\}_{n \in \bb N}$ and $\{C_n({\cl T})\}_{n \in \bb N}$ of $\cl S$ and $\cl T$, respectively.
\end{remark}

%%%%%%%%%%%%%%%%%%%%%%%%%%%%
\subsection{Factorisation of $\Delta$-equivalent operator systems}
%%%%%%%%%%%%%%%%%%%%%%%%%%%%

We would like to obtain a converse of Remark \ref{r_deltahaag} that does not depend on a linking representation for the TRO giving the $\Delta$-equivalence.
Furthermore we would like to relax the conditions for obtaining a $\Delta$- or a bihomomorphism context.
Although both versions turn out to be identical with stable isomorphisms by \cite[Theorem 3.8]{ekt}, we would like to make further connection with bihomomoprhisms as they appear in the context of non-commutative graphs \cite{stahlke} which has been a point of motivation.
This will be the subject of this subsection; its main result complements Theorem \ref{th_ekt} with one more characterisation of $\Delta$-equivalence, this time in terms of tensorial factorisations.

%%%%%%%%%%%%%%%%%%%%%%%%%%%%
\begin{definition}
Let $\cl S$ and $\cl T$ be operator systems and $\cl E$ be a left operator $\cl A_{\cl T}$-space that is also a right operator $\cl A_{\cl S}$-space. 
Suppose that both spaces $\cl E^*\otimes_{\rm s}^{\cl A_{\cl T}}\cl T \otimes_{\rm s}^{\cl A_{\cl T}}\cl E$ and $\cl E\otimes_{\rm s}^{\cl A_{\cl S}}\cl S \otimes_{\rm s}^{\cl A_{\cl S}}\cl E^*$ are operator systems.
A pair $(\al, \be)$ of unital completely positive maps
\[
\alpha \colon \cl E^*\otimes_{\rm s}^{\cl A_{\cl T}}\cl T \otimes_{\rm s}^{\cl A_{\cl T}}\cl E \to \cl S
\ \mbox{ and } \ 
\beta \colon \cl E\otimes_{\rm s}^{\cl A_{\cl S}}\cl S \otimes_{\rm s}^{\cl A_{\cl S}}\cl E^* \to \cl T,
\]
is called \emph{compatible} if the following hold:

\smallskip

\noindent (i)
Write $\alpha_1 \colon \cl E^* \otimes_{\rm s}^{\cl A_{\cl T}}\cl E \to \cl S$ and $\beta_1 \colon \cl E \otimes_{\rm s}^{\cl A_{\cl S}}\cl E^* \to \cl T$ for the restriction maps, obtained after taking into account
the inclusions
\[
 \cl E^* \otimes_{\rm s}^{\cl A_{\cl T}} \cl E \subseteq \cl E^*\otimes_{\rm s}^{\cl A_{\cl T}}\cl T \otimes_{\rm s}^{\cl A_{\cl T}}\cl E
\qand
\cl E \otimes_{\rm s}^{\cl A_{\cl S}}\cl E^* \subseteq \cl E\otimes_{\rm s}^{\cl A_{\cl S}}\cl S \otimes_{\rm s}^{\cl A_{\cl S}}\cl E^*,
\]
that hold by the injectivity of the symmetrisation (see Proposition \ref{p_tensmaps});
then
\begin{equation}\label{eq_abmult}
\alpha_1 (\cl E^* \otimes_{\rm s}^{\cl A_{\cl T}}\cl E) \subseteq \cl A_{\cl S}
\qand
\beta_1 (\cl E \otimes_{\rm s}^{\cl A_{\cl S}}\cl E^*) \subseteq \cl A_{\cl T}.
\end{equation}

\noindent
\item (ii) The maps $\al$ and $\be$ satisfy the conditions
\begin{align} \label{eq_alphabeta}
\al(x_1^* \otimes^{\cl A_{\cl T}} \be(x_2 \otimes^{\cl A_{\cl S}} s \otimes^{\cl A_{\cl S}} x_3^*) \otimes^{\cl A_{\cl T}} x_4) 
& = \nonumber \\
& \hspace{-2cm} = 
\al_1(x_1^* \otimes^{\cl A_{\cl T}} x_2) \cdot s \cdot \al_1(x_3^* \otimes^{\cl A_{\cl T}} x_4) \nonumber \\
\be(x_1 \otimes^{\cl A_{\cl S}} \al(x_2^* \otimes^{\cl A_{\cl T}} t \otimes^{\cl A_{\cl T}} x_3) \otimes^{\cl A_{\cl S}} x_4^*) 
& = \\
& \hspace{-2cm} =
\be_1(x_1 \otimes^{\cl A_{\cl S}} x_2^*) \cdot t \cdot \be_1(x_3 \otimes^{\cl A_{\cl S}} x_4^*), \nonumber
\end{align}
for all $x_1, x_2, x_3, x_4 \in \cl E$, $s \in \cl S$ and $t \in \cl T$.
\end{definition}

The next theorem is the main result of this section.

%%%%%%%%%%%%%%%%%%%%%%%%%%%%
\begin{theorem}\label{th_facDelta}
Let $\cl S$ and $\cl T$ be operator systems. 
The following are equivalent:
\begin{itemize}
\item[(i)] $\cl S\sim_{\Delta}\cl T$;

\item[(ii)] there exist a closed TRO $\cl M$ such that 
\[
\cl A_{\cl T} \stackrel{\pi_{\cl T}}{\simeq} [\cl M \cl M^*]
\qand
\cl A_{\cl S} \stackrel{\pi_{\cl S}}{\simeq} [\cl M^* \cl M],
\]
$\cl M$ is $\cl A_{\cl T}$-balanced $\cl T$-semi-unital and $\cl M^*$ is $\cl A_{\cl S}$-balanced $\cl S$-semi-unital, and a compatible pair $(\al, \be)$ of unital complete order isomorphisms
\[
\alpha \colon \cl M^*\otimes_{\rm s}^{\cl A_{\cl T}}\cl T \otimes_{\rm s}^{\cl A_{\cl T}}\cl M \to \cl S
\text{ and } 
\beta \colon \cl M\otimes_{\rm s}^{\cl A_{\cl S}}\cl S \otimes_{\rm s}^{\cl A_{\cl S}}\cl M^* \to \cl T,
\]
that are bimodule maps over $\cl A_{\cl S}$ and $\cl A_{\cl T}$, respectively, such that
\begin{equation}\label{eq_module}
\begin{split}
\al(x_1^* \otimes 1_{\cl T} \otimes x_2) = \pi_{\cl S}^{-1}(x_1^* x_2) \cdot 1_{\cl S} \\
\be(x_1 \otimes 1_{\cl S} \otimes x_2^*) = \pi_{\cl T}^{-1}(x_1 x_2^*) \cdot 1_{\cl T},
\end{split}
\end{equation}
for all $x_1, x_2 \in \cl M$;

\item[(iii)] there exist a non-degenerate operator space $\cl E$ that is an operator $\cl A_{\cl T}$-$\cl A_{\cl S}$-bimodule, such that $\cl E$ is $\cl A_{\cl T}$-balanced $\cl T$-semi-unital and $\cl E^*$ is $\cl A_{\cl S}$-balanced $\cl S$-semi-unital, and a compatible pair $(\al, \be)$ of unital completely positive maps
\[
\alpha \colon \cl E^*\otimes_{\rm s}^{\cl A_{\cl T}}\cl T \otimes_{\rm s}^{\cl A_{\cl T}}\cl E \to \cl S
\text{ and }
\beta \colon \cl E\otimes_{\rm s}^{\cl A_{\cl S}}\cl S \otimes_{\rm s}^{\cl A_{\cl S}}\cl E^* \to \cl T.
\]
\end{itemize}

If any (and thus all) of the items hold, then the maps $\al$ and $\be$ in items (ii) and (iii) are surjective.
Moreover, the TRO $\cl M$ in item (ii) can be chosen to be finitely generated as a module over $[\cl M\cl M^*]$.
\end{theorem}

\begin{proof}
\noindent
[(i)$\Rightarrow$(ii)]:
By \cite[Proposition 3.10]{ekt}, there exists a closed TRO $\cl M$ for the complete order embeddings $\psi_{\cl S} \colon \cl S \to \cenv(\cl S)$ and $\psi_{\cl T} \colon \cl T \to \cenv(\cl T)$ such that
\[
[\cl M^* \psi_{\cl T}(\cl T) \cl M] = \psi_{\cl S}(\cl S)
\qand
[\cl M \psi_{\cl S}(\cl S) \cl M^*] = \psi_{\cl T}(\cl T),
\]
and 
\[
\cl A_{\cl S} \stackrel{\pi_{\cl S}}{\simeq} [\cl M^* \cl M]
\qand
\cl A_{\cl T} \stackrel{\pi_{\cl T}}{\simeq}  [\cl M \cl M^*],
\]
for $\pi_{\cl S} = \psi_{\cl S}|_{\cl A_{\cl S}}$ and $\pi_{\cl T} = \psi_{\cl T}|_{\cl A_{\cl T}}$.
By construction we have
\[
\pi_{\cl S}(a) \psi_{\cl S}(s) = \psi_{\cl S}(a \cdot s), \ \ a \in \cl A_{\cl S}, s \in \cl S;
\]
That is, $(\psi_{\cl S}, \pi_{\cl S})$ is an $\cl A_{\cl S}$-representation.
Likewise $(\psi_{\cl T}, \pi_{\cl T})$ is an $\cl A_{\cl T}$-repre\-sen\-ta\-tion.

By Theorem \ref{t_trosym} and Remark \ref{r_unit}, we have complete order isomorphisms $\Phi_{\cl T}$ and $\Phi_{\cl S}$ so that the symmetrisations
\[
\cl M^*\otimes_{\rm s}^{\cl A_{\cl T}}\cl T \otimes_{\rm s}^{\cl A_{\cl T}}\cl M 
\stackrel{\Phi_{\cl T}}{\simeq} 
[\cl M^* \psi_{\cl T}(\cl T) \cl M] = \psi_{\cl S}(\cl S)
\]
and
\[
\cl M\otimes_{\rm s}^{\cl A_{\cl S}}\cl S \otimes_{\rm s}^{\cl A_{\cl S}}\cl M^*
\stackrel{\Phi_{\cl S}}{\simeq}
[\cl M \psi_{\cl S}(\cl S) \cl M^*] = \psi_{\cl T}(\cl T)
\]
are unital; a fortiori, $\cl M$ is $\cl A_{\cl T}$-balanced $\cl T$-semi-unital, and $\cl M^*$ is $\cl A_{\cl S}$-balanced $\cl S$-semi-unital.
Moreover, by Corollary \ref{c_trosymbim}, we have that the balanced symmetrisations are C*-bimodules and, by construction $(\Phi_{\cl T}, \pi_{\cl S})$ and $(\Phi_{\cl S}, \pi_{\cl T})$ are C*-representa\-tions over $\cl A_{\cl S}$ and $\cl A_{\cl T}$, respectively.

Let
\[
\al := \psi_{\cl S}^{-1} \circ \Phi_{\cl T} \colon \cl M^*\otimes_{\rm s}^{\cl A_{\cl T}}\cl T \otimes_{\rm s}^{\cl A_{\cl T}}\cl M \to \cl S
\]
and
\[
\be := \psi_{\cl T}^{-1} \circ \Phi_{\cl S} \colon \cl M\otimes_{\rm s}^{\cl A_{\cl S}}\cl S \otimes_{\rm s}^{\cl A_{\cl S}}\cl M^* \to \cl T;
\]
we have that $\al$ and $\be$ are the unital complete order isomorphisms, satisfying
\[
\al(x^*_1 \otimes^{\cl A_{\cl T}} t \otimes^{\cl A_{\cl T}} x_2) 
= 
\psi_{\cl S}^{-1}(x^*_1 \psi_{\cl T}(t) x_2)
\]
and
\[
\be(x_1 \otimes^{\cl A_{\cl S}} s \otimes^{\cl A_{\cl S}} x_2^*) 
= 
\psi_{\cl T}^{-1}(x_1 \psi_{\cl S}(s) x_2^*).
\]
In particular, 
\begin{align*}
\al(x^*_1 \otimes^{\cl A_{\cl T}} 1_{\cl T} \otimes^{\cl A_{\cl T}} x_2)
& = 
\psi_{\cl S}^{-1}(x^*_1 x_2)
=
\psi_{\cl S}^{-1}(x^*_1 x_2 \cdot \psi_{\cl S}(1_{\cl S})) \\
& =
\pi_{\cl S}^{-1}(x_1^* x_2) \cdot \psi_{\cl S}^{-1} (\psi_{\cl S}(1_{\cl S}))
=
\pi_{\cl S}^{-1}(x_1^* x_2) \cdot 1_{\cl S}.
\end{align*}
Note that, in addition, 
\[
\al_1(x^*_1 \otimes^{\cl A_{\cl T}} x_2) 
=
\al(x^*_1 \otimes^{\cl A_{\cl T}} 1_{\cl T} \otimes^{\cl A_{\cl T}} x_2)
=
\psi_{\cl S}^{-1}(x^*_1 x_2)
=
\pi_{\cl S}^{-1}(x_1^* x_2).
\]
We readily verify the analogous properties for $\be$, and thus $(\al, \be)$ satisfies (\ref{eq_module}).
It remains to show the compatibility of the pair.
Towards this end, we have
\begin{align*}
\al (x_1^* \otimes^{\cl A_{\cl T}} \be(x_2 \otimes^{\cl A_{\cl S}} s \otimes^{\cl A_{\cl S}} x_3^*) \otimes^{\cl A_{\cl T}} x_4)
& = \\
& \hspace{-5cm} =
\psi_{\cl S}^{-1}(x_1^* \psi_{\cl T} ( \psi_{\cl T}^{-1}(x_2 \psi_{\cl S}(s) x_3^*) ) x_4) 
=
\psi_{\cl S}^{-1}(x_1^* x_2 \psi_{\cl S}(s) x_3^* x_4) \\
& \hspace{-5cm} =
\pi_{\cl S}^{-1}(x_1^* x_2) \cdot \psi_{\cl S}^{-1}(\psi_{\cl S}(s)) \cdot \pi_{\cl S}^{-1}(x_3^* x_4) 
=
\pi_{\cl S}^{-1}(x_1^* x_2) \cdot s \cdot \pi_{\cl S}^{-1}(x_3^* x_4) \\
& \hspace{-5cm} =
\al_1(x^*_1 \otimes^{\cl A_{\cl T}} x_2) \cdot s \cdot \al_1(x^*_3 \otimes^{\cl A_{\cl T}} x_4),
\end{align*}
where we used that $(\psi_{\cl S}, \pi_{\cl S})$ is a bimodule isomorphism.
The dual relations with $\al$ and $\be$ interchanged are shown in the same way.

\smallskip

\noindent
[(ii)$\Rightarrow$(iii)]: 
By Remark \ref{r_unit}, the closed TRO's $\cl M$ and $\cl M^*$ are finitely generated with $[\cl M \cl M^*]$ and $[\cl M^* \cl M]$ being unital.
Hence $\cl M$ is non-degenerate, and we can set $\cl E = \cl M$.

\smallskip

\noindent
[(iii)$\Rightarrow$(i)]: 
By Theorem \ref{th_ekt}, it suffices to establish the existence of a bihomomorphism context.
Towards this end, define the trilinear maps
\[
[\cdot,\cdot,\cdot] \colon \cl E^* \times \cl T \times \cl E \longrightarrow \cl S
\ \mbox{ and } \ 
(\cdot,\cdot,\cdot) \colon \cl E \times \cl S \times \cl E^* \longrightarrow \cl T,
\]
by letting 
\[
[y^*,t,x] := \alpha(y^*\otimes t\otimes x) 
\ \mbox{ and } \ 
(y,s,x^*) := \beta(y\otimes s\otimes x^*),
\]
for $x,y\in \cl E$, $s\in \cl S$ and $t\in \cl T$. 
Since the pair $(\al, \be)$ is compatible, we have $[\cl E^*, 1_{\cl T}, \cl E] \subseteq \cl A_{\cl S}$ and $(\cl E, 1_{\cl S}, \cl E^*) \subseteq \cl A_{\cl T}$.
Moreover, if $x_1, x_2, x_3, x_4 \in \cl E$ and $s \in \cl S$, then 
\begin{align*}
[x_1^*, (x_2, s, x_3^*), x_4]
& =
\al(x_1^* \otimes^{\cl A_{\cl T}} \be(x_2 \otimes^{\cl A_{\cl S}} s \otimes^{\cl A_{\cl S}} x_3^*) \otimes^{\cl A_{\cl T}} x_4) \\
& =
\al_1(x_1^* \otimes^{\cl A_{\cl T}} x_2) \cdot s \cdot \al_1(x_3^* \otimes^{\cl A_{\cl T}} x_4) \\
& =
[x_1^*, 1_{\cl T}, x_2] \cdot s \cdot [x_3^*, 1_{\cl T}, x_4].
\end{align*}
In an analogous way, 
\[
(x_1, [x_2^*, t, x_3], x_4^*) = (x_1, 1_{\cl S}, x_2^*) \cdot t \cdot (x_3, 1_{\cl S}, x_4^*)
\]
for all $t \in \cl T$ and all $x_1, x_2, x_3, x_4 \in \cl E$.
By Proposition \ref{p_balsemi-unit}, there exist nets $((\un{y}_i)_i, (\un{x}_i)_i)$ and $((\un{w}_i)_i, (\un{z}_i)_i)$ of finitely supported columns over $\cl E$ and $\cl E^*$ respectively, such that
\[
\lim_i \un{y}_i^* \odot^{\cl A_{\cl T}} 1_{\cl T} \odot^{\cl A_{\cl T}} \un{x}_i = 1_{\cl E^* \otimes_{\rm s}^{\cl A_{\cl T}}\cl T \otimes_{\rm s}^{\cl A_{\cl T}}\cl E},
\]
and
\[
\lim_i \un{w}_i \odot^{\cl A_{\cl S}} 1_{\cl S} \odot^{\cl A_{\cl S}} \un{z}_i^* = 1_{\cl E \otimes_{\rm s}^{\cl A_{\cl S}} \cl S \otimes_{\rm s}^{\cl A_{\cl S}} \cl E^*}.
\]
An application of the (unital) maps $\al$ and $\be$ gives
\[
\lim_i [\un{y}_i^*, 1_{\cl T}, \un{x}_i] = \al(\lim_i \un{y}_i^* \odot^{\cl A_{\cl T}} 1_{\cl T} \odot^{\cl A_{\cl T}} \un{x}_i ) = 1_{\cl S}
\]
and 
\[
\lim_i [\un{w}_i, 1_{\cl S}, \un{z}_i^*] = \be(\lim_i \un{w}_i \odot^{\cl A_{\cl S}} 1_{\cl S} \odot^{\cl A_{\cl S}} \un{z}_i^*) = 1_{\cl T}.
\]
Therefore the quintuple $\big( \cl S, \cl T, \cl E, [\cdot, \cdot, \cdot], (\cdot, \cdot, \cdot) \big)$ is a bihomomorphism context for $\cl S$ and $\cl T$.

\smallskip

Finally we have already noted that bihomomorphism contexts are automatically surjective, and thus so are the maps $\al$ and $\be$. 
\end{proof}

%%%%%%%%%%%%%%%%%%%%%%%%%%%%

\end{document}